\documentclass[11pt, reqno]{amsart}
\usepackage[]{hyperref}
\usepackage{amsmath}
\usepackage{amsfonts}
\usepackage{amssymb}
\usepackage{tikz}
	\usetikzlibrary{arrows,snakes}
	\usetikzlibrary{decorations.pathreplacing} 
	\usetikzlibrary{matrix}
\usepackage{stmaryrd}
\usepackage{multirow}
\usepackage{mathtools}
\usepackage{subfigure}
\usepackage{verbatim}
\usepackage{pinlabel}
\usepackage[all,cmtip,knot]{xy}
\usepackage{enumerate, enumitem, comment}
\usepackage[normalem]{ulem}

\newtheorem{theorem}{Theorem}[section]

\newtheorem{lemma}[theorem]{Lemma}
\newtheorem{proposition}[theorem]{Proposition}

\newtheorem{corollary}[theorem]{Corollary}

\newtheorem{question}[theorem]{Question}
\newtheorem*{namedtheorem}{\theoremname}
\newcommand{\theoremname}{testing}

\theoremstyle{definition}
\newtheorem{definition}[theorem]{Definition}

\newtheorem{construction}[theorem]{Construction}
\theoremstyle{remark}
\newtheorem{remark}[theorem]{Remark}

\def\F{\mathbb{F}}
\def\bbI{\mathbb{I}}
\def\bbL{\mathbb{L}}
\def\bbM{\mathbb{M}}

\def\R{\mathbb{R}}
\def\Z{\mathbb{Z}}
\def\Q{\mathbb{Q}}

\def\cA{\mathcal{A}}
\def\A{\mathcal{A}}

\def\cH{\mathcal{H}}
\def\cI{\mathcal{I}}

\def\P{\mathcal{P}}

\def\cZ{\mathcal{Z}}

\def\bfa{\mathbf{a}}
\def\bfb{\mathbf{b}}
\def\bfs{\mathbf{s}}
\def\bft{\mathbf{t}}
\def\bfx{\mathbf{x}}
\def\bfy{\mathbf{y}}
\def\bfz{\mathbf{z}}
\def\bfI{\mathbf{I}}

\def\mfs{\mathfrak{s}}
\def\mfS{\mathfrak{S}}

\def\bfalpha{\boldsymbol{\alpha}}
\def\bfbeta{\boldsymbol{\beta}}
\def\bfrho{\boldsymbol{\rho}}

\def\CF{\widehat{\mathit{CF}}}
\def\CFhat{\CF}
\def\CFD{\widehat{\mathit{CFD}}}
\def\CFA{\widehat{\mathit{CFA}}}
\def\CFDA{\widehat{\mathit{CFDA}}}
\def\CFAA{\widehat{\mathit{CFAA}}}

\def\CFDD{\widehat{\mathit{CFDD}}}
\def\HFK{\widehat{\mathit{HFK}}}
\def\HFhat{\widehat{\mathit{HF}}}
\def\CFKi{\mathit{CFK}^{\infty}}

\def\sfMod{\mathsf{Mod}}
\def\sfF{\mathsf{F}}

\def\bot{\textup{bot}}
\def\dr{\textup{dr}}

\def\gr{\textup{gr}}
\def\Hom{\textup{Hom}}

\def\id{\textup{id}}
\def\Im{\textup{Im}}

\def\inv{\textup{inv}}
\def\textmid{\textup{mid}}
\def\min{\textup{min}}

\def\Mor{\textup{Mor}}

\def\op{\textup{op}}
\def\ref{\textup{ref}}

\def\sgn{\textup{sgn}}
\def\spinc{\textup{spin}^c}

\def\top{\textup{top}}

\def\Ind{\operatorname{Ind}}

\def\Span{\operatorname{Span}}
\def\Tr{\operatorname{Tr}}
\def\grTr{\operatorname{Tr}_{\operatorname{gr}}}

\def\alg{\cA(\cZ)}
\def\d{\partial}
\def\dd{\partial^\partial}
\def\FDA{\mathcal{F}_{\operatorname{DA}}}
\def\grunref{g'}
\def\grref{g}
\def\Rtilde{\widetilde{R}}
\def\Ztwo{\Z/2\Z}

\newcommand{\co}{\mskip0.5mu\colon\thinspace}

\def\t{\gamma} 
\def\ybottom{y^{\bot}}
\def\ytop{y^{\top}}
\def\id{\mathbb{I}}
\def\CFAAid{\CFAA(\mathbb{I}_{\cZ})}
\def\CFDDid{\CFDD(\mathbb{I}_{\cZ})}

\def\CFAAidzero{\CFAA(\mathbb{I}_{\cZ_0})}
\def\bfsbar{\overline{\bfs}}

\def\dualize{\eta}
\def\homize{\Omega}
\def\dualhomize{\Upsilon}

\author[Jennifer Hom]{Jennifer Hom}
\thanks{The first author was partially supported by NSF grant DMS-1307879.}
\address {School of Mathematics, Georgia Institute of Technology.}
\email{hom@math.gatech.edu}

\author[Tye Lidman]{Tye Lidman}
\thanks{The second author was partially supported by NSF grant DMS-0636643.}
\address {Department of Mathematics, North Carolina State University.}
\email {tlid@math.ncsu.edu}

\author[Liam Watson]{Liam Watson}
\thanks{The third author was partially supported by a Marie Curie Career Integration Grant (HFFUNDGRP)}
\address {D\'{e}partement de math\'{e}matiques, Universit\'{e} de Sherbrooke.}
\email {liam.watson@usherbrooke.ca}

\title{The Alexander module, Seifert forms, and categorification}

\specialcomment{jen}{%
  \begingroup\color{purple}  Jen\; $\blacktriangleright$ \;}{%
	\endgroup}
\specialcomment{cagri}{%
  \begingroup\color{blue}  Liam \; $\blacktriangleright$ \;}{%
  \endgroup}
\specialcomment{tye}{%
  \begingroup\color{green} Tye \; $\blacktriangleright$ \;}{%
  \endgroup}

\numberwithin{equation}{section}

\begin{document}

\begin{abstract}
We show that bordered Floer homology provides a categorification of a TQFT described by Donaldson \cite{Donaldson1999}. This, in turn, leads to a proof that both the Alexander module of a knot and the Seifert form are completely determined by Heegaard Floer theory. 
\end{abstract}

\vspace*{1.375in}

\maketitle

\hfill\begin{footnotesize}{\em Dedicated to the memory of Tim Cochran and Geoff Mess.}\end{footnotesize}

\newpage

\vspace*{\fill}
\tableofcontents
\vspace*{\fill}

\newpage

\section{Introduction}\label{sec:introduction}

Heegaard Floer homology is a TQFT-like invariant of three-manifolds. In its simplest form this is a $\Z/2\Z$-graded $\Z/2\Z$-vector space $\HFhat(Y)$ associated with a closed, connected, oriented three-manifold $Y$ \cite{OSz2004-invariants,OSz2004-properties}.  Defining Heegaard Floer homology requires a choice of Heegaard diagram $(\Sigma,\bfalpha,\bfbeta)$, where $\Sigma$ is a genus $g>0$, based Heegaard surface for $Y$ and $\bfalpha$ and $\bfbeta$ are $g$-tuples of attaching circles specifying the splitting of $Y$ along $\Sigma$.  A nullhomologous knot $K$ in $Y$ gives a refinement of Heegaard Floer homology: The knot induces a filtration on the Heegaard Floer chain complex known as the Alexander filtration and the homology of the associated graded complex gives rise to the knot Floer homology $\HFK(Y,K)$ \cite{OSz2004-knot, Rasmussen2003}. 
(When $Y$ is the three-sphere, or when the ambient manifold $Y$ is clear from the context, the simplified notation $\HFK(K)$ is standard.) 
This invariant of the knot comes equipped with a $\Z$-grading induced by the Alexander filtration.       

These homological invariants are related to familiar classical invariants in low-dimensional topology.  Given a Heegaard diagram $(\Sigma,\bfalpha, \bfbeta)$ for $Y$ define a matrix $A_{ij} = (\alpha_i \cdot \beta_j)$.  This $A$ is a presentation matrix for $H_1(Y;\Z)$ and, if $Y$ is a rational homology sphere, the determinant of $A$ is the order of $H_1(Y;\Z)$ (up to sign).  The $\Z/2\Z$-grading on $\HFhat(Y)$ arises from the signs of the intersections between the $\alpha$- and $\beta$-curves and it follows that $\chi\big(\HFhat(Y)\big) = |H_1(Y;\Z)|$.  In other words, the hat flavor of Heegaard Floer homology {\it decategorifies} to the order of $H_1(Y;\Z)$.  Similarly, if $K$ is a nullhomologous knot in $Y$, then we have 
\[\chi_{\operatorname{gr}}\big(\HFK(Y,K)\big) = \sum_i\chi\big(\HFK(Y,K,i)\big)  t^i = \Delta_K(t)\]  
where $\HFK(Y,K,i)$ denotes the knot Floer homology in Alexander grading $i$ and $\Delta_K(t)$ is the Alexander polynomial of $K$ \cite{OSz2004-knot,Rasmussen2003}.  By analogy with the case of $\chi\big(\HFhat(Y)\big)$, if $V$ is the Seifert matrix coming from a choice of Seifert surface for $K$ then $\Delta_K(t)$ is the determinant of $V - tV^T$. The matrix $V - tV^T$ is a presentation matrix for the first homology of the infinite cyclic cover of the knot complement as a $\Z[t,t^{-1}]$-module; this invariant of $K$ is known as the Alexander module, and $\Delta_K(t)$ naturally arises as  the total order as a $\Z[t,t^{-1}]$-module (recall that the Alexander module is always a torsion module; see Rolfsen \cite{Rolfsen1976}). This paper is concerned with the following question:

\begin{question}\label{QUESTION}
Can the Alexander module and the Seifert form be recovered from Heegaard Floer theory?
\end{question}

This question should be compared with \cite[Problem 1.4]{GeorgiaProceedings}, wherein Ruberman posits that such classical invariants may {\it not} be determined by knot Floer homology. It is indeed the case that the knot Floer homology group of a knot $K$ does not determine the Alexander module of $K$ in general, 
based on explicit examples. Recall that Kanenobu constructs an infinite family\footnote{The knots of this family are now commonly referred to as Kanenobu knots in the literature; see \cite{Watson2007} for generalizations. In the notation used here, the knot $K_i$ corresponds to $K_{0,i}$ from Kanenobu's original construction of a two-parameter family of knots $K_{p,q}$ where $p,q\in\Z$.} of knots $\{K_i\}_{i\in\Z}$ with identical Alexander polynomial for every $i$ but for which the Alexander module certifies that $K_i\simeq K_j$ if and only if $i=j$ \cite{Kanenobu1986}.  Moreover, it has been observed that the knot Floer homology $\HFK(K_i)$ is identical for every $i$ \cite{HW2014} (see also \cite{GW2013}). As a result one has an infinite family of examples for which information from the module structure cannot be extracted from the homology group $\HFK(K_i)$. Moreover, since the examples of Kanenobu are thin knots it follows that the filtered chain homotopy type of $\CFKi(K_i)$ -- a more general invariant determining knot Floer homology -- is independent of the integer $i$ \cite[Theorem 4]{Petkova}.  It is immediate then that more structure must be taken into consideration if one hopes to extract information about the Alexander module from Heegaard Floer homology.

The goal of this paper is to answer Question \ref{QUESTION}. To do so we work with  bordered Floer homology, a refinement of Heegaard Floer homology for three-manifolds with boundary due to Lipshitz, Ozsv\'ath, and Thurston \cite{LOT}. To a surface\footnote{In the context of bordered Heegaard Floer homology, all surfaces are (or must be) equipped with a choice of handle decomposition. This is a key ingredient in the definition of the differential graded algebra assigned to a surface; see Section \ref{sec:algebrabackground}.} $F$, bordered Floer homology assigns a differential graded algebra $\cA(F)$; to a manifold $Y$ with connected parameterized\footnote{That is, we have an orientation preserving diffeomorphism $\phi \co F \rightarrow \d Y$.} boundary $F$, the theory assigns a right $\cA_\infty$-module $\CFA(Y)$  over $\cA(F)$; and to a manifold $Y$ with connected parameterized boundary $-F$, the theory assigns a left differential graded module $\CFD(Y)$ over $\cA(F)$.  There is a suitable duality relating these two objects.  Furthermore, there are bimodules associated with manifolds $W$ with two parameterized boundary components $-F_0$ and $F_1$ \cite{LOTbimodules}.  Again, there are various flavors of these bimodules together with duality theorems relating them. 

For the purpose of this introduction we will focus on the bimodule $\CFDA(W)$, which has a left action (as a differential graded module) by $\cA(-F_1)$ and a right action (as an $\cA_\infty$-module) by $\cA(-F_0)$.
In the context of our main theorem, it is instructive to regard the bimodule $\CFDA(W)$ as a functor from a category of left modules over $\cA(-F_0)$ to a category of left modules over $\cA(-F_1)$ defined by tensoring (in an appropriate sense) by $ \CFDA(W)$ on the left.\footnote{We work with $-F_i$ for consistency with the most common conventions in bordered Floer homology.} 

\subsection{Categorifying Donaldson's TQFT}

Recall that a ($2+1$)-dimensional TQFT is a functor that assigns a vector space to a surface (which we always assume to be oriented and have positive genus), and  a linear map to a cobordism between surfaces. Our interest is in a TQFT considered by Donaldson \cite[Section 4]{Donaldson1999} which we will denote by $\FDA$. Let $F_0$ and $F_1$ be connected, oriented surfaces. If $W$ is a three-manifold with $\partial W = -F_0\amalg F_1$ and $H_1(W,\partial W;\mathbb{Z}) = \mathbb{Z}$, then \[\FDA(W)\co \FDA(F_0)\to\FDA(F_1)\] where $\FDA(F_i) = \Lambda^* H_1(F_i;\Z)$ and the map $\FDA(W)$ is determined by the kernel of the inclusion \[H_1(\partial W;\mathbb{Z}) \to H_1(W;\mathbb{Z})\] viewed as an element in  $\Lambda^*H_1(-F_0\amalg F_1)\cong(\Lambda^*H_1(F_0))^* \otimes\Lambda^*H_1(F_1)$. The map $\FDA(W)$ decomposes as a sum of homogeneous maps $\alpha_{i,W} \co \Lambda^iH_1(F_0) \to \Lambda^iH_1(F_1)$.

Given a (nullhomologous) knot $K$ in an integer homology sphere  $Y$, any choice of Seifert surface $F^\circ$ gives rise to a $0$-framing on $Y\smallsetminus\nu(K)$. Let $W$ be the result of removing a neighbourhood of the capped-off Seifert surface from the result of $0$-framed surgery along $K$. That is,  $W= (Y \smallsetminus \nu (F^\circ)) \cup (D^2 \times I)$. Let $F=F^\circ \cup D^2$ denote the capped-off Seifert surface; then $\partial W\cong -F\amalg F$.  In this setting, the functor $\FDA(W)$ will be denoted by  $\FDA(Y,K,F^\circ)$.  We set $\CFDA(Y,K,F^\circ)=\CFDA(W)$ (or, more briefly, $\CFDA(K,F^\circ)$ if the ambient manifold is clear), and suppress the required parameterizations from the notation. This bimodule decomposes as a direct sum along a $\Z$-grading from the underlying algebra, the components of which are denoted $\CFDA(Y,K,F^\circ,i)$. The (finite) non-trivial support of this algebra $\cA(F)=\bigoplus_{i\in\Z}\cA(F,i)$ is determined by the genus of the Seifert surface. Our main result is the following: 

\begin{theorem}\label{thm:main}
Let $K$ be a knot in an integer homology sphere $Y$ and let $F^\circ$ be a Seifert surface for $K$. Then for each $i \in \mathbb{Z}$ we have
a commutative diagram 
\[
\begin{tikzpicture}
  \matrix (m) [matrix of math nodes,row sep=2em,column sep=5em,minimum width=2em]
  {
     \CFDA(Y,K,F^\circ,i) & \HFK(Y,K,i) \\
    \alpha_{i,W} & a_i \\};
  \path[-latex]
    (m-1-1) edge node [right] {$K_0$} (m-2-1)
            edge node [above]   {$\mathit{HH}_*$} (m-1-2)
    (m-2-1.east|-m-2-2) edge node [above] {$(-1)^i\operatorname{Tr}$} (m-2-2)
    (m-1-2) edge node [right] {$(-1)^i\chi$}  (m-2-2);
\end{tikzpicture}
\]
where $\FDA(Y,K,F^\circ)=\bigoplus_{i\in\Z}\alpha_{i,W}$ as above and $a_i$ is the $i^{\text th}$ coefficient of the symmetrized Alexander polynomial of $K$.
\end{theorem}

In Theorem~\ref{thm:main}, the top arrow is established in \cite[Theorem 14]{LOTbimodules}, while the bottom arrow is due to Donaldson \cite{Donaldson1999}.  To establish Theorem \ref{thm:main} we need to first make the left-most arrow $K_0\co \CFDA(Y,K,F^\circ,i) \mapsto \alpha_{i,W}$ precise. To this end we will prove a more general statement, establishing that bordered Heegaard Floer homology provides a categorification of Donaldson's TQFT. We recall that, given a three-manifold $W$ with $\partial W\cong -F_0\amalg F_1$, the box tensor product allows us to regard  $\CFDA(W)$ as a functor (tensoring on the left) from the category of left $\cA(-F_0)$-modules to the category of left $\cA(-F_1)$-modules. 
The Grothendieck group of the category of $\Ztwo$-graded $\cA(-F)$-modules is isomorphic to $\Lambda^*H_1(F; \Z)$\cite[Theorem 1]{Petkovadecat}, and the following theorem describes the map of Grothendieck groups induced by $\CFDA(W)$.

\begin{theorem}\label{thm:donaldson}
Suppose that $W$ is a three-manifold with two parameterized boundary components $-F_0$ and $F_1$, and that $H_1(W, \d W) = \Z$. Then the functor determined by $\CFDA(W)$ from the category of left $\cA(-F_0)$-modules to the category of left $\cA(-F_1)$-modules induces a map of Grothendieck groups that agrees with Donaldson's TQFT. 
\end{theorem}
The proof of this theorem builds on Petkova's work in the case of a single boundary component \cite{Petkovadecat} (see Theorem \ref{thm:kernel}), and appeals to a suitable categorification of Hodge duality arising in bordered Floer theory (see Theorem \ref{thm:hodge}). The proof also requires that we identify a relative $\Ztwo$-grading on the bimodules in question and prove that our grading is invariant (in an appropriate sense). Recall that gradings in bordered Floer homology are by cosets of a non-commutative group \cite[Chapter 10]{LOT}.

\begin{theorem}\label{thm:introZ2}The bordered invariants admit a combinatorial $\Ztwo$-grading, defined for appropriate choices on a bordered Heegaard diagram, that coincides with a reduction modulo 2 of the non-commutative grading and is invariant up to homotopy. This promotes the relevant modules and bimodules to $\Ztwo$-graded objects in a manner that is compatible with the $\Ztwo$-grading in Heegaard Floer homology.  \end{theorem}

See Theorem \ref{thm:allthegradings} for a more precise statement. While this grading represents a key component of the proof of Theorem \ref{thm:donaldson} and is a new addition to the literature, in the interest of exposition we have opted to collect the details of the relative $\Ztwo$-grading in an appendix. As discussed above, tensoring with a bordered bimodule induces a functor between suitable categorifies of modules.  In \cite[Theorem 3]{LOTfaithful}, Lipshitz, Ozsv\'ath, and Thurston show that for the ungraded bordered bimodule associated to the mapping cylinder of a surface diffeomorphism, this functor decategorifies to the standard action of the mapping class group on $H_1(F; \Ztwo)$. As an application of Theorem \ref{thm:introZ2}, one may repeat their argument to obtain an analogous result for homology with integer rather than $\Ztwo$ coefficients.

\begin{corollary}\label{cor:MCGZ}
Let $F$ be a surface of genus $k$. The action of the mapping class group on the category of left $\cA(-F, 1-k)$-modules decategorifies to the standard action of the mapping class group on $H_1(F; \Z)$.\hfill$\Box$
\end{corollary}

Theorem \ref{thm:main} may be viewed as a categorification of Donaldson's TQFT in a particularly strong sense: In establishing the commutativity of the diagram we ultimately prove that  the graded trace $\grTr$ computing $\Delta_K(t)$ from $\FDA(K,F^\circ)$ is categorified by the Hochschild homology computing $\HFK(Y,K)$ from $\CFDA(K,F^\circ)$.  In other words, we prove that $\grTr \circ K_0 = \chi_{\operatorname{gr}} \circ HH_*$, as summarized by a commutative diagram:
\[\begin{tikzpicture}
  \matrix (m) [matrix of math nodes,row sep=4em,column sep=13em,minimum width=2em]
  {
     \CFDA(Y,K,F^\circ) & \HFK(Y,K) \\
     \mathcal{F}_{DA}(Y,K,F^\circ) & \Delta_K(t) \\};
  \path[-latex]
    (m-1-1) edge node [left] {Theorem \ref{thm:main}} node [right] {$K_0$}(m-2-1)
            edge node [above] {$\mathit{HH}_*$} node [below] {\cite[Theorem 14]{LOTbimodules}} (m-1-2)
    (m-2-1.east|-m-2-2) edge node [above] {$\grTr$}
            node [below] {\cite[Proposition 12]{Donaldson1999}} (m-2-2)
    (m-1-2) edge node [right] {$\chi_{\operatorname{gr}}$} node [left] {\cite[Proposition 4.2]{Rasmussen2003}}  (m-2-2);
\end{tikzpicture}\]
As an immediate consequence of this commutativity we obtain an alternate proof that knot Floer homology categorifies the Alexander polynomial.
\begin{corollary}\label{cor-alex}
The Alexander polynomial of a nullhomologous knot $K$ in a homology sphere $Y$ is the graded Euler characteristic of the knot Floer homology groups $\HFK(Y,K)$.\hfill$\Box$
\end{corollary}

 The original proof of Corollary~\ref{cor-alex} is found in \cite[Proposition 4.2(2)]{Rasmussen2003}; it is worth noting that Rasmussen establishes the result, more generally, for nullhomologous knots in any three-manifold. Rasmussen's approach makes use of Fox free calculus, the mechanics of which are closely tied to the module structure of the universal abelian cover (and hence the Alexander module). Rasmussen's approach is the only one (at least in print) that applies to knots in arbitrary three-manifolds.  An important antecedent in this context and closely related result was the proof of Corollary~\ref{cor-alex} for knots in $S^3$ by Ozsv\'ath and Szab\'o; see \cite[Theorem 1.2 and Corollary 1.3]{OSz2004-properties}.

Finally, using Theorem \ref{thm:main} we answer Question \ref{QUESTION}.

\begin{theorem}\label{thm:seifert}
Given a nullhomologous knot $K$ in a homology sphere $Y$ with a choice of Seifert surface $F^\circ$, $\CFDA(Y,K,F^\circ)$ determines the Seifert form for $F^\circ$ and the Alexander module of $K$.\end{theorem} 

\subsection{Properties of the bordered invariants}
For the remainder of this introduction, we fix the ambient three-manifold $Y=S^3$ (and drop it from the notation). Observe that $\CFDA(K,F^\circ)$ is not  an invariant of $F^\circ$ (since it depends on a choice of parameterization), and certainly not an invariant of $K$. However, in certain settings it is possible to extract strong geometric information from this bimodule. 



\begin{theorem}\label{thm:stronger}
There exists pairs $(K_1,F^\circ_1)$ and $(K_2,F^\circ_2)$ such that $K_1$ and $K_2$ cannot be distinguished by their knot Floer homology and $F_1$ and $F_2$ cannot be distinguished by their Seifert forms.  However, the bimodules $\CFDA(K_i,F^\circ_i)$ are not isomorphic as type DA bimodules for any choices of parameterizations.
\end{theorem}

Similarly, revisiting the examples of Kanenobu \cite{Kanenobu1986}, it is possible to construct an infinite family of pairs $\{(K_i,F^\circ_i)\}_{i\in\Z}$ with identical knot Floer homology and identical Alexander module\footnote{Kanenobu gives an explicit calculation of the Alexander module for his two parameter family of knots $\{K_{p,q}\}_{p,q\in\Z}$. This second application of Kanenobu's construction is not in contradiction with our first: For this application consider the one-parameter  family $K_i=K_{i,i}$.} that are separated by $\CFDA(K_i,F^\circ_i)$; that is, for $i \neq j$, the bimodules $\CFDA(K_i,F^\circ_i)$ and $\CFDA(K_j,F^\circ_j)$ are distinct for all choices of parameterization. This follows readily from the fact that the $K_i$ (with unique minimal genus Seifert surfaces $F^\circ_i$) are distinct fibered knots for $i\in\Z$. These examples and Theorem \ref{thm:stronger} suggest that the bimodules  provided by bordered Floer homology are particularly strong invariants. On the other hand, it is not the case that $\CFDA(\cdot,\cdot)$ gives rise to a complete invariant in general. 

\begin{theorem}\label{thm:same}
There exist infinitely many pairs $(K_i,F^\circ_i)$, with $K_i \not \simeq K_j$ such that the bimodules $\CFDA(K_i,F^\circ_i)$ and $\CFDA(K_j,F^\circ_j)$ are isomorphic for some choice of parametrization. 
\end{theorem}

\subsection{Related work} 
In addition to Petkova's work  \cite{Petkovadecat} already mentioned above, we note some other related recent work. Tian \cite{Tian2012,Tian2013} gives a categorification of the Burau representation using a different framework than the one considered here. In a different direction, sutured Floer homology provides an alternate extension of Heegaard Floer homology to manifolds with boundary, and the decategorification of this theory has been studied by Friedl, Juh\'asz, and Rasmussen \cite{FJR2011}. Since Hedden, Juh\'asz, and Sarkar have applied sutured Floer homology to distinguishing Seifert surfaces \cite{HJS2013}, it seems that similar applications of bordered Floer homology might be possible in light of Theorem \ref{thm:stronger}. We leave this problem as motivation to search for effective methods for computing and studying the bordered invariant $\CFDA(K,F)$ given a Seifert surface $F$ for a knot $K$.

\subsection{Outline} 
We first review the relevant classical objects in Section \ref{sec:linearalgebra}, namely the Alexander module, Seifert forms, and Donaldson's TQFT. We then introduce the required elements of bordered Floer homology and define our $\Ztwo$-grading in Section \ref{sec:background} before turning to the proofs of Theorem \ref{thm:main}, Theorem \ref{thm:donaldson}, and Theorem \ref{thm:seifert} in Section \ref{sec:decategorification}. The constructions establishing Theorems \ref{thm:stronger} and \ref{thm:same} are part of Section \ref{sec:background} and a complete example recovering the Alexander module and the Seifert form (applying Theorem \ref{thm:seifert}) is included at the end of Section \ref{sec:decategorification}. The technical details regarding the $\Ztwo$-gradings, representing the bulk of the paper, are collected in Appendix \ref{sec:gradings}.

\subsection*{Acknowledgments}The authors thank Sam Lewallen for many enthusiastic conversations about this work, particularly in the early days of the project. This paper also benefited from conversations with Robert Lipshitz, Allison Moore, and Tim Perutz. We gratefully acknowledge the American Institute of Mathematics (AIM) for providing a fantastic working environment as part of a SQuaRE program where some of this work was carried out.


\section{TQFTs, the Alexander module, and the Seifert form}\label{sec:linearalgebra}
This section reviews some of the necessary perspectives on linear algebra, the Alexander module, and the Seifert form. A good reference for this material is \cite{Rolfsen1976}, though for  certain aspects we draw heavily on \cite{Donaldson1999}. When computing singular homology we will always work with $\Z$ coefficients and therefore omit this choice from the notation.  Throughout this paper, for convenience, all surfaces will have positive genus.  

\subsection{Multilinear algebra and Donaldson's TQFT}
Let $G$ be a finitely-generated, free abelian group and fix a subgroup $H \subset G$ of rank $n$.  It will be convenient to record $H$ by its Pl\"ucker point $|H| \in \Lambda^n(G)$ (which we will only define up to sign) as follows.  Choose a basis $h_1,\ldots,h_n$ for $H$.  We define $|H| = h_1 \wedge \ldots \wedge h_n$, which is independent of the choice of basis, up to sign.  While not every element of $\Lambda^n(G)$ arises in this way, any non-zero element of the form $g_1 \wedge \ldots \wedge g_n$ (that is, any decomposable element) specifies a subgroup.  Recall that if $\alpha \in \Lambda^*(G_1)^* \otimes \Lambda^*(G_2)$, then we can think of $\alpha$ as in fact specifying a homomorphism from $\Lambda^*(G_1)$ to $\Lambda^*(G_2)$.

Suppose that $W$ is a compact, connected, oriented three-manifold with boundary $\partial W=-F_0\amalg F_1$.  We restrict attention to the case that $i_*\co H_1(\partial W) \to H_1(W)$ is surjective.\footnote{This condition is satisfied if $H_1(W,\d W) = \mathbb{Z}$, unless some $F_i=\emptyset$, in which case $H_1(W,\d W) = 0$ is sufficient.} 
Following Donaldson \cite{Donaldson1999}, this setup gives rise to a $(2+1)$-dimensional topological quantum field theory which we will denote by $\FDA$.  In particular, to a closed, oriented surface $F$, we assign $\FDA(F)=\Lambda^*H_1(F)$;  and to a cobordism $W\co F_0 \to F_1$, we assign the element 
\[ \FDA(W) = | \ker i_* | \in \Lambda^* H_1(-F_0 \amalg F_1) \subset \Hom(\Lambda^* H_1(F_0),\Lambda^*H_1(F_1)), \]
which is defined up to sign.  For the moment, following Donaldson, we will ignore the sign ambiguity.   
Under this identification, given cobordisms $W_0\co F_0 \to F_1$ and $W_1\co F_1 \to F_2$, we have that 
\[
\FDA(W_1) \circ \FDA(W_0) = \FDA(W_0 \cup_{F_1} W_1).
\]  
Ultimately, we will show that bordered Floer homology provides a natural categorification of this TQFT (cf. Theorem \ref{thm:donaldson}).  In doing so, a choice of signs will be made using the bordered structure. This resolves the ambiguity in Donaldson's construction above in a manner consistent with the choices in Heegaard Floer theory (see, in particular, the conventions of Proposition \ref{prp:conv}, below).  

Let $V$ be a finite-dimensional, oriented vector space of dimension $n$.  For each $0 \leq j \leq n$ and $v \in \Lambda^j V$, we can define a dual vector in $(\Lambda^{n-j}V)^*$ by $x \mapsto \star(x \wedge v)$, and this identification establishes an isomorphism that is essentially Hodge duality.  Given $v \in \Lambda^j V$, we will use $\dualize_V(v)$ to denote the associated dual vector in $(\Lambda^{n-j} V )^*$.  When the vector space $V$ is clear, we will omit the subscript from the map $\dualize$.  Consequently, we obtain a natural identification of $\Lambda^* V \otimes \Lambda^* V'$ with $\Hom(\Lambda^*V,\Lambda^*V')$.  The discussion translates to finitely-generated free abelian groups $L$ with a choice of ordered basis, where we induce a lattice structure on $L$ by choosing the basis to be orthonormal.  From this, we obtain a volume form.  

In the present setting, $V$ will be $H_1(F)$ for a given surface $F$ of genus $k>0$.  Our surfaces will come equipped with a distinguished ordered basis $x_1,\ldots,x_{2k}$ for $H_1(F)$ corresponding to a handle decomposition for $F$; the associated basis for $H_1(-F)$ will be precisely $-x_1,\ldots,-x_{2k}$.  Observe that the two ordered bases induce the same orientation on $H_1(F) = H_1(-F)$.  Thus, $\dualize_{H_1(F)} = \dualize_{H_1(-F)}$.  We will still keep track of orientations as this will help for clarity in certain contexts.       

\subsection{The Alexander module}
Let $Y$ be an integer homology sphere and consider a knot $K \subset Y$. Denote by $X_K$ the exterior of $K$, that is, $X_K=Y\smallsetminus\nu (K)$ where $\nu (K)$ is an open tubular neighborhood of the knot.  The universal abelian cover  $\widetilde{X}_K$ is induced by the abelianization  $\pi_1(X_K) \to H_1(X_K) \cong \Z$. This cyclic group is generated by a meridian for the knot, and the corresponding deck transformations on $\widetilde{X}_K$  give rise to a natural module structure on $H_1(\widetilde{X}_K)$. The Alexander module of $K$ is the $\Z[t,t^{-1}]$-module $H_1(\widetilde{X}_K)$.  Note that this module may be presented by a square matrix the determinant of which, denoted $\Delta_K(t)$, generates the order ideal (also called the Alexander ideal). The generator $\Delta_K(t)\in\Z[t,t^{-1}]$ of this principal ideal, well defined up to units in $\Z[t,t^{-1}]$, is the Alexander polynomial of $K$.


With only minor adjustments, we can just as easily work with the closed manifold $Y_0$ resulting from $0$-framed surgery along $K$.  Indeed, given a Seifert surface $F^\circ$ for $K$ the closed surface $F=F^\circ\cup D^2$, where $D^2$ is a meridional disk for the surgery solid torus, generates $H_2(Y_0)$. It is straightforward to check that $H_1(\widetilde{Y}_0)\cong H_1(\widetilde{X}_K)$ as groups; that these are isomorphic as modules follows from the fact that the knot meridian and its image in $Y_0$ are (geometrically) dual to $F^\circ$ and $F$, respectively. That is, the deck transformations on the universal abelian cover $\widetilde{Y}_0$ are generated by the image of a meridian for $K$ in $H_1(Y_0)$. As a result, given $K\subset Y$, we will treat $H_1(\widetilde{Y}_0)$ as the Alexander module of $K$ and write $\Delta_K(t)=\Delta_{Y_0}(t)$ for the corresponding Alexander polynomial. We consider $H_1(\widetilde{Y}_0)$ rather than $H_1(\widetilde{X}_K)$ because we prefer to work with the closed surface $F$ rather than the surface with boundary $F^\circ$.

Let $F^\circ$ be a genus $k$ Seifert surface for $K$ and, setting $\nu (F^\circ) = F^\circ \times [0,1]$, let $W =  (S^3 \smallsetminus \nu (F^\circ)) \cup D^2 \times I$. From the preceding paragraph, notice that $W$ is the complement of a  neighborhood of the capped-off Seifert surface $F=F^\circ\cup D^2$ in $Y_0$.  In particular, $H_1(W,\partial W) \cong \Z$ as in Donaldson's setup. Since $H_1(\partial W) \cong H_1(F) \oplus H_1(F)$, let $\Gamma \subset H_1(F) \oplus H_1(F)$ be the kernel of $i_*\co H_1(\partial W) \to H_1(W)$.  Note that the rank of $\Gamma$ is $2k$.  We have the $\mathbb{Z}[t,t^{-1}]$-module presentation 
\begin{equation}\label{eqn:infinitecycliccover}
H_1(\widetilde{Y}_0) \cong H_1(F) \otimes \Z[t,t^{-1}]/\sim
\end{equation}
where $x \otimes p(t) \sim -y \otimes tp(t)$ if $(x,y) \in \Gamma$, since $x \otimes 1 + y \otimes t$ corresponds to the inclusion of the boundary of $W$ which is zero if $(x,y)$ is in the kernel of $i_*$.  We conclude that the Pl\"ucker point $|\Gamma|$ determines the Alexander module of $K$, though note that $|\Gamma|$ is certainly not an invariant of $K$.  Rearranging further, 
\begin{align*}
\Lambda^{2k}(H_1(-F) \oplus H_1(F)) & \subset \bigoplus^{k}_{i = -k} \Lambda^{k-i} H_1(-F) \otimes \Lambda^{k+i} H_1(F) \\
										  & \cong \bigoplus^{k}_{i=-k} (\Lambda^{k+i} H_1(F))^* \otimes \Lambda^{k+i} H_1(F)
\end{align*}
so that $|\Gamma|$ in fact corresponds to a family of maps $\alpha_{i,W} \co \Lambda^{k+i} H_1(F) \to \Lambda^{k+i} H_1(F)$ (the identifications above are made explicit in Section \ref{sec:decategorification}).  As a result, 
\begin{equation}\label{eqn:alexanderpoly}
\Delta_K(t) = \sum^{k}_{i = -k} (-1)^i t^i\Tr(\alpha_{i,W})
\end{equation}
as observed by Donaldson. In terms of the symmetrized Alexander polynomial:

\begin{proposition}[Donaldson {\cite[Proposition 12]{Donaldson1999}}]\label{prp:conv} The graded trace $\grTr$ of the map $\FDA(W)$ determines the Alexander polynomial of $K$. That is, \[a_i = (-1)^i\operatorname{Tr}\big(\alpha_{i,W}\co \Lambda^{(i)}H_1(F)\to \Lambda^{(i)}H_1(F)\big)\] where $\Delta_K(t) = a_0+\sum_i a_i(t^i+t^{-i})$ and  $\Lambda^{(i)}H_1(F)$ represents one of two isomorphic copies of $\Lambda^{k\pm i}(H_1(F))$. \end{proposition}

Finally, we remark that Donaldson opts to work with $\Lambda^*H^1(F)$ instead of  $\Lambda^*H_1(F)$ but notes that this choice is a matter of convenience; the latter is more convenient for our purposes. 

\subsection{The Seifert form}
Given a genus $k>0$ Seifert surface $F^\circ$ for $K$, there are two natural bilinear forms we can place on $H_1(F^\circ)$.  The first is simply the intersection form on the surface, $\omega_{F^\circ}$.  The second is the {\em Seifert form}, defined by 
\[
V\co H_1(F^\circ) \times H_1(F^\circ) \to \Z, \quad V(x,y) = lk(x,y^+), 
\]
where $y^+$ denotes the positive push-off of $y$ in $Y$, determined by the orientation of $F^\circ$.  Note that $V$ is an invariant of the pair $(K,F^\circ)$, not $K$.  Given a basis for $H_1(F^\circ)$, then $V$ determines a $2k \times 2k$-matrix, called the {\em Seifert matrix}.  By abuse of notation, we will also write this as $V$.  

Given a basis for $H_1(F^\circ)$, the corresponding matrix $V - tV^T$ gives a presentation matrix for $H_1(\widetilde{X}_K)$ over $\Z[t,t^{-1}]$.  On its own, the Alexander module is not able to recover the Seifert form for a given Seifert surface.  However, we will show that given a presentation matrix of the Alexander module in an appropriate basis for $H_1(F^\circ)$, one can reconstruct the Seifert matrix in that basis, and thus the Seifert form.    

\subsection{Reconstructing the Seifert form from the Alexander module}
Fix a knot $K$ in an integer homology sphere $Y$, a genus $k$ Seifert surface $F^\circ$ for $K$, and let $\omega$ denote the intersection form for $H_1(F^\circ)$, which is symplectic.  Given a basis $\mathcal{B} = \{e_i\}$ for $H_1(F^\circ)$, we have the intersection matrix (where intersections are defined using a right-hand-rule), which we also write as $\omega$.  Finally, let $\Gamma_W$ be the subgroup of $H_1(F) \oplus H_1(F)$ determined by $W = (Y \smallsetminus \nu (F)) \cup D^2 \times I$ as described above.  

Following Donaldson, we fix a basis $\mathcal{C} = \{(\sigma_1,\tau_1),\ldots,(\sigma_{2k},\tau_{2k})\}$ for $\Gamma_W$.  We write $\sigma_i = \sum_j \sigma_{ij} e_j$ and $\tau_i = \sum_j \tau_{ij} e_j$.  We may rephrase \eqref{eqn:infinitecycliccover} by saying that $(\sigma_{ij} + t \tau_{ij})$ gives a presentation matrix for $H_1(\widetilde{Y}_0)$.  For what follows, given a matrix $A + tB$, where $A, B$ have integer entries, we write $A_i, B_i$ for the rows of $A, B$ respectively.  When referring to presentation matrices of the Alexander module, we will mean explicitly that the span of the $(A_i, B_i)$ is $\Gamma_W$, not that the modules presented are isomorphic.  It also follows from the constructions that the Seifert form $V$ has the property that $V - tV^T$ is a presentation matrix for $H_1(\widetilde{Y}_0)$ in this regard with respect to the basis $\mathcal{B}$.  In particular, we observe that the rowspace of $V - tV^T$ is necessarily the same as that of any presentation matrix $A + tB$ for the Alexander module with respect to the basis $\mathcal{B}$.

\begin{lemma}
\label{lem:V=X}
Suppose that $A + tB$ gives a presentation matrix for the Alexander module with respect to the basis $\mathcal{B}$ as described above.  Further, suppose that $A + B = -\omega$.  Then $A$ is the Seifert matrix for $K$ with respect to the basis $\mathcal{B}$.  
\end{lemma}

\begin{proof}
Let $V$ be the Seifert matrix for $F^\circ$ with respect to the basis $\mathcal{B}$. It is well-known that $V - V^T = -\omega$; see, for example, \cite[Chapter 8]{Rolfsen1976}, noting that Rolfsen uses left-handed coordinates, while we use right-handed coordinates. As discussed, the row space of $V - t V^T$ is the same as the row space of $A + tB$.  We first consider the case where $\omega = \oplus_{k}\left(\begin{smallmatrix}  0 & 1 \\ \text{-}1 & 0 \end{smallmatrix}\right)$.  For compactness, we write $\sigma$ for the permutation of $\{1,\ldots,2k\}$ that switches $2i$ and $2i-1$ for each $i$.  Since $A + B = -\omega$, we have that in each row of $A + tB$, there is exactly one element of that row which is not a multiple of $(1-t)$, namely the entry $(A+tB)_{i,\sigma(i)}$.  The same statements also hold for $V - tV^T$.  

Since the rowspaces of $V - tV^T$ and $A + tB$ agree, we may write each row $A_i + tB_i$ as a linear combination of the rows of $V - tV^T$.  We claim that $A_i + tB_i$ must equal the $i$th row of $V - tV^T$.  Indeed, writing $A_i + tB_i = \sum a_j (V_j - t(V^T)_j)$, if $a_j \neq 0$ for $j \neq i$ it follows immediately from the observation in the previous paragraph that  $(A+tB)_{i,\sigma(j)}$ cannot be a multiple of $(1-t)$.   Again, since $A+B = V - V^T = -\omega$, we see that we must have that the rows $A_i + t B_i$ and $V_i - t (V^T)_i$ agree for each $i$.  In particular, we have $A + tB = V - tV^T$.  This clearly gives the desired result in the case that $\omega$ is in standard form.      

Now, we work with general $\omega$.  Choose $P \in \textup{SL}_{2k}(\Z)$ such that $P\omega P^T = \oplus_{k}\left( \begin{smallmatrix} 0 & 1 \\ \text{-}1 & 0 \end{smallmatrix}\right)$, i.e., we apply a change of basis to obtain the standard symplectic matrix, which we denote by $\Omega$.  Note that we have $PVP^T - tPV^TP^T$ and $PAP^{T} + tPBP^{T}$ have the same rowspace.  Further, $PVP^{T}$ is the Seifert matrix for $K$ in the basis $P(\mathcal{B})$.  We also obtain that $PAP^{T} + tPBP^{T}$ is a presentation matrix for the Alexander module with respect to the basis $P(\mathcal{B})$.  Since $PAP^{T} + PBP^{T} = -\Omega$, we see that $PAP^{T} = PVP^{T}$, by the above argument.  Therefore, $A = V$ and the proof is complete.    
\end{proof}

We can apply Lemma~\ref{lem:V=X} to read off the Seifert form from any presentation matrix for the Alexander module.   

\begin{lemma}\label{lem:intersectiondetermines}
Let $K$ be a knot in an integer homology sphere $Y$, and fix a genus $k$ Seifert surface $F^\circ$ for $K$, together with a basis $\mathcal{B}$ for $H_1(F)$ and let $\omega$ denote the intersection matrix for $H_1(F)$ in this basis.  Let $A + tB$ be a presentation for the Alexander module in terms of the basis $\mathcal{B}$ with the conditions given above (i.e., the span of the $(A_i,B_i)$ is precisely $\Gamma_W$). Then $-\omega (A+B)^{-1} A$ is the Seifert matrix for $K$ with respect to $\mathcal{B}$.
\end{lemma}

\begin{proof}
Let $V$ be the Seifert matrix for $K$ with respect to $\mathcal{B}$. As discussed, we have that $V - V^T = -\omega$ and that $V - t V^T$ is a presentation for the Alexander module.

Since $\omega, (A+B) \in  \textup{GL}_{2g}(\mathbb{Z})$ (recall that $(A+B)$ is a presentation matrix for the first homology of the integer homology sphere $Y$),  we have that $-\omega (A+B)^{-1}(A+tB)$ and $(A+tB)$ necessarily have the same rowspace.  Denote this new presentation by $Y + tZ$, where
\[ Y = -\omega (A + B)^{-1} A \qquad \qquad Z = -\omega (A + B)^{-1} B. \]
Note that $Y + Z =-\omega$.  The result now follows from Lemma~\ref{lem:V=X}.  
\end{proof}

\begin{proposition}\label{prop:intersectiondetermines}
Let $K$ be a knot in an integer homology sphere $Y$, and fix a Seifert surface $F^\circ$ for $K$.  Given the Pl\"ucker point $|\Gamma_W|$, up to sign, and the intersection form on $H_1(F)$, we may determine the Seifert form of $F^\circ$.
\end{proposition}
\begin{proof}
Choose a basis $\mathcal{B}$ for $H_1(F)$.  The construction of Donaldson described throughout this section describes how to turn the Pl\"ucker point into a presentation matrix for the Alexander module in terms of $\mathcal{B}$, where the result may depend on the choice of basis $\mathcal{C}$ for $\Gamma_W$.  Note that changing the Pl\"ucker point by an overall sign does not change the rowspace of the presentation matrix.  Since we have the intersection form on $H_1(F)$, we can construct the intersection matrix $\omega$ in terms of the basis $\mathcal{B}$.  The result now follows from Lemma~\ref{lem:intersectiondetermines}.  
\end{proof}

\section{Background on the bordered invariants}\label{sec:background}

This section has two principal aims: we review the requisite background from bordered Heegaard Floer homology and give a $\Z/2\Z$-grading on the relevant objects. The identification of this grading with one in the literature (namely, Petkova's \cite{Petkovadecat} specialization of the non-commutative grading of \cite{LOT}), as well as the proof of invariance, is deferred to Appendix \ref{sec:gradings}. In \cite{Petkovadecat}, Petkova also defines a $\Ztwo$-grading on $\cA(\cZ)$ and $\CFD$ which agrees with the definitions below, and also outlines a proof that the two different gradings agree. We extend these definitions to $\CFA$ and bimodules, as well as prove that they behave well with respect to pairing.  

The bordered invariants are defined over the field of two elements, which we denote by $\F$. We use $\Ztwo$ to refer to the grading on these objects.

\subsection{The algebra}\label{sec:algebrabackground} Following \cite{LOT,LOTnotes}, we recall the definition of the algebra associated with a pointed matched circle. We let $[n] = \{ 1, \dots, n \}$ and $[n, n+k] = \{ n, n+1, \dots, n+k \}$.

\begin{definition}
\label{defn:pmc}
A \emph{pointed matched circle} $\cZ$ is a quadruple $(Z, z, \bfa, M)$ consisting of an oriented circle $Z$ together with a basepoint $z \in Z$, a set of $4k$ distinct oriented points $\bfa \subset Z$, and a $2$-to-$1$ map $M\co \bfa \rightarrow [2k]$. We require further that for each $i \in [2k]$, $M^{-1}(i)$ is an oriented $S^0$ (i.e., one positively oriented and one negatively oriented point), and that the result of performing surgery on $(Z, \bfa)$ according to $M$ is connected.
\end{definition}

\noindent The map $M \co \bfa \rightarrow [2k]$ is called the \emph{matching} and we say that the two elements in $M^{-1}(i)$ are \emph{matched} by $M$. Note that attaching two-dimensional 1-handles to $D^2$ along $Z = \partial D^2$ at $\bfa$ specified by $M$ always yields a genus $k$ surface with a single boundary component.  Further, we observe that the orientation of $Z$ induces an order $\lessdot$ on $Z \setminus z$.  

To a pointed matched circle $\cZ$, we associate a closed, oriented surface $F(\cZ)$ (or just $F$ when the pointed matched circle is clear). The surface $F$ is obtained by taking a disk with boundary $Z$ (preserving orientation), attaching two-dimensional $1$-handles according to $M$, and filling in the resulting boundary component with a disk.  

Given a pointed matched circle $\cZ=(Z, z, \bfa, M)$, we can form the pointed matched circle $-\cZ=(-Z, z, -\bfa, M \circ R |_{-\bfa})$, where $-\bfa$ denotes the set $\bfa$ with the orientation of each element reversed and $R$ denotes the (orientation-reversing) identity map $-Z \rightarrow Z$. Given two pointed matched circles $\cZ_1=(Z_1, z_1, \bfa_1, M_1)$ and $\cZ_2=(Z_2, z_2, \bfa_2, M_2)$ where $|\bfa_i| = 4k_i$, we can form the connected sum $\cZ_1 \# \cZ_2 = (Z_1 \# Z_2, z, \bfa_1 \sqcup \bfa_2, M)$, where
\begin{itemize}
	\item $Z_1 \# Z_2$ denotes the connected sum of $Z_1$ and $Z_2$, taken near $z_1$ and $z_2$,
	\item $z$ is a point on $Z_1 \# Z_2$ such that $a_1 \lessdot a_2$ for all $a_1\in \bfa_1$ and $a_2 \in \bfa_2$ where $\lessdot$ is the order on $Z_1 \# Z_2 \smallsetminus z$ induced by the orientation of $Z_1 \# Z_2$, and
	\item $M\co \bfa_1 \sqcup \bfa_2 \rightarrow [2(k_1+k_2)]$ is defined to be
	\begin{align*}
		M (a) &= \left\{
		\begin{array}{ll}
			M_1(a) & \text{if } a \in \bfa_1 \\
			M_2(a)+2k_1 & \text{otherwise}.
		\end{array} \right.
	\end{align*}
\end{itemize}
Geometrically, this corresponds to ``stacking $Z_2$ on top of $Z_1$''.  Note that the connected sum operation is not commutative. See Figure \ref{fig:PMCoperations} for examples. It is straightforward to verify that $F(-\cZ)=-F(\cZ)$ and $F(\cZ_1 \# \cZ_2)=F(\cZ_1) \# F(\cZ_2)$.

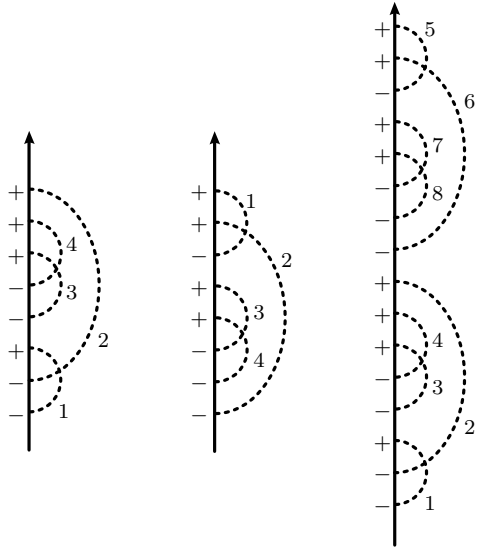
\begin{figure}[htb!]
\centering
\labellist
\scriptsize \pinlabel $1$ at 10 34
\pinlabel $2$ at 20 52
\pinlabel $3$ at 12 64
\pinlabel $4$ at 12 76
\pinlabel $-$ at -2 33
\pinlabel $-$ at -2 41
\pinlabel $+$ at -2 49
\pinlabel $-$ at -2 57
\pinlabel $-$ at -2 65
\pinlabel $+$ at -2 73
\pinlabel $+$ at -2 81
\pinlabel $+$ at -2 89

\pinlabel $1$ at 57 87
\pinlabel $2$ at 66 72
\pinlabel $3$ at 59 59
\pinlabel $4$ at 59 47
\pinlabel $-$ at 44 33
\pinlabel $-$ at 44 41
\pinlabel $-$ at 44 49
\pinlabel $+$ at 44 57
\pinlabel $+$ at 44 65
\pinlabel $-$ at 44 73
\pinlabel $+$ at 44 81
\pinlabel $+$ at 44 89

\pinlabel $1$ at 102 11
\pinlabel $2$ at 112 30
\pinlabel $3$ at 104 40
\pinlabel $4$ at 104 52
\pinlabel $5$ at 102 130
\pinlabel $6$ at 112 112
\pinlabel $7$ at 104 101
\pinlabel $8$ at 104 89
\pinlabel $-$ at 90 10
\pinlabel $-$ at 90 18
\pinlabel $+$ at 90 26
\pinlabel $-$ at 90 34
\pinlabel $-$ at 90 42
\pinlabel $+$ at 90 50
\pinlabel $+$ at 90 58
\pinlabel $+$ at 90 66
\pinlabel $-$ at 90 74 
\pinlabel $-$ at 90 82
\pinlabel $-$ at 90 90
\pinlabel $+$ at 90 98
\pinlabel $+$ at 90 106
\pinlabel $-$ at 90 114
\pinlabel $+$ at 90 122
\pinlabel $+$ at 90 130

\endlabellist
\includegraphics[scale=1.5]{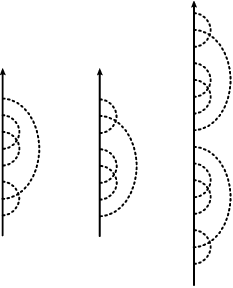}
\caption{Left, a pointed matched circle $\cZ$, cut open at $z$. The dashed lines indicate the matching. Center, $-\cZ$. Right, $\cZ \# -\cZ$.}
\label{fig:PMCoperations}
\end{figure}

Fix a positive integer $n$. A \emph{partial permutation} $(S, T, \sigma)$ is a bijection $\sigma \co S \rightarrow T$ between two subsets $S, T \subset [n]$. One can view a partial permutation graphically in a strands diagram as follows. Consider two columns of $n$ vertices, and place an edge between the $i^\textup{th}$ vertex in the first column and the $j^\textup{th}$ vertex in the second column if $i \in S$, $j \in T$ and $\sigma(i)=j$. Below, we will define an algebra $\cA$ in terms of partial permutations; for a geometric treatment of the algebra in terms of strands diagrams see \cite[Section 3.1.2]{LOT}.

\begin{figure}[htb!]
\centering
\includegraphics[scale=1.2]{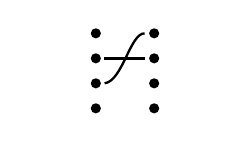}
\caption{A graphical picture of a partial permutation $\sigma \co \{2, 3\} \rightarrow \{ 3, 4 \}$ where $\sigma(2)=4$ and $\sigma(3)=3$.}
\label{fig:partperm}
\end{figure}

Let $(S,T,\sigma)$ be a partial permutation.  An \emph{inversion} is a pair $(i, j)$ (for $i, j \in S$) such that $i<j$ and $\sigma(j) < \sigma(i)$. Denote the number of inversions of $\sigma$ by $\textup{inv}(\sigma)$. Recall that if $\sigma$ is a permutation (rather than merely a partial permutation), then $\inv(\sigma)$ is congruent modulo $2$ to the number of terms in a factorization of $\sigma$ into transpositions. A partial permutation is \emph{upward-veering} if $\sigma(i) \geq i$ for all $i \in S$. Let $\cA(n)$ denote the vector space over $\F$ generated by all upward-veering partial permutations. There is a product and differential on $\cA(n)$, promoting it to a differential algebra \cite[Lemma 3.1]{LOT}.  The multiplication is given by
\[ (S, T, \phi) \cdot (U, V, \psi) = \left\{
	\begin{array}{lll}
		0 & \text{if } T \neq U\text{;}\\
		0 & \text{if } \inv(\psi \circ \phi) \neq \inv(\psi) + \inv(\phi)\text{; and}\\
		(S, V, \psi \circ \phi) & \text{otherwise}.
	\end{array} \right.
\]
Note that if $\inv(\psi)$ or $\inv(\phi)$ is zero, then we have $\inv(\psi \circ \phi) = \inv(\psi) + \inv(\phi)$.
The differential on $\cA(n)$ is
 \[ d(S, T, \phi) = \sum_{\substack{\textup{inversions } (i, j) \\ \inv( \phi \circ \tau_{i,j})=\inv(\phi)-1}}(S, T,  \phi \circ  \tau_{i,j}), \]
where $\tau_{i,j}$ is the transposition interchanging $i$ and $j$. 

Our interest is in a particular subalgebra  $\cA(\cZ) \subset \cA(4k)$ associated with a pointed matched circle $\cZ=(Z, z, \bfa, M)$. Identify $\bfa$ with $[4k]$ by the ordering $\lessdot$. An upward-veering partial permutation $(S, T, \phi)$ is \emph{$M$-admissible} if $M|_S$ and $M|_T$ are injective.

Let $\textup{Fix}(\phi)$ denote the fixed points of $\phi$. For each $U \subset\textup{Fix}(\phi)$, define a new element $\big( (S \setminus U) \cup U_M, (T \setminus U) \cup U_M, \phi_U \big)$ where $U_M= M^{-1} ( M(U)) \setminus U$, $\phi_U|_{S\setminus U}=\phi|_{S\setminus U}$, and $\phi_U|_{U_M}=\textup{id}_{U_M}$. For $(S, T, \phi)$ $M$-admissible, define
\[a(S, T, \phi) = \sum_{U \subset \textup{Fix}(\phi)} \big( (S \setminus U) \cup U_M, (T \setminus U) \cup U_M, \phi_U \big), \]
that is, for each subset  $U$ of the fixed points of $\phi$, replace $U$ with the points $U_M$ that they are matched with and define $\phi_U$ to agree with $\phi$ away from the fixed points $U$  and to be constant on $U_M$.
Then $\cA(\cZ)$ is the subalgebra of $\cA(4k)$ generated by $\{a(S, T, \phi) \ | \ (S, T, \phi) \ M \textup{-admissible} \}$. The algebra $\cA(\cZ)$  decomposes as
\[ \cA(\cZ) = \bigoplus_{i=-k}^k \cA(\cZ, i) \]
where $i=|S|-k$. We refer to $i$ as the \emph{strands grading} of $\cA(\cZ, i)$.

Note that if $(S, T, \phi) \cdot (U, V, \psi) = (S, V, \psi \circ \phi)$, then $a (S, T, \phi) \cdot a(U, V, \psi) = a(S, V, \psi \circ \phi)$. Also note that any element $a(S, S, \textup{id}_S) \in \cA(\cZ)$ is an idempotent, i.e., an element which squares to itself. A \emph{minimal idempotent} is an idempotent  of the form $a(S, S, \textup{id}_S)$, where the partial permutation $(S, S, \textup{id}_S)$ is $M$-admissible. Define the element $\mathbf{I}$ (respectively $\bfI_i$) to be the sum over all distinct minimal idempotents in $\cA(\cZ)$ (respectively $\cA(\cZ,i)$). Let $\mathcal{I}$ (respectively $\cI_i$) denote the ring of idempotents in $\cA(\cZ)$ (respectively $\cA(\cZ,i)$. We will also use the notation $\cI(\cZ)$ and $\cI(\cZ,i)$ when we want to emphasize the choice of pointed matched circle.  Given $\bfs \subset [2k]$, we write $I(\bfs)$ to denote the idempotent $a(S, S, \textup{id}_S)$ where $M(S) = \bfs$.

Given an $M$-admissible partial permutation $(S, T, \phi)$, the map $\phi$ induces a bijection $\overline{\phi}\co M(S) \rightarrow M(T)$ where $\overline{\phi}$ sends $M(i)$ to $M(\phi(i))$  for $i \in S$. It is straightforward to verify that $\overline{\phi}$ is well-defined and that, given $(S, T, \phi)$ and $(T, U, \psi)$, the maps $\overline{\psi \circ \phi}$ and $\overline{\psi} \circ \overline{\phi}$ agree.  We now describe a $\Ztwo$-grading on $\cA(\cZ)$.  

\begin{definition}\label{defn:alggr}
The function $\gr \co \cA(\cZ) \rightarrow \Z/2\Z$ is specified by
\[ \gr(a(S, T, \phi)) = \sum_{i \in S} o(i) + \sum_{i \in T} o(i) + \inv(\overline{\phi}) \pmod 2, \]
where $o(i)$ denotes the orientation of $i \in \bfa$, with the convention that $o(i)=0$ if $i$ is positively oriented and $o(i)=1$ if $i$ is negatively oriented.
\end{definition}

Given $(S', T', \phi')$ such that $a(S', T', \phi')=a(S, T, \phi)$ we have that $\inv(\overline{\phi}) = \inv(\overline{\phi'})$ since, for each $U \subset \textup{Fix }(\phi)$, the map $\overline{\phi_U}$ agrees with $\overline{\phi}$. Further, for $i \in \textup{Fix }(\phi)$, there are two contributions of $o(i)$ in $\gr(a(S,T,\phi))$.  It follows that $\gr(a(S', T', \phi')) = \gr (a(S, T, \phi)) \pmod 2$ so that $\gr$ is well-defined.

\begin{remark}
It is straightforward to verify that the above definition agrees with the $\Z/2\Z$-grading on $\cA(\cZ)$ defined in \cite[Section 7]{Petkovadecat} using the diagram $\textsf{AZ}(\cZ)$ \cite[Section 4]{LOTmorphism}. We provide a direct proof that this grading makes $\cA(\cZ)$ into a differential graded algebra without appealing to $\textsf{AZ}(\cZ)$.
\end{remark}

\begin{lemma}\label{lem:Adiffgr}
The function $\gr$ makes $\cA(\cZ)$ into a differential graded algebra.
\end{lemma}

\begin{proof}
We first show that $\gr(a(S, T, \phi) \cdot a(T, U, \psi))=\gr(a(S, T, \phi))+\gr(a(T, U, \psi)) \pmod 2$. Recall that  $\overline{\psi \circ \phi}=\overline{\psi} \circ \overline{\phi}$ and note that $\inv(\overline{\psi} \circ \overline{\phi}) = \inv(\overline{\psi})+\inv(\overline{\phi}) \pmod 2$. Thus, we have
\begin{align*}
\gr \big(a(S, T, \phi) \cdot a(T, U, \psi) \big) &= \gr(a(S, U, \psi \circ \phi)) \\
	&= \sum_{i \in S} o(i) + \sum_{i \in U} o(i) + \inv(\overline{\psi \circ \phi})  \\
	&= \sum_{i \in S} o(i) + \sum_{i \in T} o(i)+\sum_{i \in T} o(i)+\sum_{i \in U} o(i) + \inv(\overline{\psi})+\inv(\overline{\phi}) \\
	&=\gr(a(S, T, \phi))+\gr(a(T, U, \psi)) \pmod 2 .
\end{align*}
To show that $\gr (d(a(S, T, \phi)) = \gr(a(S, T, \phi)) +1 \pmod 2$, we simply notice that for each inversion $(i, j)$ of $\phi$, 
\begin{align*}
\inv (\overline{\phi \circ \tau_{i,j}}) &= \inv(\overline{\tau_{i,j}}) + \inv (\overline{\phi})  \\
	& = 1+ \inv (\overline{\phi}) \pmod 2,
\end{align*}
which implies that $\gr(a(S, T, \phi \circ \tau_{i,j})) = \gr(a(S, T, \phi)) +1 \pmod 2$.
\end{proof}

\subsection{Bordered Heegaard diagrams}\label{sec:borderedHD} Pointed matched circles arise naturally from the boundary of a bordered three-manifold, that is, a manifold with connected boundary suitably parametrized.  More precisely, a {\em parameterization} of a surface is a choice of orientation-preserving diffeomorphism with $F(\cZ)$ for a pointed matched circle $\cZ$.  For the purpose of the invariants used here, these manifolds are described by bordered Heegaard diagrams \cite{LOT}, that is, a 5-tuple $\cH=(\Sigma, \bfalpha^c,\bfalpha^a,\bfbeta,z)$ where:
\begin{itemize}
\item $\Sigma$ is a compact orientable genus $g>0$ surface with a single boundary component;
\item $\bfbeta$ is a $g$-tuple of non-intersecting, essential simple closed curves in $\Sigma$;
\item $\bfalpha^c$ is a $(g-k)$-tuple of non-intersecting, essential simple closed curves in $\Sigma$;
\item $\bfbeta$ and $\bfalpha^c$ generate $g$- and $(g-k)$-dimensional subspaces, respectively, of $H_1(\Sigma;\Q)$;
\item $z$ is a point in the boundary circle $\partial\Sigma$; and
\item $\bfalpha^a$ is an ordered $2k$-tuple of non-intersecting, properly embedded, oriented arcs in $\Sigma$ so that \[\cZ_\cH=(\partial\Sigma,z,\partial\bfalpha^a,M_{\bfalpha^a})\] is a pointed matched circle, where $M_{\bfalpha^a}\co [|\partial \bfalpha^a|] \to [|\bfalpha^a|]$ is the matching determined by inclusion under the identification of $\bfalpha^a$ with $[2k]$).
\end{itemize}
Given a bordered Heegaard diagram, there is a natural way to construct a three-manifold with boundary of genus $k$, as in \cite[Construction 4.6]{LOT}. The collection $\bfalpha^c$ is referred to as the $\alpha$-circles and the collection $\bfalpha^a$ is referred to as the $\alpha$-arcs; where required, elements of $\bfalpha=\bfalpha^c\cup\bfalpha^a$ are called  $\alpha$-curves. Note that the orientation required on $\partial \bfalpha^a$ by Definition \ref{defn:pmc} is induced from the choice of orientation on the collection $\bfalpha^a$. (Alternatively, given an orientation on $\partial \bfalpha^a$ coming from a pointed matched circle $\cZ$, there is a unique compatible choice of orientation on the $\alpha$-arcs of a bordered Heegaard diagram inducing $\cZ$.) Also, observe that the $\alpha$- and $\beta$-curves may intersect (and we require that such intersections be transverse); indeed, these intersections generate the modules arising in bordered Floer homology. The intersections of interest are $g$-tuples of the form 
\begin{equation}\label{eqn:intersect}
\mfS(\cH)=\bigcup_{\sigma}\{(\beta_1\cap\alpha_{\sigma(1)}, \ldots,    \beta_{g}\cap\alpha_{\sigma(g)})\}
\end{equation} 
where the union is taken over injections $\sigma \co [g] \rightarrow [g+k]$ satisfying $\bfalpha^c \subset \{ \alpha_{\sigma(i)} \mid i \in [g]\}$. This ensures that precisely half of the curves in $\bfalpha^a$ are occupied, and each of the circles in $\bfalpha^c$ is occupied, in any given $\bfx\in \mfS(\cH)$. Each such $\bfx$ thus determines a $k$-element subset $\bfs_\bfx \subset [2k]$ by applying the matching to the boundary of the $\alpha$-arcs occupied by $\bfx$. Two examples, in the case $g=k=1$, are given in Figure \ref{fig:CFACFDsample}.

\begin{remark}\label{rmk:diagrams-vs-manifolds}
Following \cite[Chapter 4]{LOT}, the combinatorial data $(\cH,\cZ_\cH)$ determines a pair $(W, F)$, where $W$ is a three-manifold with connected boundary $F$ parameterized by $\cZ_\cH$. On the other hand, given a three-manifold with connected, parameterized boundary, it is possible to choose an associated $\cH$ (and hence $(\cH,\cZ_\cH))$, though not uniquely. Since the bordered invariants considered below do not, up to homotopy, depend on this choice \cite{LOT, LOTbimodules} we will (as in the notation in the introduction) describe these objects in terms of  $(W, F)$.  
\end{remark}

\subsection{Type A  structures} \label{sec:typeAbackground} The type A structure $\CFA(W)$ associated with a bordered manifold $W$ with connected boundary $F$ (by way of a suitable choice of bordered Heegaard diagram $(\cH, \cZ)$) is a right $\cA_\infty$-module over $\cA(\cZ)$ generated (as a vector space over $\F$) by the set of intersection points $\mfS(\cH)$ (see  \eqref{eqn:intersect}). By abuse of notation, and continuing in the vein of Remark \ref{rmk:diagrams-vs-manifolds}, we will identify $g$-tuples of intersection points $\bfx\in\mfS(\cH)$ with generators $\bfx\in\CFA(W)$.

Given $\bfx\in\CFA(W)$, we define 
\[ I_A(\bfx) = I(\bfs_\bfx) \in \cA(F) \]
where $\bfs_{\bfx} \in [2k]$ is as described above. We define a right action of $\cI$ on $\CFA(W)$ by $\bfx \cdot I_A(\bfx) = \bfx$ and $\bfx \cdot \iota =0$ for all other minimal idempotents $\iota$. There are multiplication maps \[m_i\co \CFA(W)\otimes\cA(\cZ)^{\otimes(i-1)}\to \CFA(W),\]  defined for all $i>0$, that satisfy compatibility conditions ensuring multiplication is associative up to homotopy; see \cite[Definition 2.5]{LOT}. This compatibility guarantees, in particular, that $m_1\circ m_1=0$. We say that a type A structure is \emph{bounded} if $m_i \equiv 0$ for $i$ sufficiently large. (This is also called {\em operationally bounded}; compare \cite[Definition 2.2.18 and Remark 2.2.19]{LOTbimodules}). It is always the case that there is a choice of bordered Heegaard diagram such that $\CFA(W)$ is bounded \cite[Proposition 4.25, Lemma 6.5]{LOT}.  In particular, for any choice of bordered Heegaard diagram, $\CFA(W)$ is homotopy equivalent to a bounded type A structure.  

Referring the reader to \cite[Chapter 7]{LOT} for the details, we recall that these operations arise from particular holomorphic curve counts in $\Sigma\times\R\times[0,1]$ with boundary conditions specified by the curves $(\bfalpha,\bfbeta)$ in the (choice of) Heegaard surface. 

In order to define a $\Z/2\Z$-grading on $\CFA(W)$ we study partial permutations in a different context. Our goal is to define the sign of a partial permutation in a way that behaves well under ``gluing'' two partial permutations together. These partial permutations will ultimately be induced by generators in bordered Floer homology (as in \eqref{eqn:intersect}), and our terminology is designed to reflect this relationship.

\begin{definition}
A \emph{bordered partial permutation} $(g, k, B, \sigma)$ is an injection $\sigma \co [g] \rightarrow [g+k]$, where 
\begin{itemize}
	\item $B \subset [g+k]$
	\item $g \geq k$
	\item $|B| =2k$
	\item $[g+k] \setminus B \subset \Im(\sigma)$.
\end{itemize}
\end{definition}

\begin{remark}
It follows that for a bordered partial permutation, $|\Im(\sigma) \cap B| =k$.
\end{remark}

\begin{remark}
In terms of bordered Heegaard diagrams, the domain $[g]$ corresponds to $\bfbeta$, the codomain $[g+k]$ corresponds to $\bfalpha$, and $B \subset [g+k]$ corresponds to $\bfalpha^a$.
\end{remark}

\begin{remark}\label{rmk:footnote4}
For convenience, we will always suppose that the elements of $\bfalpha^a$ are oriented such that for each $\alpha \in \bfalpha^a$, we have $\alpha^- \lessdot \alpha^+$ where $\alpha^\pm$ denotes the positively/negatively oriented point in $\d \alpha$. (This assumption does not cause any loss of generality by Remark \ref{rem:idemconj}.) This choice of orientation specifies an oriented basis for $H_1(F(\cZ))$.  We record the orientation of an intersection point according to a right-hand-rule, where $\alpha$ and $\beta$ are identified with the $x$- and $y$-axis, respectively. 
\end{remark}

Denote by $o(x) \in \{0,1\}$ the orientation of $x \in \bfx$ (compare \eqref{eqn:intersect}), with the convention that $o(x)=0$ if the orientation at $x$ is positive and $o(x)=1$ if the orientation at $x$ is negative. 

We also suppose that an order on $\bfalpha$ and $\bfbeta$ has been fixed. For type A structures, we choose this order so that the circles in $\bfalpha^c$ come before the arcs in $\bfalpha^a$. (For the type D structures defined below in Section \ref{subsec:typeD} we will choose the opposite: the arcs in $\bfalpha^a$ will come before the circles in $\bfalpha^c$.)

\begin{definition}
\label{defn:Asgn}
A \emph{type A partial permutation} is a bordered partial permutation $(g, k, A, \sigma)$ with $A = [g-k+1, g+k]$. The \emph{type A sign} of a type A partial permutation is
\[ \sgn_A(\sigma) = \inv(\sigma) \pmod 2.\]
\end{definition}

We will often abbreviate a type A partial permutation $(g, k, A, \sigma_A)$ by the function $\sigma_A$ when $g$ and $k$ are implicit.  Note that every generator $\bfx\in\CFA(W)$ determines a partial permutation $\sigma_\bfx$ as in \eqref{eqn:intersect}, viewed as a function from the $\beta$-curves to the $\alpha$-curves. 

\begin{definition}\label{defn:grA}
	The \emph{type A grading} of a generator $\bfx \in \CFA(W)$ is 
\[ \gr_A(\bfx) = \sgn_A(\sigma_{\bfx}) + \sum_{x \in \bfx} o(x) \pmod 2. \]
\end{definition}

Note that $\gr_A$ is well-defined up to a possible overall shift if we reorder the $\alpha$- and $\beta$-circles or change their orientations. That this gives rise to a grading on $\CFA(W)$ which is invariant in an appropriate sense is established in Theorem \ref{thm:allthegradings}.    

\subsection{Type D  structures}\label{subsec:typeD} We now define the type D structure $\CFD(W)$ associated with a bordered manifold $W$ with parametrized, connected boundary $F$.  This is generated, as an $\F$-vector space, by the set of intersection points $\mfS(\cH)$. As above, by abuse, we write $\bfx\in\CFD(W)$ for these generators. By contrast with $\CFA$, this invariant is defined over $\cA(-F)=\cA(-\cZ)$.  Let  
\[ I_D(\bfx) = I([2k] \setminus \bfs_\bfx) \in \cA(-F), \] 
where $\bfs_\bfx$ corresponds to the $\alpha$-arcs occupied by $\bfx$. We define a left action of $\cI$ on $\CFA(W)$ by $I_D(\bfx) \cdot \bfx= \bfx$ and $\iota \cdot \bfx =0$ for all other minimal idempotents $\iota$. A type D structure admits a map
\[\delta^1\co \CFD(W) \to \cA(-F)\otimes \CFD(W)\]
satisfying a compatibility condition \cite[Definition 2.18]{LOT}. As a result, $\cA(-F)\otimes_{\mathcal{I}} \CFD(W)$ is a left differential module over  $\cA(-F)$ having specified multiplication $a\cdot(b\otimes \bfx)=ab\otimes \bfx$ and differential $\partial(a\otimes \bfx) = a\cdot \delta^1(\bfx) + da\otimes \bfx$; compatibility ensures that $\partial\circ\partial=0$.  By abuse, when considering a type D structure $N$, we may also use $N$ to refer to the associated differential module.   

The map $\delta^1$ can be iterated inductively to define maps 
	\[ \delta^j \co \CFD(W) \rightarrow \cA(-F)^{\otimes j} \otimes \CFD(W)\]
by fixing $\delta^0 = \operatorname{id}_{\CFD(W)}$ and setting \[\delta^j=\left(\operatorname{id}_{\cA(-F)^{\otimes j-1}}\otimes\delta^1\right)\circ\delta^{j-1} \] for all $j>0$. We say that a type D structure is \emph{bounded} if $\delta^j \equiv 0$ for $j$ sufficiently large.  As in the case of type A structures, $\CFD(W)$ is always, up to homotopy, bounded.

\begin{definition}
A \emph{type D partial permutation} is a bordered partial permutation $(g, k, D, \sigma)$ with $D = [2k]$. The \emph{type D sign} of a type D partial permutation is
\[ \sgn_D(\sigma) = \inv(\sigma) + \sum_{i \in \text{Im}(\sigma)} \# \{ \ j \ | \ j>i, j \notin \text{Im}(\sigma) \}  \pmod 2.\]
\end{definition}

\begin{remark}
Note that $j \notin \Im(\sigma)$ implies that $j \in D$; compare to Definition \ref{defn:sgnDA} below.
\end{remark}

As above, we will often abbreviate a type D partial permutation $(g, k, D, \sigma_D)$ by the function $\sigma_D$ when $g$ and $k$ are implicit.   

\begin{definition}
	The \emph{type D grading} of a generator $\bfx \in \CFD(W)$ is 
\[ \gr_D(\bfx) = \sgn_D(\sigma_{\bfx}) + \sum_{x \in \bfx} o(x) \pmod 2. \]
\end{definition}

Note that $\gr_D$ is well-defined up to a possible overall shift if we reorder or change the orientations on the $\alpha$- and $\beta$-circles. That this gives rise to an invariant grading, in an appropriate sense, on $\CFD(W)$ is established in Theorem~\ref{thm:allthegradings}.

\begin{remark}
When the $\alpha$-arcs are ordered according to the condition $\alpha^-_i \lessdot \alpha^-_j$ (where $\lessdot$ is again the order induced by the orientation of $Z$ and the base point $z$, and $\alpha^-$ denotes the negatively oriented point in $\partial \alpha$), our definition of $\gr_D$ agrees with the definition of the $\Ztwo$-grading $s(\bfx)$ in \cite[Section 7]{Petkovadecat}, up to a possible overall shift. (Recall that per Remark \ref{rmk:footnote4}, the $\alpha$-arcs are oriented such that $\alpha^- \lessdot \alpha^+$.)
\end{remark}

\begin{figure}[ht]
\labellist
\pinlabel $m_2(\bfx\otimes\rho_2)=\bfy$ at 10 5
\pinlabel $\delta^1(\bfa)=\rho_2\otimes\bfb$ at 185 5
\pinlabel $\cZ$ at 66 122
\pinlabel $-\cZ$ at 101 122

\scriptsize \pinlabel $\bfx$ at 37 53
\scriptsize \pinlabel $\bfy$ at 44 84
\scriptsize \pinlabel $\bfb$ at 145 61
\scriptsize \pinlabel $\bfa$ at 145 46
\endlabellist
\includegraphics[scale=1.4]{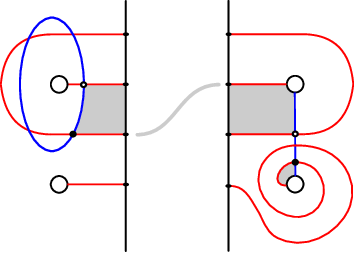}
\caption{Bordered Heegaard diagrams for two solid tori with opposite orientations. On the left the boundary is positively oriented by $F(\cZ)$ compatible with the conventions for a type A structure over $\cA(\cZ)$, and on the right the boundary is given by $-F(\cZ)$ compatible with the conventions for a type D structure over $\cA(\cZ)$. Two sample operations are given where, in each case, the relevant algebra element is $\rho_2 = (\{2\}, \{3\}, 2 \mapsto 3)$.}
\label{fig:CFACFDsample}
\end{figure}
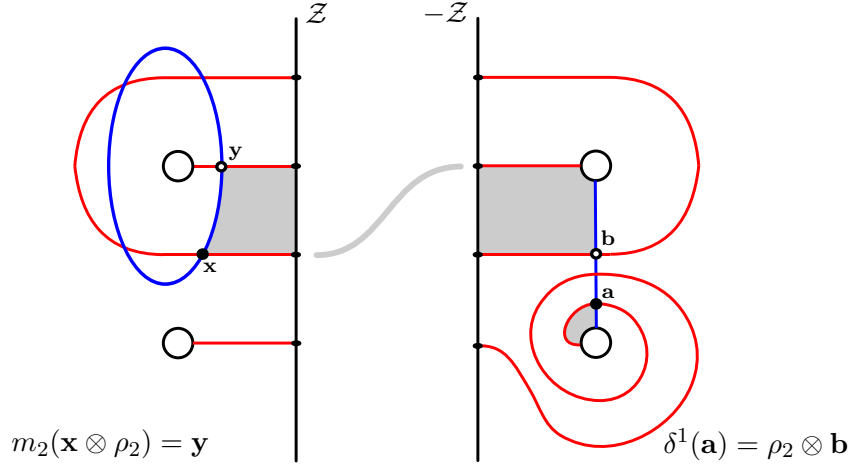

\subsection{Pairing}\label{subsec:pair} A key motivation for the bordered invariants described above is a gluing formula for the Heegaard Floer homology of a closed oriented three-manifold $Y$. More precisely, if $Y=W_1\cup_F W_2$, then the pairing theorem \cite[Theorem 1.3]{LOT} states that
\[\HFhat(Y) \cong H_*\left(\CFA(W_1)\boxtimes\CFD(W_2)\right)\]
for a compatible choice of parameterizations of the (common) boundaries of $W_1$ and $W_2$. Note that if $(\cH_1, \cZ)$ and $(\cH_2, -\cZ)$ describe $(W_1, F)$ and $(W_2, -F)$ respectively, then $\cH_1 \cup_\cZ \cH_2$ describes $Y=W_1 \cup_F W_2$. The product $\boxtimes$ is a special form of the derived tensor product (see \cite[Section 2.4]{LOT}) that, despite a duality between type A and type D structures described below in Section~\ref{subsec:dualityrevisited}, highlights the utility of the latter. In particular, the generators of   $\CFA(W_1)\boxtimes\CFD(W_2)$ are of the form $\bfx\boxtimes\bfy=\bfx\otimes_\cI\bfy$ with differential defined by \[\partial^\boxtimes(\bfx\boxtimes\bfy) = \sum_j m_{j+1}(\bfx\otimes\delta^j(\bfy))\]
which vanishes for $j\gg0$ provided that at least one of $\CFA(W_1)$, $\CFD(W_2)$ is bounded. This tensor product (and hence the homology) can be effectively computed in many settings. 

We review from \cite[Section 5]{OSz2004-properties} the relative $\Ztwo$-grading on $\CFhat$, which agrees with the relative Maslov grading modulo $2$. Given a Heegaard diagram $\cH=(\Sigma, \bfalpha, \bfbeta, z)$ for a closed manifold, we fix an order and orientations on the $\alpha$-circles and similarly for the $\beta$-circles. Then given $\bfx \in \mfS(\cH)$, a generator of $\CFhat$, we have an associated permutation $\sigma_\bfx$ as before, and we define
	\begin{equation}\label{eq:closed-z2-grading} \gr(\bfx) = \inv(\sigma_\bfx) + \sum_{x \in \bfx} o(x) \pmod 2. \end{equation}

We now return to the discussion in the introduction and establish that the gradings defined above on $\CFA(W_1)$ and $\CFD(W_2)$ recover the $\Z/2\Z$-grading on $\HFhat(Y)$ (with respect to which $|\chi(\HFhat(Y))|=|H_1(Y;\Z)|$). Note that this amounts to a statement about the gradings of the $\bfx\otimes_\cI\bfy$ at the chain level and does not depend on $\d^\boxtimes$.  

Given a type A partial permutation $(g_A, k, A, \sigma_A)$ and a type D partial permutation $(g_D, k, D, \sigma_D)$, we define a permutation 
\begin{equation}
\label{eqn:sumdefn}
\sigma_A+\sigma_D\co [g_A+g_D] \rightarrow [g_A+g_D]
\end{equation}
as follows. If $( \Im(\sigma_D) \cap D ) \cup \{ i-(g_A-k) \mid i \in \Im(\sigma_A) \cap A) \} \neq [2k]$, then  $\sigma_A+\sigma_D$ is undefined. If $( \Im(\sigma_D) \cap D ) \cup \{ i-(g_A-k) \mid i \in \Im(\sigma_A) \cap A) \} = [2k]$, then
\begin{align*}
	\sigma_A+\sigma_D(i) &= \left\{
	\begin{array}{ll}
		\sigma_A(i) & \text{if } i\in [g_A] \\
		\sigma_D(i-g_A) +g_A -k & \text{otherwise}.
	\end{array} \right.
\end{align*}
Note that if $\sigma_A$ and $\sigma_D$ are partial permutations corresponding to $\bfx \in \CFA(W_1)$ and $\bfy \in \CFD(W_2)$ respectively, then $\sigma_A+\sigma_D$ is defined exactly when $I_A(\bfx)=I_D(\bfy)$, i.e., when the $\alpha$-arcs occupied by $\bfx$ are complementary to those occupied by $\bfy$. Injectivity and surjectivity of  $\sigma_D +\sigma_A$ follow by hypothesis. See Figure \ref{fig:permutations}.

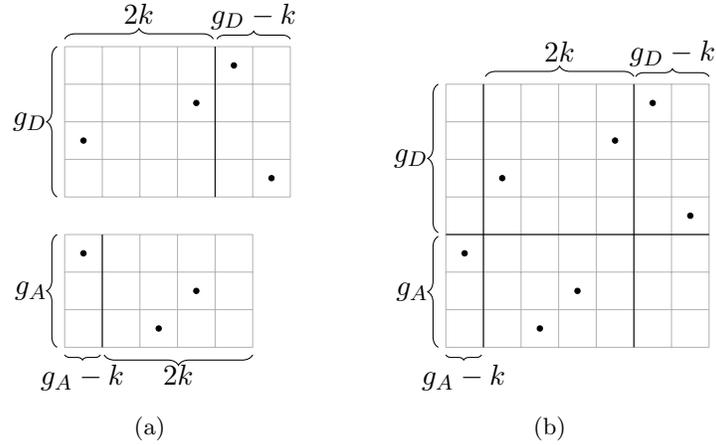
\begin{figure}[htb!]
\centering
\subfigure[]{
\begin{tikzpicture}[scale=0.50]
	\draw[step=1, black!30!white, very thin] (0, 0) grid (5, 3);
	\draw[-] (1, 0) -- (1, 3);
	\filldraw (2.5, 0.5) circle (2pt);
	\filldraw (3.5, 1.5) circle (2pt); 
	\filldraw (0.5, 2.5) circle (2pt); 
	\draw [decorate,decoration={brace,amplitude=4pt}] (-0.2, 0) -- (-0.2, 3) node [midway, left] {$g_A$};
	\draw [decorate,decoration={brace,amplitude=4pt,mirror}] (1.05, -0.2) -- (5, -0.2) node [midway, below] {$2k$};
	\draw [decorate,decoration={brace,amplitude=2pt,mirror}] (0, -0.2) -- (0.95, -0.2) node [midway, below] {$g_A-k$};
	
	\draw[step=1, black!30!white, very thin] (0, 4) grid (6, 8);
	\draw[-] (4, 4) -- (4, 8);
	\filldraw (0.5, 5.5) circle (2pt);
	\filldraw (3.5, 6.5) circle (2pt); 
	\filldraw (4.5, 7.5) circle (2pt); 
	\filldraw (5.5, 4.5) circle (2pt); 
	\draw [decorate,decoration={brace,amplitude=4pt}] (-0.2, 4) -- (-0.2, 8) node [midway, left] {$g_D$};	
	\draw [decorate,decoration={brace,amplitude=4pt}] (0, 8.2) -- (3.95, 8.2) node [midway, above, yshift=2pt] {$2k$};
	\draw [decorate,decoration={brace,amplitude=3pt}] (4.05, 8.2) -- (6, 8.2) node [midway, above] {$g_D-k$};
\end{tikzpicture}
}
\hspace{15pt}
\subfigure[]{
\begin{tikzpicture}[scale=0.50]
	\draw[step=1, black!30!white, very thin] (0, 0) grid (7, 7);
	\draw[-] (1, 0) -- (1, 7);
	\draw[-] (5, 0) -- (5, 7);
	\draw[-] (0, 3) -- (7, 3);
	\filldraw (2.5, 0.5) circle (2pt);
	\filldraw (3.5, 1.5) circle (2pt); 
	\filldraw (0.5, 2.5) circle (2pt); 

	\filldraw (1.5, 4.5) circle (2pt);
	\filldraw (4.5, 5.5) circle (2pt); 
	\filldraw (5.5, 6.5) circle (2pt); 
	\filldraw (6.5, 3.5) circle (2pt); 	
	
	\draw [decorate,decoration={brace,amplitude=4pt}] (-0.2, 0) -- (-0.2, 2.95) node [midway, left] {$g_A$};
	\draw [decorate,decoration={brace,amplitude=4pt}] (-0.2, 3.05) -- (-0.2, 7) node [midway, left] {$g_D$};
	\draw [decorate,decoration={brace,amplitude=4pt}] (1.05, 7.2) -- (5, 7.2) node [midway, above, yshift=2pt] {$2k$};
	\draw [decorate,decoration={brace,amplitude=2pt,mirror}] (0, -0.2) -- (0.95, -0.2) node [midway, below] {$g_A-k$};
	\draw [decorate,decoration={brace,amplitude=3pt}] (5.05, 7.2) -- (7, 7.2) node [midway, above] {$g_D-k$};
\end{tikzpicture}
}
\caption[]{Top left, an example of a type D partial permutation $\sigma_D$. The domain corresponds to the vertical axis. Bottom left, an example of a type A partial permutation $\sigma_A$. Right, the resulting permutation $\sigma_A +\sigma_D$.}
\label{fig:permutations}
\end{figure}

\begin{lemma}
\label{lem:sgnsum}
Let $\sigma_A$, $\sigma_D$ be type A and type D partial permutations respectively.  If $\sigma_A + \sigma_D$ is well-defined, then 
\[ \sgn(\sigma_A+\sigma_D) = \sgn_A (\sigma_A) +\sgn_D(\sigma_D).\]
In particular, the relative sign of a sum is equal to the sum of the relative signs of the components.
\end{lemma}

\begin{proof}
We will show that the number of inversions in $\sigma = \sigma_A+\sigma_D$ is congruent modulo $2$ to $ \sgn_A (\sigma_A) +\sgn_D(\sigma_D)$. It is clear that $\inv(\sigma|_{[g_A]}) = \inv(\sigma_A)$ and $\inv(\sigma|_{[g_A+1, g_A+g_D]}) = \inv(\sigma_D)$. We now must count inversions of $\sigma$ of the form $(\ell, m)$, where $\ell \in [g_A]$ and $m \in [g_A+1, g_A+g_D]$, which we claim is
\begin{align}
\label{eqn:sgnDsum}
	 \sum_{i \in \text{Im}(\sigma_D)} \# \{ \ j \in [g_D+k] \ | \ j>i, j \notin \text{Im}(\sigma_D) \}.
\end{align}
Indeed, $j \notin \text{Im}(\sigma_D)$ is equivalent to the existence of $\ell \in [g_A]$ such that $\sigma_A+\sigma_D(\ell)=j + g_A -k$. Then $(\ell, m)$ is an inversion of $\sigma_A + \sigma_D$ for each $m \in [g_A+1, g_A + g_D]$ with $\sigma_D(m-g_A)<j$, since $\sigma_A + \sigma_D (\ell) = j + g_A -k > \sigma_D(m-g_A) + g_A -k = \sigma_A + \sigma_D (m)$. The number of such pairs $(\ell, m)$ is equal to the value of \eqref{eqn:sgnDsum}.

Equivalently, the sum of two partial permutations can be visualized as in Figure \ref{fig:permutations}, where an inversion corresponds to a northwest-southeast pair. It is clear that the number of NW-SE pairs is equal to the sum of
\begin{itemize}
	\item the number of NW-SE pairs below the horizontal line (i.e., inversions of $\sigma |_{[g_A]}$)
	\item the number of NW-SE pairs above the horizontal line (i.e., inversions of $\sigma |_{[g_A+1, g_A + g_D]}$)
	\item the number of NW-SE pairs with the NW element above the horizontal line and the SE element below the horizontal line.
\end{itemize}
The last bullet point is equal to the value of \eqref{eqn:sgnDsum}.

Thus, the number of inversions of $\sigma_A+\sigma_D$ is
\[ \inv(\sigma_A) + \inv(\sigma_D) + \sum_{i \in \text{Im}(\sigma_D)} \# \{ \ j \ | \ j>i, j \notin \text{Im}(\sigma_D) \}, \]
completing the proof.
\end{proof}

We now verify that the gradings $\gr_A$ and $\gr_D$ behave well with respect to pairing.  
\begin{proposition}
\label{prop:respectspairing}
Let $\bfx \in \mfS(\cH_1)$ and $\bfy \in \mfS(\cH_2)$ be such that $-\cZ_2= \cZ_1$ where $\cZ_i = \d \cH_i$. Then, up to a possible overall shift independent of both $\bfx$ and $\bfy$, 
\[ \gr(\bfx \otimes_\cI \bfy) = \gr_A(\bfx) + \gr_D(\bfy) \pmod  2. \]
\end{proposition}

\begin{proof}
Assume that $\bfx \otimes_\cI \bfy \neq 0$.  This condition implies that $\sigma_A(\bfx) + \sigma_D(\bfy)$ is well-defined.  There is an ordering of the $\alpha$-circles on $\cH = \cH_1 \cup_\cZ \cH_2$ induced by the ordering on $\cH_1$ and $\cH_2$. More precisely, we order the $\alpha$-circles beginning with the $\alpha$-circles (in order) from $\cH_1$, followed by the glued up $\alpha$-arcs (in order), and then the $\alpha$-circles (in order) from $\cH_2$.  Similarly, we order the $\beta$-circles on $\cH$ by writing the $\beta$-circles from $\cH_1$ followed by the $\beta$-circles from $\cH_2$. Notice that this order corresponds with the ordering used in the definition of $\sigma_A +\sigma_D$ in Equation (\ref{eqn:sumdefn}). In particular, $\sigma(\bfx \otimes \bfy) = \sigma_A(\bfx) +\sigma_D(\bfy)$

It follows from the definition of $-\cZ_1$ that the orientations on the $\alpha$-arcs in $\cH_1$ and in $\cH_2$ induce canonical orientations on the resulting $\alpha$-circles in $\cH$; further, we assume that the $\alpha$- and $\beta$-circles in each $\cH_i$ are oriented the same as in $\cH$.  We see that 
\[ \sum_{x \in \bfx, y \in \bfy} o(\bfx \otimes_\cI \bfy) = \sum_{x \in \bfx} o(x) + \sum_{y \in \bfy} o(y) \pmod 2.\] 
This, together with Lemma \ref{lem:sgnsum}, establishes the result.
\end{proof}

\subsection{Bimodules} All of the forgoing discussion may be promoted to various types of bimodules as described in \cite{LOTbimodules}. This first requires an expansion of the class of bordered Heegaard diagrams in order to describe a bordered three-manifold with a pair of boundary components. Following the notation above, these are 6-tuples $\cH=(\Sigma, \bfalpha^c,\bfalpha^a_L,\bfalpha^a_R,\bfbeta, \bfz)$ where:
\begin{itemize}
\item $\Sigma$ is a compact orientable genus $g>1$ surface with two boundary components $\d_L \Sigma$ and $\d_R \Sigma$;
\item $\bfbeta$ is a $g$-tuple of non-intersecting, essential simple closed curves in $\Sigma$;
\item $\bfalpha^c$ is a $(g-k_L-k_R)$-tuple of non-intersecting, essential simple closed curves in $\Sigma$;
\item $\bfbeta$ and $\bfalpha^c$ generate $g$- and $(g-k_L-k_R)$-dimensional subspaces, respectively, of $H_1(\Sigma;\Q)$;
\item $\bfz$ is a properly embedded arc in $\Sigma$, disjoint from $\bfbeta$ and $\bfalpha$, with boundary $z_L \cup z_R$ where $z_L \in \d_L \Sigma$ and $z_R \in \d_R \Sigma$;  and
\item $\bfalpha^a_L$ and $\bfalpha^a_R$ are ordered $2k_L$- and $2k_R$-tuples, respectively, of non-intersecting, properly embedded, oriented arcs in $\Sigma$ so that \[\cZ_L=(\partial_L\Sigma,z_L,\partial\bfalpha^a_L,M_{\bfalpha^a_L}), \quad \cZ_R=(\partial_R\Sigma, z_R,\partial\bfalpha^a_R,M_{\bfalpha^a_R})\]  are pointed matched circles.
\end{itemize}
We call such an $\cH$ an \emph{arced bordered Heegaard diagram}. Given an arced bordered Heegaard diagram, there is a natural way to construct a \emph{strongly bordered three-manifold with two boundary components}, which is a three-manifold with two boundary components parametrized by $\cZ_1$ and $\cZ_2$ together with a framed arc $\bfz$ connecting $z_1 \in \cZ_1$ to $z_2 \in \cZ_2$ \cite[Construction 5.3]{LOTbimodules}.  As in the case of a single boundary component, the generating set $\mfS(\cH)$ consists of $g$-tuples of intersection points in $\bfalpha\cap\bfbeta$ where:
\begin{itemize}
\item Each $\beta$-circle is occupied exactly once;
\item each $\alpha$-circle is occupied exactly once; and
\item exactly $k_L+k_R$ of the  $\alpha$-arcs in $\bfalpha^a_L\cup\bfalpha^a_R$ are occupied.
\end{itemize}
As before, to a generator $\bfx$, there is an associated injection $\sigma_\bfx \co [g] \rightarrow [g+k_L+k_R]$ satisfying $\bfalpha^c \subset \{ \alpha_{\sigma(i)} \mid i \in [g]\}$, where the $\alpha$-curves are ordered $\bfalpha^a_L$, $\bfalpha^c$, $\bfalpha^a_R$.  For brevity, given a three-manifold with two boundary components, we will use the term bordered in place of strongly bordered. 

Given a bordered three-manifold $W$ with two boundary components which we label $F_L$ and $F_R$ presented by an arced bordered Heegaard diagram $\cH$, a type DA structure is a bimodule $\CFDA(W)$ generated (as a vector space over $\F$) by $\mfS(\cH)$ admitting commuting actions on the left, by $\cA(-F_L)$, of type D and on the right, by $\cA(F_R)$, of type A. The left idempotent action $I_{L, D}(\bfx)$ (respectively right, $I_{R, A}(\bfx)$) is defined as for type D (respectively type A) structures. There is a natural splitting
	\[ \CFDA(\cH) = \bigoplus^{k_L + k_R}_{i = -(k_L + k_R)} \CFDA(\cH, i) \]
where
	\[  \CFDA(\cH, i) = \CFDA(\cH) \cdot \bfI_i(F_R). \]
Given $\bfx \in \CFDA(\cH, i)$, we say that $i$ is the \emph{strands grading} of $\bfx$. Equivalently, the strands grading of $\bfx$ is equal to $|\bfs|-k_R$ where $I(\bfs)=I_{R,A}(\bfx)$.
There are maps that take the form
\[\delta_i^1\co \CFDA(W)\otimes\cA(F_R)^{\otimes (i-1)} \to \cA(-F_L)\otimes \CFDA(W)\] 
which simultaneously generalize the type A and type D operations, again satisfying appropriate compatibility conditions; see \cite[Definition 2.2.43]{LOTbimodules}. An example is shown in Figure \ref{fig:CFDAsample}. From this the $\delta_i^k$ may be defined recursively as before, for all $k>0$, where $\delta_i^0$ take the form of the usual $\cA_\infty$-products.  As in the case of pairing type A with type D structures, there is an analogously defined box tensor product $\boxtimes$ for bimodules which satisfy appropriate pairing theorems \cite{LOTbimodules}.  For example, given $W = W_1 \cup_F W_2$, we have $\CFDA(W) = \CFDA(W_1) \boxtimes_{\cA(F)} \CFDA(W_2)$; there are also pairing theorems for type DA structures with type A structures, etc.

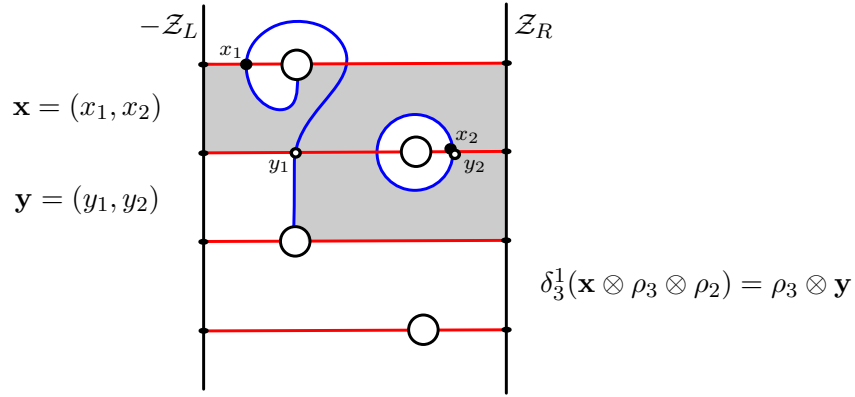
\begin{figure}[ht]
\labellist
\pinlabel $\delta^1_3(\bfx\otimes\rho_3\otimes\rho_2)=\rho_3\otimes\bfy$ at 135 30
\pinlabel $\cZ_R$ at 91 100
\pinlabel $-\cZ_L$ at -8 100
\pinlabel $\bfx=(x_1,x_2)$ at -30 78
\pinlabel $\bfy=(y_1,y_2)$ at -30 53

\scriptsize \pinlabel $x_1$ at 9 93
\scriptsize \pinlabel $y_1$ at 22 62
\scriptsize \pinlabel $x_2$ at 73 70
\scriptsize \pinlabel $y_2$ at 75 62
\endlabellist
\includegraphics[scale=1.4]{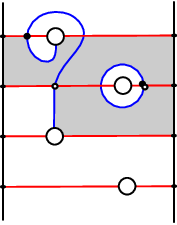}
\caption{A bordered Heegaard diagram in the case $g=2, k_L=k_R=1$. This example encodes a the mapping cylinder of a Dehn twist on the torus; see \cite[Section 10]{LOTbimodules}. Note that the relevant algebra elements for this example are $\rho_2 = (\{2\}, \{3\}, 2 \mapsto 3)$ and $\rho_3 = (\{3\}, \{4\}, 3 \mapsto 4)$.}
\label{fig:CFDAsample}
\end{figure}

The gradings given above extend in a straightforward manner to type DA bimodules, which we now describe. 
\begin{definition}
\label{defn:sgnDA}
A \emph{type DA partial permutation} $(g, k_L, k_R, D, A, \sigma)$ is an injection
	\[ \sigma \co [g] \rightarrow [g + k_L + k_R], \]
such that
\begin{itemize}
	\item $D = [2k_L]$
	\item $A = [g + k_L - k_R + 1, g + k_L + k_R]$
	\item $[g + k_L + k_R] \setminus (D \cup A) \subset \Im (\sigma)$.
\end{itemize}
The \emph{type DA sign} of a type DA partial permutation is
\[ \sgn_{DA}(\sigma) = \inv(\sigma) + \sum_{i \in \Im(\sigma)} \# \{ j \in D \ | \ j>i, j \notin \Im(\sigma) \} + t(g-k_L-k_R) \ \pmod  2, \]
where $t=|\Im(\sigma)\cap A|$.
\end{definition}

\begin{remark}
	The purpose of the $t(g-k_L-k_R)$ term in the above definition will become clear in Proposition \ref{prop:Hochschildgr} where we show that the grading on $\CFDA$ behaves well under Hochschild homology.
\end{remark}

\begin{definition}
\label{defn:grDA}
The \emph{type DA grading} of a generator $\bfx \in \mfS(\cH)$ is
\[ \gr_{DA}(\bfx) = \sgn_{DA}(\sigma_\bfx) + \sum_{x \in \bfx} o(x) \pmod 2. \]
\end{definition}

\begin{definition}
The sum of two type DA partial permutations 
\[ (g_1, k_L, k_\textmid, D_1, A_1, \sigma_1) \qquad \textup{and} \qquad (g_2, k_\textmid, k_R, D_2, A_2, \sigma_2) \]
is the type DA partial permutation $(g_1+g_2, k_L, k_R, D_1, A, \sigma_1 + \sigma_2)$ where $A = [g_1 + g_2 + k_L - k_R + 1, g_1 + g_2 + k_L + k_R]$ and
\begin{align*}
	\sigma_1+\sigma_2(i) &= \left\{
	\begin{array}{ll}
		\sigma_1(i) & \text{if } i\in [g_1] \\
		\sigma_2(i-g_1)+g_1+k_L-k_\textmid & \text{if } i\in [g_1+1, g_1+g_2]
	\end{array} \right.
\end{align*}
if $\{ i-(g_1+k_L-k_\textmid) \mid i \in \Im(\sigma_1) \cap A_1 \} \cup (\Im(\sigma_2) \cap D_2) = [2k_\textmid]$ and undefined otherwise.
\end{definition}

\begin{lemma}\label{lem:CFDAsignaddition}
With $\sigma_1, \sigma_2$ as above, if $\sigma_1 + \sigma_2$ is defined, then 
\[ \sgn_{DA} (\sigma_1 + \sigma_2) = \sgn_{DA} (\sigma_1) + \sgn_{DA} (\sigma_2) + (k_R+ k_\textmid)(g_1+k_L+k_\textmid) \pmod 2. \]
In particular, the relative sign of a sum of two type DA partial permutations is equal to the sum of the relative signs of the components. Moreover, if $k_R=k_\textmid$, then
\[ \sgn_{DA} (\sigma_1 + \sigma_2) = \sgn_{DA} (\sigma_1) + \sgn_{DA} (\sigma_2) \pmod 2. \]
\end{lemma}

\begin{proof}
Straightforward generalization of the proof of Lemma \ref{lem:sgnsum}.
\end{proof}

\begin{proposition}
\label{prop:CFDApairing}
Let $\bfx \in \mfS(\cH_1)$ and $\bfy \in \mfS(\cH_2)$ where $\d_R \cH_1 = - \d_L \cH_2$. Then, up to a possible overall shift as in Lemma \ref{lem:CFDAsignaddition},
\[ \gr_{DA} (\bfx \otimes \bfy) = \gr_{DA}(\bfx) + \gr_{DA}(\bfy) \pmod 2. \]
\end{proposition}

\begin{proof}
Straightforward generalization of the proof of Proposition \ref{prop:respectspairing}.
\end{proof}

\begin{remark}
\label{rem:CFDApairing}
The analogue of Lemma~\ref{lem:CFDAsignaddition} holds for the sum of a type A partial permutation with a type DA partial permutation and for the sum of a type DA partial permutation with a type D partial permutation.  Consequently, we have analogues of Proposition~\ref{prop:CFDApairing} in these settings as well.
\end{remark}

There are other bimodules associated to $W$, which we describe in more detail in the subsequent subsection; we develop the flavor here.  For instance, the type AA bimodule $\CFAA(W)$ has a module structure consisting of two commuting right $\cA_\infty$-module structures over $\cA(F_L)$ and $\cA(F_R)$.  In other words, there are operations of the form 
\[m_{1,i,j}\co\CFAA(W)\otimes\big( \cA(F_L)^{\otimes i} \otimes \cA(F_R)^{\otimes j}\big)\to\CFAA(W)\]
with appropriate notions of compatibility as described in \cite[Definition 2.2.55]{LOTbimodules} and \cite[Definition 2.2.38]{LOTbimodules}.  Alternatively, we can think of $\CFAA(W)$ as a type A structure (right $\cA_\infty$-module) over $\cA(F_L) \otimes \cA(F_R)$.  We will not make a distinction between these two viewpoints.  

We also have the type DD bimodule $\CFDD(W)$, which is a type D structure over $\cA(-F_L) \otimes \cA(-F_R)$.   This comes equipped with a type D structure map denoted 
\[\delta^{1,1}\co \CFDD(W) \to \big(\cA(-F_L) \otimes \cA(-F_R)\big)\otimes\CFDD(W).\]  Again, there is an extension of the box tensor product and suitable pairing theorems.  

Moreover, there is a way to convert between type D structures and type A structures using bimodules \cite{LOTbimodules}.   As an example, there is a type AA identity bimodule $\CFAA(\bbI)$ converting $\CFD(W)$ to $\CFA(W)$ by way of the operation $\CFAA(\bbI)\boxtimes -$ thought of as a functor between ${}^{\cA(-F)}\mathsf{Mod}$ and $\mathsf{Mod}_{\cA(F)}$, which are appropriate categories of type D structures over $\cA(-F)$ and type A structures over $\cA(F)$ respectively, defined in Section~\ref{sec:catofmod}.  This duality will be discussed more precisely in Section~\ref{subsec:dualityrevisited}.   

\subsection{Categories of modules}\label{sec:catofmod} Following the notation set out by Lipshitz, Ozsv\'ath and Thurston, we review the relevant categories of modules that have appeared to this point \cite{LOTbimodules}. Given any fixed dg-algebra $\cA$, write $\mathsf{Mod}_{\cA}$ for the category of $\Ztwo$-graded type A structures over $\cA$ (i.e., $\Ztwo$-graded right $\cA_\infty$-modules over $\cA$) which are homotopy equivalent to a bounded structure.  We write ${}^{\cA}\mathsf{Mod}$ for the category of $\Ztwo$-graded left type D structures over $\cA$ which are homotopy equivalent to bounded type D structures; as mentioned above, we will not make the distinction between a type D structure and the associated dg-module.  By abuse, we will refer to the dg-module arising from a bounded type D structure as a bounded dg-module.  As a result, the type D structure $\CFD(W)$ and the type A structure $\CFA(W)$ are members of the categories  ${}^{\cA(-F)}\mathsf{Mod}$ and $\mathsf{Mod}_{\cA(F)}$, respectively. This notation generalizes to bimodules so that for manifolds $W$ with $\partial W = F_L\amalg F_R$
\[\CFDA(W) \in  {}^{\cA(-F_L)}\mathsf{Mod}_{\cA(F_R)}, \
\CFDD(W) \in  {}^{\cA(-F_L),\cA(-F_R)}\mathsf{Mod}, \
\CFAA(W) \in \mathsf{Mod}_{\cA(F_L),\cA(F_R)},\]
categories of bounded, up to homotopy, type DA, type DD, and type AA structures respectively. 

We now define $\CFAA(W)$ (respectively $\CFDD(W)$) using restriction (respectively induction) functors. 

Given a bordered manifold $W$ with disconnected boundary 	$F_L\amalg F_R$ there is a closely related bordered manifold  $W_\dr$ with connected boundary $F_L\# F_R$ such that attaching a 2-handle to $W_\dr$ along the neck of the connected sum yields $W$.  We can also view this in terms of bordered Heegaard diagrams as follows.  Let $\cH$ be an arced bordered Heegaard diagram for $W$.  We can obtain a bordered Heegaard diagram for $W_\dr$ by removing an open neighbourhood of the arc $\bfz$; after an appropriate choice of basepoint, the result is a bordered Heegaard diagram describing $W_\dr$ which has pointed matched circle $\cZ_L \# \cZ_R$ describing $F_L\# F_R$.   

First, $\CFAA(W)$ and $\CFA(W_\dr)$ have the same generating sets.  To define the type AA structure, we note that the above construction gives rise to a dg-algebra morphism  
	\[\varphi_\dr\co \cA(F_L) \otimes \cA(F_R) \to \cA(F_L\# F_R)\] 
coming from an obvious inclusion map.  Using $\varphi_\dr$, we obtain the operations $m_{1,i,j}$ mentioned above for the bimodule $\CFAA(W)$ from the $\cA_\infty$-products \[m_{i+j+1}\co \CFA(W_\dr)\otimes \cA(F_L\# F_R)^{\otimes (i+j)} \to \CFA(W_\dr).\]  More generally, this process converts a type AA structure over $\cA(F_L \# F_R)$ to one over $\cA(F_L) \otimes \cA(F_R)$,  giving rise to a functor 
\[\operatorname{Res}_{\varphi_\dr}\co \mathsf{Mod}_{\cA(F_L\# F_R)} \to \mathsf{Mod}_{\cA(F_L),\cA(F_R)}.\]
In order for this to be well-defined, we need to place a $\Ztwo$-grading on $\CFAA(W)$.  We use $\gr_{AA}$, the $\Ztwo$-grading on $\CFAA(W)$ induced from $\CFA(W_\dr)$ by the restriction functor.  We remark that the construction above is equivalently described by the functor $\operatorname{Rest}_{\varphi_\dr}$ given by \[-\boxtimes\big({}^{\cA(F_L\# F_R)}[\varphi_{\dr}]_{\cA(F_L)\otimes\cA(F_R)}\big),\] where this module is a 1-dimensional vector space (over $\F$) defined in \cite[Section 2.4.2]{LOTbimodules}. Finally, note that $\gr_{AA}$ depends on a choice of order of $F_L$ and $F_R$; we have chosen $F_L$ to come first and $F_R$ second.  We have that $\sfMod_{\cA(F_L) \otimes \cA(F_R)}$ is equivalent to the category of type AA structures, $\sfMod_{\cA(F_L), \cA(F_R)}$, so we will use the notations interchangeably. 

We now are interested in defining $\CFDD(W)$.  For this construction, we work with a parameterization of the boundary of $W_\dr$ as $\cZ = \cZ_R \# \cZ_L$.  This ordering is simply a choice of convention for compatibility with the way that we study the behavior of gradings under pairing (see Proposition~\ref{prop:CFAACFDDgrpairing} and Remark~\ref{rmk:pairingorderchoices}).  Note that $-\cZ = -\cZ_L \# -\cZ_R$.  Writing $F = F(\cZ)$, there is now a natural dg-algebra morphism
\[\varphi^{\textup{split}}\co \cA(-F)\to \cA(-F_L) \otimes \cA(-F_R) \]
where all algebra elements for which the partial permutation $\phi$ does not respect the splitting of $-\cZ$ are sent to $0$. This gives rise to a type DD operation
\[\delta^{1,1}\co \CFDD(W) \to (\cA(-F_L) \otimes \cA(-F_R))\otimes \CFDD(W_\dr)\] determined by the operation $\delta^1$ from the type D structure $\CFD(W_\dr)$, where as before, the generating sets for both modules are identical.  Again, this construction generalizes to yield the induction functor 
\[\Ind_{\varphi^{\textup{split}}} \co {}^{\cA(-F)}\sfMod\to {}^{\cA(-F_L),\cA(-F_R)}\sfMod.\]
This construction coincides with the functor $\operatorname{Induct}_{\varphi^{\textup{split}}}$ defined by 
	\[\big({}^{ \cA(-F_L)\otimes \cA(-F_R)}[\varphi^{\textup{split}}]_{\cA(-F)}\big)\boxtimes-\] 
as in \cite[Section 2.4.2]{LOTbimodules}. We let $\gr_{DD}$ denote the $\Ztwo$-grading on $\CFDD(W)$ induced from $\CFD(W_\dr)$ by the induction functor.  Again, note that $\gr_{DD}$ depends on a choice of order of $F_L$ and $F_R$.

\subsection{Duality}\label{subsec:dualityrevisited} The bimodules described above also give rise to functors, via pairing, that describe the duality between type D structures and type A structures. Consider the {\it identity bimodule} $\CFAA(\mathbb{I})\in\sfMod_{\cA(-F),\cA(F)}$ associated with the product $F\times I$. This is described in detail in \cite{LOTbimodules}; see also Figure \ref{fig:sample-build} in Section \ref{subsec:hodgeduality}.  

This object may be regarded as a functor \[\CFAA(\mathbb{I)}\boxtimes - \co{}^{\cA(-F)}\sfMod \to\sfMod_{\cA(F)}\] inducing a powerful duality: 
\begin{theorem}[Lipshitz-Ozsv\'ath-Thurston {\cite[Corollary 1.1]{LOTbimodules}}]\label{thm:AA-duality}
The functor $\CFAA(\mathbb{I})\boxtimes -$ induces an equivalence of categories \[{}^{\cA(-F,i)}\sfMod \simeq\sfMod_{\cA(F,-i)}\] for each strands grading $i$. 
\end{theorem}

We remark that the duality in Theorem~\ref{thm:AA-duality} can be thought of as a categorification of Hodge duality.  This is made precise in Theorem~\ref{thm:hodge} below and will be a key step in the proof of Theorem~\ref{thm:main}.  
 
\subsection{Hochschild homology}\label{sub:HH}
Given $W$ with two boundary components parameterized by $-\cZ$ and $\cZ$ respectively, we can form a chain complex $CH_*(\CFDA(W))$ with generators of the form $\bfx^\circ = \iota \cdot \bfx \cdot \iota$ for $\bfx \in \CFDA(W)$ and $\iota$ a minimal idempotent. The grading of $\bfx^\circ$ is $\gr_{DA}(\bfx)+i$, where $i$ denotes the strands grading of $\bfx$. The differential on this complex is defined in \cite[Section 2.3.5]{LOTbimodules}. We denote the homology of $CH_*(\CFDA(W))$ by $HH_*(\CFDA(W))$.  An analogous construction holds more generally for arbitrary bounded type DA bimodules over $\cA(\cZ)$, and this construction agrees with Hochschild homology \cite[Proposition 2.3.54]{LOTbimodules}. 

By \cite[Theorem 14]{LOTbimodules}, we have that
	\[ \HFK(W^\circ, K) \cong HH_*(\CFDA(W)) \]
where $(W^\circ,K)$ is the generalized open book obtained by identifying $F(-\cZ)$ with $F(\cZ)$ and $K$ is the binding; see \cite[Construction 5.20]{LOTbimodules} for the definition of a generalized open book. In the case that $W$ is a mapping cylinder, this is an honest open book. Moreover, the above isomorphism identifies the strands grading $i$ with the Alexander grading on knot Floer homology, i.e.,
		\[ \HFK(W^\circ, K,i) \cong HH_*(\CFDA(W,i)), \]
and identifies the relative Maslov grading modulo 2 with the $\Ztwo$-grading defined above.

\subsection{Fine print} In order to introduce the relevant objects in a streamlined and accessible manner, we have opted to defer the majority of the associated proofs, such as showing the grading functions such as $\gr_A$, $\gr_D$, and $\gr_{DA}$ are $\Ztwo$-gradings which are invariants, to Appendix \ref{sec:gradings}. We prove the desired results by identifying these with a generalization of Petkova's reduction \cite[Section 3]{Petkovadecat} of the non-commutative gradings of \cite[Chapter 10]{LOT}. Impatient and/or suspicious readers may consult Figure \ref{fig:grading-summary} for an overview of how this identification is ultimately established. 

\begin{figure}[ht!]
\begin{tikzpicture}
\node at (2,3) {$\cA$};\node at (6,3) {$\CFA$};\node at (10,3) {$\CFAA$};
\node at (6,1) {$\CFD$};\node at (10,1) {$\CFDD$};\node at (14,1) {$\CFDA$};

\draw [ultra thick] (1.25,2.5) rectangle (2.75,3.5);
\draw[ultra thick, ->,>=latex] (2.75,3) -- (5.25,3);
\draw [ultra thick] (5.25,2.5) rectangle (6.75,3.5);
\draw[ultra thick, ->,>=latex] (6.75,3) -- (9.25,3);\node at (8,3.375) {$\operatorname{Res}$};
\draw [ultra thick] (9.25,2.5) rectangle (10.75,3.5);

\draw [ultra thick] (5.25,0.5) rectangle (6.75,1.5);
\draw[ultra thick, ->,>=latex] (6.75,1) -- (9.25,1);\node at (8,1.375) {$\operatorname{Ind}$};
\draw [ultra thick] (9.25,0.5) rectangle (10.75,1.5);
\draw[ultra thick, ->,>=latex] (10.75,1) -- (13.25,1);\node at (12,1.375) {$\boxtimes\CFAA(\bbI)$};
\draw [ultra thick] (13.25,0.5) rectangle (14.75,1.5);

\draw[ultra thick, ->,>=latex] (6,2.5) -- (6,1.5);\node at (5.625,2) {$\boxtimes$};
\end{tikzpicture}
\caption{A graphical summary of the operations used in the identification of the two $\Z/2\Z$-gradings in bordered Floer homology; details in Appendix \ref{sec:gradings}.}\label{fig:grading-summary}
\end{figure}
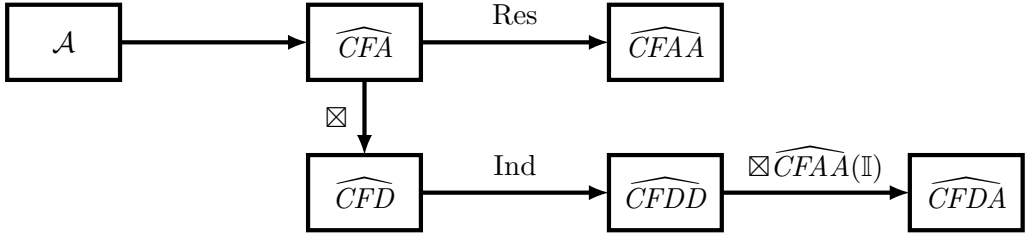

\begin{remark}
We will not need to make use of $\spinc$-structures on our three-manifolds throughout the body of the paper, only in Appendix~\ref{sec:gradings}.  Therefore, we will not make mention of these until then.
\end{remark}

\subsection{Examples}\label{sec:examples}
Before turning to the proofs of our main results, we consider some examples of the objects introduced above; these will establish Theorem \ref{thm:stronger} and Theorem \ref{thm:same}. In particular, we consider the strength of the bimodule $\CFDA$ as an invariant of a Seifert surface complement.  

To make this precise, let $K$ be a knot in an integer homology sphere $Y$ with Seifert surface $F^\circ$.  Let $W = (Y \smallsetminus \nu (F^\circ)) \cup D^2 \times I$ and choose parameterizations of the boundary of $\partial W$ such that gluing according to this parameterization gives $Y_0(K)$.  Note that by construction, both boundary components are parameterized by the same pointed matched circle.  As in Section \ref{sec:linearalgebra}, $F$ denotes the resulting capped off surface in $Y_0(K)$, and thus we can think of $W$ as a cobordism from $F$ to itself.    We define an arc $\bfz$ on $W$ given by $\{\text{pt}\} \times I \subset D^2 \times I$.  We write $\CFDA(K,F^\circ)$ to mean $\CFDA(W)$ with this additional data.   


We begin by showing that this invariant can distinguish certain knots with isomorphic knot Floer homology and Seifert forms. Let $K_1$ be the pretzel knot $P(3,5,3,-2)$ and let $K_2$ be $P(5,3,3,-2)$; these are illustrated in Figure \ref{fig:HFK}. The following proposition establishes Theorem~\ref{thm:stronger}.
\begin{proposition}\label{prop:noparameterization}
	The fibered knots $K_1$ and $K_2$ have isomorphic knot Floer homology and Seifert form, but 
	\[ \CFDA(K_1,F_1^\circ) \not\simeq\CFDA(K_2,F_2^\circ) \] 
for all choices of parametrization on $W_1$ and $W_2$, where $F_i^\circ$ is the unique minimal genus Seifert surface for $K_i$ and $W_i$ is as defined above.
\end{proposition}

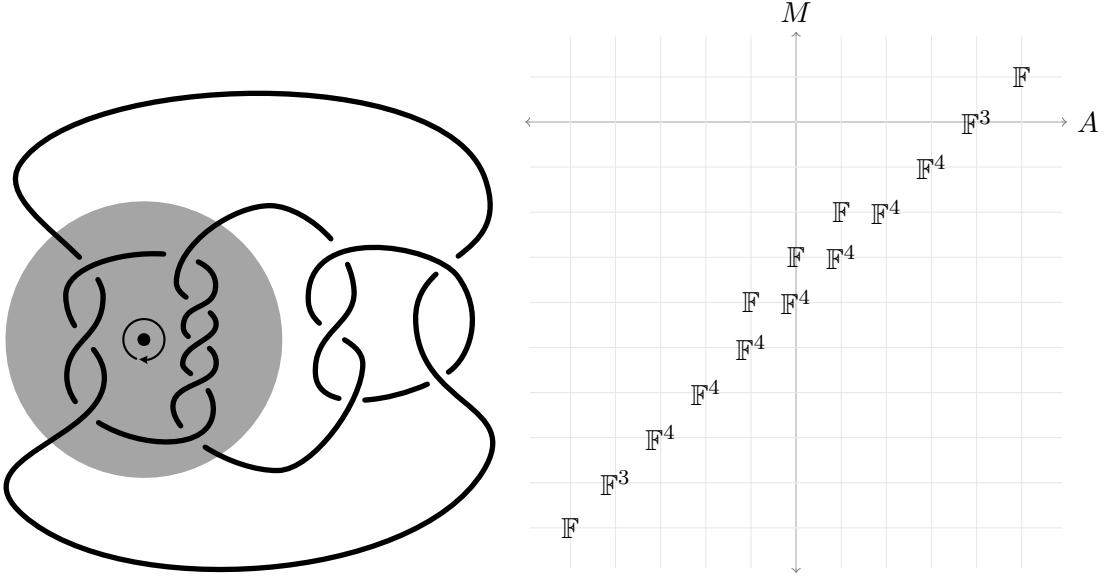
\begin{figure}[htb!]
\vspace{5pt}
\includegraphics[scale=0.5]{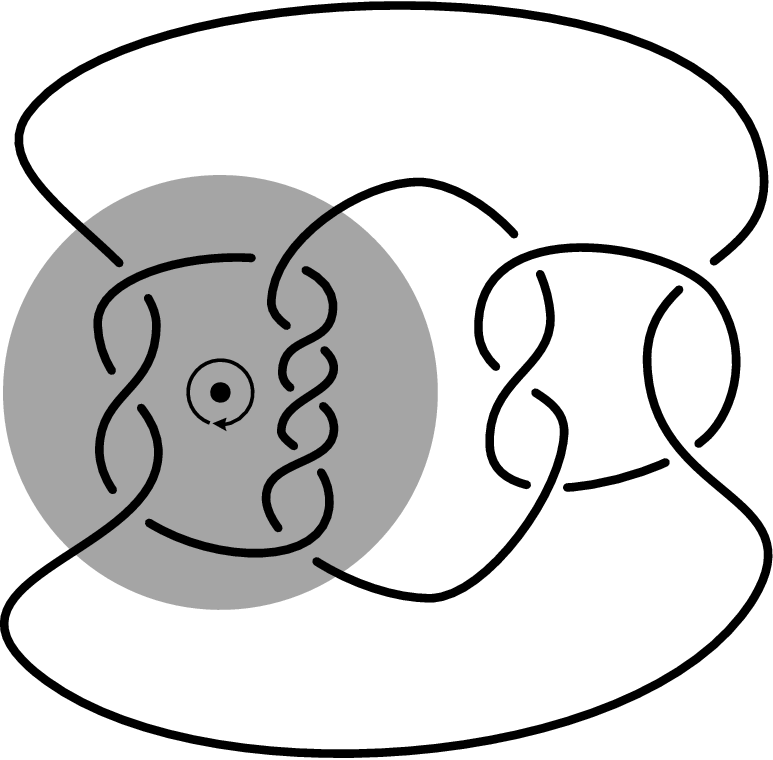}\quad
\labellist
\small \hair 2pt
\endlabellist
\centering
\begin{tikzpicture}[scale=0.6]

	\begin{scope}[thin, gray]
		\draw [<->] (-6, 0) -- (6, 0);
		\draw [<->] (0, -10) -- (0, 2);
	\end{scope}
	\draw (6, 0) node [right] {$A$};
	\draw (0, 2) node [above] {$M$};
	
	\draw[step=1, black!10!white, very thin] (-5.9, -9.9) grid (5.9, 1.9);

	\draw (-5, -9) node {$\F$};
	\draw (-4, -8) node {$\F^3$};
	\draw (-3, -7) node {$\F^4$};
	\draw (-2, -6) node {$\F^4$};
	\draw (-1, -5) node {$\F^4$};
	\draw (0, -4) node {$\F^4$};
	\draw (1, -3) node {$\F^4$};
	\draw (2, -2) node {$\F^4$};
	\draw (3, -1) node {$\F^4$};
	\draw (4, 0) node {$\F^3$};
	\draw (5, 1) node {$\F$};
	
	\draw (-1, -4) node {$\F$};
	\draw (0, -3) node {$\F$};
	\draw (1, -2) node {$\F$};
\end{tikzpicture}
\caption{A diagram for the $(3,5,3,-2)$-pretzel knot shown with a mutating disk illustrating that the $(5,3,3,-2)$-pretzel knot may be obtained by a positive mutation (left), and the homology $\widehat{HFK}\big(P(3, 5, 3, -2)\big) \cong \widehat{HFK}\big(P(5, 3, 3, -2)\big)$ (right).}
\label{fig:HFK}
\end{figure}
\begin{proof}
The knots $K_1$ and $K_2$ are distinct (see, for example \cite[Theorem 2.3.1]{Kawauchi}) and each $K_i$ is fibered (this follows from Gabai's classification of pretzel knots \cite{Gabai:Detecting}). Using the program of Droz \cite{Droz}, Allison Moore computed the knot Floer homology for each of these knots to be the bigraded vector space shown in Figure \ref{fig:HFK} \cite{Allison}. 

To see that the two knots have the same Seifert form, we notice that they are positive mutants, i.e., they are related by a mutation that does not require any changes in orientation. Then by \cite[Theorem 2.1]{KirkLivingston}, it follows that for an appropriate choice of basis, their Seifert matrices are the same.

Now, suppose for contradiction that there exist parameterizations of the boundaries of $W_1$ and $W_2$ such that $\CFDA(K_1,F_1)$ and $\CFDA(K_2,F_2)$ are homotopy equivalent. Clearly, if these modules are homotopy equivalent, each $W_i$ must be parameterized by the same pointed matched circle $\cZ$.  Now, because the modules are homotopy equivalent, by work of \cite[Theorem 1]{LOTfaithful}, since $W_1$ and $W_2$, as bordered three-manifolds, are each the mapping cylinder of a self-diffeomorphism $\overline{\phi}_i$ of $F(\cZ)$, we can conclude that $\overline{\phi}_1$ and $\overline{\phi}_2$ are isotopic rel $D^2$.  This implies that $K_1$ and $K_2$ are isotopic, which is a contradiction.  
\end{proof}

\begin{remark}
In general, the faithful linear-categorical action of the mapping class group defined in \cite{LOTfaithful} shows that $\CFDA(K,F^\circ)$, considered modulo homotopy and conjugation by the mapping class groupoid, is a complete invariant of the pair $(K,F^\circ)$ when $K$ is a fibered knot and $F^\circ$ is the fiber surface.  Recall that the fiber surface for a knot is unique, so this technique cannot be used to construct examples of non-isotopic Seifert surfaces for the same knot.   
\end{remark}

Given a pattern knot $P$ and a companion knot $C$, write $P(C)$ for the knot resulting from the satellite  of $C$ by the pattern $P$. Towards the proof of Theorem ~\ref{thm:same}, we will be interested in an infinite family of knots of the form  $K_i=P(C_i)$, for $i\in\Z$, where $\{C_i\}_{i\in\Z}$ is an infinite family of thin\footnote{See \cite{MOQA} for a definition of thinness.} knots with identical knot Floer homology. The existence of such a family $\{C_i\}_{i\in\Z}$ is provided by \cite[Proposition 11]{GW2013} or by  \cite[Section 6.4]{HW2014}. (Note that the latter produces an infinite family for which $\HFK(C_i)\cong\HFK(6_1)$.) Since these $C_i$ have thin knot Floer homology they must have homotopy equivalent $\mathit{CFK}^-$ \cite[Theorem 4]{Petkova}. The knots $C_i$ are distinct for all $i\in\Z$ by considering the Turaev torsion of their branched double covers \cite[Section 4]{GW2013} or by calculating their Khovanov homology \cite[Section 3]{HW2014}. The following proposition establishes Theorem~\ref{thm:same}.  

\begin{proposition}
Given any infinite family of distinct thin knots $\{C_i\}_{i\in\Z}$ with isomorphic knot Floer homology, let $P$ be a non-trivial nullhomologous pattern knot, and set $K_i=P(C_i)$ for $i\in\Z$.  Then the $K_i$ form  an infinite family of distinct knots  for which there is a choice of Seifert surface $F_i^\circ$ for $K_i$ and a choice of parametrization for $W_i=(S^3 \smallsetminus \nu (F_i^\circ)) \cup D^2 \times I$ such that \[ \CFDA(K_i,F_i^\circ) \simeq \CFDA(K_j,F_j^\circ) \] for all $i$ and $j$.
\end{proposition}

\begin{proof}
	Fix a non-trivial nullhomologous pattern $P \subset S^1 \times D^2$. Note that for such a pattern $P$, there exists a Seifert surface $F^\circ \subset S^1 \times D^2$ for $P$. Let $P(C_i)$ be the satellite with pattern $P$ and companion $C_i$. Then the image $P(F^\circ)$ gives a Seifert surface for $P(C_i)$. We denote this Seifert surface by $F_i^\circ$. Let $W_i=(S^3 \smallsetminus \nu (F_i^\circ)) \cup D^2 \times I$.
We will show that there is a choice of parametrization on the boundary such that $\CFDA(W_i)$ is independent of $i$.

Consider $X=(S^1 \times D^2 \smallsetminus \nu (F^\circ)) \cup D^2 \times I$, attached such that $X$ is the exterior of the capped-off Seifert surface in 0-surgery on $P$.  Observe that $X$ has three boundary components.    Roughly, the idea is that the bordered invariants for $S^3 \smallsetminus \nu (C_i)$ will be identical for all $i$, and thus by an appropriate pairing theorem, we obtain identical bordered invariants when we glue $X$ to $S^3 \smallsetminus \nu (C_i)$. We make this more precise below.

We begin with $X$ and a fixed parametrization on $F \times \{0\}$ and $F \times \{1\}$, where $F$ denotes the capped off Seifert surface for $P$, together with a path $\gamma$ in $X_i$ from the basepoint $z_L$ of the parameterization for $F \times \{0\}$ to the basepoint $z_R$ of the parameterization for $F \times \{1\}$ .  We also parametrize the boundary component of $X$ coming from $\d (S^1 \times D^2)$. We denote the basepoint of this parametrization by $z_M$, and we take a path $\eta$ from $z_M$ a point in $\gamma$. Associated with this data is a trimodule $\widehat{\mathit{CFDAA}}(X)$, where $F \times \{0\}$ is treated in a type D manner, and $F \times \{1\}$ and $\d (S^1 \times D^2)$ are treated in a type A manner. See \cite{Hanselman} and \cite{DLM} for discussions of trimodules. Alternatively, one could use bordered sutured Heegaard Floer homology \cite{Zarev}.

By \cite[Theorem A.11]{LOT}, we have that $\CFD(S^3 \smallsetminus \nu (C_i))$, where we parameterize the boundary using a meridian and a 0-framed longitude, is completely determined by $\mathit{CFK}^-(C_i)$. Note that if we glue $S^3 \smallsetminus \nu (C_i)$ to $X$ along $\d(S^1 \times D^2)$ according to the given parametrization, the resulting bordered manifold is $W_i$. We consider the tensor product $\widehat{\mathit{CFDAA}}(X) \boxtimes \CFD(S^3 \smallsetminus \nu (C_i))$, where the pairing corresponds to identifying $\d(S^1 \times D^2)$ and $\d (S^3 \smallsetminus \nu (C_i))$, which is homotopy equivalent to $\CFDA(W_i)$ by an analogue of \cite[Theorem 12]{LOTbimodules}.  Since the homotopy type of $\CFD(S^3 \smallsetminus \nu (C_i))$ is independent of $i$ with the given parameterizations, it follows that $\CFDA(W_i)$ is as well.
\end{proof} 

\section{Categorification}
\label{sec:decategorification}
With this background in place, from now on we treat all differential graded objects as $\Ztwo$-graded.  We briefly recall some algebraic preliminaries before turning to the proofs of Theorem \ref{thm:main}, Theorem \ref{thm:donaldson}, and Theorem~\ref{thm:seifert}.  
These theorems rely on an algebraic formalism, described presently,  which will allow us to extend Petkova's  work (see Theorem \ref{thm:kernel}) and categorify Hodge duality in this context (see Theorem \ref{thm:hodge}). 

\subsection{Algebraic preliminaries}\label{sub:collect}
The dg-algebras $\cA(F)$ over which the bordered invariants are defined satisfy particularly nice properties, as described in \cite[Section 2]{LOTbimodules}. The results proved below, while in some sense general, depend on these properties in an essential way. As such, we collect all of the requisite properties here for easy reference.   

A dg-algebra $\cA$ is {\it strictly unital} if there exists an identity element $\bfI\in\cA$ satisfying $\bfI\cdot a=a\cdot\bfI = a$ \cite[Definition 2.1.1]{LOTbimodules}. Observe that $\bfI$ is the sum of the minimal idempotents. 

Such a dg-algebra is naturally an algebra over the idempotent subring $\cI$. An {\it augmentation} of $\cA$ is a dg-algebra map $\epsilon\co\cA\to\cI$ satisfying $\epsilon(\bfI)=\bfI$ \cite[Definition 2.1.1]{LOTbimodules}. We will make use of the augmentation $\epsilon$  specified by the quotient to $\cI$, so $\cA/\ker(\epsilon)\cong\cI$.  Such a dg-algebra $\cA$ is called {\it augmented}. Strictly speaking, we should specify the pair $(\cA,\epsilon)$; since $\epsilon$ has been fixed it will be suppressed from the notation. The notation $\cA_+=\ker(\epsilon)$ will be used below; $\cA_+$ is the {\em augmentation ideal}.
An augmented dg-algebra $\cA$ is called {\em nilpotent} if it has a nilpotent augmentation ideal $\cA_+$, in the sense that $(\cA_+)^n=0$ for some integer $n$ \cite[Definition 2.1.8]{LOTbimodules}.  

All of the above properties are preserved under tensor product. Namely, if $\cA_1$ and $\cA_2$ are strictly unital, augmented, and nilpotent dg-algebras then so is $\cA_1\otimes\cA_2$, where the tensor is taken over $\F$. Note that  $\cA_1\otimes\cA_2$ is viewed as an algebra over $\cI_1\otimes\cI_2$. 

In summary, we restrict attention to dg-algebras that are strictly unital and augmented (compare \cite[Convention 2.1.5]{LOTbimodules}) in the strong sense recorded above for the remainder of this paper.

We follow the setup described by Khovanov \cite{Khovanov2010} to compute $K_0(\cA)$, the Grothendieck group for a dg-algebra $\cA$. Given a dg-algebra $\cA$ form the homotopy category of left dg-modules over $\cA$ and let  $\P(\cA)$ denote the full subcategory of compact, projective modules. Consider the free abelian group $G$ generated by symbols $[A]$ where $A$ is an object in $\P(\cA)$. The Grothendieck group $K_0(\cA)$ is the quotient of $G$ obtained by adding the relation $[A]=[B]+[C]$ whenever there is an exact triangle in $\P(\cA)$ of the form
\[
\begin{tikzpicture}
  \matrix (m) [matrix of math nodes,row sep=2em,column sep=2em,minimum width=2em]
  {
  &A& \\
    B & &C \\
 };
  \path[-latex]
  (m-1-2)  edge (m-2-1)
    (m-2-1) edge node [above] {$\{1\}$}(m-2-3)
    (m-2-3)  edge (m-1-2);
  \end{tikzpicture}
\] where $\{1\}$ indicates that the grading is shifted by one. Note that this identifies the object $A$ with the mapping cone on $B\to C$, up to homotopy. As an immediate consequence, it follows that $[A\{1\}]=-[A]$, that is $[A]$ and its inverse are related by the grading shift.  

Note that a functor $\sfF\co\P(\cA) \to \P(\cA')$ (or more generally between a pair of triangulated categories) induces a homomorphism $K_0(\sfF)$ between Grothendieck groups.  We recall from Section~\ref{sec:background} that we defined $\CFDD(W)$ in terms of $\CFD(W_\dr)$ by an induction functor coming from a dg-algebra morphism.  We point out that more generally, given a dg-algebra morphism $\varphi: \cA_1 \to \cA_2$, one obtains an induction functor $\Ind_\varphi$ between the associated categories of type D modules over the $\cA_i$.  
   
\begin{definition}Two dg-algebras $\cA_1$ and $\cA_2$ are {\em $K_0$-equivalent} if there is a dg-algebra morphism $\varphi$, called a {\em $K_0$-equivalence}, such that $\Ind_\varphi\co \mathcal{P}(\cA_1) \to \mathcal{P}(\cA_2)$ is a $K_0$-equivalence of categories, that is, $K_0(\Ind_\varphi)$ is an isomorphism.  By abuse of notation, we write $K_0(\varphi)$ for the map $K_0(\Ind_\varphi)$.     
\end{definition}
\noindent In particular, any dg-algebra isomorphism is a $K_0$-equivalence.  

Some more notation is required in order to state the next result.  Let $\iota$ be a minimal idempotent of $\cA$.  We define an {\em elementary projective module} $N$ to be one of the form $\cA \iota$, where the differential and grading is precisely the one coming from $\cA$.  In terms of type D structures, this means that considered as a module over the idempotent subalgebra, $N$ is a one-dimensional $\F$-vector space where $\iota$ is the only minimal idempotent which acts non-trivially; further, $N$ is supported in grading 0 and has $\delta^1 \equiv 0$.  We will write $\F \iota$ for this type D structure, even though the action of $\iota$ is on the {\em left}.  This is to make statements like $\Ind_\epsilon(\cA \iota) =\F \iota$ consistent with the placement of the idempotents.  We have an analogous convention for type A structures.    

\begin{lemma}\label{lem:idempotentdecomp}
Let $\cA$ be a differential $\Ztwo$-graded algebra equipped with an idempotent decomposition $\cA_1 \oplus \cA_2$.  Then, $K_0(\cA)$ is canonically isomorphic to $K_0(\cA_1) \oplus K_0(\cA_2)$.  Consequently, if $\cI$ is the idempotent subalgebra of $\cA$, such that $\cI = \oplus^n_{i=1} \F \bfx_i$, where $\bfx_i$ are the minimal idempotents, then $K_0(\cI)$ is the free abelian group with basis $[\F \bfx_i]$.    
\end{lemma}
\begin{proof}
The first statement is immediate.  For the second case, we have $K_0(\cI) = \oplus^n_{i=1} K_0(\F \bfx_i)$.  As discussed, $K_0(\F \bfx_i) \cong \mathbb{Z}$, generated by $[\F \bfx_i]$.  The result follows.  
\end{proof}

We are now in a position to state (and outline the proof of) a key result.      

\begin{theorem}\label{thm:decat-general}
Suppose that $\cA$ is a dg-algebra satisfying the properties described above; in particular, $\cA$ is strictly unital and augmented. Then, the augmentation $\epsilon \co \cA \to \cI$ is a $K_0$-equivalence.  In particular, $K_0(\cA)$ is free abelian, generated by the symbols $[\cA\iota]$, where $\iota$ is a minimal idempotent.  
\end{theorem}

The proof of this fact follows that of \cite[Theorem 21]{Petkovadecat} (see also \cite[Theorem 3]{LOTfaithful}), where the result is proved in a special case. Thus, we only outline the argument.  
We also remark that this result is well known to experts in various other contexts.  

\begin{proof}[Outline of the proof of Theorem \ref{thm:decat-general}]
The proof has two major steps.  The first is to show that the category $\P(\cA)$ is generated by elementary projective modules, as defined above.  This is achieved by induction using the triangle 
\[
\cA \iota(\bfx)\{\gr(\bfx)\} \to N \to N/\cA \bfx,
\]
where $\bfx$ is a cycle in $N$, the existence of which is guaranteed by the boundedness of the type D structure on $N$.  It follows that the symbols $[\cA \iota]$ generate $K_0(\cA)$ as an abelian group.    

The second step is to consider the induction functor for the augmentation map $\epsilon \co \cA \to \cI$.  Observe that 
\begin{equation}\label{eqn:K0epsilon}
K_0(\epsilon)([\cA \iota]) = [\Ind_\epsilon(\cA \iota)] = [\cI \iota] = [\F \iota].
\end{equation}
We remark that since the type D structure maps vanish on the elementary projective $\cA \iota$, they vanish on $\Ind_{\varphi_\dr}(\cA \iota)$ as well.  By Lemma~\ref{lem:idempotentdecomp}, the symbols $[\F \iota]$ give a basis for the free abelian group $K_0(\cI)$.  Since the symbols $[\cA \iota]$ generate $K_0(\cA)$, it follows that $K_0(\epsilon)$ is surjective, and thus $\epsilon$ is a $K_0$-equivalence.  The latter statement follows from \eqref{eqn:K0epsilon}.              
\end{proof}

\begin{remark}\label{rmk:checkelementary}
For strictly unital, augmented dg-algebras, Theorem~\ref{thm:decat-general} allows us to work with their Grothendieck groups more concretely.  In particular, given a compact projective module $N \in \P(\cA)$, 
in order to understand its image in the Grothendieck group, it suffices to understand its decomposition into elementary projectives.  Further, given a functor $\sfF$ from $\P(\cA_1)$ to $\P(\cA_2)$, in order to understand $K_0(\sfF)$, it suffices to compute $K_0(\sfF)(\cA_1 \iota)$ for the minimal idempotents $\iota$ in $\cA_1$.  
\end{remark}

Finally, we note that for a given bounded type D structure over a strictly unital, nilpotent, and augmented dg-algebra $\cA$, the associated left dg-module is projective according to \cite[Corollary 2.3.25]{LOTbimodules}. Since the type D structure $\CFD(W)$ associated with a bordered three-manifold $W$ is an invariant of $W$ up to homotopy, we will always work with bounded, compact, projective type D structures up to homotopy. This homotopy category of projectives will, replacing the notation from Section~\ref{sec:catofmod}, be denoted by ${}^\cA\sfMod$;  according to the discussion above $K_0({}^\cA\sfMod)$ is well-defined in this setting.  As we will often not distinguish between type D structures and their corresponding dg-modules, we have $K_0({}^\cA\sfMod) = K_0(\cA)$.  Note that other instances of ${}^\cA\sfMod$ in the bordered Floer homology literature also specify projective modules, however not all instances take equivalence up to homotopy as we do here. In particular, our   ${}^\cA\sfMod$ is denoted by $H({}^\cA\sfMod)$ in \cite{LOTbimodules}; we are abusing notation in order to lighten notation somewhat.  We similarly abuse notation and write $\sfMod_{\cA}$ for its homotopy category.  

\subsection{A reformulation in terms of type D structures}\label{subsec:kernel} 
Before going further, we need a quick discussion about gradings.  For each bordered three-manifold $W$ with parameterized boundary, the work of Section~\ref{sec:background} and Appendix~\ref{sec:gradings} equips the bordered Floer invariants associated to $W$ with a differential grading; as an absolute grading, this is not shown to be an invariant, as it depends on many choices.  In order to consider the bordered invariants as an object of the appropriate category (e.g., $^{\cA(\cZ)}\sfMod$), when working with the bordered Floer invariants of $W$, we assume this comes with an arbitrary fixed choice of absolute grading; in Lemma~\ref{lem:CFAAidgr} and its invocations, we will make an explicit choice of Heegaard diagram with orders and orientations of the relevant data, thus determining a particular lift.  Consequently, the decategorification of the bordered Floer invariants of $W$ will only be defined up to sign.  This is compatible with the fact that in Section~\ref{sec:linearalgebra} we have only defined Donaldson's TQFT up to sign.  Therefore, we will abuse notation and write statements like $\FDA(W) = K_0(\CFDA(W)) \in \Hom(\FDA(F_0),\FDA(F_1))$, when we actually mean that these objects agree up to this sign discrepancy.  A similar statement applies for the Pl\"ucker points constructed in Section~\ref{sec:linearalgebra}.  

With this discussion, we are now ready to begin our work towards Theorem~\ref{thm:main}.   A first essential result will be the following theorem, which generalizes a result of Petkova (see Theorem~\ref{thm:inakernel}).

 \begin{theorem}\label{thm:kernel}
Suppose that $W$ is a three-manifold with two parameterized boundary components $F_0$ and $F_1$, with genus $k_0$ and $k_1$ respectively, and an arc $\bfz$ connecting them.  Then, as an element in the Grothendieck group of ${}^{\A(-F_0) \otimes \A(-F_1)}\sfMod$, $[\CFDD(W)]$ is given by 
\[
|H_1(W,\partial W \cup \bfz)| \Lambda^{k_0+k_1} \Big(\ker(i_*\co H_1(\partial W) \to H_1(W) )\Big) \in \Lambda^*H_1(F_0) \otimes \Lambda^*H_1(F_1)
\]  
if $|H_1(W,\partial W \cup \bfz)| < \infty$ and $[\CFDD(W)] = 0$ otherwise.
\end{theorem}


Here we are utilizing the canonical isomorphism between $\Lambda^*H_1(\partial W) \cong \Lambda^*H_1(F_0) \otimes \Lambda^*H_1(F_1)$.  In order to prove Theorem~\ref{thm:kernel}, we begin with a review of Petkova's results for one boundary component.  For a pointed matched circle $\cZ$, there is a natural identification between the idempotents of $\cA(\cZ)$ and a basis for $\Lambda^*H_1(F(\cZ))$ given as follows.  

Recall for an element $j \in [2k]$, we write $M^{-1}(j) = \{a^-_j, a^+_j\}$, where $a^-_j \lessdot a^+_j$.  We let $\gamma_j \in H_1(Z',\bfa)$ denote the oriented segment from $a^-_j$ to $a^+_j$, where $Z'=Z \smallsetminus z$.  For elements of $H_1(Z',\bfa)$ of this form, there is an obvious assignment of these to elements in $H_1(F(\cZ))$.   For $\bfs \subset [2k]$, recall from Section \ref{sec:algebrabackground} that $I(\bfs)$ is the idempotent $a(S, S, \operatorname{id}_S)$ where $M(S)=\bfs$, that is, the idempotent consisting of horizontal strands for each $j \in \bfs$. Let $\bfs=\{j_1, \dots, j_n\}$ where $j_1 < \dots < j_n$. We then assign $I(\bfs)$ to $\gamma_{j_1} \wedge \dots \wedge \gamma_{j_n} \in \Lambda^*H_1(F(\cZ))$ and by abuse of notation write this wedge product as $\bigwedge_{j \in \bfs} \gamma_j$ (that is, we use the prescribed linear order induced by $[2k]$ when taking wedge products).   

Observe that from the above construction, $H_1(F(\cZ))$ comes equipped with an orientation by the ordered basis $\gamma_1,\ldots,\gamma_{2k}$.  Applying the definition of $-\cZ$ from Section~\ref{sec:algebrabackground},  we see that $H_1(-F(\cZ))$ is oriented by the ordered basis $-\gamma_1,\ldots,-\gamma_{2k}$.  Consequently, given a surface parameterized by $\cZ$, $H_1(F(\cZ))$ and $H_1(-F(\cZ))$ are canonically identified, even as oriented vector spaces  (i.e., a homology orientation is determined by the choice of parameterization of a surface and not the orientation).  In order to stay consistent with both Donaldson's TQFT and Petkova's work from \cite{Petkova} on the decategorification of type D structures, we treat $\bigwedge_{j \in \bfs} \gamma_j$ as an element of $\Lambda^*H_1(-F)$ (or equivalently, we treat the homology elements assigned to idempotents of $\cA(-\cZ)$ as elements of $\Lambda^*H_1(F)$).  Of course this choice of sign is simply cosmetic, but this is chosen to create a clearer parallel with the work of Section~\ref{sec:linearalgebra} (see especially Section~\ref{subsec:functorvalued}).  The displeased reader should think of this as counteracting the fact that if $W$ has boundary parameterized by $F$, then $\CFD(W)$ is an object in $^{\cA(-F)}\sfMod$.  

Finally, for a finitely-generated left module $N$ over $\cA(\cZ)$, choose a minimal set of generators $\mfS(N)$ for $N$ as an $\F$-vector space such that each element has a unique minimal idempotent $I(\bfx)$ that acts non-trivially on that generator by the identity.  For an element $\bfx \in \mfS(N)$, let $\overline\bfs_\bfx \subset [2k]$ be such that $I(\overline\bfs_\bfx )$ is the unique minimal idempotent such that $\bfx = I(\overline\bfs_\bfx) \cdot \bfx$. (Note that when $N$ is $\CFD(W)$ for a given bordered Heegaard diagram of $W$, we have that $\overline\bfs_\bfx=[2k]\setminus \bfs_\bfx$ where $\bfs_\bfx$ corresponds to the $\alpha$-arcs occupied by $\bfx$, as in Section \ref{sec:typeAbackground}.)


\begin{theorem}[Petkova {\cite[Theorem 1]{Petkovadecat}}]\label{thm:inadecat}
Let $\cZ$ be a pointed matched circle with $F(\cZ) = F$, a surface of genus $k$.  There is a well-defined isomorphism $K_0(\cA(\cZ)) \cong \Lambda^*H_1(-F)$ given by extending the map  
\begin{equation}
\Psi^{\cA(\cZ)}([\cA I(\bfs)]) = \bigwedge_{j \in \bfs} \gamma_j
\end{equation}
by linearity. Written more generally: 
\begin{equation}\label{eqn:decatmap}
\Psi^{\cA(\cZ)}([N])= \sum_{\bfx \in \mfS(N)} (-1)^{\gr(\bfx)} \bigwedge_{j \in \overline\bfs_\bfx} \gamma_j. 
\end{equation}
\end{theorem}
We repeat Petkova's proof --- but instead with the notation of Section~\ref{sub:collect} --- as the constructions will be necessary for the proof of Theorem~\ref{thm:kernel}.
\begin{proof}[Proof of Theorem \ref{thm:inadecat}]
We begin by first studying the idempotent subalgebra of $\cA(\cZ)$.  A basis for $K_0(\cI(\cZ))$ is given by $[\F I(\bfs)]$, for $\bfs \subset [2k]$, according to Lemma~\ref{lem:idempotentdecomp}.  Extending by linearity, the map
\begin{equation}\label{eqn:psioneidempotent}
\Psi^{\cI(\cZ)} \co K_0(\cI(\cZ)) \to \Lambda^* H_1(-F(\cZ)), \; \; \; \Psi^{\cI(\cZ)}([\F I(\bfs)]) = \bigwedge_{j \in \bfs} \gamma_j
\end{equation}
induces an isomorphism from $K_0(\cI(\cZ))$ to $\Lambda^* H_1(-F(\cZ))$.  More explicitly, the map $\Psi^{\cI(\cZ)}$ is given by $[N] \mapsto \sum_{\bfx \in \mfS(N)} (-1)^{\gr(\bfx)} \bigwedge_{j \in \overline\bfs_\bfx} \gamma_j$.  By \eqref{eqn:K0epsilon} we have that that 
\begin{equation}\label{eqn:psioneboundary}
\Psi^{\cA(\cZ)} = \Psi^{\cI(\cZ)} \circ K_0(\epsilon)
\end{equation}     
and since $K_0(\epsilon)$ is an isomorphism by Theorem~\ref{thm:decat-general}, the result follows.     
\end{proof}

\begin{remark}\label{rmk:decatstrands}
Keeping track of $|\bfs|$ for our idempotents, we obtain a more refined statement, namely that \eqref{eqn:decatmap} induces an isomorphism $K_0(\cA(\cZ,i)) \cong \Lambda^{k+i} H_1(-F)$. Indeed,
 if $I(\bfs) \in \cA(\cZ,i)$ then $|\bfs| = k+i$.  We may write this isomorphism either as $\Psi^{\cA(\cZ,i)}$ or we may write it as $\Psi^{\cA(\cZ)}$ which we restrict to $K_0(\cA(\cZ,i))$.    
\end{remark}

As a consequence of Theorem~\ref{thm:inadecat}, if $Y$ is a bordered three-manifold with parameterized boundary $F = F(\cZ)$, then $\Psi^{\cA(-\cZ)}([\CFD(W)]) \in \Lambda^*H_1(F)$.  Using these identifications, Petkova studies this element explicitly.    
\begin{theorem}[Petkova {\cite[Theorem 4]{Petkovadecat}}]\label{thm:inakernel} 
Let $Y$ be a bordered three-manifold with one boundary component $F$.  If $|H_1(Y,\partial Y)| = \infty$, then $[\CFD(Y)] = 0$ in $K_0(\cA(-F))$.  Otherwise, 
\begin{equation}\label{eqn:inaspan}
\Span(\Psi^{\cA(-F)}([\CFD(Y)]))= |H_1(Y,\partial Y)| \Lambda^k(\ker(i_*\co H_1(\partial Y) \to H_1(Y))).    
\end{equation}
Equivalently,  given a basis $\omega_1,\ldots,\omega_{k}$ for this kernel, when $|H_1(Y,\partial Y)| < \infty$, we have 
\begin{equation}\label{eqn:inaspan2}
\Psi^{\cA(-F)}([\CFD(Y)]) = \pm |H_1(Y,\partial Y)| \omega_1 \wedge \ldots \wedge \omega_{k}.
\end{equation}
\end{theorem}

\begin{remark}
In \cite[Section 7]{Petkovadecat}, Petkova orders the $\alpha$-arcs according to the orientation of $\d \cH$ (i.e., $(\partial\alpha_j)^- \lessdot (\partial \alpha_{j'})^-$ if and only if $j < j'$). In fact, the result is independent of this choice of order. Indeed, reordering the $\alpha$-arcs changes both the $\Ztwo$-grading and the order in which basis elements of $H_1(F)$ are wedged together when considering $\Psi^{\cA(-F)}([\CFD(Y)])$. It is straightforward to verify that the overall effect of these two changes leaves $\Span(\Psi^{\cA(-F)}([\CFD(Y)]))$ unchanged. See also Remark \ref{rem:idemconj}.
\end{remark}

As in Theorem \ref{thm:kernel}, let $W$ have $\partial W = F_0 \amalg F_1$.  By abuse of notation, we will often simply write $\Lambda^{k_0+k_1}(\ker(i_*\co H_1(\partial W) \to H_1(W)))$ to mean a choice of generator for this rank one subgroup of $\Lambda^*(H_1(\partial W))$, which we canonically identify with $\Lambda^* H_1(F_0) \otimes \Lambda^* H_1(F_1)$.  To prove Theorem~\ref{thm:kernel}, we would first like an extension of Theorem~\ref{thm:inadecat} for the algebra $\cA(-\cZ_0) \otimes \cA(-\cZ_1)$.  Once we understand this Grothendieck group, we will be able to use the induction functor and appeal to \eqref{eqn:inaspan}, since $\CFDD(W)$, for $W$ a bordered three-manifold with two boundary components $F(\cZ_0)$ and $F(\cZ_1)$, is defined in terms of $\CFD(W_\dr)$; see Section \ref{sec:catofmod}.  

We must first begin by computing $K_0(\cA(-\cZ_0) \otimes \cA(-\cZ_1))$ for  pointed matched circles $\cZ_0$ and $\cZ_1$.  Observe that the idempotent subalgebra of $\cA(-\cZ_0) \otimes \cA(-\cZ_1)$ is exactly $\cI(-\cZ_0) \otimes \cI(-\cZ_1)$.  We let $\epsilon^\otimes$ be the quotient map to the idempotent subalgebra.  As discussed in Section~\ref{sub:collect}, $\cA(-\cZ_0) \otimes \cA(-\cZ_1)$ is strictly unital and augmented.  Let $N$ be a type D module over $\cA(-\cZ_0) \otimes \cA(-\cZ_1)$.  For an element $\bfx \in \mfS(N)$, we write $\overline\bfs_{0,\bfx}$ and $\overline\bfs_{1,\bfx}$ to denote the subsets of $[2k_0]$ and $[2k_1]$ such that $I(\overline\bfs_{0,\bfx}) \otimes I(\overline\bfs_{1,\bfx})$ is the minimal idempotent which acts by the identity on $\bfx$.  Let $\mu_1,\ldots,\mu_{2k_0}$ (respectively $\nu_1,\ldots,\nu_{2k_1}$) be the homology classes in $H_1(F(\cZ_0))$ (respectively $H_1(F(\cZ_1))$) associated with elements of $[2k_0]$ (respectively $[2k_1]$) via $\cA(-\cZ_0)$ (respectively $\cA(-\cZ_1)$) described above.  

As in Theorem~\ref{thm:inadecat}, we begin with a discussion about $K_0$ for the idempotent subalgebras.  It again follows from Lemma~\ref{lem:idempotentdecomp} that the map 
\begin{equation}\label{eqn:psitwoidempotent}
\Psi^{\cI(-\cZ_0) \otimes \cI(-\cZ_1)}\co [\F (I(\bfs_0) \otimes I(\bfs_1))] \mapsto \bigwedge_{j \in \bfs_0} \mu_j \otimes \bigwedge_{j' \in \bfs_1}\nu_{j'}
\end{equation} 
induces an isomorphism from $K_0(\cI(-\cZ_0) \otimes \cI(-\cZ_1))$ to $\Lambda^*H_1(F(\cZ_0))\otimes \Lambda^*H_1(F(\cZ_1))$.  

\begin{proposition}\label{prop:DDdecat}
There is an isomorphism \[K_0(\cA(-\cZ_0) \otimes \cA(-\cZ_1)) \cong \Lambda^*H_1(F(\cZ_0)) \otimes \Lambda^* H_1(F(\cZ_1))\] given by 
\begin{equation}\label{eqn:tensordecatmap}
\Psi^{\cA(-\cZ_0) \otimes \cA(-\cZ_1)} \co [N] \mapsto \sum_{\bfx \in \mfS(N)} (-1)^{\gr(\bfx)} \bigwedge_{j \in \overline\bfs_{0,\bfx}} \mu_j \otimes \bigwedge_{j' \in \overline\bfs_{1,\bfx}} \nu_{j'}.  
\end{equation}
\end{proposition}
\begin{proof}
The result follows from the fact that $\Psi^{\cA(-\cZ_0) \otimes \cA(-\cZ_1)} = \Psi^{\cI(-\cZ_0) \otimes \cI(-\cZ_1)} \circ K_0(\epsilon^\otimes)$ combined with Theorem~\ref{thm:decat-general}.
\end{proof}

\begin{remark}\label{rmk:decatstrandstwoboundary}
It follows that $K_0(\cA(-\cZ_0,i_0) \otimes \cA(-\cZ_1,i_1))\cong \Lambda^{k_0 + i_0}H_1(F(\cZ_0)) \otimes \Lambda^{k_1 + i_1}H_1(F(\cZ_1))$. 
\end{remark}

To prove Theorem~\ref{thm:kernel} we must determine where $[\CFDD(W)] \in K_0(\cA(-\cZ_0) \otimes \cA(-\cZ_1))$ is sent under the identifications of Proposition~\ref{prop:DDdecat}.  One approach could be to repeat the proof of Theorem~\ref{thm:inakernel}, which corresponds to reading the homological information directly off a bordered Heegaard diagram.  We instead take a purely categorical approach which allows us to simply apply Petkova's work.  We recall the definition of $\CFDD(W)$ as $\Ind_{\varphi_\dr}(\CFD(W_\dr))$.  Note that while we know what $[\CFD(W_\dr)]$ corresponds to under the identification $K_0(\cA(-F_0 \# -F_1)) \cong \Lambda^*H_1(F_0) \otimes \Lambda^*H_1( F_1)$, that $[\CFDD(W)] = K_0(\varphi_\dr)([\CFD(W_\dr)])$, and that $K_0(\cA(-F_0 \# -F_1)) \cong K_0(\cA(-F_0) \otimes \cA(-F_1))$, we do not know that this isomorphism is realized {\em a priori} by $K_0(\varphi_\dr)$, and thus we cannot {\em a priori} determine $[\CFD(W_\dr)]$ from $[\CFDD(W)]$ and apply Theorem~\ref{thm:inakernel}.  We are therefore interested in understanding the effect of $K_0(\varphi_\dr)$ under the identifications of Theorem~\ref{thm:inadecat} and Proposition~\ref{prop:DDdecat}.  For notation, given $\bfs \subset [2k_0 + 2k_1]$, this determines $\bfs_0 \subset [2k_0]$ and $\bfs_1 \subset [2k_1]$ by the split matching on $-\cZ_0 \# -\cZ_1$.    

\begin{proposition}\label{prop:decatsquare}
There is a commutative square of isomorphisms 
\[\begin{tikzpicture}
  \matrix (m) [matrix of math nodes,row sep=2em,column sep=4em,minimum width=2em]
  {
   K_0(\cA(-F_0 \# -F_1))& K_0(\cA(-F_0) \otimes \cA(-F_1)) \\
   \Lambda^* H_1(F_0 \# F_1) & \Lambda^* H_1(F_0) \otimes \Lambda^*H_1(F_1)\\};
  \path[-latex]
    (m-1-1) edge node[left] {$\Psi^{\cA(-F_0 \#-F_1)}$} (m-2-1)
            edge node[above] {$K_0(\varphi_\dr)$}(m-1-2)
    (m-2-1) edge node[above] {$\cong$} (m-2-2)
    (m-1-2) edge node[right] {$\Psi^{\cA(-F_0) \otimes \cA(-F_1)} $} (m-2-2);
\end{tikzpicture}\]
where the bottom horizontal arrow is the obvious map.  
\end{proposition}

\begin{proof}
By Theorem~\ref{thm:decat-general} and Remark~\ref{rmk:checkelementary}, it suffices to determine where the map $K_0(\varphi_\dr)$ sends the symbol of an elementary projective $\cA(-F_0 \# -F_1)I(\bfs)$, where $\bfs \subset [2k_0+2k_1]$.  In other words, we are interested in the symbol 
\[
K_0(\varphi_\dr)([\cA(-F_0 \# -F_1)I(\bfs)]) = [\Ind_{\varphi_\dr} (\cA(-F_0 \# -F_1) I(\bfs))].
\]
Changing perspective, consider $\cA(-F_0 \# -F_1) I(\bfs)$ as its associated type D structure.  We write this module as $\F I(\bfs)$, where $\delta^1(\bfx) \equiv 0$ and $I(\bfs)$ is the unique minimal idempotent such that $I(\bfs) \cdot \bfx = \bfx$.  From the definition, we see $\Ind_{\varphi_\dr}(\F I(\bfs))$, as a type D structure over $\cI(-F_0) \otimes \cI(-F_1)$, is given by $\F (I(\bfs_0) \otimes I(\bfs_1))$, where the structure maps are trivial.  We can interpret this as identifying $K_0(\varphi_\dr)([\cA(-F_0 \# -F_1)I(\bfs)])$ as the symbol of the elementary projective $\cA(-F_0) \otimes \cA(-F_1)(I(\bfs_0) \otimes I(\bfs_1))$.  The diagram commutes by the definition of the maps in Theorem~\ref{thm:inadecat} and Proposition~\ref{prop:DDdecat} and, since these maps were shown to be isomorphisms, we conclude that $\varphi_\dr$ is a $K_0$-equivalence.     
\end{proof}

For ease of notation, we will omit the maps $\Psi$ from our notation and work implicitly with the identifications established in this section between Grothendieck groups and exterior algebras.

\begin{proof}[Proof of Theorem~\ref{thm:kernel}]
Let $W$ be a bordered three-manifold with boundary $F_0 \amalg F_1$.  Proposition~\ref{prop:DDdecat} identified $K_0(\cA(-F_0) \otimes \cA(-F_1))$ with $\Lambda^*H_1(F_0) \otimes \Lambda^*H_1(F_1)$ and, as discussed, we would like to compute $[\CFDD(W)]$ under this identification.   Recall that $\CFDD(W)$, as a type D structure over $\cA(-\cZ_0) \otimes \cA(-\cZ_1)$ is defined by $\Ind_{\varphi_\dr}(\CFD(W_\dr))$.  By the identifications given in Theorem~\ref{thm:inakernel} and Proposition~\ref{prop:decatsquare}, we can identify $[\CFDD(W)]$ with the element \[\pm |H_1(W_\dr,\partial W_\dr)|\Lambda^{k_0 + k_1}\ker(i_*\co H_1(\partial W_\dr)\rightarrow H_1( W_\dr)).\]  Notice that this is almost precisely what we want, except everything is in terms of $W_\dr$ instead of $W$. Recall that $W$ is obtained from $W_\dr$ by attaching a 2-handle along a nullhomologous curve (namely along the connected sum annulus).  Because the attachment is along a nullhomologous curve, we see that $|H_1(W, \partial W \cup \bfz)| = |H_1(W_\dr, \partial W_\dr)|$ and we have a canonical identification between $\Lambda^{k_0+k_1}\ker(i_*\co H_1(\partial W)\rightarrow H_1(W))$ and $\Lambda^{k_0+k_1}\ker(i_*\co H_1(\partial W_\dr)\rightarrow H_1( W_\dr))$.  This gives the desired result. 
\end{proof}


\subsection{Hodge duality}\label{subsec:hodgeduality}


In order to prove Theorem~\ref{thm:donaldson} we need to precisely formulate a passage between type DD bimodules and type DA bimodules in the present context. This will ultimately amount to a categorification of Hodge duality. 

Let $W$ be a cobordism from $F_0$ to $F_1$, that is, $\partial W= -F_0\amalg F_1$.  As shown in Theorem~\ref{thm:kernel} we are able to obtain the Pl\"ucker point $|\Gamma_W| \in \Lambda^*(H_1(-F_0) \oplus H_1(F_1))$ from $\CFDD(W)$.  As discussed in Section~\ref{sec:linearalgebra}, this allows us to obtain Donaldson's TQFT from bordered Floer homology.  
However, this is not sufficient for our purposes.  In order to obtain the commutative square in Theorem~\ref{thm:main} and thus obtain an independent proof that knot Floer homology categorifies the Alexander polynomial, we need to approach this problem more categorically.  

Suppose the boundary components $-F_0$ and $F_1$ of $W$ are parameterized by $-\cZ_0$ and $\cZ_1$, respectively.  Then, there is an associated functor
\[
\CFDA(W) \boxtimes_{\cA(-\cZ_0)} - \co{}^{\cA(-\cZ_0)}\sfMod \to {}^{\cA(-\cZ_1)}\sfMod
\]
where we treat $\CFDA(W)$ as a bimodule with a left (type D) action by $\cA(-\cZ_1)$ and a right (type A) action by $\cA(-\cZ_0)$.  The induced map on Grothendieck groups gives an element of \[\Hom\big(K_0(\cA(-\cZ_0)), K_0(\cA(-\cZ_1))\big)\] which we denote by $K_0(\CFDA(W))$.  By Theorem~\ref{thm:inadecat}, $K_0(\CFDA(W))$ may be regarded as an element of $\Hom(\Lambda^* H_1(F_0), \Lambda^*H_1(F_1))$.  Therefore, in order to prove Theorem~\ref{thm:donaldson}, we would like to show that under these identifications, $K_0(\CFDA(W))= \FDA(W)$.  In other words, the way that we will prove Theorem~\ref{thm:donaldson} in a manner that is suitable for Theorem~\ref{thm:main} is to show that the duality between type DD structures and type DA structures determined by Theorem~\ref{thm:AA-duality} categorifies the isomorphism between $\Lambda^*H_1(-F_0) \otimes \Lambda^*H_1(F_1)$ and $\Hom(\Lambda^*H_1(F_0), \Lambda^*H_1(F_1))$ (see Section \ref{Bimodules_and_Homs}, below, where this isomorphism is made explicit).  From this we will then be able to appeal to Theorem~\ref{thm:kernel} to obtain $\FDA(W)$, but still maintain enough structure to obtain contact with both the decategorified invariants and the Hochschild homology of $\CFDA(W)$ to prove Theorem~\ref{thm:main}.  To make this categorification easier to follow, we warm-up with the categorification of a version of Hodge duality.
  

\subsubsection{Type $A$ structures and Hodge duality}\label{subsubsec:typeAhodge}
Let $Q$ be a bounded type A structure over $\cA(-\cZ)$. This gives rise to a functor from ${}^{\cA(-\cZ)}\sfMod$ to $\mathsf{Kom}$, the homotopy category of $\Ztwo$-graded chain complexes, given by $N \mapsto Q \boxtimes_{\cA(-\cZ)} N$.  Therefore, $K_0(Q) \in \Hom(K_0(\cA(-\cZ)), K_0(\mathsf{Kom}))$.  Since $K_0(\mathsf{Kom}) \cong \mathbb{Z}$, with the isomorphism induced by the Euler characteristic, we see that $K_0(Q)$ yields an element of $K_0(\cA(-\cZ))^* \cong (\Lambda^*H_1(F(\cZ)))^*$.  

Given a bounded type D structure $N$ over $\cA(\cZ)$, $\CFAAid \boxtimes_{\cA(\cZ)} N$ is a bounded type A structure over $\cA(-\cZ)$.  The next proposition shows that the relationship between $[N] \in K_0(\cA(\cZ))$ and $K_0(\CFAAid \boxtimes_{\cA(\cZ)} N) \in K_0(\cA(-\cZ))^*$ is described essentially by Hodge duality. Recall, that $H_1(F(\cZ)) \cong H_1(-F(\cZ))$ came naturally equipped with two equivalent ordered bases: $\gamma_1,\ldots,\gamma_{2k}$ arising from $\cZ$, as described in Section~\ref{subsec:kernel}, and $-\gamma_1,\ldots,-\gamma_{2k}$ coming from $-\cZ$.  Note that $\dualize_{H_1(-F)} = \dualize_{H_1(F)}$, where $\dualize_{H_1(-F)}\co \Lambda^*H_1(-F) \to (\Lambda^*H_1(F))^*$ takes $v$ to the linear functional $x \mapsto \star(x \wedge v)$ as in Section~\ref{sec:linearalgebra}.  While this cosmetic insertion of signs may seem artificial, this is done to obtain $\FDA(-F) = \FDA(F)^*$ (see Section~\ref{subsec:functorvalued}), which is the duality axiom for topological quantum field theories (see \cite{Atiyah}).       


\begin{theorem}[Categorified Hodge duality]\label{thm:hodge}
Let $\cZ$ be a pointed matched circle.  Given $N \in {}^{\cA(\cZ)}\sfMod$ we have $K_0(\CFAAid \boxtimes_{\cA(\cZ)} N) = \dualize([N])$ as elements of $(\Lambda^*H_1(F))^*$.  
\end{theorem}

In order to prove Theorem~\ref{thm:hodge}, we first understand $\CFAAid$, the bimodule associated to the mapping cylinder $M(\id_{F(\cZ)})$ of the identity map on $F(\cZ)$, described in \cite{LOTbimodules}.  Recall the definition of $-\cZ \# \cZ$ from Section \ref{sec:algebrabackground} and let $\cZ^\#$ denote $-\cZ \# \cZ$.  
Let $\bfsbar = [2k] \setminus \bfs$ and define
\begin{equation}\label{eqn:grAAtheta}
\theta(\bfs) = |\bfs| + \sum_{j' \in \bfsbar} \# \{ j \in \bfs \mid j < j'\}.
\end{equation}
\begin{lemma}\label{lem:CFAAidgr}
There exists an arced, bordered Heegaard diagram $\cH$ for $M(\id_{F(\cZ)})$ and an ordering and orientation of the $\beta$-circles such that as $\Ztwo$-graded right $\cI(-\cZ)$-, right $\cI(\cZ)$-modules,
\[
\CFAAid = \bigoplus_{\bfs \subset [2k]} (I_{-\cZ}(\bfsbar) \otimes I_{+\cZ}(\bfs)) \mathbb{F} \{\theta(\bfs)\}.
\]    
\end{lemma}
\begin{proof}
For the following construction, we rely on Figure~\ref{fig:sample-build} for the picture which we now describe.  We will use the ``canonical'' arced bordered Heegaard diagram $\cH$ for the mapping cylinder of the identity map on $F(\cZ)$ as described in \cite[Definition 5.35]{LOTbimodules}.  Recall that to compute $\CFAAid$ we need to work with the drilled manifold, which is a solid handlebody with boundary $F(-\cZ) \# F(\cZ)$.  From $\cH$, we construct the Heegaard diagram $\cH_{\dr}$ for the drilled manifold which has exactly $4k$ $\alpha$-arcs, no $\alpha$-circles, and $2k$ $\beta$-circles.  While the picture is clear up to isotopy, we will give a very explicit diagram.  We must also choose an ordering and orientations for the $\beta$-circles in $\cH$ (which consequently induces these choices on $\cH_\dr$).  

Recall that we orient our $\alpha$-arcs so that $(\partial \alpha)^- \lessdot (\partial \alpha)^+$.  We order $\bfalpha$ as $\alpha^{\bot}_1,\ldots,\alpha^{\bot}_{2k}, \alpha^{\top}_{1},\ldots,\alpha^{\top}_{2k}$, where $\alpha^{\bot,-}_j = M_{\cZ^\#}^{-1}(j)^-$ and $\alpha^{\top,-}_j = M_{\cZ^\#}^{-1}(j+2k)^-$.  By our labeling, $\partial \alpha^{\bot}_j \subset -Z'$ and $\partial \alpha^{\top}_j \subset Z'$.  
We begin with a planar diagram describing the Heegaard surface in $\cH_\dr$ (i.e., we record the feet of the handles).  Isotope the $\beta$-circles such that in the (punctured) planar diagram given by removing the feet of the handles, a $\beta$-circle which intersects $\alpha_j^{\bot}$ (respectively $\alpha_j^{\top}$) does so in the component containing $\alpha_j^{\bot,-}$ (respectively $\alpha_j^{\top,+}$); further, we arrange that a given $\beta$-circle intersects exactly two $\alpha$-arcs (in one point each).  We orient the $\beta$-circles clockwise and order them so that $\beta_j \cap \bfalpha$ consists of two points, one on $\alpha^{\bot}_j$ and the other on $\alpha^{t\op}_{j}$.  

\begin{figure}[ht]
\vspace{10pt}
\subfigure[]{
\labellist
\pinlabel $-\cZ$ at 10 80
\pinlabel $\cZ$ at 74 80
\scriptsize
\pinlabel $1$ at 7 61
\pinlabel $2$ at -1 46
\pinlabel $3$ at 7 35
\pinlabel $4$ at 7 23
\pinlabel $1$ at 79 61
\pinlabel $2$ at 88 46
\pinlabel $3$ at 79 35
\pinlabel $4$ at 79 23

\endlabellist
\includegraphics[scale=1.5]{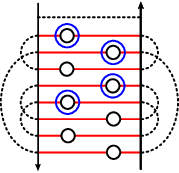}
}
\hspace{20pt}
\subfigure[]{
\labellist
\pinlabel {$-\cZ \# \cZ$} at 122 142
\scriptsize \pinlabel $1$ at 123 7
\pinlabel $2$ at 133 23
\pinlabel $3$ at 124 34
\pinlabel $4$ at 124 46
\pinlabel $5$ at 123 135
\pinlabel $6$ at 133 119
\pinlabel $7$ at 124 108
\pinlabel $8$ at 124 94
\pinlabel $\alpha^{\top}_1$ at 45 130
\pinlabel $\alpha^{\bot}_1$ at 45 10
\endlabellist
\includegraphics[scale=1.5]{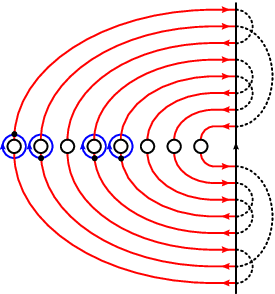}
}

\caption{Computing the grading on $\CFAAid$: Recall that we first pass to a bordered diagram with a single pointed matched circle (from the left-hand figure to the right-hand figure) via the drilling operation. The reader should check that in this example, the grading of the intersection point illustrated is 0.
}
\label{fig:sample-build}
\end{figure}
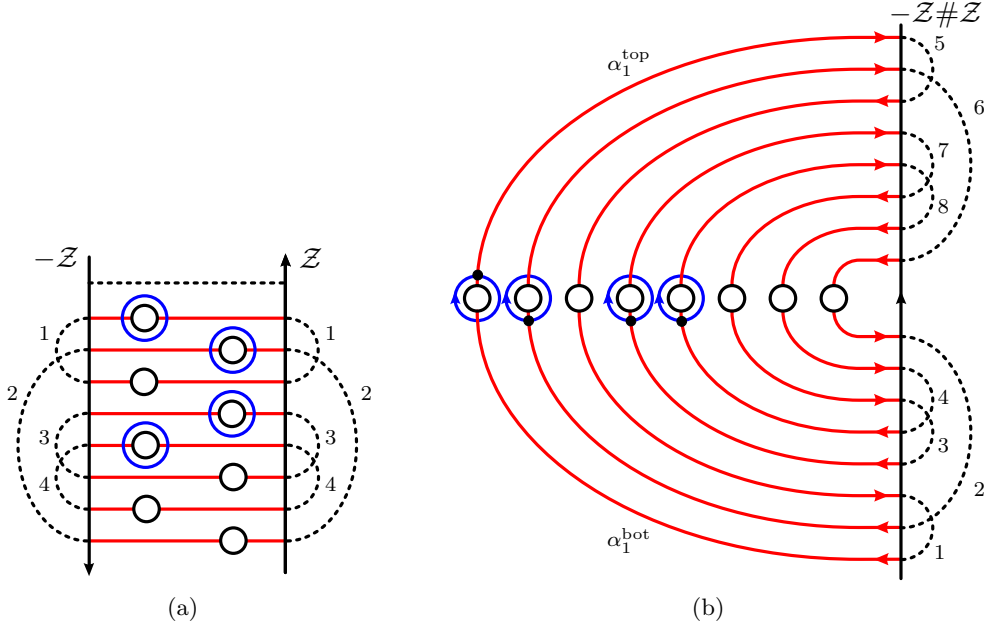

As an $\F$-vector space, $\CFAAid$ is generated by $2k$-tuples of intersection points, one on each $\beta$-circle.  For $1 \leq j \leq 2k$, on $\beta_j$, there is a pair of intersection points $\ybottom_j$ and $\ytop_j$, where $\ybottom_j = \beta_j \cap \alpha^{\bot}_j$ and $\ytop_j = \beta_j \cap \alpha^{\top}_{j}$.  Thus, a generator of $\CFAAid$ is given by a choice of top or bottom for each $j \in [2k]$.  We denote such a generator by $\bfy$.  Recall that $\CFAAid$ comes equipped with two commuting right $\cA_\infty$-module structures over $\cA(-\cZ)$ and $\cA(\cZ)$.  The minimal idempotent of $\cA(\cZ)$ (respectively $\cA(-\cZ)$) which acts non-trivially on $\bfy$ is precisely $I_{\cZ}(\bfs)$ (respectively $I_{-\cZ}(\bfsbar)$) where $\bfs$ is the set of $j \in [2k]$ such that $\ytop_j \in \bfy$.  We see that there is a correspondence between the generators of $\CFAAid$ and $\bfs \subset [2k]$; in terms of idempotents, we have $(I_{-\cZ}(\bfsbar) \otimes I_{\cZ}(\bfs)) \CFAAid  = \mathbb{F}$, up to grading shift, for each $\bfs \subset [2k]$.    

We claim that the $\Ztwo$-grading of a generator $\bfy$ is given by 
\begin{equation}\label{eqn:grAAid}
\gr_{AA}(\bfy) = \theta(\bfs) \pmod 2,    
\end{equation}
where $\theta(\bfs)$ is as defined in \eqref{eqn:grAAtheta}.  Since $\gr_{AA}$ is defined using the restriction functor and Definition~\ref{defn:grA}, we see that $\gr_{AA}(\bfy) = \sum_{j \in \bfs} o(y_j) + \sgn(\sigma_{\bfy}) \pmod 2$.  For each $j$, an intersection point of the form $\ybottom_j$ has intersection sign 0 while $\ytop_j$ has intersection sign 1.  This shows that the contribution of the intersection signs to $\gr_{AA}$ is precisely $|\bfs|$.  We now count inversions. By our ordering on the $\alpha$-arcs, the inversions are exactly the pairs $(j, j')$ such that $\ybottom_{j'}$ and $\ytop_{j}$ are in $\bfy$ and $j< j'$. Thus, the total number of inversions is $\sum_{j' \in \bfsbar} \# \{ j \in \bfs \mid j < j'\}$.   This now completes the proof, having established \eqref{eqn:grAAid}.
\end{proof}

With this, we are now ready to prove 
Theorem \ref{thm:hodge}.  For notation, recall that given $\bfs \subset [2k]$, we write $\bigwedge_{j \in \bfs} x_j$ to mean $x_{s_1} \wedge \ldots \wedge x_{s_\ell}$ where $\bfs = \{s_1,\ldots,s_\ell\}$ with $s_1 < \ldots < s_\ell$.

\begin{proof}[Proof of Theorem~\ref{thm:hodge}]
Per usual, by Theorem~\ref{thm:decat-general}, it suffices to establish the claim in the case $N$ is the elementary projective module $\cA(\cZ)I_{\cZ}(\bfs)$, or equivalently, the type D structure $\F I_{\cZ}(\bfs)$ equipped with trivial type D structure maps and supported in grading zero; we then extend by linearity.  We first compute $\CFAAid \boxtimes_{\cA(\cZ)} \F I_{\cZ}(\bfs)$.  By Lemma~\ref{lem:CFAAidgr} and the definition of the box tensor product, 
\begin{equation}\label{eqn:boxCFAAid}
\CFAAid \boxtimes_{\cA(\cZ)} \F I_{\cZ}(\bfs) = I_{-\cZ}(\bfsbar) \F \{\theta(\bfs)\},  
\end{equation}  
where $\theta(\bfs)$ is as defined in \eqref{eqn:grAAtheta}.  
Observe that \eqref{eqn:boxCFAAid} is only a statement about $\cI(-\cZ)$-modules.  However, as shown in Theorem~\ref{thm:decat-general} (see also Theorem~\ref{thm:inadecat}), the graded $\cI(-\cZ)$-module structure completely determines the decategorification, so we are content to work at this level.   To determine $K_0(I_{-\cZ}(\bfsbar) \F \{\theta(\bfs)\})$ as a functional from $K_0(\cA(-\cZ))$ to $K_0(\mathsf{Kom})$, we will see what it does on a basis for $K_0(\cA(-\cZ))$, namely $[\F I_{-\cZ}(\bfs')]$ for $\bfs' \subset [2k]$.  For each $\bfs' \subset [2k]$, we compute  
\begin{align}
K_0\Big(I_{-\cZ}(\bfsbar) \F \{\theta(\bfs)\}\Big)([\F I_{-\cZ}(\bfs')]) &= [I_{-\cZ}(\bfsbar) \F \{\theta(\bfs)\} \boxtimes_{\cA(-\cZ)} \F I_{-\cZ}(\bfs')]\notag\\ &= 
\begin{cases}   
[\F\{\theta(\bfs)\}] & \text{ if } \bfs' = \bfsbar \\
0 &\text{ if } \bfs' \neq \bfsbar.
\end{cases}\label{eqn:K0duality}
\end{align} 
Using the computations above, we will determine $K_0(I_{-\cZ}(\bfsbar) \F \{\theta(\bfs)\})$ as a map from $\Lambda^* H_1(F(\cZ))$ to $\mathbb{Z}$, using the identifications $K_0(\cA(-\cZ)) \cong \Lambda^* H_1(F(\cZ))$ from Theorem~\ref{thm:inadecat} and $K_0(\mathsf{Kom}) \cong \mathbb{Z}$ via Euler characteristic.  

Using the definition of $-\cZ$ from Section~\ref{sec:algebrabackground}, we have that $[\F I_{-\cZ}(\bfs')] = (-1)^{|\bfs'|} \bigwedge_{j \in \bfs'} \gamma_j \in \Lambda^*H_1(F)$.  (The factor of $(-1)^{|\bfs'|}$ appears since for $-\cZ$, we use the basis $-\gamma_1,\ldots,-\gamma_{2k}$.)  Further, we observe that $\chi(\F\{\theta(\bfs)\}) = (-1)^{\theta(\bfs)}$.  

We thus deduce from \eqref{eqn:K0duality} that when $\bfs' = \bfsbar$, under the appropriate identifications,
\[
K_0(I_{-\cZ}(\bfsbar) \F \{\theta(\bfs)\})\Big(  (-1)^{|\bfsbar|} \bigwedge_{j \in \bfsbar} \gamma_j \Big) = (-1)^{\theta(\bfs)},
\]   
and $K_0(I_{-\cZ}(\bfsbar) \F \{\theta(\bfs)\})$ vanishes on all other basis elements in $\Lambda^* H_1(F)$ arising from the remaining elementary projectives.  More concisely, we have
\[
K_0(I_{-\cZ}(\bfsbar) \F \{\theta(\bfs)\})\Big( \bigwedge_{j \in \bfs'} \gamma_j \Big) = (-1)^{|\bfsbar|} (-1)^{\theta(\bfs)} \delta_{\bfs', \bfsbar},
\]
where $\delta_{\bfs', \bfsbar}$ is the Kronecker delta function on the subsets of $[2k]$.  

However, we also have that $[\F I_{\cZ}(\bfs)] = \bigwedge_{j \in \bfs} \gamma_j \in \Lambda^* H_1(-F(\cZ))$ by \eqref{eqn:psioneidempotent}.  In summary, 
\[
K_0(\CFAAid \boxtimes_{\cA(\cZ)} -)\co \Lambda^* H_1(-F) \to \left( \Lambda^* H_1(F)  \right)^*
\] 
is the map which takes $\bigwedge_{j \in \bfs} \gamma_j$ to the functional 
\[
\bigwedge_{j' \in \bfs'} \gamma_{j'} \mapsto (-1)^{|\bfsbar|} (-1)^{\theta(\bfs)} \delta_{\bfs', \bfsbar}.
\]
Notice that $\theta(\bfs) + |\bfsbar| = \sum_{j' \in \bfsbar} \#\{j \in \bfs \mid j < j'\} \pmod{2}$ by \eqref{eqn:grAAtheta}.  

Write $\bfs = \{s_1,\ldots,s_\ell\}$ and $\bfsbar = \{s_{\ell+1},\ldots,s_{2k}\}$ in increasing order.  We see that the ordered set  $\{s_{\ell+1},\ldots,s_{2k}, s_1,\ldots,s_\ell\}$ requires exactly $\sum_{j' \in \bfsbar} \#\{j \in \bfs \mid j < j'\}$ transpositions to obtain the ordered set $[2k]$.  Thus, we see that $(-1)^{|\bfsbar|} (-1)^{\theta(\bfs)} \delta_{\bfs', \bfsbar}$ agrees with $\star(\bigwedge_{j' \in \bfs'} \gamma_{j'} \wedge \bigwedge_{j \in \bfs} \gamma_j)$.  Therefore, $K_0(\CFAAid \boxtimes_{\cZ} -)\co \Lambda^*H_1(-F) \to \left( \Lambda^*H_1(F)  \right)^*$ is precisely $\dualize_{H_1(-F)}$.
\end{proof}

\begin{remark}
An alternative to Theorem~\ref{thm:hodge} would be to consider a category of type A modules, to which we would assign exterior algebra elements.  If one follows a construction analogous to that for type D modules (where we work with simple modules instead of elementary projectives), we would see that the process of box tensoring with $\CFDDid$ is precisely Hodge duality.  We take the viewpoint given in Theorem~\ref{thm:hodge} as a similar picture will be used for the proof of Theorem~\ref{thm:donaldson}. 
\end{remark}

\begin{remark}
Keeping track of gradings of the exterior algebra, we have that $\dualize$ induces an isomorphism from $\Lambda^{k+i} H_1(-F)$ to $\left(\Lambda^{k-i} H_1(F) \right)^*$.  This grading corresponds to the strands grading on $\cA(\cZ)$; more precisely, this isomorphism is lifted by Theorem~\ref{thm:AA-duality}, which shows that $\CFAAid \boxtimes_{\cA(\cZ)} -$ induces an equivalence of categories from $^{\cA(\cZ,i)} \sfMod$ to $\sfMod_{\cA(-\cZ,-i)}$.  
\end{remark}

\subsubsection{Bimodules and Homs}\label{Bimodules_and_Homs}
We now generalize the discussion for 
Theorem \ref{thm:hodge} with type D and type A structures replaced by type DD and type DA structures.  We begin with the relevant linear algebra.  Let $V, V'$ be finitely-generated, oriented, vector spaces (or free abelian groups with ordered bases), and consider the isomorphism 
\[
\homize_{V,V'}\co \Lambda^*(V)^* \otimes \Lambda^*(V') \to \Hom(\Lambda^*(V), \Lambda^*(V')),
\]
which takes $v^* \otimes v'$ to the functional $x \mapsto v^*(x) v'$.  Therefore, we have an isomorphism
\begin{equation}\label{eqn:homdecat}
\dualhomize_{V,V'}\co \Lambda^*(V) \otimes \Lambda^*(V') \to \Hom(\Lambda^*(V), \Lambda^*(V')), \qquad v \otimes v' \mapsto \homize_{V,V'}(\dualize(v) \otimes v').  
\end{equation} 
In the context of Donaldson's TQFT, we have, as discussed in Section~\ref{sec:linearalgebra} without this notation, that $\dualhomize_{H_1(-F_0), H_1(F_1)}(|\Gamma_W|) = \FDA(W)$.  Again, we will omit the subscripts from $\homize$ and $\dualhomize$ when the groups in question are clear.    

In 
Theorem \ref{thm:hodge}, we showed that the duality between type D and type A structures given by box tensoring with $\CFAAid$ (which is an equivalence of categories by Theorem~\ref{thm:AA-duality}) decategorifies to the map $\dualize \co \Lambda^*H_1(-F) \to \left( \Lambda^*H_1(F)\right)^*$ defined above.  We seek to extend this construction to obtain a categorification of $\dualhomize_{V,V'}$.  Observe that given a type DD structure $N$ over $\cA(\cZ_0)$ and $\cA(-\cZ_1)$, we may obtain a left $\cA(-\cZ_1)$-, right $\cA(-\cZ_0)$-bimodule of type DA by considering $\CFAAidzero \boxtimes_{\cA(\cZ_0)} N$.  Consequently, we obtain an induced map $K_0(\CFAAidzero \boxtimes_{\cA(\cZ_0)} N) \in \Hom(\Lambda^*H_1(F_0), \Lambda^*H_1(F_1))$.

\begin{proposition}\label{prop:hodgehom}
Let $N \in \,  ^{\cA(\cZ_0) \otimes \cA(-\cZ_1)}\sfMod$.  Considering $[N]$ as an element of $\Lambda^*H_1(-F_0)\otimes \Lambda^*H_1(F_1)$ via Proposition~\ref{prop:DDdecat}, we have 
\begin{equation}\label{eqn:homdecatupsilon}
K_0(\CFAAidzero \boxtimes_{\cA(\cZ_0)} N)  = \dualhomize_{H_1(-F_0), H_1(F_1)} ([N])     
\end{equation}
as elements of $\Hom(\Lambda^*H_1(F_0), \Lambda^*H_1(F_1))$.
\end{proposition}

\begin{remark}
Here is a moral argument for Proposition~\ref{prop:hodgehom}.  Given a type DD bimodule $N$, this gives us an element of $\Lambda^*H_1(-F_0) \otimes \Lambda^* H_1(F_1)$.  The box tensor product with $\CFAAidzero$ only affects the ``$\cZ_0$-side'', so by 
Theorem \ref{thm:hodge}, this corresponds to applying $\dualize$ to the ``$\cZ_0$-component'' of the tensor product.  Therefore, we obtain an element in $(\Lambda^*H_1(F_0))^* \otimes \Lambda^*H_1(F_1)$, which we can apply $\homize$ to obtain an element of $\Hom(\Lambda^*H_1(F_0), \Lambda^*H_1(F_1))$.  Consequently, we obtain $\dualhomize$.    
\end{remark}

\begin{proof}[Proof of Proposition \ref{prop:hodgehom}]
It suffices to show \eqref{eqn:homdecatupsilon} holds on a basis for $K_0(\cA(\cZ_0)\otimes \cA(-\cZ_1))$. Using Proposition~\ref{prop:DDdecat}, we have a basis given by $N_{\bfs_0,\bfs_1} = \cA(\cZ_0) \otimes \cA(-\cZ_1) (I_{\cZ_0}(\bfs_0) \otimes I_{-\cZ_1}(\bfs_1))$, for $\bfs_0 \subset [2k_0]$ and $\bfs_1 \subset [2k_1]$.  First, by Lemma~\ref{lem:CFAAidgr}, we have
\begin{equation}\label{eqn:CFAAidbimod}
\CFAAidzero \boxtimes_{\cA(\cZ_0)} N_{\bfs_0, \bfs_1} = I_{-\cZ_0}(\bfsbar_0) \F I_{-\cZ_1}(\bfs_1)\{\theta(\bfs_0)\}.   
\end{equation}
Again this is only a statement about graded modules over the appropriate idempotent subalgebras; however, as shown in Section~\ref{sub:collect}, this structure is all we need to keep track of to decategorify.  To compute $K_0(\CFAAidzero \boxtimes_{\cA(\cZ_0)} N_{\bfs_0,\bfs_1})$, we check this on a basis for $K_0(\cA(-\cZ_0))$.  In other words, we compute 
$K_0(\CFAAidzero \boxtimes_{\cA(\cZ_0)} N_{\bfs_0,\bfs_1})([\cA(-\cZ_0)I_{-\cZ_0}(\bfs'_0)])$ for each $\bfs'_0 \subset [2k_0]$. 
Using \eqref{eqn:CFAAidbimod}, we see 
\begin{align}
\nonumber &K_0(\CFAAidzero \boxtimes_{\cA(\cZ_0)} N_{\bfs_0,\bfs_1})([\cA(-\cZ_0)I_{-\cZ_0}(\bfs'_0)]) \\
\nonumber&= [(\CFAAidzero \boxtimes_{\cA(\cZ_0)} N_{\bfs_0,\bfs_1}) \boxtimes_{\cA(-\cZ_0)} \cA(-\cZ_0)I_{-\cZ_0}(\bfs'_0)] \\
\label{eqn:CFAAidbimodbox} &= [I_{-\cZ_0}(\bfsbar_0) \F I_{-\cZ_1}(\bfs_1) \{\theta(\bfs_0)\}\boxtimes_{\cA(-\cZ_0)}  \cA(-\cZ_0)I_{-\cZ_0}(\bfs'_0) ] \\
\nonumber &= 
\begin{cases} (-1)^{\theta(\bfs_0)}[\F I_{-\cZ_1}(\bfs_1)] & \text{ if } \bfs'_0 = \bfsbar_0 \\ 0  &\text{ if } \bfs'_0 \neq \bfsbar_0. \end{cases}
\end{align}

We now rephrase these computations in terms of our identifications with exterior algebras.  We let $x_1,\ldots,x_{2k_0}$ and $y_1,\ldots,y_{2k_1}$ be the ordered bases associated with $H_1(F_0)$ and $H_1(F_1)$ for $\cZ_0$ and $\cZ_1$ respectively, as described in Section~\ref{subsec:kernel}.  Recall that $-y_1,\ldots,-y_{2k_1}$ is the basis associated to $H_1(-F_1)$ associated to $-\cZ_1$.  From Proposition~\ref{prop:DDdecat}, we see 
\begin{equation}\label{eqn:Ns1s2}
[N_{\bfs_0,\bfs_1}] = (-1)^{|\bfs_1|} \Big(\bigwedge_{j \in \bfs_0} x_j \otimes \bigwedge_{j' \in \bfs_1} y_{j'}\Big) \in \Lambda^*H_1(-F_0)\otimes \Lambda^*H_1(F_1). 
\end{equation}  
By \eqref{eqn:CFAAidbimodbox}, we also have 
\begin{equation}\label{eqn:CFAAids1s2}
K_0(\CFAAidzero \boxtimes_{\cA(\cZ_0)} N_{\bfs_0,\bfs_1})\Big((-1)^{|\bfs'_0|} \bigwedge_{j \in \bfs'_0} x_j\Big) =  \begin{cases} (-1)^{\theta(\bfs_0)}(-1)^{|\bfs_1|} \bigwedge_{j' \in \bfs_1} y_{j'} & \text{ if } \bfs'_0 = \bfsbar_0 \\ 0  & \text{ if } \bfs'_0 \neq \bfsbar_0. \end{cases}
\end{equation}

As shown in the proof of Theorem~\ref{thm:hodge}, the map from $\Lambda^*H_1(-F_0)$ to $(\Lambda^*H_1(F_0))^*$ given by sending $\bigwedge_{j \in \bfs_0} x_j$ to the functional determined by 
\[
\bigwedge_{j \in \bfs'_0} x_j \mapsto (-1)^{|\bfsbar_0|} (-1)^{\theta(\bfs_0)} \delta_{\bfs'_0, \bfsbar_0}
\]
is precisely $\dualize_{H_1(-F_0)}$.  It therefore follows from \eqref{eqn:Ns1s2} and \eqref{eqn:CFAAids1s2} that 
\[
K_0(\CFAAidzero \boxtimes_{\cA(\cZ_0)} -)\co\Lambda^*H_1(-F_0) \otimes \Lambda^*H_1(F_1) \to \Hom(\Lambda^*H_1(F_0), \Lambda^*H_1(F_1))
\]
is given by sending $v_0 \otimes v_1$ to the functional $v'_0 \mapsto (\dualize_{H_1(-F_0)}(v_0))(v'_0) v_1$.  By definition, this map is precisely $\dualhomize$.  This completes the proof.
\end{proof}

\subsection{A functor-valued categorification of Donaldson's TQFT}\label{subsec:functorvalued}

We are now in a position to complete the proof of Theorem \ref{thm:donaldson} that $K_0(\CFDA(W)) = \FDA(W)$ for a cobordism $W$ from $F_0$ to $F_1$.  To begin, recall that Donaldson's TQFT, denoted $\FDA$,  assigns to a closed, orientable surface $F$ the abelian group $\Lambda^*H_1(F)$; and to a manifold $W$ with boundary $F$ an element of $\FDA(F)$ determined by the kernel of the map induced by the inclusion $i\co \partial W \to W$ on homology.  For a cobordism $W$ between closed surfaces $F_0$ and $F_1$ (so that $\partial W = -F_0\amalg F_1$) we obtain an element of $\Hom\big(\Lambda^*H_1(F_0),\Lambda^*H_1(F_1)\big)$. Throughout, we will continue to restrict to cobordisms for which $i_*:H_1(\partial W) \to H_1(W)$ is surjective.  We now relate the structure of this TQFT to the bordered Floer homology package.

Recall that, given a pointed matched circle $\cZ$ describing a surface $F$, there are two natural ordered bases for $H_1(F)$ determined by $\cZ$ and $-\cZ$ (compare Section \ref{subsec:hodgeduality}). To ensure consistency with the conventions of bordered Floer theory, we associated to $F$ the exterior algebra $\Lambda^*H_1(F)$, using the basis prescribed by $-\cZ$. As a result, we may describe the categorification of the two-dimensional part of Donaldson's TQFT, coming from Theorem~\ref{thm:inadecat}, diagrammatically as follows:
\[\begin{tikzpicture}
  \matrix (m) [matrix of math nodes,row sep=2em,column sep=3em,minimum width=2em]
  {
    &  {}^{\cA(-F)}\sfMod \\
 F & \Lambda^*H_1(F) \\};
  \path[-latex]
    (m-2-1) edge[snake=snake,
segment amplitude=.4mm, segment length=2mm, line after snake=2mm] (m-1-2)
    (m-2-1.east|-m-2-2) edge node [below] {$\FDA$} (m-2-2)
    (m-1-2) edge node [right] {$K_0$} (m-2-2);
\end{tikzpicture}\] 
Given a three-manifold $W$, with connected boundary $F$, we obtain $\FDA(W) \in \FDA(F)$.  Bordered Floer homology provides a categorical lift of this by assigning a bordered three-manifold $W$ with parameterized boundary $F$ the object $\CFD(W)$ in $^{\cA(-F)}\sfMod$.  Theorem~\ref{thm:kernel} shows that this is precisely the categorification of $\FDA(W) \in \FDA(F)$.     

Another axiom of topological quantum field theories is duality: If $\mathcal{X}$ is a $(2+1)$-dimensional TQFT, then $\mathcal{X}(-F) = \mathcal{X}(F)^*$ (see \cite{Atiyah}).  We can see the categorification of this duality in bordered Floer homology as the equivalence of categories ${}^{\cA(F)}\sfMod \simeq\sfMod_{\cA(-F)}$ from Theorem \ref{thm:AA-duality}.  To see this,  consider the pairing theorem as a means of constructing a functor, via box  tensor product, from the category ${}^{\cA(-F)}\sfMod$ to $\mathsf{Kom}$ (the homotopy category of $\Ztwo$-graded chain complexes over $\F$) as discussed in Section~\ref{subsubsec:typeAhodge}.  
The categorification of Hodge duality from Theorem \ref{thm:hodge} may be summarized in a commutative diagram
\[\begin{tikzpicture}
  \matrix (m) [matrix of math nodes,row sep=2em,column sep=3em,minimum width=2em]
  {
    {}^{\cA(F)}\sfMod & \sfMod_{\cA(-F)} \\
   \Lambda^*H_1(-F) & \left(\Lambda^*H_1(F)\right)^*\\};
  \path[-latex]
    (m-1-1) edge node [left] {$K_0$} (m-2-1)
      (m-1-1.east|-m-1-2) edge node [above] {$\simeq$}  (m-1-2)
    (m-2-1) edge node [above] {$\eta_{H_1(-F)}$} (m-2-2)
    (m-1-2) edge node [right] {$K_0$}  (m-2-2);
\end{tikzpicture}\] 
where $\eta_{H_1(-F)}$ is as defined in Section~\ref{sec:linearalgebra} using the Hodge star operator. This results in the following diagramatic illustration of the desired categorification of duality:
\[\begin{tikzpicture}
  \matrix (m) [matrix of math nodes,row sep=2em,column sep=3em,minimum width=2em]
  {
    &  {}^{\cA(F)}\sfMod & \sfMod_{\cA(-F)} \\
 -F & & (\Lambda^*H_1(F))^* \\};
  \path[-latex]
    (m-2-1) edge[snake=snake,
segment amplitude=.4mm, segment length=2mm, line after snake=2mm] (m-1-2)
     (m-1-2.east|-m-1-3) edge node [above] {$\simeq$}  (m-1-3)
    (m-2-1.east|-m-2-3) edge node [below] {$\FDA$} (m-2-3)
    (m-1-3) edge node [right] {$K_0$} (m-2-3);
\end{tikzpicture}\] 
With this setup in hand we may complete the proof that bordered Floer homology recovers Donaldson's TQFT. 

\begin{proof}[Proof of Theorem~\ref{thm:donaldson}]
Let $W$ be a cobordism from $F_0$ to $F_1$, with boundary parametrized by $-\cZ_0$ and $\cZ_1$ respectively, which satisfies $H_1(W,\partial W) = \Z$.  By Theorem~\ref{thm:kernel}, the homological conditions imply $[\CFDD(W)] = |\Gamma_W| \in \Lambda^*H_1(-F_0) \otimes \Lambda^*H_1(F_1)$, where $|\Gamma_W|$ is the Pl\"ucker point corresponding to $\ker(i_*\co H_1(\partial W) \to H_1(W))$.  
As shown in Section~\ref{sec:linearalgebra}, we have $\dualhomize_{H_1(-F_0), H_1(F_1)}(|\Gamma_W|) = \FDA(W)$.  Therefore, $\dualhomize_{H_1(-F_0), H_1(F_1)}([\CFDD(W)]) = \FDA(W)$.  Further, $\CFAAidzero \boxtimes_{\cA(\cZ_0)} \CFDD(W) \simeq \CFDA(W)$ by \cite[Theorem 12]{LOTbimodules}.   By Proposition~\ref{prop:hodgehom}, we have 
\begin{align*}
K_0(\CFDA(W)) &= K_0(\CFAAidzero \boxtimes_{\cA(\cZ_0)} \CFDD(W)) \\
&= \dualhomize_{H_1(-F_0),H_1(F_1)}([\CFDD(W)]) \\  
&= \FDA(W).\qedhere
\end{align*}
\end{proof}

We now conclude by placing this final step of the proof in the context of the discussion preceding it. Consider a cobordism $W$ with $\partial W = -F_0\amalg F_1$. Since $\CFAA(\mathbb{I}_{\cZ_0})\boxtimes-$ induces an equivalence of categories  \[{}^{\cA(F_0)\otimes\cA(-F_1)}\sfMod \simeq {}^{\cA(-F_1)}\sfMod_{\cA(-F_0)}\] we obtain, by appealing to  Proposition \ref{prop:hodgehom}, the following mnemonic diagram:
\[\begin{tikzpicture}
  \matrix (m) [matrix of math nodes,row sep=2em,column sep=3em,minimum width=2em]
  {
  & {}^{\cA(F_0)\otimes\cA(-F_1)}\sfMod &{}^{\cA(-F_1)}\sfMod_{\cA(-F_0)} \\
 W& & \Hom(\Lambda^*H_1(F_0), \Lambda^*H_1(F_1))\\};
  \path[-latex]
   (m-2-1) edge[snake=snake,
segment amplitude=.4mm, segment length=2mm, line after snake=2mm] (m-1-2)
     (m-1-2.east|-m-1-3) edge node [above] {$\simeq$}  (m-1-3)
    (m-2-1.east|-m-2-3) edge node [below] {$\FDA$} (m-2-3)
    (m-1-3) edge node [right] {$K_0$} (m-2-3);
\end{tikzpicture}\] 
As discussed, given $\CFAA(\mathbb{I}_{\cZ_0})\boxtimes_{\cA(\cZ_0)} \CFDD(W)\cong \CFDA(W)\in {}^{\cA(-F_1)}\sfMod_{\cA(-F_0)}$ we obtain a functor 
\[\CFDA(W)\boxtimes-\co {}^{\cA(-F_0)}\sfMod \to {}^{\cA(-F_1)}\sfMod\] taking type D structures over $\cA(-F_0)$ to type D structures over $\cA(-F_1)$. Theorem~\ref{thm:donaldson} shows this functor is precisely the desired categorification of Donaldson's TQFT, that is, $K_0(\CFDA(W)) = \FDA(W)$ as elements of $\Hom(\Lambda^*H_1(F_0), \Lambda^*H_1(F_1))$.

\subsection{Completing the proof of Theorem~\ref{thm:main}}
We now specialize Theorem~\ref{thm:donaldson} in order to prove Theorem~\ref{thm:main}.  We recall the setup.  

Let $F^\circ$ be a Seifert surface for a knot $K$ in a homology sphere $Y$.  Let $F = F^\circ \cup D^2$ be the capped-off Seifert surface in $Y_0(K)$.  Construct the manifold $W = Y_0(K) \smallsetminus \nu (F) = (Y \smallsetminus \nu (F^\circ)) \cup D^2 \times I$. Then, $\partial W = -F \amalg  F$.  Observe that since $H_1(Y) = 0$, we have that $H_1(W, \partial W) = \mathbb{Z}$.  We choose a parameterization of $\partial W$ such that gluing the two boundary components together results in $Y_0(K)$. Recall that the Alexander module is the first homology of the universal abelian cover of $Y_0(K)$ regarded as a $\Z[t,t^{-1}]$-module, and that $\Delta_{Y_0(K)}(t) = \Delta_K(t)$. Define $\CFDA(K,F^\circ) = \CFDA(W)$, noting that this depends on the choice of parameterization which, as in the previous sections, we will fix and suppress from the notation; however, we do keep in mind that we treat $\CFDA(K,F^\circ)$ as a bimodule over $\cA(-F)$. Following this notation, write $\FDA(K,F^\circ)=\FDA(W)$.  

Recall that Theorem~\ref{thm:main} asserts that there is a commutative diagram 
\[
\begin{tikzpicture}
  \matrix (m) [matrix of math nodes,row sep=3em,column sep=10em,minimum width=2em]
  {
     \CFDA(K,F^\circ) & \HFK(Y,K) \\
     \mathcal{F}_{DA}(K,F^\circ) & \Delta_K(t) \\};
  \path[-latex]
    (m-1-1) edge node [right] {$K_0$}(m-2-1)
       (m-1-1.east|-m-1-2)     edge node [above] {$HH_*$} (m-1-2)
    (m-2-1.east|-m-2-2) edge node [above] {$\grTr$}
            (m-2-2)
    (m-1-2) edge node [right] {$\chi_{\operatorname{gr}}$} (m-2-2);
\end{tikzpicture}
\]
We have shown in Theorem \ref{thm:donaldson} that $K_0(\CFDA(K,F^\circ)) = \FDA(K,F^\circ)$, which establishes the left-hand side of the diagram in the present setting. Donaldson proves that $\grTr(\FDA(K,F^\circ)) = \Delta_{Y_0(K)}(t) = \Delta_K(t)$ \cite{Donaldson1999} while $HH_*(\CFDA(K,F^\circ)) \cong \HFK(Y,K)$ is due to Lipshitz, Ozsv\'ath, and Thurston \cite{LOTbimodules}. The statement in the latter compares the non-commutative gradings and strands gradings on $\CFDA(K,F^\circ)$ to the $\mathbb{Z}$-valued Maslov grading and Alexander grading on $\HFK(K)$ respectively. The fact that this isomorphism also respects the relevant $\Ztwo$-gradings is established in Section \ref{app-sub-Hoch}. It is straightforward to show that the above diagram splits along strands/Alexander gradings.

Given that $\chi_{\operatorname{gr}}( \HFK(Y,K))= \Delta_K(t)$ \cite{Rasmussen2003}, we might be content to end the proof of Theorem \ref{thm:main}. However, in the interest of giving an independent proof of this decategorification (see, in particular, Corollary \ref{cor-alex}), it remains to establish commutativity independent of Rasmussen's work. In particular, it suffices to prove the following lemma verifying that Hochschild homology categorifies the (graded) trace in the present context.

\begin{lemma}
Let $N$ be a bounded bimodule of type DA over $\cA(-F)$.  Then $\grTr(K_0(N)) = \chi_{\operatorname{gr}}(CH_*(N))$.  
\end{lemma}
\begin{proof}
Let $N$ be a bounded type DA bimodule over $\cA(-F)$, and recall that Proposition \ref{prop:hodgehom} identifies $K_0(N)$ with an element of $\Hom(\Lambda^*H_1(F), \Lambda^*H_1(F))$.  Generalizing the arguments used in Section~\ref{sub:collect}, it suffices to establish the result for the trivial type DA structure $N = I_{-\cZ}(\bfs_1) \F I_{-\cZ}(\bfs_0)$ and extend linearly.  (To see this, as in the proof of Theorem~\ref{thm:decat-general}, we reduce the problem to studying type DA bimodules over the idempotent subalgebras.)  

We begin by studying $\grTr(K_0(N))$.  As is now familiar, to determine $K_0(N)$, it suffices to consider, for each $\bfs \subset [2k]$,  $K_0(N)([\F I_{-\cZ}(\bfs)])$ where $\F I_{-\cZ}(\bfs)$ is the one-dimensional vector space over $\F$ with trivial type D structure and minimal idempotent $I_{-\cZ}(\bfs)$.  We compute 
\begin{align*}
K_0\big(I_{-\cZ}(\bfs_1)\F I_{-\cZ}(\bfs_0)\big)([\F I_{-\cZ}(\bfs)])
&= [I_{-\cZ}(\bfs_1)\F I_{-\cZ}(\bfs_0)\boxtimes \F I_{-\cZ}(\bfs)]\\
&=\begin{cases}[\F I_{-\cZ}(\bfs_0)] & \bfs_1=\bfs \\ 0 & \bfs_1\ne\bfs \end{cases}
\end{align*} which is represented by the matrix $(\delta_{\bfs_0,\bfs_1})$. In particular, 
\[\operatorname{Tr}\big(K_0(I_{-\cZ}(\bfs_1)\F I_{-\cZ}(\bfs_0))\big) =
\begin{cases} 1 & \bfs_0=\bfs_1 \\ 0 & \bfs_0\ne\bfs_1. \end{cases}
\]
On the other hand, \[CH_*\big(I_{-\cZ}(\bfs_1)\F I_{-\cZ}(\bfs_0)\big) \cong \begin{cases}\mathbb{F} & \bfs_0=\bfs_1 \\ 0 &\bfs_0\ne\bfs_1,\end{cases}\] where non-trivial generators are in grading $|\bfs_1| - k \pmod{2}$, according to Section \ref{sub:HH}. In particular, the resulting group has Euler characteristic $(-1)^{|\bfs_1| - k}$ or $0$ depending only on the idempotents $\bfs_0$ and $\bfs_1$. It now follows from the definition of $\grTr$ that  \[\grTr (K_0(N)) = \chi_{\operatorname{gr}}(CH_*(N))\] as claimed. 
\end{proof}

This completes the proof of Theorem \ref{thm:main}; Corollary \ref{cor-alex} follows, since $\grTr(\FDA(K,F^\circ)) = \Delta_K(t)$ by \cite[Proposition 12]{Donaldson1999}.


\subsection{Intersection pairings and decategorification}
Again, let $F^\circ$ be a Seifert surface for a  knot $K$ in an integer homology sphere $Y$.  By the work of Section~\ref{sec:linearalgebra}, we can determine the Seifert form for $F^\circ$ from the Pl\"ucker point for $W = (Y \setminus \nu (F^\circ)) \cup D^2 \times S^1$ and the intersection form on $F^\circ$.  We have already seen in Theorem~\ref{thm:kernel} how to obtain the Pl\"ucker point from the bordered invariants.   Therefore, we now focus on recovering the intersection form on the first homology of a surface from the bordered Floer homology package.  The following was described to us by Robert Lipshitz.  

We will recover the intersection form on $F^\circ$ using Theorem~\ref{thm:inadecat} and the definition of the $\Ztwo$-grading on $\cA(\cZ)$.  Consider a functor $\mathcal{F}$ from ${}^{\cA(\cZ,1-k)}\sfMod \times {}^{\cA(\cZ,1-k)}\sfMod$ to $\mathsf{Kom}$, the homotopy category of $\Ztwo$-graded chain complexes over $\F$.  Recall that $C_* \mapsto \chi(C_*)$ induces an isomorphism $K_0(\mathsf{Kom}) \cong \Z$.  Thus, by Theorem~\ref{thm:inadecat} and Remark~\ref{rmk:decatstrands}, the induced map on Grothendieck groups, $K_0(\mathcal{F})$, is a bilinear form $K_0(\mathcal{F})\co H_1(F(\cZ)) \times H_1(F(\cZ)) \to\Z$.  Here, we are returning to the identifications with $H_1(F(\cZ))$ as opposed to $H_1(-F(\cZ))$ as the orientation of the surface (as opposed to the homology orientation) is what is relevant for the intersection form.  Given a closed, connected, orientable surface $F$ and a pointed matched circle $\cZ$ such that $F(\cZ) = F$, a categorification of the intersection form is such a functor $\mathcal{F}$ where the induced bilinear form is exactly the intersection form on $H_1(F)$.  It turns out that there is a very natural categorification of the intersection form.  

\begin{proposition}\label{prop:intersectioncat}
Let $F$ be a closed, connected, orientable surface.  Let $\cZ$ be a pointed matched circle such that $F(\cZ) = F$.  Consider the functor
\begin{align*}
\Mor \co {}^{\cA(\cZ,1-k)}\sfMod \times {}^{\cA(\cZ,1-k)}\sfMod &\to \mathsf{Kom},
\end{align*}
where $\Mor$ indicates $\cA(\cZ,1-k)$-module homomorphisms which need not respect the differential graded structure.  Then $K_0(\Mor)$ is the intersection form on $H_1(F)$.  \end{proposition}  
\begin{proof}
Recall from the identifications in Theorem~\ref{thm:inadecat} that we have a basis for $K_0({}^{\cA(\cZ,1-k)}\sfMod) \cong H_1(F(\cZ))$ given by the $[\cA I(\bfs)]$, where $I(\bfs)$ is an idempotent in $\cA(\cZ,1-k)$ (i.e., idempotents in $\cA(\cZ)$ which consist of exactly one pair of matched, horizontal strands, or equivalently, $\bfs$ is a singleton).  We label the idempotents in $\cA(\cZ,1-k)$ as $I_1,\ldots, I_{2k}$, where $I_i = I(\{i\})$.  Let $\gamma_1,\ldots,\gamma_{2k}$ denote the corresponding elements in $H_1(F(\cZ))$ that they induce. 

We use the notation of upward-veering strands, introduced in Section \ref{sec:alg}. We define $\rho_{j^\pm,j'^\pm}$ as the upward-veering strand in $Z'$ (when it exists) from $a^\pm_j$ to $a^\pm_{j'}$, and similarly for $\rho_{j^\pm, j'^\mp}$.  Furthermore, we abuse notation and identify $\rho_{j^\pm, j'^\pm}$ with the algebra element $I_j \cdot a(\rho_{j^\pm, j'^\pm}) \cdot I_{j'}$.

Since $\gamma_1,\ldots,\gamma_{2k}$ form a basis for $H_1(F(\cZ))$, it thus suffices to show that 
\[
\chi(\Mor(\cA(\cZ)I_j,\cA(\cZ)I_{j'})) = \gamma_j \cdot \gamma_{j'}, \quad \ 1 \leq j,j' \leq 2k.  
\]
Note that $\Mor(\cA(\cZ)I_j, \cA(\cZ)I_{j'})$ is naturally identified with $I_j\cA(\cZ)I_{j'}$ as $\Ztwo$-graded modules.  Thus, we are interested in determining the generators of $I_j\cA(\cZ)I_{j'}$ and their gradings, $\gr$.  Since the elements of $I_j\cA(\cZ)I_{j'}$ have only one moving strand, there are no inversions in the associated partial permutations.  Thus, we are only interested in the sum of the orientations of the initial and final points of the upward veering strand to compute gradings.  Therefore, it follows that 
\begin{align}
\label{eqn:reebgr0} \gr(I_j) = \gr(\rho_{j^+,j'^+}) = \gr(\rho_{j^-,j'^-}) &= 0 \pmod{2},  \\
\label{eqn:reebgr1} \gr(\rho_{j^+,j'^-}) = \gr(\rho_{j^-,j'^+}) &= 1 \pmod{2}.
\end{align}     

We begin with an example.  Consider the case $j = j'$.  In this case, we see that $I_j\cA(\cZ)I_{j}$ is a $\Ztwo$-vector space of dimension 2, generated by $I_j$ and $\rho_{j^-, j^+}$.   By \eqref{eqn:reebgr0} and \eqref{eqn:reebgr1}, we see 
\[
\chi(\Mor(\cA(\cZ)I_j,\cA(\cZ)I_j)) = \chi(I_j\cA(\cZ)I_j) = 0 = \gamma_j \cdot \gamma_j.
\]
More generally, we can apply \eqref{eqn:reebgr0} and \eqref{eqn:reebgr1} to compute the graded dimension of $I_j\cA(\cZ)I_{j'}$.  It is straightforward to verify that 
\[
\chi(\Mor(\cA(\cZ)I_j,\cA(\cZ)I_{j'})) = 
	\begin{cases} 
		0 & \text{ if } \rho_{j^-,j^+} \cap \rho_{j'^-,j'^+} = \emptyset \text{ or } \rho_{j^-,j^+} \subset \rho_{j'^-,j'^+} \text{ or } \rho_{j'^-,j'^+} \subset \rho_{j^-,j^+} \\
		1 & \text{ if } a^-_j \lessdot a^-_{j'} \lessdot a^+_j \lessdot a^+_{j'} \\
		-1 & \text{ if } a^-_{j'} \lessdot a^-_{j} \lessdot a^+_{j'} \lessdot a^+_{j}.
	\end{cases}
\]
By our orientation conventions, this value is precisely $\gamma_j \cdot \gamma_{j'}$.   
\end{proof}

\begin{remark}
In the definition of the grading $\gr$ on $\cA(\cZ)$, per Remark~\ref{rmk:footnote4}, we chose conventions for how to orient $\bfa$ in $\cZ$.  If we chose different conventions, this would result in a different absolute $\Ztwo$-grading on $\cA(\cZ)$.  We could have made a corresponding change of basis for $\Lambda^* H_1(F)$ with which to identify with $K_0({}^{\cA(\cZ)}\sfMod)$, and consequently Proposition~\ref{prop:intersectioncat} would hold for any of these other absolute gradings as well.   
\end{remark}

With the technical work of this section complete, we are now ready to obtain the Seifert form from bordered Floer homology.   
\begin{proof}[Proof of Theorem~\ref{thm:seifert}]
By Theorem~\ref{thm:main}, we can obtain $\FDA$ from $\CFDA(K,F^\circ)$.  As mentioned in Section~\ref{sec:linearalgebra},  Donaldson's TQFT $\FDA$ determines the Alexander module.  

To recover the Seifert form, we proceed as follows.  First, we may compute $\CFDD(W)$ as $\CFDA(K,F^\circ) \boxtimes \CFDD(\mathbb{I})$ by \cite[Theorem 12]{LOTbimodules}.  Donaldson's construction converts the Pl\"ucker point associated to $[\CFDD(W)]$ to a presentation matrix for the Alexander module, thought of as an endomorphism of $H_1(F) \otimes \mathbb{Z}[t,t^{-1}]$.  It's important to point out that in these constructions, we have considered $\CFDA(W)$ as a left $\cA(-F)$-, right $\cA(-F)$-bimodule, and not a bimodule over $\cA(F)$.  From $\cA(-F)$, we obtain the intersection form on $H_1$ for $-F$ by Proposition~\ref{prop:intersectioncat}.  Certainly, from this we can obtain the intersection form for $F$.  The proof is now completed by Proposition~\ref{prop:intersectiondetermines}.        
\end{proof}

\subsection{Example: The right-hand trefoil}

Let $W = (S^3 \smallsetminus \nu (F^\circ)) \cup D^2 \times I$ as before. We illustrate the process of obtaining the Seifert form from $\CFDD(W)$, where $K$ is the right-handed trefoil and $F^\circ$ is its unique minimal genus Seifert surface. We begin with an arced bordered Heegaard diagram for $W$ in Figure \ref{fig:trefoilbimodule}. Let $\cZ$ denote the pointed matched circle induced by this diagram which parameterizes $F = F^\circ \cup D^2$.  Note that the pointed matched circle parameterizing $-F$ is $-\cZ$ (where we recall from Section~\ref{sec:algebrabackground} that $-\cZ$ has particular matchings and orientations induced by $\cZ$).  We then pass to a bordered Heegaard diagram for the drilled manifold, which induces the single pointed matched circle $\cZ \# -\cZ$ as in Figure \ref{fig:trefoilCFDD}.  

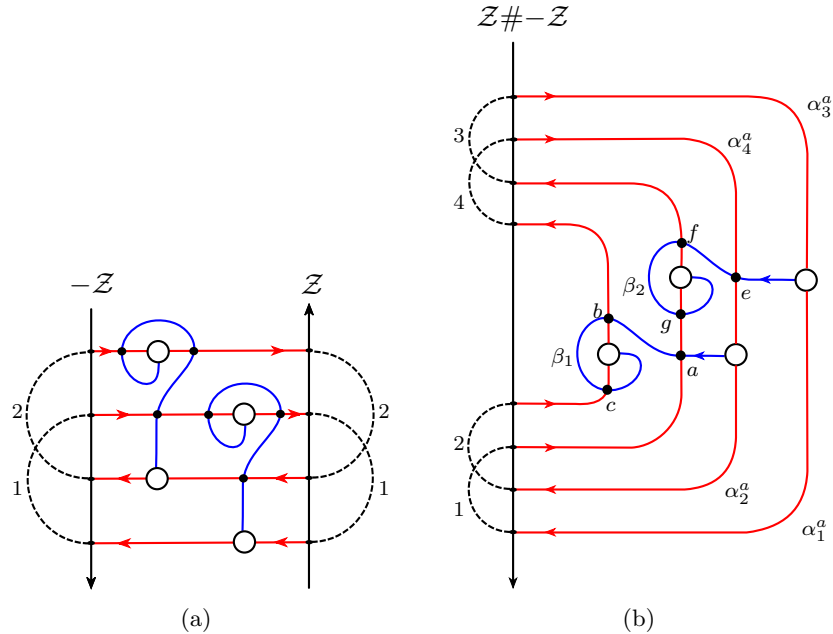
\begin{figure}[ht]
\vspace{10pt}
\subfigure[]{
\labellist
\pinlabel ${-\cZ}$ at 25 115
\pinlabel ${\cZ}$ at 108 115
\scriptsize \pinlabel $1$ at -3 38
\pinlabel $2$ at -3 67
\pinlabel $1$ at 135 38
\pinlabel $2$ at 135 67
\endlabellist
\includegraphics[scale=1]{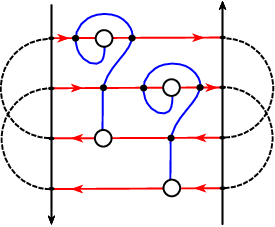}
\label{fig:trefoilbimodule}
}
\hspace{20pt}
\subfigure[]{
\labellist

\pinlabel ${\cZ \# {-\cZ}}$ at 21 215

\scriptsize \pinlabel $1$ at -3 30
\pinlabel $2$ at -3 55
\pinlabel $4$ at -3 146
\pinlabel $3$ at -3 171

\pinlabel $a$ at 85 82
\pinlabel $b$ at 49 105
\pinlabel $c$ at 54 69
\pinlabel $e$ at 105 111
\pinlabel $f$ at 85 134
\pinlabel $g$ at 75 99

\pinlabel $\beta_1$ at 36 87
\pinlabel $\beta_2$ at 63 114

\pinlabel $\alpha^a_1$ at 131 22
\pinlabel $\alpha^a_2$ at 102 36
\pinlabel $\alpha^a_3$ at 133 182
\pinlabel $\alpha^a_4$ at 103 168

\endlabellist
\includegraphics[scale=1]{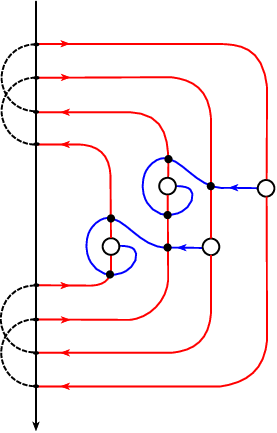}
\label{fig:trefoilCFDD}
}
\caption{Bordered diagrams associated with the right-hand trefoil relative to the unique genus one Seifert surface.  Recall that the matching on a pointed matched circle induced by a Heegaard diagram determines the labeling of the $\alpha$-arcs in the diagram.}
\label{fig:trefoil}
\end{figure}

Recall from Section~\ref{subsec:kernel} that given a pointed matched circle $\cZ$, a basis for $H_1(F(\cZ))$ is prescribed by the arcs on $Z'$ connecting pairs of matched points, oriented from the negatively-oriented point to the positively-oriented point it is matched with; further, if $\cZ$ is induced by a bordered Heegaard diagram, then it follows from the construction that this basis is ordered by the $\alpha$-arcs in the diagram.  From Figure~\ref{fig:trefoilbimodule}, we can thus construct ordered bases for each of $H_1(-F)$ and $H_1(F)$ as follows.  First, we have the ordered basis $x_1, x_2$ for $H_1(-F)$ induced by $\alpha^a_1$ and $\alpha^a_2$ respectively, where $x_i$ in $Z'$ is the arc connecting $(\partial \alpha_i^a)^-$ to $(\partial \alpha_i^a)^+$.   We also obtain $y_1$ and $y_2$, induced by $\alpha^a_3$ and $\alpha^a_4$ respectively, as the ordered basis for $H_1(F)$. It is clear that we have $y_i = -x_i$.  

We see from Figure~\ref{fig:trefoilCFDD} that the set 
$\{ae, \ af, \ bf, \ bg, \ ce, \ cf, \ cg\}$ generates $\CFDD(W)$ as a type DD structure.
We begin by considering $ae$, which lies on $\alpha^a_1$ and $\alpha^a_4$. Thus, the element of $\Lambda^*H_1(-F) \otimes  \Lambda^*H_1(F)$ corresponding to the idempotent of $ae$ is $x_2 \otimes y_1$. The grading of $ae$ is
\begin{align*} 
	\gr_{DD}(ae) &=\sgn_D (\sigma) + \sum_{p \in ae} o(p) \\
		&= \inv(\sigma) + \sum_{i \in \Im (\sigma)} \# \{ j \mid j>i, j \notin \Im(\sigma)\} + \sum_{p \in ae} o(p) \\
		&= 0 + 0 + 1  \\
		&= 1 \pmod 2,
\end{align*}
where $\sigma \co [2] \rightarrow [4]$ is the partial permutation corresponding to $ae$, that is, $1 \mapsto 1$ and $2 \mapsto 4$.  Therefore, under the identifications in Proposition~\ref{prop:DDdecat}, the generator will ultimately contribute $-x_2 \otimes y_1$ to $K_0(\CFDD(W))$.  We summarize the generators, their gradings, and the element of $\Lambda^*H_1(-F) \otimes \Lambda^*H_1(F)$ determined by the associated idempotent:

\begin{center}
\begin{tabular}{cccc}
\\ \hline
&Generator \qquad & Sign \qquad & Basis element \\ 
\hline
&$ae$ & 1 & $x_2 \otimes y_1$ \\
&$af$ & 1 & $x_2 \otimes  y_2$ \\
&$bf$ & 1 & $(x_1 \wedge x_2) \otimes 1$ \\
&$bg$ & 0 & $x_2 \otimes y_1$ \\
&$ce$ & 1 & $x_1 \otimes y_1$ \\
&$cf$ & 1 & $x_1 \otimes y_2$\\
&$cg$ & 1 & $1 \otimes (y_1 \wedge y_2)$\\
\hline \\
\end{tabular}
\end{center}

By Theorem~\ref{thm:kernel}, we obtain the following Pl\"ucker point $|\Gamma| = [\CFDD(W)]  \in \Lambda^*H_1(-F) \otimes \Lambda^* H_1(F)$:
\begin{align*}  
|\Gamma| &= -(x_1 \wedge x_2) \otimes 1  -x_1 \otimes (y_1+y_2) - x_2 \otimes y_2 -  1 \otimes (y_1 \wedge y_2).  
\end{align*}
Observe that treating $|\Gamma|$ as an element of $\Lambda^{2k}(H_1(-F) \oplus H_1(F))$, we can write $|\Gamma| = (x_1 + x_2, y_1) \wedge (x_1, - y_2)$.  

Following Section \ref{sec:linearalgebra}, in order to obtain a presentation matrix for the Alexander module, we would like to express the Pl\"ucker point as an element of $\Lambda^*(H_1(F) \oplus H_1(F))$ where we use the same basis in each copy of $H_1(F)$.  Therefore, we write $|\Gamma|$ as 
\begin{equation}\label{eq:trefoilplucker}
	 |\Gamma| = -(y_1 \wedge y_2) \otimes 1 + y_1 \otimes (y_1 + y_2) + y_2 \otimes y_2 - 1 \otimes (y_1 \wedge y_2)
\end{equation}
where we recall $y_i = -x_i$.  As an element of $\Lambda^*(H_1(F) \oplus H_1(F))$, we have $|\Gamma| =  (-y_1 -y_2, y_1) \wedge (-y_1, - y_2)$.  Before computing the Alexander module and Seifert form from $|\Gamma|$ by following the recipe in Section~\ref{sec:linearalgebra}, to provide relevance with Theorem~\ref{thm:main}, we convert $|\Gamma|$ to $\FDA(K,F^\circ)$ and compute the Alexander polynomial of the trefoil.  

We first obtain an element of $\Hom (\Lambda^* H_1(F), \Lambda^* H_1(F))$ from $|\Gamma|$.  Recall the isomorphism
\[ \eta_{H_1(F)}\co \Lambda^j H_1(F) \cong (\Lambda^{k-j} H_1(F) )^* \]
sending $x \in \Lambda^j H_1(F)$ to the map $v \mapsto \star (v \wedge x)$.  For $x$ of the form $1, y_1, y_2, y_1 \wedge y_2$, we will use $x^*$ to be the element of $(\Lambda^*H_1(F))^*$ which is precisely 1 on $x$ and vanishes on the other three basis vectors for $\Lambda^*H_1(F)$.  Since $y_1 \wedge y_2$ is our chosen volume form on $H_1(F)$, we have
\[ \eta_{H_1(F)}(1) = (y_1 \wedge y_2)^* \qquad \eta_{H_1(F)}(y_1) = -(y_2)^* \qquad \eta_{H_1(F)}(y_2) = (y_1)^* \qquad \eta_{H_1(F)}(y_1 \wedge y_2) = 1^*. \]
Since $\Upsilon_{H_1(F),H_1(F)}(v \otimes w)$ is the element of $\Hom(\Lambda^*H_1(F),\Lambda^*H_1(F))$ defined by $v' \mapsto \left(\eta_{H_1(F)}(v)\right)(v')w$, we see from \eqref{eq:trefoilplucker} that $\FDA(K,F^\circ) = \dualhomize_{H_1(F), H_1(F)}(|\Gamma|) \in \Hom(\Lambda^*H_1(F),\Lambda^*H_1(F))$ is given by
\begin{align*}
	1 &\mapsto -1 \\
	y_1 &\mapsto y_2 \\
	y_2 &\mapsto -y_1 - y_2 \\
	y_1 \wedge y_2 &\mapsto -y_1\wedge y_2. \\
\end{align*}
Using the ordered basis $1, y_1, y_2, y_1 \wedge y_2$ for $\Lambda^*H_1(F)$, this map in matrix form is
\[\begin{tikzpicture}
\draw (-1.4,0.5)--(1.4,0.5);
\draw (-1.4,-0.45)--(1.4,-0.45);
\draw (-0.75,-0.9)--(-0.75,0.9);
\draw (0.75,-0.9)--(0.75,0.9);
\node at (0,0) {$\begin{pmatrix} -1 & 0 & 0 & 0\\ 0  & 0 &-1& 0 \\ 0 & 1 & -1 & 0 \\ 0 &0 &0&-1 \end{pmatrix}$};
\end{tikzpicture}\]
and by taking the graded trace, we obtain 
\begin{equation*}
\Delta(t) =  \left(
-t^{-1}\Tr\begin{pmatrix}-1\end{pmatrix} 
+\Tr\begin{pmatrix}0&-1\\1&-1\end{pmatrix} 
-t\Tr\begin{pmatrix}-1\end{pmatrix} \right) \\
	=  t^{-1} -1+t,
\end{equation*}
which indeed is the Alexander polynomial of the trefoil.  

We next compute the Alexander module.  Since $|\Gamma| = (-y_1 - y_2,y_1) \wedge (-y_1, - y_2) \in \Lambda^*(H_1(F) \oplus H_1(F))$, we have that 
\[
\Gamma = \Span\left\{(-y_1 - y_2,y_1), (-y_1, -y_2)\right\}.  
\]
In terms of the ordered basis $\mathcal{B}$ for $H_1(F) \oplus H_1(F)$ given by $(y_1,0), (y_2,0), (0,y_1), (0,y_2)$, the $\mathcal{B}$-coordinates of $(-y_1 - y_2,y_1)$ are $(-1, -1, 1, 0)$ and the coordinates of $(-y_1, - y_2)$ are $(-1, 0, 0, -1)$.  These vectors are the rows in the matrix
\[
(A \ B)=
\begin{pmatrix}
    -1 & -1 & 1 & 0 \\        
    -1 & 0 & 0 & -1
\end{pmatrix},
\]
which the construction in Section~\ref{sec:linearalgebra} guarantees a presentation of the Alexander module of the right-handed trefoil by 
\[
A+tB=
\begin{pmatrix}
    -1+t & -1  \\        
    -1 & -t   
\end{pmatrix}.
\]
We also compute that the intersection form on $H_1(F)$ is
\[ \omega = \begin{pmatrix}
    0 & 1  \\        
    -1 & 0   
\end{pmatrix}
\]
by Proposition~\ref{prop:intersectioncat}. Finally, by Lemma \ref{lem:intersectiondetermines}, it follows that the Seifert form is
\begin{equation*}
-\omega (A+B)^{-1}A = 
-\begin{pmatrix}
    0 & 1  \\        
    -1 & 0   
\end{pmatrix}
\begin{pmatrix}
    0 & -1  \\        
    -1 & -1 
\end{pmatrix}^{-1}
\begin{pmatrix}
    -1 & -1  \\        
    -1 & 0 
\end{pmatrix} 
= 
\begin{pmatrix}
    -1 & -1  \\        
    0 & -1   
\end{pmatrix}.
\end{equation*}
Note that this is indeed the Seifert form for the right-handed trefoil.

\appendix\section{Properties and invariance of the $\Z/2\Z$-grading}\label{sec:gradings}

In Section \ref{sec:background}, we defined a relative $\Ztwo$-grading, $\gr$, on the bordered Floer invariants of three-manifolds.  The goal of this appendix is to prove the following theorem.  
\begin{theorem}\label{thm:allthegradings}
Let $\cZ$ be a pointed matched circle.  The function $\gr$ gives $\cA(\cZ)$ the structure of a differential $\Ztwo$-graded algebra.  Further, if $Y$ is a bordered three-manifold with boundary parameterized by $\cZ$, the function $\gr_A$ gives $\CFA(Y)$ the structure of a $\Ztwo$-graded $\cA_\infty$-module over $\cA(\cZ)$ which is invariant up to $\Ztwo$-graded $\cA_\infty$-homotopy equivalence.  Analogous statements hold for $\gr_D$, $\gr_{AA}$, $\gr_{DD}$, and $\gr_{DA}$.
\end{theorem}

\noindent In order to prove the above theorem, we relate our gradings to those coming from the noncommutative gradings in bordered Floer homology, which are defined in \cite[Chapter 10]{LOT} and \cite[Section 6.5]{LOTbimodules}. More specifically, we identify our gradings with Petkova's $\Ztwo$-reduction of the noncommutative gradings \cite[Section 3]{Petkovadecat}. Note that Petkova defines both $\gr$ and a $\Ztwo$-reduction of the noncommutative group gradings on $\alg$.  However, the equivalence between $\gr$ and this ({\em a priori} different) $\Ztwo$-reduction was not required in that setting and consequently does not appear in Petkova's work. As this equivalence is necessary for our proof of Theorem \ref{thm:allthegradings}, we carry out the proof below.

The appendix is organized as follows. We begin with  an alternate description of the algebra in terms of Reeb chords (Section \ref{sec:alg}). We then identify our grading on $\cA(\cZ)$ with Petkova's reduction (Section \ref{sec:idalgebra}), followed by an identification of the two gradings on $\CFA$ and $\CFD$ (Section \ref{sec:grequivCFACFD}). We proceed to identify the gradings on bimodules (Section \ref{sec:grequivbimodules}). These identifications complete the proof of Theorem \ref{thm:allthegradings}, since the noncommutative gradings (and thus also Petkova's $\Ztwo$ reduction) are invariant up to the appropriate notions of equivalence.


\begin{remark} \label{rem:idemconj}
Given a bordered Heegaard diagram $\cH$, different choices of order and orientation on the $\alpha$-arcs change the pointed matched circle and our $\Ztwo$-grading $\gr$; cf. Remark \ref{rmk:footnote4}.  However, different such choices give different gradings that are conjugate in an idempotent-dependent way, as in \cite[Remark 3.46]{LOT}. Consider any two choices of order and orientation of the $\alpha$-arcs yielding bordered Heegaard diagrams $\cH_1$ and $\cH_2$ giving gradings $\gr_i$ on both $\cA(\cZ_i)$ and $\CFA(\cH_i)$. Then $\cZ_i = (Z, z, \bfa_i, M_i)$ where $M_2=\sigma \circ M_1$ for some permutation $\sigma \co [2k] \rightarrow [2k]$ and $M^{-1}_2(i)=o(i) \cdot M^{-1}_1 \circ \sigma^{-1}(i)$ for some function $o \co [2k] \rightarrow \{\pm1\}$. The permutation $\sigma$ corresponds to the reordering of the $\alpha$-arcs and the sign $o$ corresponds to the reorienting. There is a function $\xi \co \cI(\cZ_1) \rightarrow \Ztwo$ so that given a grading homogenous element $a=I(\bfs) \cdot a \cdot I(\bft) \in \cA(\cZ_1)$,
\[ \gr_2(a')=\xi(I(\bfs)) + \gr_1(a) - \xi(I(\bft)), \]
where $a' \in \cA(\cZ_2)$ is the element with the same strands diagram (see, for example, Figure \ref{fig:partperm}) as $a \in \cA(\cZ_1)$.
The map $\xi$ is given by $\xi(\bfs)=\inv(\sigma|_\bfs) + \prod_{i\in\bfs} o(i)$. Similarly, given $\bfx \in \mfS(\cH)$ thought of as an element of $\CFA(\cH)$ with a non-trivial right action by $I(\bfs_\bfx)=I_A(\bfx)$, we have
\[ \gr_2(\bfx) = \xi(I(\bfs_\bfx)) + \gr_1(\bfx).\]
Analogous statements hold for $\CFD$ and $\CFDA$.  Finally, we note that if $\gr_1$ is a differential $\Ztwo$-grading, then it follows that $\gr_2$ is as well.

This will be particularly relevant when comparing our grading to Petkova's reduction of the noncommutative gradings of \cite[Chapter 10]{LOT}, where we work with pointed matched circles where the matching is induced by the orientation of $Z$ and a specific choice of grading refinement data (see \cite[Section 3.3.2]{LOT}).  The above discussion will allow us to compare the two $\Ztwo$-gradings on a very specific family of pointed matched circles where we can compute the gradings on the associated algebras quite explicitly.  
\end{remark}

\subsection{The algebra}
\label{sec:alg}

There is an alternative description of the algebra from that in Section~\ref{sec:background} -- in terms of Reeb chords -- which will at times be useful. Let $\cZ = (Z, z, \bfa, M)$ be a pointed matched circle. Viewing $Z$ as a contact one-manifold and $\bfa$ as a Legendrian submanifold, a \emph{Reeb chord $\rho$} in $(Z \setminus z, \bfa)$ is an embedded arc in $Z \setminus z$ with endpoints in $\bfa$. A Reeb chord inherits an orientation from $Z$. We denote the initial point of $\rho$ by $\rho^-$ and the terminal point by $\rho^+$. A set of Reeb chords $\bfrho =\{ \rho_1, \ldots, \rho_\ell \}$ is \emph{consistent} if no two $\rho_i$ share initial or terminal points, i.e., $\bfrho^- = \{ \rho_1^-, \ldots, \rho_\ell^- \}$ and $\bfrho^+ = \{ \rho_1^+, \ldots, \rho_\ell^+ \}$ each contain $\ell$ distinct points.

Given a consistent set of Reeb chords $\bfrho$, we obtain an algebra element $a(\bfrho) \in \cA(\cZ)$, defined as 
\[ a(\bfrho) = \mathbf{I} \cdot \sum_{ \{ S  \;|\;  S \cap ( \bfrho^- \cup \bfrho^+) = \emptyset \} } (S \cup \bfrho^-, S \cup \bfrho^+, \phi_S) \cdot \mathbf{I}, \]
where $\phi_S(\rho_i^-)=\rho_i^+$ and $\phi_S |_S = \textup{id}_S$ and $\mathbf{I}$ is the unit for the algebra $ \cA(\cZ)$. The element $a(\bfrho)$ consists of all ways of enlarging the partial permutation given by $\bfrho$ to include any set of constant functions such that the result is $M$-admissible. Note that in general, $a(\bfrho)$ is not homogenous with respect to $\gr$. 

We write $[\bfrho]$ to denote the homology class in $H_1(Z, \bfa)$ given by $\sum_i \rho_i$.  

Recall that 
	\[ \cA(\cZ) = \bigoplus_{\t=-k}^k \cA(\cZ, \t), \]
where $k$ denotes the genus of $F(\cZ)$. Throughout this appendix, we will use $\t$ to denote the strands grading.

\subsection{Equivalence of the two gradings on the algebra}
\label{sec:idalgebra}

In \cite{Petkovadecat}, Petkova defines a $\Ztwo$-reduction of the noncommutative gradings for $\alg$, $\CFA$, and $\CFD$. Since the noncommutative gradings are differential, type A, and type D gradings on $\alg$, $\CFA$, and $\CFD$ respectively, it follows that Petkova's reduction necessarily inherits these properties. Therefore, it is our goal to identify $\gr$ with Petkova's $\Ztwo$-reduction, and then generalize to bimodules.  We begin the process with this section, where we work solely with $\alg$.  In particular, we review the noncommutative group gradings on $\alg$ and Petkova's $\Ztwo$-reduction.  Finally, we identify this with $\gr$.  

\subsubsection{Noncommutative group gradings on $\cA(\cZ)$ and Petkova's $\Z/2\Z$-reduction}\label{sec:algebranoncomm}
We quickly review the material in \cite[Section 3.3]{LOT} on noncommutative group gradings for $\cA(\cZ)$.  We begin with the definition of a grading by a noncommutative group.  Recall that we say that a {\em grading of $\cA$ by $(G,\lambda)$}, where $G$ is a group and $\lambda$ is central in $G$, is a decomposition $\cA = \oplus_{g \in G} A_g$ such that for $a \in A_g$ and $b \in A_{g'}$, we have $a \cdot b \in A_{gg'}$ and $\partial a \in A_{\lambda^{-1} g}$.  Furthermore, we call $a \in \cA$ {\em homogeneous} if it is an element of $A_g$ for some $g \in G$ and we say the grading of $a$ is $g$.  Instead of giving the decomposition of $\cA$ for a grading, we will usually just specify a function from a generating set for $\cA$ to $G$. Unless otherwise specified, we use the same notation as \cite[Section 3.3]{LOT}.  Finally, we note that if $\cA$ is graded by $(G,\lambda)$, and there exists a homomorphism from $G$ to $\Ztwo$ which sends $\lambda$ to 1, then this mod 2 reduction induces a $\Ztwo$-grading on $\cA$.

We fix a pointed matched circle $\cZ = (Z,z,\bfa,M)$ and let $Z' = Z \setminus z$. A \emph{parity change} in $\eta \in H_1(Z',\bfa)$ is a point $p \in \bfa$ such that the multiplicity of $\eta$ to the left of $p$ has different parity than the multiplicity to the right of $p$; note that there is always an even number of such points. Consider the group
\[
G'(4k) = \{(j,\eta) \mid j \in \frac{1}{2}\Z, \ \eta \in H_1(Z',\bfa), \ j = \epsilon(\eta) \pmod 1\}, 
\]
where $\epsilon(\eta) = \frac{1}{4} \#(\text{parity changes in } \eta) \pmod 1$. For $p \in \bfa$ and $\eta \in H_1(Z', \bfa)$, let $m(\eta, p)$ denote the average multiplicity with which $\eta$ covers the regions on either side of $p$, and extend bilinearly to $m: H_1(Z', \bfa) \times H_0(\bfa) \rightarrow \frac{1}{2} \Z$.
 Multiplication in $G'(4k)$ is 
\[
(j_1,\eta_1) \cdot (j_2, \eta_2) = (j_1 + j_2 + L(\eta_1,\eta_2), \eta_1 + \eta_2),
\]  
where $L(\eta_1, \eta_2) = m( \eta_2, \d \eta_1)$.
Note that $\lambda = (1,0)$ is a central element.  Given $a=(S, T, \phi)$, define $[a] \in H_1(Z', \bfa)$ to be the sum of the intervals corresponding to the strands, i.e.,
\[ [a] = \sum_{s \in S} [s, \phi(s)],\]
where $[s,\phi(s)]$ denotes the Reeb chord from $s$ to $\phi(s)$ and we have identified $[4k]$ with $\bfa$ as in Section \ref{sec:algebrabackground} .        

The {\em unrefined grading}, $\grunref$, on $\cA(\cZ)$ by $(G'(4k),\lambda)$ is 
\begin{equation}
\label{eqn:unrefgralg}
\grunref(a) = (\iota(a),[a]),
\end{equation}
where $\iota(a) = \inv(\phi) - m([a], S)$ for $a=(S, T, \phi)$. It follows from \cite[Proposition 3.40]{LOT} that $a(S, T, \phi)$ is homogeneous with respect to $\grunref$.

We are interested in a {\em refined grading}, which takes values in the subgroup $G(\cZ)$ of $G'(4k)$ defined by 
\[
G(\cZ) = \{(j,\eta) \in G'(4k) \mid M_*(\partial \eta) = 0\}.
\]  
(Note that we use $\grunref$ and $\grref$ to denote the unrefined and refined gradings, respectively,  rather than $\gr'$ and $\gr$ as in \cite[Section 3.3]{LOT} since we would like to reserve $\gr$ to refer to the $\Ztwo$-grading defined in Section \ref{sec:algebrabackground}.)
We will grade $\cA(\cZ,\t)$ separately for each $\t$.  However, in order to define a grading of $\cA(\cZ,\t)$ by $(G(\cZ),\lambda)$ some choices are required.  Fix an idempotent $I(\bfs_0) \in \cA(\cZ,\t)$, for $\bfs_0 \subset [2k]$ with $|\bfs_0| = k + \t$.  Now, for each $I(\bfs) \neq I(\bfs_0)$ with $|\bfs| = k+\t$, choose $\psi(\bfs) \in G'(4k)$ such that $M_*(\partial([\psi(\bfs)])) = \bfs - \bfs_0$, as elements of $H_0(M(\bfa);\Z)$, where $[\psi(\bfs)]$ denotes the second component of $\psi(\bfs)$.  Define $\psi(\bfs_0)$ to be the identity element in $G'(4k)$. These choices $\bfs_0, \psi$ are called the {\em grading refinement data} and $I(\bfs_0)$ is called the \emph{base idempotent}.  With this, we define the grading $\grref^{\bfs_0,\psi}$ on $\cA(\cZ,\t)$ by 
\[
\grref^{\bfs_0,\psi}(I(\bfs)a(\bfrho)I(\bft)) = \psi(\bfs) \grunref(a(\bfrho))\psi(\bft)^{-1},
\]
or equivalently,
	\[ \grref^{\bfs_0,\psi}(a(S,T,\phi))=\psi(\bfs)g'(a(S,T,\phi))\psi(\bft)^{-1}, \]
where $\bfs=M(S)$ and $\bft=M(T)$.

The function $\grref^{\bfs_0,\psi}$ takes values in $G(\cZ)$ and consequently determines a grading of $\cA(\cZ,\t)$ by $(G(\cZ),\lambda)$.  A choice of grading refinement data for each $\t$ thus determines a grading on $\cA(\cZ)$. 

Our choice of grading refinement data will rely on a certain subset $\bbL$ of $\cA(\cZ)$ (where $\bbL$ can refer to ``least'' or ``lowest''). Define
\[ L(j) = \min \{M^{-1}(j)\} \]
where $\min$ is taken with respect to $\lessdot$, the order induced by the orientation on $Z \setminus z$. In other words, the elements of $L(j)$ are the ``bottoms'' of each pair of matched points in $\mathbf{a}$.  Let $S_0$ denote the first $\t+k$ points (with respect to $\lessdot$) in $L([4k])$. Consider the set $\mathcal{L}$ of upward-veering, $M$-admissible partial permutations $(S_0, T, \phi)$ such that
\begin{itemize}
	\item $T \subset L([4k])$
	\item $\inv(\phi)=0$.
\end{itemize}
Given $S$ and $T$, the condition that $\inv(\phi)=0$ uniquely specifies $\phi$. 
Recall that, for $a(S, T, \phi) \in \cA(\cZ, \t)$, we have $|S|=\t+k$.  Now define 
\[ \bbL = \{ a(S_0, T, \phi) \ | \ (S_0, T, \phi) \in \mathcal{L} \}. \]
See Figure \ref{fig:algebraL} for examples. 

\begin{figure}[htb!]
\centering
\includegraphics[scale=1]{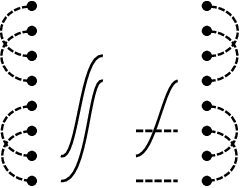}
\caption[]{Examples of elements in $\bbL$. }
\label{fig:algebraL}
\end{figure}

We define $\psi(\bft) = \lambda^{\gr (a(S_0, T, \phi))}\grunref(a(S_0, T, \phi))$ where $a(S_0, T, \phi)$ is the unique element in $\bbL$ with $\bft=M(T)$ and $\gr$ is the $\Ztwo$-grading defined in Section \ref{sec:algebrabackground}. Our base idempotent is $I(\bfs_0)$ where $\bfs_0 = M(S_0)$. It follows from the definitions that this prescribes grading refinement data.




\begin{remark}\label{rmk:reverse}
Fix a pointed matched circle $\cZ$ and a choice of grading refinement data $(\bfs_0,\psi)$ for $\cA(\cZ)$ as above.  It is shown in \cite[Chapter 10.5]{LOT} (see also \eqref{eqn:reverserefinement}) how to produce from $(\bfs_0,\psi)$ grading refinement data for $-\cZ$, called the {\em reverse of the grading refinement data}.  Note that this is {\em not} the same as the refinement data given above for $-\cZ$.  
\end{remark}

For the rest of the paper, given a pointed matched circle $\cZ$, we will work with the fixed choice of grading refinement data defined above for each $-k \leq \t \leq k$; we therefore omit this from the notation by simply writing $\grref$ and refer to $\grref$ as {\em the} refined grading on $\alg$. 

We will be particularly interested in the refined gradings of the elements in the following subset $\bbM$ (where $\bbM$ refers to ``matched''). Consider the set $\mathcal{M}$ of upward-veering, $M$-admissible partial permutations $(S, T, \phi)$ with $\phi \neq \textup{id}_S$ such that there exists a $j \in S$ satisfying
\begin{itemize}
	\item $|S| = \t+k$
	\item $\phi |_{S \backslash \{j\}} = \textup{id}_{S \backslash \{j\}}$
	\item $M(j) = M(\phi(j))$.
\end{itemize}
Then define $\bbM$ to be
\[ \bbM = \{ a(S, T, \phi) \ | \ (S, T, \phi) \in \mathcal{M} \}. \]
See Figure \ref{fig:algebraM} for examples. Note that the elements in $\bbM$ are exactly the elements $\iota \cdot a(\rho_j)$ for each minimal idempotent $\iota \in \cI_\t$ and $j \in [2k]$ where $\rho_j$ is the Reeb chord connecting the two elements of $M^{-1}(j)$. 

\begin{figure}[htb!]
\centering
\includegraphics[scale=1]{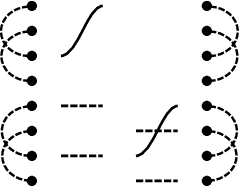}
\caption[]{Examples of elements in $\bbM$. }
\label{fig:algebraM}
\end{figure}

It follows from  \cite[Lemma 14]{Petkovadecat} that $\bfI_\t \cdot a(\rho_j)$ is $\grref$-homogeneous (although $a(\rho_j)$ is not in general $\grref$-homogenous), where $\bfI_\t$ is the unit in $\cA(\cZ, \t)$. We define $\mu_{\t,j} = \grref(\bfI_\t \cdot a(\rho_j))$ for $1 \leq j \leq 2k$.\footnote{Note that in \cite[Section 3]{Petkovadecat}, Petkova only needs the central strands grading, i.e., $ \grref(\bfI_0 \cdot a(\rho_j))$, since she does not deal with bimodules, although \cite[Lemma 14 and Proposition 15]{Petkovadecat} holds for any strands grading. }
For each $\t$, the group $G(\cZ)$ is generated by $\{ \lambda, \mu_{\t,1}, \dots, \mu_{\t,2k}\}$ subject to the relations
	\[ \mu_{\t,j} \mu_{\t,{j'}} = \mu_{\t,{j'}} \mu_{\t,j} \lambda^{2 (\rho_j \cap \rho_{j'})} \quad \textup{ and } \lambda \textup{ central}, \]  
where $\rho_j$ is viewed as an element of $H_1(F(\cZ))$ and $\rho_j \cap \rho_{j'}$ indicates the signed intersection number. 

Petkova's $\Ztwo$-reduction, $m$, of the noncommutative gradings on $\cA(\cZ,\t)$ can now be defined as follows.  Consider the assignment
\begin{equation}
\label{eqn:defnf}
	\lambda \mapsto  1,  \quad \mu_{\t,j} \mapsto 1 \text{ for } 1 \leq j \leq 2k.
\end{equation}
This data in fact determines a homomorphism $f_\t^{\cZ}:G(\cZ) \to \Ztwo$.  We choose to keep $\cZ$ in the notation for the homomorphism from $G(\cZ)$ to $\Ztwo$ since, when working with bimodules, we will consider more than one pointed matched circle at a time. Note that $f_\t^{\cZ}$ depends on a choice of grading refinement data for $\cA(\cZ)$, which we suppress from the notation.

For $a \in \cA(\cZ, \t)$, define 
\begin{equation}\label{eqn:Inaalg}
m(a) = f_\t^{\cZ} \circ \grref(a).
\end{equation}  
Since $\grref$ grades $\alg$ by $(G(\cZ),\lambda)$, it is clear that $m$ makes $\alg$ into a differential $\Ztwo$-graded algebra.  
\vspace{20pt}


\subsubsection{Identifying $m$ and $\gr$} 
Recall that $\gr$ depends on a choice of matching $M$ and choice of orientations of the $4k$ points. (In terms of a bordered Heegaard diagram, this will come from a choice of order and orientation on the $\alpha$-arcs.) This data is built into our definition of a pointed matched circle. Recall also that Petkova's $\Z/2\Z$-grading $m$ depends on a choice of grading refinement data. Our goal for the rest of the current section is to prove:

\begin{theorem}\label{thm:alggradingsagree}
For the choice of refinement data above, the $\Z/2\Z$-gradings on $\cA(\cZ)$ agree, i.e., we have $\gr = m$.  
\end{theorem}


The key step in the proof of Theorem~\ref{thm:alggradingsagree} is showing that $\gr$ is determined by its values on a small subset of $\cA(\cZ)$, namely $\bbL$ and $\bbM$. We fix a single $\t$ throughout, where $-k \leq \t \leq k$. We will then explicitly compute the different $\Z/2\Z$-gradings on this subset of $\cA(\cZ, \t)$ and see that they agree.

\begin{proposition}\label{prop:gradingreduction}
Any $\Z/2\Z$-grading on $\cA(\cZ)$ which respects the differential and the algebra structure is determined by its values on $\bbL$ and $\bbM$.  
\end{proposition}

In order to prove Proposition~\ref{prop:gradingreduction}, for what follows, we will assume that we have an arbitrary $\Z/2\Z$-grading on $\cA(\cZ,\t)$.

\begin{lemma}\label{lem:smoothings}
Let $a(S, T, \phi) \in \cA(\cZ,\t)$.  There exists an element $a(S, T, \phi') \in \cA(\cZ,\t)$ with $\inv(\phi')=0$ whose $\Z/2\Z$-grading determines that of $a(S, T, \phi)$.  
\end{lemma}


\begin{proof}
If $\inv(\phi)=0$, we are done. Otherwise, let $(i, j)$ be the smallest inversion of $\phi$ with respect to the lexicographical ordering given by $\lessdot$ (recall that the condition $i<j$ is built into the definition of an inversion). Then consider the partial permutation $(S, T, \phi \circ \tau_{i,j})$, where $\tau_{i,j}$ is the transposition interchanging $i$ and $j$. Note that since we chose $(i, j)$ to be the smallest inversion, we have that $\inv(\phi \circ \tau_{i,j}) = \inv(\phi)-1$. Indeed, the inversions of $\phi \circ \tau_{i,j}$ are exactly the inversions $(i', j')$ of $\phi$ such that $(i', j',) \neq (i, j)$.
Therefore, $(S, T, \phi \circ \tau_{i,j})$ is a non-zero term in $d(a(S, T, \phi))$. Hence, the grading of $a(S, T, \phi \circ \tau_{i,j})$ determines the gradings of $a(S, T, \phi)$; namely, their gradings differ by one.
Iterating this process, and noting that at each step we decrease the number of inversions by one, completes the proof of the lemma.
\end{proof}

\begin{lemma}\label{lem:leftendpoints}
Let $a(S, T, \phi) \in \cA(\cZ,\t)$.  The $\Ztwo$-grading of $a(S, T, \phi)$ is determined by the gradings on $\bbL$ and $\bbM$, and the grading of the element $a(S_0, T, \phi') \in \cA(\cZ,\t)$, where  
\begin{itemize}
	\item $S_0$ consists of the first $\t+k$ points in $L([4k])$
	\item $\inv(\phi')=0$.
\end{itemize}
\end{lemma}

\begin{proof}
We first apply Lemma \ref{lem:smoothings} to reduce the problem to the case that $\phi$ has no inversions. Next, if $S \not \subset L([4k])$, then choose $j \in S$ such that $j \notin L([4k])$, and consider $(S \cup \{j'\} \backslash \{j \}, S, \tau_{j, j'})$, where $j$ and $j'$ are matched by $M$ (i.e., $M(j)=M(j')$ and $j \neq j'$) and $\tau$ is the transposition interchanging $j$ and $j'$. Note that $a(S \cup \{j'\} \backslash \{j \}, S, \tau_{j, j'}) \in \bbM$. Then the grading of the product
\[ a(S \cup \{j'\} \backslash \{j \}, S, \tau_{j, j'}) \cdot a(S, T, \phi) \]
together with the grading on $\bbM$ determines the grading of $a(S, T, \phi)$. Note that since $\inv(\phi)=0$, it follows that $\inv(\tau_{j, j'} \circ \phi) = \inv (\tau_{j, j'}) + \inv(\phi)$ and hence the product above is non-zero.

By iterating the procedure in the preceding paragraph and applying Lemma \ref{lem:smoothings}, we may conclude that the grading of $a(S, T, \phi)$ is determined by the gradings on $\bbM$ and an element $a(U, T, \psi)$ such that $\inv(\psi)=0$ and $U \subset L([4k])$.
 
Since $U \subset L([4k])$, there is an element $a(S_0, U, \omega) \in \bbL$, where $S_0$ consists of the first $\t+k$ points in $L([4k])$. Then $a(S_0, U, \omega) \cdot a(U, T, \psi)$ is non-zero and equal to $a(S_0, T, \psi \circ \omega)$, where $a(S_0, T, \psi \circ \omega)$ is an algebra element satisfying the desired properties.
\end{proof}

\begin{lemma}\label{lem:rightendpoints}
If $a(S_0, T, \phi) \in \cA(\cZ,\t)$ satisfies
\begin{itemize}
	\item $S_0$ consists of the first $\t+k$ points in $L([4k])$
	\item $\inv(\phi)=0$,
\end{itemize}
as in the conclusion of Lemma~\ref{lem:leftendpoints},  then the $\Z/2\Z$-grading of $a(S_0, T, \phi)$ is determined by the $\Z/2\Z$-grading on $\bbL$ and $\bbM$.  
\end{lemma}

\begin{proof}
We will show that $a(S_0, T, \phi)$ can be constructed from the sets $\bbM$ and $\bbL$ together with applications of Lemma \ref{lem:smoothings}.

 If $T \subset L([4k])$, then $a(S_0, T, \phi) \in \bbL$ and we are done. If $T \not \subset L([4k])$, consider the unique element $a(S_0, T', \phi') \in\bbL$ with the same left and right minimal idempotents as $a(S_0, T, \phi)$. That is, $T'=L(M(T))$ and $\phi':S_0 \rightarrow T'$ is the unique partial permutation such that $\inv(\phi')=0$.
 
Choose $j \in T$ such that $j \notin L([4k])$. Consider
\[ a(S_0, T', \phi') \cdot a(T', T' \cup \{j\} \backslash \{ j' \} , \tau_{j', j}), \]
where $j$ and $j'$ are matched by $M$, and $\tau_{j', j}$ denotes the transposition interchanging $j'$ and $j$. The second factor above is clearly in $\bbM$. Note that $j' \in L([4k])$ since $j$ was not, and that the product above is non-zero since $\inv(\phi')=0$. Now apply Lemma \ref{lem:smoothings} to the above product.

Iterating this procedure for all $j \in T$ with $j \notin L([4k])$ will produce $a(S_0, T, \phi)$.  Therefore, the grading of $a(S_0, T, \psi)$ is determined by the gradings on $\bbL$ and $\bbM$, since $a(S_0, T, \psi)$ was constructed using elements in $\bbL$ and $\bbM$, together with applications of Lemma \ref{lem:smoothings}. 
\end{proof} 

\begin{proof}[Proof of Proposition~\ref{prop:gradingreduction}]
Apply Lemmas \ref{lem:leftendpoints} and \ref{lem:rightendpoints} in sequence.    
\end{proof}

We are now ready to prove Theorem \ref{thm:alggradingsagree}, showing that for our choice of refinement data, the $\Z/2\Z$-gradings $m$ and $\gr$ agree.

\begin{proof}[Proof of Theorem~\ref{thm:alggradingsagree}]
By Proposition~\ref{prop:gradingreduction}, it suffices to show that $\gr$ and $m$ agree on $\bbL$ and $\bbM$.  

Recall  that our choice of grading refinement data can be written as follows:
\begin{itemize}
	\item $\bfs_0=M(S_0)$
	\item $\psi(\bft)= \lambda^{\gr (a(S_0, T, \phi))} \grunref(a(S_0, T, \phi))$ for $a(S_0, T, \phi)$ the unique element in $\bbL$ with $\bft = M(T)$.
\end{itemize}

We first look at $\bbL$. Consider $m(a(S_0, T, \phi))$ for $a(S_0, T, \phi) \in \bbL$, using the above refinement data. Recall that
\[
m(a(S_0, T, \phi)) = f_\t^{\cZ}\big(\psi(\bfs_0) \grunref(a(S_0, T, \phi))\psi(\bft)^{-1}\big),
\]
where $\bft=M(T)$. Then, for $a(S_0, T, \phi) \in \bbL$, we have that 
\begin{align*}
  \psi(\bfs_0) \grunref(a(S_0, T, \phi))\psi(\bft)^{-1} = \lambda^{-\gr (a(S_0, T, \phi))},
\end{align*}
implying that
\[ m(a(S_0, T, \phi)) = \gr(a(S_0, T, \phi)) \pmod 2.\]

Next, we look at $\bbM$. Given $a(S, T, \phi) \in \bbM$, note that for all $j \in S$, $\phi(j)$ either equals $j$ or $j'$, where $j$ and $j'$ are matched by $M$. It follows that $\inv(\overline{\phi})=0$ since $\overline{\phi} = \id_{M(S)}$ and that 
\[ \sum_{i \in S} o(i) + \sum_{i \in T} o(i) = 1 \pmod 2\]
since for all $i \in [2k]$, $M^{-1}(i)$ is an oriented $S^0$ and elements of $\bbM$ disagree with the identity at exactly one point in the domain of the partial permutation. Hence $\gr(a(S, T, \phi)) = 1 \pmod 2$ for all $a(S, T, \phi) \in \bbM$. The fact that $m(a(S, T, \phi)) = 1 \pmod 2$ for all $a(S, T, \phi) \in \bbM$ follows from \eqref{eqn:defnf}. 
\end{proof}

\begin{remark}
Note that the proof of Theorem~\ref{thm:alggradingsagree} holds for any choice of pointed matched circle $\cZ$, that is, we do not need to assume that $M^{-1}(i)^- \lessdot M^{-1}(i)^+$, where $M^{-1}(i)^\pm$ denotes the positively/negatively oriented point in $M^{-1}(i)$; cf. Remark \ref{rmk:footnote4}. However, by Remark \ref{rem:idemconj}, such generality was not strictly necessary.
\end{remark}

\subsection{Equivalence of the two $\Z/2\Z$-gradings on $\CFA$ and $\CFD$}
\setcounter{secnumdepth}{5}
\label{sec:grequivCFACFD}

We briefly recall the non-commutative gradings on $\CFA$ and $\CFD$ from \cite[Chapter 10]{LOT}. Let $\cH$ be a bordered Heegaard diagram for a three-manifold $Y$ with parameterized boundary $F$ of genus $k$.  Let $\d \cH = \cZ$.  By \cite[Lemma 4.21]{LOT}, there is a homology class connecting $\bfx, \bfy \in \mfS(\cH)$ if and only if the associated  $\spinc$-structures coincide, i.e., $\mfs(\bfx)=\mfs(\bfy)$.  The collection of such homology classes is denoted by $\pi_2(\bfx,\bfy)$ and, as with the Heegaard Floer homology of a closed manifold,  the resulting theory decomposes according to $\spinc$-structures.  

We begin by fixing a $\spinc$-structure $\mfs$ on $Y$.  For $\bfx, \bfy \in \mfS(\cH, \mfs)$, we define a map
	\[ g': \pi_2(\bfx, \bfy) \rightarrow G'(\cZ) \]
as
	\[ g'(B) = (-e(B)-n_\bfx(B)-n_\bfy(B), \d^\d B), \]
where
\begin{itemize}
	\item $e(B)$ denotes the Euler measure of $B$
	\item $n_\bfx(B) = \sum_{x \in \bfx} n_x(B)$ where $n_x(B)$ denotes the average of the local multiplicities of $B$ in the four regions surrounding $x$
	\item $\d^\d B$ denotes the portion of the boundary of $B$ contained in $\d \Sigma$.
\end{itemize}
(Recall that the Euler measure of a region in $\Sigma \setminus (\bfalpha \cup \bfbeta)$ is the Euler characteristic minus $1/4$ the number of corners, since the corners of $\Sigma \setminus (\bfalpha \cup \bfbeta)$ are all acute, and is additive under unions.) 

Given refinement data $\bfs_0, \psi$ for $\cA(\cZ)$, we define
	\[ g: \pi_2(\bfx, \bfy) \rightarrow G(\cZ) \]
as
	\[ g(B) = \psi(I_A(\bfx)) g'(B) \psi(I_A(\bfy))^{-1}. \]
Given a fixed reference point $\bfx_0 \in \mfS(\cH, \mfs)$, we grade $\CFA(\cH,\mfs)$ (in the sense of \cite[Definition 2.42]{LOT}) by the right $G(\cZ)$-set 
	\[ S_A(\cH, \bfx_0) = P_{\bfx_0} \backslash G(\cZ) \]
where $P_{\bfx_0} = g(\pi_2(\bfx_0, \bfx_0)) \subset G(\cZ)$. This grading is given by
	\[ g_A(\bfx) = P_{\bfx_0} \cdot g(B) \]
for $B \in \pi_2(\bfx_0, \bfx)$.

We grade $\CFD(\cH)$ by the left $G(-\cZ)$-set
	\[ S_D(\cH, \bfx_0) = G(-\cZ) / R(P_{\bfx_0}),  \]
where $R: G(\cZ) \rightarrow G(-\cZ)$ is the map induced by the (orientation reversing) identity map $r: Z \rightarrow -Z$ \cite[Chapter 10.4]{LOT}. Then
	\[ g_D(\bfx) = R(g(B))\cdot R(P_{\bfx_0}). \]

\begin{remark}
	Note that we use $g_A$ and $g_D$ rather than $\gr$ to denote the $G(\cZ)$-set gradings on $\CFA(\cH)$ and $\CFD(\cH)$ from \cite[Chapter 10]{LOT}, since we have chosen to use $\gr$ to denote our $\Ztwo$-grading.
\end{remark}

In \cite[Definition 13]{Petkova}, Petkova defines a relative $\Ztwo$-grading on $\CFA(\cH, \mfs)$ as
	\[ m_A(\bfy)-m_A(\bfx)=f_0^\cZ \circ g(B) \pmod 2\]
and a relative $\Ztwo$-grading on $\CFD(\cH, \mfs)$ as
	\[ m_D(\bfy)-m_D(\bfx)=f_0^{-\cZ} \circ R \circ g(B) \pmod 2. \]
The goal of this subsection is to prove the following theorem.

\begin{theorem}
\label{thm:gradingsagree}
Let $\cH$ be a bordered Heegaard diagram with $\d \cH = \cZ$. For generators $\bfx, \bfy \in \mfS(\cH, \mfs)$ and $B \in \pi_2(\bfx, \bfy)$, we have
\begin{enumerate}
\item\label{gradingsagreethm:1}$\gr_A(\bfy) - \gr_A(\bfx) = f_0^{\cZ} \circ g(B) \pmod 2$
\item \label{gradingsagreethm:2} $\gr_D(\bfy) - \gr_D(\bfx) = f_0^{- \cZ} \circ R \circ g(B) \pmod 2$;
\end{enumerate}
that is, the two relative $\Z/2\Z$-gradings on bordered Floer homology agree.
\end{theorem}

As discussed earlier, Theorem~\ref{thm:gradingsagree} is our method of proof that $\gr_A$ (respectively $\gr_D$) induces $\cA_\infty$-gradings on $\CFA(\cH)$ (respectively gradings on $\CFD(\cH)$).  

\subsubsection{$\CFA$}\label{subsec:CFAagree}
We will now prove Theorem \ref{thm:gradingsagree}\eqref{gradingsagreethm:1} as follows. We begin by constructing a Heegaard diagram for a closed manifold $\cH'$, which is associated to $\cH$. Generators $\bfx, \bfy \in \mfS(\cH, \mfs)$ will extend to generators $\bfx', \bfy' \in \mfS(\cH')$, and a domain $B \in \pi_2(\bfx, \bfy)$ in $\cH$ will extend to a domain $B' \in \pi_2(\bfx', \bfy')$ in $\cH'$. We then use the fact that we have a $\Ztwo$-grading on $\CFhat(\cH')$ to compare the gradings of $\bfx'$ and $\bfy'$, and hence also of $\bfx$ and $\bfy$.

In order to construct the closed Heegaard diagram $\cH'$, we follow the approach of \cite[Proposition 3.5]{HuangRamos}, in particular Step 2 and the end of Step 5 of their proof. 

Given a pointed matched circle $\cZ$, recall that the orientation of $Z$ and basepoint $z$ induces an ordering $\lessdot$ on $\bfa$.  For $j \in [2k]$, suppose that
\begin{enumerate}
	\item $M^{-1}(j) = \{a^-_j,a^+_j\}$, where $a^-_j \lessdot a^+_j$ and $a^-_j$ (respectively $a^+_j$) is oriented negatively (respectively positively).
	\item $M$ is chosen such that $a^-_j \lessdot a^-_{i}$ if and only if $j < i$. 
\end{enumerate}
A matching $M$ satisfying the two conditions above is said to be \emph{subordinate to $\lessdot$}.
Throughout the current section (\ref{subsec:CFAagree}), we will assume that the matching $M$ associated to $\cZ$ is subordinate to $\lessdot$.  Recall that, as described in Section~\ref{sec:borderedHD}, the matching determines the ordering on the $\alpha$-arcs in $\cH$.  This eases the exposition and does not cause any loss of generality by Remark \ref{rem:idemconj}.  
\vspace{.07in}

\paragraph{{\em A warm-up for Theorem~\ref{thm:gradingsagree}\eqref{gradingsagreethm:1}}}\label{subsubsec:sameidem}
We first prove Theorem~\ref{thm:gradingsagree}\eqref{gradingsagreethm:1} in the case where $I_A(\bfx) = I_A(\bfy) = I(\bfs_0)$ for a warm-up, where $\bfs_0 = \{1,\ldots,k\}$.  

\begin{proposition}
\label{prop:sameidem}
Let $\bfx, \bfy \in \mfS(\cH, \mfs)$, $B \in \pi_2(\bfx, \bfy)$, and $\d \cH = \cZ$. If $I_A(\bfx)=I_A(\bfy)=I(\bfs_0)$, then
	\[ \gr_A(\bfy) - \gr_A(\bfx) = f_0^\cZ \circ g(B) \pmod 2. \]
\end{proposition}

In order to prove Proposition~\ref{prop:sameidem}, we will need a number of constructions.  We begin with two constructions that will build the closed Heegaard diagram $\cH'$ discussed above.  See Figure~\ref{fig:constructions_1_and_2} for a graphical depiction.
\begin{construction}
\label{con:closeddiagram}
Let $\Sigma'$ be the closed surface obtained from $\Sigma$ by gluing a compact surface of genus $k$ with boundary $-Z$ to $\Sigma$ along the boundary. We define a closed Heegaard diagram $\cH' = (\Sigma', \bfalpha', \bfbeta', z)$ as follows. Complete each $\alpha$-arc $\alpha^a_i$ on $\Sigma$ to an $\alpha$-circle $\alpha'_i$ in $\Sigma'$ such that the collection $\bfalpha'=\{ \alpha^c_1, \dots, \alpha^c_{g-k}, \alpha'_1, \dots, \alpha'_{2k} \}$ is pairwise-disjoint and linearly independent in $H_1(\Sigma')$. Consider $k$ translates of $Z \setminus \nu ( z )$ in a collar neighborhood of $\partial \Sigma$ in $\Sigma' \setminus \Sigma$, and order them, increasingly, according to their distance to $\partial \Sigma$.  Add $k$ $\beta$-circles, $\beta'_1,\ldots,\beta'_k$, to $\Sigma' \setminus \nu(\Sigma)$ such that $\beta'_i$ contains the $i^\textup{th}$ translate of $Z \setminus \nu(z)$ and the collection $\bfbeta' = \{ \beta_1, \dots, \beta_g, \beta'_1, \dots, \beta'_k\}$ is linearly independent in $H_1(\Sigma')$. Lastly, isotope the $\beta'_1, \ldots, \beta'_k$ such that $\beta'_i$ intersects $\alpha'_j$ in $\Sigma' \setminus \nu(\Sigma)$ for all $i, j$.  
\end{construction}

In our setting, any $\alpha$-circles which intersect $Z$ will always intersect in exactly two points.  For notation, we denote these intersections by $\alpha^\pm$, where $\alpha^- \lessdot \alpha^+$ and we equip these points with positive orientations.  Since $I_A(\bfx)=I_A(\bfy)$, given $B \in \pi_2(\bfx, \bfy)$, we can express $[ \dd B]$ as $(h_1, \ldots, h_{2k}) \in H_1(F(\cZ))$, represented by the sum of the intervals $a_i$ in $Z\setminus \nu(z)$ from $\alpha'^-_i$ to $\alpha'^+_i$ with multiplicity $h_i$.   

\begin{construction}\label{con:doubling}
For each $i \in \bfs_0$ with $h_i \neq 0$, remove a neighborhood of a point $p_i \in \alpha'_i \cap (\Sigma' \setminus \nu(\Sigma))$ which avoids the $\beta$-circles. Add an $\alpha$-arc parallel to $\alpha'_i$ in $\Sigma' \setminus \nu(p_i)$ and identify the boundaries of $\Sigma' \setminus \nu(p_i)$ and $\Sigma_1$, where $\Sigma_1$ is the decorated punctured torus in Figure \ref{fig:sigmai}, so as to obtain two new $\alpha$-circles (instead of one).   We denote these new $\alpha$-circles by $\alpha'_{i, 1}$ and $\alpha'_{i, 2}$, ordered such that $\alpha'^-_{i, 1}$ appears immediately before $\alpha'^-_{i, 2}$ along $Z$.  Perform an isotopy such that the new $\beta$-circles coming from the copies of $\Sigma_1$ each contains a translate of $Z \setminus \nu(z)$ in a collar neighborhood of $\d \Sigma$ in $\Sigma' \setminus \Sigma$.  Isotope the resulting $\beta$-circles to intersect each $\alpha'_j$ for all $j$.  We abuse notation and also refer to the resulting surface as $\Sigma'$ and the collection of $\beta$-circles as $\bfbeta' = \{ \beta_1, \dots, \beta_g, \beta'_1, \dots, \beta'_\ell\}$ where $\ell = k + \# \{i \mid i \in \bfs_0, h_i \neq 0\}$ and $\beta'_i$ contains the $i^\textup{th}$ translate of $Z \setminus \nu(z)$, which are again ordered by distance to $\partial \Sigma$. We give the $\alpha'_{i, 1}$ and $\alpha'_{i, 2}$ the same orientation as $\alpha'_i$ in $\Sigma' \setminus \nu(p_i)$, and we may orient the $\beta'_i$ however we would like.  Observe that when restricting to $\Sigma$, we have an identification of $\alpha'_i$ and $\alpha'_{i,1}$; we will implicitly use this identification, especially for intersection points on these curves.
\end{construction}

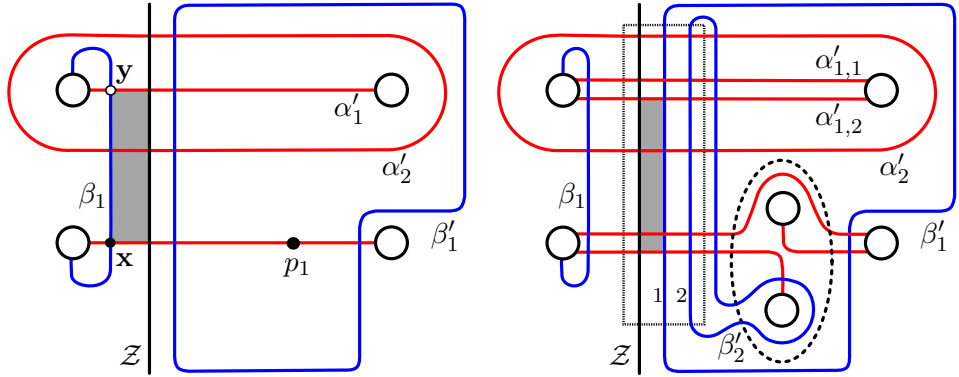
\begin{figure}[htb!]
\labellist
\pinlabel $\bfy$ at 29.5 75.5 \pinlabel $\bfx$ at 29.5 29
	\pinlabel $\cZ$ at 31 5 \pinlabel $\cZ$ at 154 5
	\pinlabel $\beta_1$ at 21.5 45 \pinlabel $\beta_1$ at 142 45
	\pinlabel $\beta_1'$ at 110 35 \pinlabel $\beta_1'$ at 233 35
	\pinlabel $\alpha_2'$ at 98 52 \pinlabel $\alpha_2'$ at 223 52 
	\pinlabel $\alpha_1'$ at 86 67 
	\pinlabel $p_1$ at 73 27
	\pinlabel $\alpha_{1,2}'$ at 209 64 \pinlabel $\alpha_{1,1}'$ at 209 78.5
	\pinlabel $\beta_2'$ at 182 7
	\scriptsize
	\pinlabel $1$ at 163.5 20 \pinlabel $2$ at 169.5 20
\endlabellist
\includegraphics[scale=1.5]{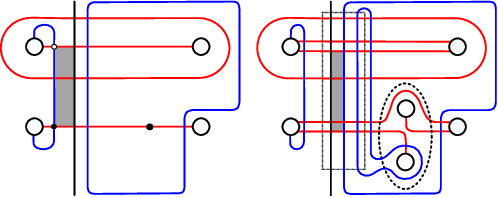}
\caption{A sample implementation of Construction \ref{con:closeddiagram} (left) followed by Construction \ref{con:doubling} (right). Left, the Heegaard diagram $\cH'$ built out of the original diagram $\cH$, which consists of everything to the left of $\cZ$.  Right, in this example, $h_1 \neq 0$, and the corresponding rectangle $R_1$ has been shaded in the Heegaard diagram resulting from Construction \ref{con:doubling}; the grid formed by the new $\alpha$- and $\beta$-curves has been highlighted with portions of the relevant $\beta$-curves labelled according to their corresponding translates of the pointed matched circle.}
\label{fig:constructions_1_and_2}
\end{figure}

\begin{remark}\label{rem:alphaorientations}
Throughout this section, we will repeat the construction of creating parallel $\alpha$-curves and attaching the decorated surface $\Sigma_n$ to our Heegaard diagram. Whenever we do so, we identify the boundaries in such as way as to create $n+1$ new $\alpha$-circles (the maximum possible) and will orient the new $\alpha$-circles so that their orientation agrees with $\alpha'_i$ in $\Sigma$.  Further, when restricting to $\Sigma$, as in the previous construction, we continue to make an identification of $\alpha'_{i,1}$ with $\alpha'_i$, including the intersection points on these curves.
\end{remark}

The result of these constructions contains a rectangular grid inside of a collar neighborhood of $\partial \Sigma$ in $\Sigma'$, with segments of $\beta$-circles as specified translates of $Z \setminus \nu(z)$, which are perpendicular to segments of $\alpha$-circles.  See Figure~\ref{fig:constructions_1_and_2}.  Note that there may be other segments of $\beta$-circles which are contained in this grid (which may also run parallel to $Z \setminus \nu(z)$), which we would like to ignore to keep this grid as clear as possible.  From now on, whenever we refer to a portion of the $\beta$-circle in a neighborhood of $Z$, by abuse of notation, we will only be considering the portion which is on the specified translate of $Z \setminus \nu(z)$ that it has been isotoped to contain.  Note that inside of $\nu(\Sigma)$ the change to the $\alpha$-curves from Constructions~\ref{con:closeddiagram} and \ref{con:doubling} simply consists of the addition of a parallel copy of $\alpha_i$ whenever $i \in \bfs_0$ and $h_i \neq 0$.  In $\nu(\Sigma)$, there has been no change to the set of $\beta$-circles.

\begin{figure}[htb!]
\centering
\labellist
	\pinlabel $\alpha$ at 6 95
	\pinlabel $\alpha$ at 24 80
	\pinlabel $\alpha$ at 45 49
	\pinlabel $\beta$ at 60 57
	\pinlabel $\beta$ at 70 144
\endlabellist
\includegraphics[scale=1]{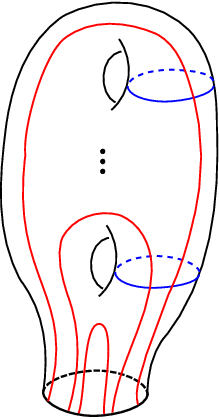}
\vspace{-10pt}
\caption[]{The surface $\Sigma_n$, a genus $n$ surface with boundary.}
\label{fig:sigmai}
\end{figure}
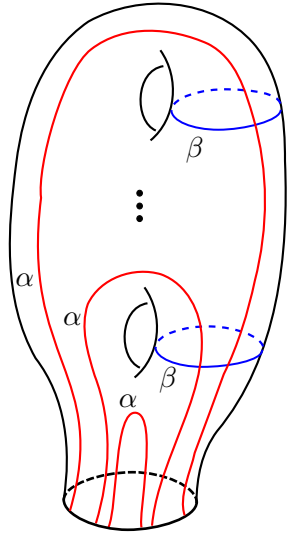

If $i \notin \bfs_0$, let $R_i$ denote the rectangle in a collar neighborhood of $Z \setminus \nu(z)$ with boundary consisting of 
\begin{itemize}
	\item $a_i$, the segment in $Z\setminus \nu(z)$ between $\alpha'^-_i$ and $\alpha'^+_i$
	\item the segment in $\alpha'_i \cap \nu(Z)$ from $\alpha'^-_i$ to $\beta'_{j_i}$
	\item the segment in $\alpha'_i \cap \nu(Z)$ from $\alpha'^+_i$ to $\beta'_{j_i}$
	\item the translate of $a_i$ to $\beta'_{j_i}$, 
\end{itemize}
where $j_i = \# \{ i' \leq i \mid i' \notin \bfs_0 \textup{ or } h_{i'} \neq 0\}$. The index $j_i$ is so defined because for each $n \in \bfs_0$ with $h_{n} \neq 0$ and for each $n \notin \bfs_0$, we have added a new $\beta$-circle to $\bfbeta'$; see Figure \ref{fig:constructions_1_and_2}.
We denote the segment in the last bullet point by $a'_i$.

If $i \in \bfs_0$ and $h_i \neq 0$, let $R_i$ denote the rectangle in a collar neighborhood of $Z \setminus \nu(z)$ with boundary consisting of 
\begin{itemize}
	\item the segment in $Z\setminus \nu(z)$ between $\alpha'^-_{i, 1}$ and $\alpha'^+_{i, 2}$
	\item the segment in $\alpha'_{i, 1} \cap \nu(Z)$ from $\alpha'^-_{i, 1}$ to $\beta'_{j_i}$
	\item the segment in $\alpha'_{i, 2} \cap \nu(Z)$ from $\alpha'^+_{i, 2}$ to $\beta'_{j_i}$ 
	\item the 
	translate in $\beta'_{j_i}$ such that these four segments bound a rectangle. 
\end{itemize}
We denote the segment in the last bullet point by $a'_i$. See Figure \ref{fig:Ri} and compare to Figure \ref{fig:constructions_1_and_2}.  

In fact, Figure \ref{fig:Ri} suggests how we could {\em attempt} to extend $B$ to a domain between generators for $\cH'$; namely one could append $R_i$ with multiplicity $h_i$ to (a slight perturbation of) $B$ and add to $\bfx$ and $\bfy$ the intersection points between $\beta'_{j_i}$ and $\alpha'_i$ (or $\alpha'_{i,1}$ and $\alpha'_{i,2}$).  Observe the necessity in the distinction made between the cases of $i \in \bfs_0$ and $i \notin \bfs_0$ for this attempt.  The issue is that when $|h_i| > 1$, this process does not produce a domain and thus we have to make further modifications.  The idea is to construct a domain $B'$, which will  (roughly) consist of $B \subset \Sigma$ together with the rectangles $R_i$ with multiplicity $h_i$; the higher multiplicities will be dealt with by further refining the grid and using multiple translates of the $R_i$.  More precisely, we begin by modifying $\bfx, \bfy,$ and $B$ as follows.

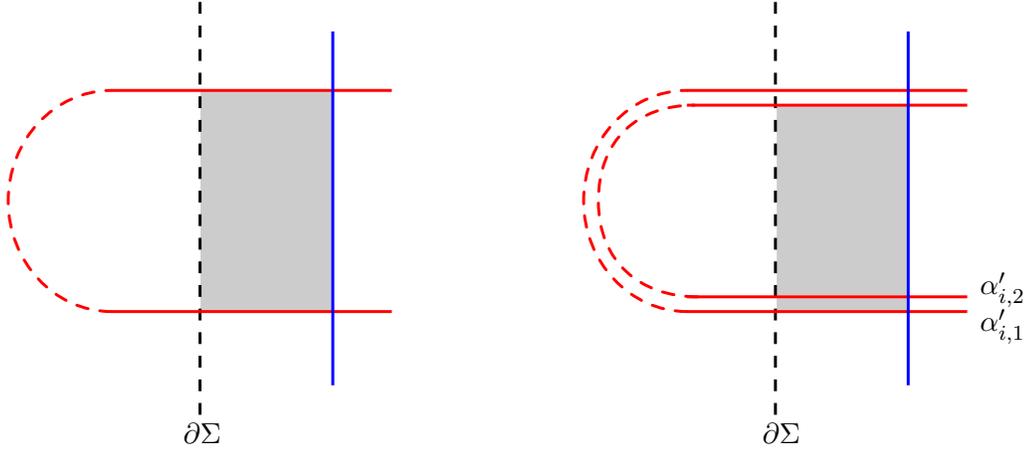
\begin{figure}[htb!]
\centering
\labellist
	\pinlabel $\d \Sigma$ at 53 -5
	\pinlabel $\d \Sigma$ at 210 -5
	\pinlabel $\alpha'_{i, 1}$ at 270 24
	\pinlabel $\alpha'_{i, 2}$ at 270 34
\endlabellist
\includegraphics[scale=1.4]{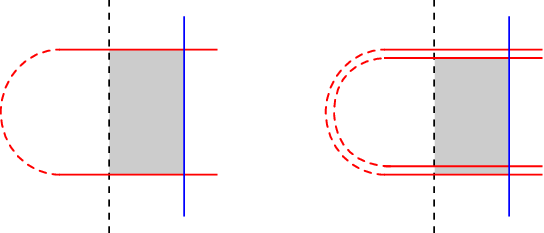}
\vspace{10pt}
\caption[]{A rectangle $R_i$ when $i \notin \bfs_0$ (left) and when $i \in \bfs_0$ (right). }
\label{fig:Ri}
\end{figure}

\begin{construction}\label{con:doublepoints}
\begin{itemize}
	\item[] 
	\item For each $i \in \bfs_0$ with $h_i > 0$, we move the intersection point $y_i \in \bfy$ between $\alpha'_i$ and $\beta'_j$ to the nearby intersection point $y'_i$ between $\alpha'_{i, 2}$ and $\beta'_j$. We then subtract the thin rectangle in $\Sigma$ bounded by $\alpha'_{i, 1}$ and $\alpha'_{i,2}$ with corners at $y_i, y'_i, \alpha'^+_{i, 2}$, and $\alpha'^+_{i, 1}$ from $B$.
	\item For each $i \in \bfs_0$ with $h_i < 0$, we move the intersection point in $x_i \in \bfx$ between $\alpha'_i$ and $\beta'_j$ to the nearby intersection point $x'_i$ between $\alpha'_{i, 2}$ and $\beta'_j$. We then add the thin rectangle in $\Sigma$ bounded by $\alpha'_{i, 1}$ and $\alpha'_{i,2}$ with corners at $x_i, x'_i, \alpha'^+_{i, 2}$, and $\alpha'^+_{i, 1}$ to $B$.
\end{itemize}	
We refer to these modified tuples of intersections as $\bfx_*$ and $\bfy_*$, and the modified domain as $B_*$.  These are {\em not} generators of any sensible Floer complex.  
\end{construction}

We modify the closed Heegaard diagram $\cH'$ and add intersection points to $\bfx_*$ (respectively $\bfy_*$) to obtain generators $\bfx'$ (respectively $\bfy'$) in $\CFhat(\cH')$ as follows. 
\begin{construction}\label{con:sameidempotent}
\begin{enumerate}
	\item[]
	\item\label{same:1} For $i \notin \bfs_0$ with $h_i=0$, we add a single fixed intersection point between $\alpha'_i$ and $\beta'_{j_i}$ in $\Sigma' \setminus \nu(\Sigma)$ to both $\bfx_*$ and $\bfy_*$. 
	\item\label{same:2}  For $i \in \bf{s}_0$ and $h_i = 0$, no change is needed.
	\item\label{same:3} For $i \notin\bfs_0$ with $|h_i| = 1$, we add $a'^{\, \sgn \; h_i}_i$ to $\bfx_*$  (respectively $a'^{ -\sgn \; h_i}_i$ to $\bfy_*$), where $a'^{\pm}_i$ denote the two endpoints of $a'_i$, oriented so that $a'^{\pm}_i$ is a translate of $\alpha'^\pm_i$ to $\beta'_{j_i}$ along $\alpha'_i$ in $\nu(Z)$.  For consistency with notation to appear shortly, we rename $R_i$ as $R_{i,1}$.  
	\item\label{same:4} For $i \in\bfs_0$ with $|h_i| = 1$, we add $a'^{\, \sgn \; h_i}_i$ to $\bfx_*$  (respectively $a'^{ -\sgn \; h_i}_i$ to $\bfy_*$), where $a'^{\pm}_i$ denote the two endpoints of $a'_i$, oriented so that $a'^-_i$ (respectively $a'^+_i$) is a translate of $\alpha'^-_{i,1}$ (respectively $\alpha'^+_{i,2}$) to $\beta'_{j_i}$  along $\alpha'_{i,1}$ (respectively $\alpha'_{i,2}$) in $\nu(Z)$.  We rename $R_i$ as $R_{i,1}$.  
	\item\label{same:5}  For $i \notin \bfs_0$ and $|h_i| > 1$, we remove a neighborhood of a point $p _i \in \alpha'_i \cap \Sigma' \setminus \nu(\Sigma)$, where $p_i$ is disjoint from the $\beta$-circles, add $(|h_i| -1)$ $\alpha$-arcs parallel to $\alpha'_i$, and identify the boundaries of $\Sigma' \setminus \nu(p_i)$ and $\Sigma_{|h_i|-1}$, where $\Sigma_{|h_i|-1}$ is the surface in Figure \ref{fig:sigmai}. Perform an isotopy such that the new $\beta$-circles each contain a translate of $Z \setminus \nu(z)$ in a collar neighborhood of $\d \Sigma$, and these translates appear between $\beta'_{j_i}$ and $\beta'_{j_i+1}$.  We relabel $\beta'_{j_i}$ as $\beta'_{j_i, 1}$ and label the new $\beta$-circles (in the obvious order) $\beta'_{j_i,2}, \dots, \beta'_{j_i, |h_i|}$. We label the resulting $\alpha$-circles as $\alpha'_{i, 1}, \dots, \alpha'_{i, |h_i|}$, ordered so that along $Z$, the endpoints are ordered as $\alpha'^-_{i, 1}, \dots, \alpha'^-_{i, |h_i|}$. Let $a'_{i, n}$ denote the translate to $\beta'_{j_i, n}$ of the segment $a_{i,n}$ from $\alpha'^-_{i, n}$ to $\alpha'^+_{i, n}$. We add $a'^{\sgn \; h_i}_{i, n}$ to $\bfx_*$ (respectively $a'^{-\sgn \; h_i}_{i, n}$ to $\bfy_*$), where $a'^\pm_{i,n}$ denote the two endpoints of $a'_{i, n}$, for $1 \leq n \leq |h_i|$. We define $R_{i,n}$ to be the rectangle in a neighborhood of $Z \setminus \nu(z)$ bounded by $a'_{i,n}$ and $a_{i,n}$.  See Figure \ref{fig:hi}.
	\item\label{same:6}  For $i \in \bfs_0$ and $|h_i| > 1$, we proceed essentially as above. That is, we remove a neighborhood of a point $p _i \in \alpha'_{i,2} \cap \Sigma' \setminus \nu(\Sigma)$, add $(|h_i| -1)$ $\alpha$-arcs parallel to $\alpha'_{i,2}$, and identify the boundaries of $\Sigma' \setminus \nu(p_i)$ and $\Sigma_{|h_i|-1}$, where $\Sigma_{|h_i|-1}$ is the surface in Figure \ref{fig:sigmai}. Perform an isotopy so that the new $\beta$-circles each contain a translate of $Z \setminus \nu(z)$ in a collar neighborhood of $\d \Sigma$, and these translates appear between $\beta'_{j_i}$ and $\beta'_{j_i+1}$. We relabel $\beta'_{j_i}$ as $\beta'_{j_i, 1}$ and label the new $\beta$-circles $\beta'_{j_i, 2}, \dots, \beta'_{j_i, |h_i|}$ based on the obvious order. We label the resulting $\alpha$-circles $\alpha'_{i, 2}$ through $\alpha'_{i, |h_i|+1}$, ordered so that along $Z$, the endpoints are ordered as $\alpha'^-_{i, 1}, \dots, \alpha'^-_{i, |h_i|+1}$. Let $a'_{i,1}$ be the translate to $\beta'_{i,1}$ of the segment $a_{i,1}$ from $\alpha'^-_{i,1}$ to $\alpha'^+_{i,2}$.  For $2 \leq n \leq |h_i|$, let $a'_{i, n}$ denote the translate to $\beta'_{j_i, n}$ of the segment $a_{i,n}$ from $\alpha'^-_{i, n+1}$ to $\alpha'^+_{i, n+1}$.  We add $a'^{\sgn \; h_i}_{i, n}$ to $\bfx_*$ and $a'^{-\sgn \; h_i}_{i,n}$ to $\bfy_*$ for $1 \leq n \leq |h_i|$.  We define $R_{i,n}$ to be the rectangle in $\nu(Z)$ bounded by $a'_{i,n}$ and $a_{i,n}$ for $1 \leq n \leq |h_i|$.   
\end{enumerate}
This yields a new closed Heegaard diagram, which we continued to call $\cH'$ by abuse.  Further, $\bfx'$ and $\bfy'$ determine generators of $\CFhat(\cH')$.  Let $r_{i,n}$ denote the thin rectangle in $\Sigma$ between $\alpha'_{i,n}$ and $\alpha'_{i,n+1}$, that is, $r_{i,n}$ has boundary $\alpha'_{i,n} \cap \Sigma$, $\alpha'_{i,n+1} \cap \Sigma$, and two segments in $\d \Sigma$, one between $\alpha'^-_{i, n}$ to $\alpha'^-_{i, n+1}$ and the other between $\alpha'^+_{i, n}$ to $\alpha'^+_{i, n}$. There is a natural way to add a linear combination of the thin rectangles $r_{i, n}$ in $\Sigma$ between each $\alpha'_{i,n}$ and $\alpha'_{i,n+1}$ to $B_* + \sum_i \sum_{1 \leq n \leq |h_i|} (\sgn \; h_i) R_{i,n}$ such that the result is an element of $\pi_2(\bfx',\bfy')$.   More explicitly, we subtract
\begin{equation}
\label{eqn:thinrectangles}
	 \sum_{i \notin \bfs_0}\Bigg(\sum_{1 \leq n \leq |h_i|-1} (\sgn \; h_i) \, (|h_i| -n) r_{i, n} \Bigg) + \sum_{i \in \bfs_0} \Bigg( (\sgn \; h_i) \, (|h_i| -1) r_{i, 1} + \sum_{1 \leq n \leq |h_i|-1} (\sgn \; h_i) \, (|h_i| -n) r_{i, n+1}\Bigg)
\end{equation}
from $B_* + \sum_i \sum_{1 \leq n \leq |h_i|}(\sgn \; h_i) R_{i,n}$ and denote the resulting element of $\pi_2(\bfx',\bfy')$ by $B'$. 
See Figures \ref{fig:hi} and \ref{fig:examples} for examples of these modifications.
\end{construction}

\begin{figure}[htb!]
\centering
\labellist
\pinlabel $\d \Sigma$ at 113 140
\pinlabel $\alpha'_i$ at 175 173
\pinlabel $\beta'_{j_i}$ at 152 150

\pinlabel $\d \Sigma$ at 268 140
\pinlabel $\alpha'_{i,3}$ at 330 183
\pinlabel $\alpha'_{i,1}$ at 330 171
\pinlabel $\beta'_{j_i,1}$ at 303 152
\pinlabel $\beta'_{j_i,3}$ at 321 152

\pinlabel $\d \Sigma$ at 58 -7
\pinlabel $\alpha'_i$ at 118 22
\pinlabel $\beta'_{j_i}$ at 96 2

\pinlabel $\d \Sigma$ at 210 -7

\pinlabel $\d \Sigma$ at 332 -7
\pinlabel $\alpha'_{i,4}$ at 395 33
\pinlabel $\alpha'_{i,1}$ at 395 17
\pinlabel $\beta'_{j_i,1}$ at 368 0
\pinlabel $\beta'_{j_i,3}$ at 385 0
\endlabellist
\includegraphics[scale=1.2]{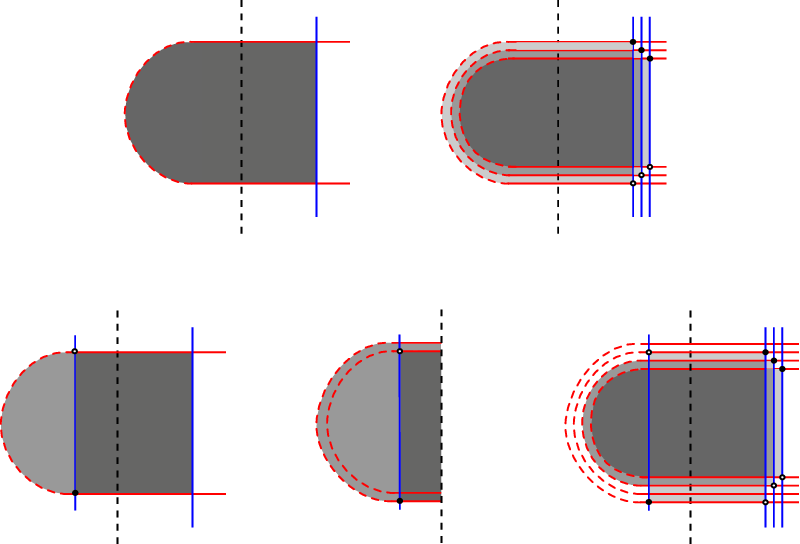}
\vspace{10pt}
\caption[]{Top row, the modification to the domain $B$ when $i \notin \bfs_0$ for $h_i=3$, to create the domain $B'$ between two generators of the Heegaard Floer chain complex of the closed Heegaard diagram.
Bottom row, the modification to the domain $B$ when $i \in \bfs_0$ for $h_i=3$, where the bottom center figure is $B_*$.}
\label{fig:hi}
\end{figure}
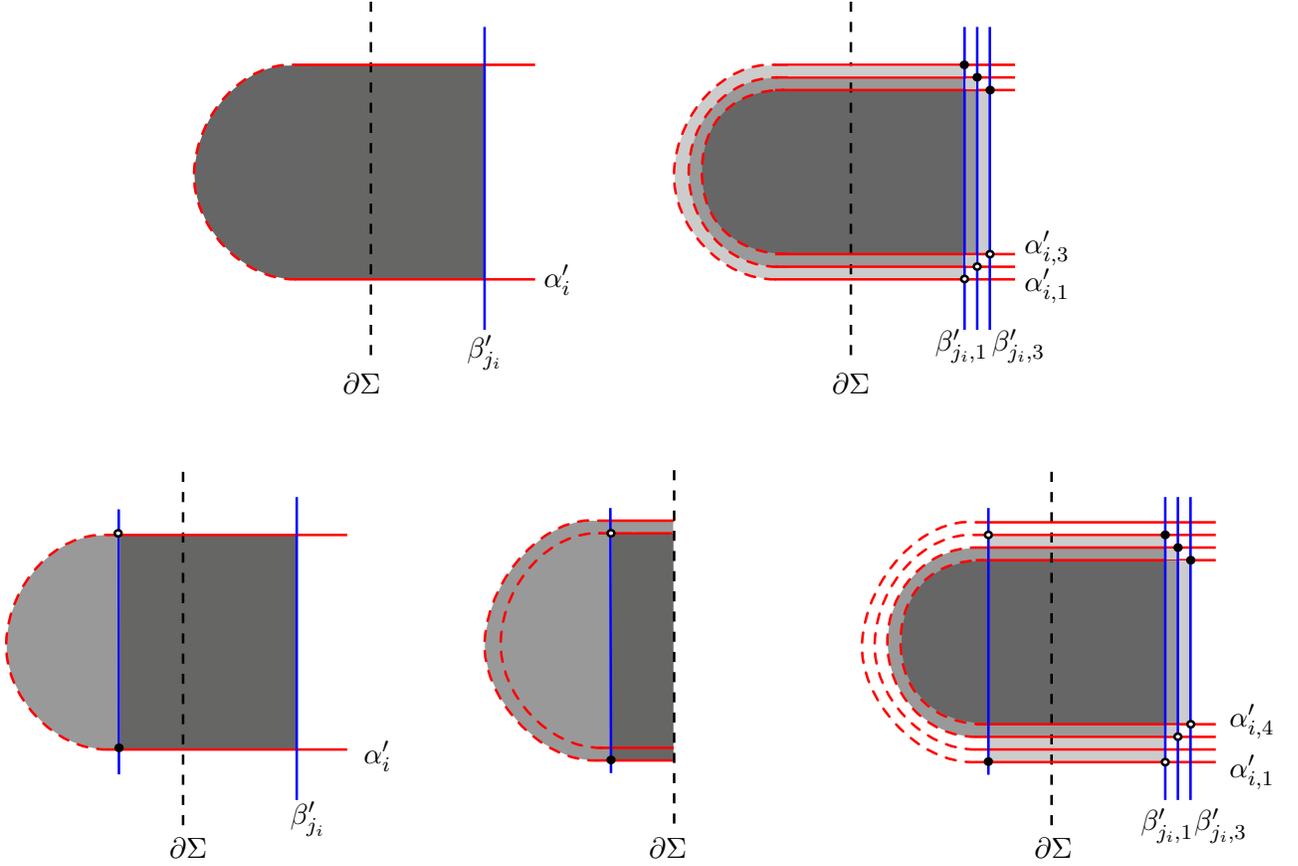

\begin{figure}[htb!]
\centering
\subfigure{
\labellist
\pinlabel $B$ at 50 88
\pinlabel $\d \Sigma$ at 60 -7
\pinlabel $\alpha^a_2$ at 50 50
\pinlabel $\alpha^a_1$ at 50 26

\pinlabel $B'$ at 198 88
\pinlabel $\alpha'_1$ at 254 32
\pinlabel $\alpha'_{2,1}$ at 256 48
\pinlabel $\alpha'_{2,2}$ at 256 58
\pinlabel $\d \Sigma$ at 195 -7
\pinlabel $\beta'_1$ at 235 -8
\pinlabel $\beta'_2$ at 244 -8
\endlabellist
\includegraphics[scale=1.4]{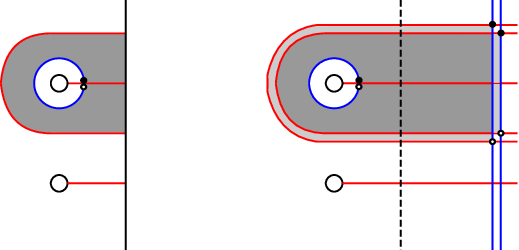}
}
\vspace{20pt}

\subfigure{
\labellist
\pinlabel $B$ at 52 70
\pinlabel $\d \Sigma$ at 64 -7
\pinlabel $\alpha^a_1$ at 52 28
\pinlabel $\alpha^a_2$ at 52 51

\pinlabel $B_*$ at 138 69
\pinlabel $\d \Sigma$ at 145 -7

\pinlabel $B'$ at 237 69
\pinlabel $\alpha'_{1,1}$ at 282 25
\pinlabel $\alpha'_{1,2}$ at 282 34
\pinlabel $\alpha'_{2}$ at 280 57
\pinlabel $\d \Sigma$ at 230 -7
\pinlabel $\beta'_1$ at 253 -8
\endlabellist
\includegraphics[scale=1.4]{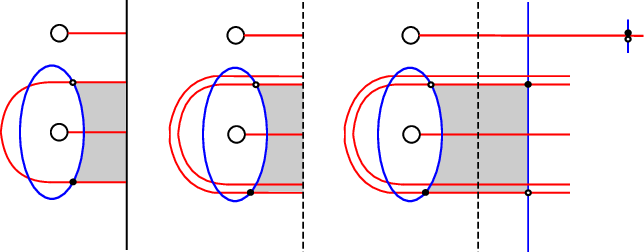}
\label{subfig:BB*B'}
}
\vspace{20pt}
\caption[]{Examples of $B \in \pi_2(\bfx, \bfy)$ with $I(\bfx)=I(\bfy)=I(\bfs_0)$ and the modified domains $B' \in \pi_2(\bfx', \bfy')$. Top, $h_2=2$. Bottom, $h_1=1$.}
\label{fig:examples}
\end{figure}

We now compare the gradings of $\bfx, \bfy \in \mfS(\cH)$ with those of $\bfx', \bfy' \in \mfS(\cH')$, where $\cH'$, $\bfx'$, and $\bfy'$ are as described above in Construction~\ref{con:sameidempotent}. Let $g'(B)=(-e(B) -n_\bfx(B) -n_\bfy(B) ; h_1, \dots, h_{k})$ as described above.  Since $I_A(\bfx) = I_A(\bfy) = I(\bfs_0)$, we have $g(B) = (-e(B) -n_\bfx(B) -n_\bfy(B) ; h_1, \dots, h_{k})$ as well.  

\begin{lemma}\label{lem:grAMsame}  For $\bfx, \bfy, \bfx', \bfy'$ as above,  
\begin{equation}\label{eqn:grAM} \gr_A(\bfy) - \gr_A(\bfx) = M(\bfy') - M(\bfx') + \sum_{i \notin \bfs_0} h_i + \sum_{\substack{ i \in \bfs_0 \\ h_i \neq 0}} (|h_i| -1) \pmod 2
\end{equation} 
where $M$ denotes the Maslov grading modulo $2$. 
\end{lemma}

\begin{proof}
We use $\sigma^{-1}_\bfx$ to denote the inverse of $\sigma_\bfx$ restricted to its image.  Observe that $\sigma_\bfy\sigma^{-1}_\bfx$ is a permutation of a $g$-element subset of $[g+k]$, since $I_A(\bfx) = I_A(\bfy)$.  Further, note that $\sgn_A(\sigma_\bfy) + \sgn_A(\sigma_\bfx) =\inv(\sigma_\bfy\sigma^{-1}_\bfx) \pmod 2$, and that $\inv (\sigma_{\bfy'} \sigma_{\bfx'}^{-1})$ and $\inv(\sigma_{\bfy} \sigma_{\bfx}^{-1})$ are independent of the choice of order of the $\alpha$- and $\beta$-circles.  We make the following observations.  
\begin{enumerate}
	\item By Construction~\ref{con:sameidempotent}\eqref{same:1}, if $h_i=0$ and $i \notin \bfs_0$, then the same fixed intersection point is added to both $\bfx'$ and $\bfy'$; hence there is no contribution to the change in their relative gradings.
	\item By Construction~\ref{con:sameidempotent}\eqref{same:2}, if $h_i=0$ and $i \in \bfs_0$, then there is no change to the intersection points on $\alpha'_i$ between $\bfx'$ and $\bfx$ (respectively $\bfy'$ and $\bfy$).
	\item By Construction~\ref{con:sameidempotent}\eqref{same:3} and \eqref{same:5}, for $i \notin \bfs_0$ and $h_i \neq 0$, compared to $\bfx$ and $\bfy$, the generators $\bfx'$ and $\bfy'$ each contain $|h_i|$ additional intersection points $x^i_1, \dots, x^i_{|h_i|}$ and $y^i_1, \dots, y^i_{|h_i|}$ such that $x^i_j$ and $y^i_j$ both lie on the same $\alpha$- and $\beta$-circles, but with different intersection signs.  See tops of Figures~\ref{fig:hi} and \ref{fig:examples}.  Thus, there is no contribution, mod 2, to the inversion number.  We conclude the new intersection points in $\bfx'$ and $\bfy'$ coming from such an $i$ induce a change in the relative gradings by $h_i$.
	\item If $i \in \bfs_0$ and $h_i=1$, the intersection point in $\bfx$ (respectively $\bfy$) on $\alpha^a_i$ has been translated to $\alpha'_{i, 1}$ in $\bfx_*$ and thus also in $\bfx'$ (respectively $\alpha'_{i, 2}$ for $\bfy'$).  By Construction~\ref{con:sameidempotent}\eqref{same:4}, we have also that the generator $\bfx'$ (respectively $\bfy'$) contains the additional intersection point $a'^+_{i,2}$ on $\alpha'_{i,2} \cap \beta'_{j_i}$ (respectively $a'^-_{i,1}$ on $\alpha'_{i,1} \cap \beta'_{j_i}$); note that $a'^+_{i,2}$ and $a'^-_{i,1}$ have opposite intersection signs. The combination of these two changes from $\bfx$ and $\bfy$ to $\bfx'$ and $\bfy'$ introduces an extra inversion to $\sigma_{\bfy'} \sigma_{\bfx'}^{-1}$ relative to $\sigma_{\bfy} \sigma_{\bfx}^{-1}$. This cancels with the change from the intersection signs, and thus in this case, there is no contribution to the difference in gradings. The case of $i \in \bfs_0$ and $h_i = -1$ is similar.    
	\item The case $i \in \bfs_0$ and $h_i > 1$ is a combination of the two preceding items.  Namely, we apply the argument in the preceding item, using Construction~\ref{con:sameidempotent}\eqref{same:6}, with the additional caveat that $\bfx'$ (respectively $\bfy'$) contains the $h_i - 1$ intersection points $a'^+_{i,n}$ (respectively $a'^-_{i,n}$) for $2 \leq n \leq h_i$.  We observe that $a'^\pm_{i,n}$ both lie on $\alpha'_{i,n+1} \cap \beta'_{j_i,n}$ for such $n$, but with different intersection signs.  These additional $h_i - 1$ intersection points therefore contribute a total of $h_i - 1$ to the change in the intersection signs between $\bfx', \bfy'$ and $\bfx, \bfy$. Also, these additional $h_i - 1$ intersection points do not contribute to the change in the number of inversions of $\sigma_{\bfy'} \sigma_{\bfx'}^{-1}$ compared to $\sigma_{\bfy} \sigma_{\bfx}^{-1}$.  Therefore, by the arguments given in the previous item, the change to the inversion number is one, while the change to the intersection signs is $h_i$.  See bottom of Figure~\ref{fig:hi}. It now follows that the total contribution to the change in the relative gradings for such an $i$ is given by $h_i - 1$.  The case of $i \in \bfs_0$ and $h_i < -1$ is similar. 	
\end{enumerate}
The lemma now follows from the above observations and the fact that $M(\bfx') - M(\bfy') = \gr(\bfx') - \gr(\bfy')$, where $\gr$ is the $\Ztwo$-grading for $\CFhat$ of closed three-manifolds defined in \eqref{eq:closed-z2-grading}.  
\end{proof}

We next note that
\begin{equation}\label{eqn:eulerBB'}
 e(B') = e(B). 
 \end{equation} 
Indeed, $B'$ differs from $B$ by a collection of rectangles, which have Euler measure zero. These rectangles are the thin rectangles in Construction~\ref{con:doublepoints}, the $R_{i,n}$ in $\nu(Z)$ defined in Construction~\ref{con:sameidempotent}, and the $r_{i,n}$ in $\Sigma$ that lie between the various $\alpha'_{i,n}$ and $\alpha'_{i,n+1}$ as in \eqref{eqn:thinrectangles}.

We next compare the multiplicities of the generators along the domains $B$ and $B'$; that is, we determine $n_{\bfx'}(B') + n_{\bfy'}(B') - \Big( n_\bfx(B) + n_\bfy(B) \Big)$.  Recall that $\rho_i$ denotes the Reeb chord in $Z$ from $\alpha'^-_i$ to $\alpha'^+_i$.

\begin{lemma}\label{lem:samenxy}
For $\bfx, \bfy, \bfx', \bfy', B, B'$ as above, 
\begin{equation}\label{eqn:multiplicityBB'}
n_{\bfx'}(B') + n_{\bfy'}(B') - \Big( n_\bfx(B) + n_\bfy(B) \Big) = \frac{1}{2} \sum_i h_i  + \sum_{i < j} h_i h_j L([\rho_i], [\rho_j]) + \sum_{\substack{ i \in \bfs_0 \\ h_i \neq 0}} (|h_i| -1) \pmod 2. 
\end{equation}
\end{lemma}
\begin{proof}
We can compute the right hand side of \eqref{eqn:multiplicityBB'} using Construction \ref{con:sameidempotent} as follows.
\begin{enumerate}
	\item The first term on the right hand side of \eqref{eqn:multiplicityBB'} consists of the contribution of the multiplicity of $B'$ at the intersection points in $\bfx'$ and $\bfy'$ on the NE and SE corners of $R_{i,n}$ for $1 \leq n \leq |h_i|$. 
	\item The second term on the right hand side comes from the overlapping rectangles $R_{i_1,n_1}$ and $R_{i_2,n_2}$ in $\Sigma' \setminus \Sigma$, as in Figure \ref{fig:Lij}.  
\item The third term on the right hand side comes from the fact that for $i \in \bfs_0$ and $h_i \neq 0$, the sum of the multiplicities at the intersection points $x_i$ and $y_i$ on $\alpha_i$ in $\Sigma$ is changed by $| h_i | - 1$ because of the addition of $|h_i| - 1$ new $\alpha$-circles and the subsequent modification to $B$, coming from the addition of the thin rectangles $r_{i,n}$ as in \eqref{eqn:thinrectangles}. See the bottom row of Figure \ref{fig:hi}.  
\end{enumerate}
The first and second items account for the contributions from the addition of the appropriate rectangles $R_{i,n}$, and the third item accounts for the the addition of the appropriate thin rectangles $r_{i,n}$.  Thus, these are all of the contributions, and the lemma is complete. 
\end{proof}

\begin{figure}[htb!]
\centering
\labellist
\pinlabel $\d \Sigma$ at 56 -5
\endlabellist
\includegraphics[scale=1.4]{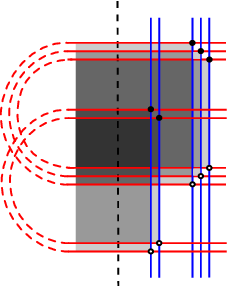}
\vspace{10pt}
\caption[]{Two $\alpha$-circles in $\cH'$ which lead to the linking term in \eqref{eqn:multiplicityBB'}.}
\label{fig:Lij}
\end{figure}

We are now ready to complete the proof of Proposition~\ref{prop:sameidem}.  
\begin{proof}[Proof of Proposition~\ref{prop:sameidem}]
Recall that 
\[ g(B) = ( -e(B) - n_\bfx(B) - n_\bfy(B); h_1, \ldots, h_{2k}). \]
By \cite[Proposition 15]{Petkovadecat}, 
\begin{equation}\label{eqn:inaB}
f_0^\cZ \circ g(B) =  -e(B) -n_\bfx(B) -n_\bfy(B) -\frac{1}{2} \sum_{i \in \bfs_0} h_i +\frac{1}{2} \sum_{i \notin \bfs_0} h_i + \sum_{i < j} h_i h_j L([\rho_i], [\rho_j]) \pmod 2.
\end{equation}

Observe that while in general $L(\gamma_1,\gamma_2)$ need not be integral for $\gamma_i \in H_1(Z',\bfa)$, it is straightforward to verify that $L([\rho_i], [\rho_j])$ is integral. Also, recall that given $B' \in \pi_2(\bfx', \bfy')$ for $\bfx', \bfy' \in \mfS(\cH')$, by \cite[Corollary 4.10]{Lipshitz} we have
\begin{equation}\label{eqn:Mclosed}
M(\bfy') - M(\bfx') = -e(B') -n_{\bfx'}(B') -n_{\bfy'}(B').
\end{equation}
Now, combining \eqref{eqn:grAM}, \eqref{eqn:eulerBB'}, \eqref{eqn:multiplicityBB'}, \eqref{eqn:inaB}, and \eqref{eqn:Mclosed}, we have

\begin{align*}
	\gr_A(\bfy) - \gr_A(\bfx) &= M(\bfy') - M(\bfx') + \sum_{i \notin \bfs_0} h_i + \sum_{\substack{ i \in \bfs_0 \\ h_i \neq 0}} (h_i -1) \\
		&= -e(B') -n_{\bfx'}(B') -n_{\bfy'}(B') + \sum_{i \notin \bfs_0} h_i + \sum_{\substack{ i \in \bfs_0 \\ h_i \neq 0}} (h_i -1) \\
		&= -e(B) -n_\bfx(B) -n_\bfy(B) -\frac{1}{2} \sum_i h_i  - \sum_{i < j} h_i h_j L([\rho_i], [\rho_j]) + \sum_{\substack{ i \in \bfs_0 \\ h_i \neq 0}} (h_i -1) \\
			& \quad + \sum_{i \notin \bfs_0} h_i + \sum_{\substack{ i \in \bfs_0 \\ h_i \neq 0}} (h_i -1) \\
		&= -e(B) -n_\bfx(B) -n_\bfy(B) -\frac{1}{2} \sum_{i \in \bfs_0} h_i +\frac{1}{2} \sum_{i \notin \bfs_0} h_i - \sum_{i < j} h_i h_j L([\rho_i], [\rho_j]) \\
		&= f_0^\cZ \circ g(B) \pmod 2,
\end{align*}
as desired.
\end{proof}

\paragraph{{\em Towards the general case of Theorem~\ref{thm:gradingsagree}\eqref{gradingsagreethm:1}}}
Suppose there exists $\bfx \in \mfS(\cH, \mfs)$ such that $I_A(\bfx) = I(\bfs_0)$. We return to the proof when this hypothesis is not satisfied in Section \ref{subsec:generalCFAagree}. To prove Theorem~\ref{thm:gradingsagree}\eqref{gradingsagreethm:1} under this hypothesis, it is sufficient to consider $B \in \pi_2(\bfx, \bfy)$ for $\bfy \in \mfS(\cH, \mfs)$ where $I_A(\bfy) = I(\bft)$ for some $\bft \subset [2k]$. This is sufficient since for arbitrary $\bfz \in \mfS(\cH, \mfs)$ we will have 
\begin{align*}
	\gr_A(\bfz) - \gr_A(\bfy) &= \gr_A(\bfz) - \gr_A(\bfx) + \gr_A(\bfx) - \gr_A(\bfy) \\
		&= f_0^\cZ \circ g(B_1) - f_0^\cZ \circ g(B_2) \\
		&= f_0^\cZ \circ g(-B_2 * B_1) \pmod 2
\end{align*}
where $B_1 \in \pi_2(\bfx, \bfz)$ and $B_2 \in \pi_2(\bfx, \bfy)$, and so $-B_2 * B_1 \in \pi_2(\bfy, \bfz)$ where $*$ denotes concatenation of homology classes corresponding to addition of domains.  Here, we are using the fact that $g(-B_2 * B_1) = g(-B_2)g(B_1)$, which is \cite[Lemma 10.32]{LOT}.  By this discussion, the following proposition suffices to complete the proof of Theorem~\ref{thm:gradingsagree}\eqref{gradingsagreethm:1} in the presence of $\bfx$ with $I_A(\bfx) = I(\bfs_0)$.

\begin{proposition}
\label{prop:diffidem}
Let $\bfx, \bfy \in \mfS(\cH, \mfs)$ and $B \in \pi_2(\bfx, \bfy)$. If $I_A(\bfx) = I(\bfs_0)$ and $I_A(\bfy) = I(\bft)$, then
	\[ \gr_A(\bfy) - \gr_A(\bfx) = f_0^\cZ \circ g(B). \]
\end{proposition}

The strategy will follow the same approach as that of Proposition~\ref{prop:sameidem}.  Namely, we would like to construct a closed Heegaard diagram $\cH'$ with generators $\bfx', \bfy'$ whose relative gradings we can compare to those of $\bfx, \bfy$.  In this general case, we require an additional modification to the Heegaard diagram, which we will give more motivation for shortly. Recall that for $B \in \pi_2(\bfx, \bfy)$ with $I_A(\bfx)=I(\bfs_0)$, we have
\[ g(B) = g'(B) \psi(\bft)^{-1}. \]
The additional modification to the Heegaard diagram we will make is related to the grading refinement data $\psi(\bft)$.

Let
\begin{align*}
	\bfs_0 ' &= \bfs_0 \setminus (\bfs_0 \cap \bft) \\
	\bft ' &= \bft \setminus (\bfs_0 \cap \bft) \\
	c &= k - | \bfs_0 \cap \bft | = |\bfs'_0| = | \bft'|. 
\end{align*}
Given a set $X = \{ x_1, \dots, x_n\} \subset \Z$, let $J(X)$ be the ordered tuple $(x_1, \ldots, x_n)$ such that $x_1 < \dots < x_n$.  Let $J(\bfs'_0)=(s_1, \dots, s_c)$ and $J(\bft') = (t_1, \dots, t_c)$. Recall that $\bfs_0 = \{1,\ldots,k\}$.  Note that $t_i>k$ for each $i$. Define $\rho^{i,i'}$ to be the Reeb chord from $\alpha^-_{i}$ to $\alpha^-_{i'}$ whenever $i \neq i'$, and $\bfrho^{\bfs'_0, \bft'}$ to be $\{ \rho^{s_i, t_i} \mid 1 \leq i \leq c \}$. We will also consider $\bfrho^{\bfs_0, \bft}$, defined analogously.  Note that $\bfrho^{\bfs'_0,\bft'}$ and $\bfrho^{\bfs_0,\bft}$ induce the same element $[\bfrho^{\bfs'_0,\bft'}]=[\bfrho^{\bfs_0,\bft}] \in H_1(Z',\bfa)$.  The grading refinement data specified in Section \ref{sec:algebranoncomm} is equivalent to
\[ \psi(\bft) = \grunref(I(\bfs_0) \cdot a(\bfrho^{\bfs_0, \bft})) \]
when $M$ is subordinate to $\lessdot$, which we have assumed throughout the current section (\ref{subsec:CFAagree}).
Our modifications to the Heegaard diagram will be more closely related to $\bfrho^{\bfs'_0, \bft'}$ rather than $\bfrho^{\bfs_0, \bft}$.


As before, we follow Construction \ref{con:closeddiagram} to form a closed surface $\Sigma'$ and an associated closed Heegaard diagram $\cH' = (\Sigma,\bfalpha', \bfbeta', z)$ for $\bfalpha' = \{\alpha_1,\ldots\alpha_{g-k}, \alpha'_1,\ldots\alpha'_{2k}\}$, $\bfbeta' = \{\beta_1,\ldots\beta_g,\beta'_1,\ldots\beta'_k\}$.  Let $B \in \pi_2(\bfx, \bfy)$ and $g(B) = (-e(B) -n_\bfx(B) -n_\bfy(B); h_1, \dots, h_{2k})$.  Recall that we are using the Reeb chords $\rho_i$ from $\alpha'^-_i$ to $\alpha'^+_i$ as our basis for $H_1(F(\cZ))$, and $(h_1,\ldots,h_{2k})$ represents the coordinates in this basis.  Our goal, as in Section~\ref{subsubsec:sameidem}, is to modify $\cH'$ to create the relevant generators $\bfx', \bfy'$ and domain $B'$ to allow for the comparison of the relative gradings of $\bfx'$ and $\bfy'$ to those of $\bfx$ and $\bfy$. For notation, let $\tilde{Z}$ denote a translate of $Z$ which lies between the portions of $\beta'_c$ and $\beta'_{c+1}$ in $\nu (Z)$.  We call the region between $Z$ and $\tilde{Z}$ the {\em refinement neighborhood}, denoted $RN(Z)$.  

Here is an outline of the argument we will use.  The main issue that prevents us from applying the arguments in Section~\ref{subsubsec:sameidem}, is that $\partial^\partial B$ corresponds to a class in $H_1(Z',\bfa)$, but does not correspond to an element of $H_1(F(\cZ))$ if $\bft \neq \bfs_0$.  Therefore, there is not a natural way to extend $B$ into a closed domain by adding rectangles $R_{i,n}$, $1 \leq i \leq 2k$ as in Construction~\ref{con:sameidempotent}.  In order to construct the domain $B'$ in the closed Heegaard diagram we will first extend $B$ across the refinement neighborhood and make further modification to obtain $\tilde{B}$ such that $\dd \tilde{B}$ naturally corresponds to the element $\vec{h} \in H_1(F(\cZ))$.\footnote{We consider $\tilde{B}$ as a domain in $\Sigma \cup RN(Z)$. Then the definition of $\dd$ will carry over, where we use $\tilde{Z}$ in place of $Z$.}  From this, we can extend $\tilde{B}$ to obtain $B'$ by repeating the constructions in Section~\ref{subsubsec:sameidem}.  We now describe the construction of $\tilde{B}$.  

\begin{construction}\label{con:refine}
	For each $\rho^{s_i, t_i}$, where $1 \leq i \leq c$, let $b_i$ denote the translate of $\rho^{s_i, t_i}$ to $\beta'_i$. Add $b_i^-$ to $\bfy$ and $b_i^+$ to $\bfx$.  Call the resulting tuples of intersection points $\tilde{\bfx}$ and $\tilde{\bfy}$.  Define rectangles $R_{s_i,t_i}$ bounded by $b_i$ and its translate to $\tilde{Z}$.  We are now ready to define $\tilde{B}$.  First, extend $B$ productwise across the refinement neighborhood.  We now obtain $\tilde B$ by appending $-R_{s_i,t_i}$ for each $1 \leq i \leq c$.  See Figure~\ref{fig:examplesdiffidem}.  
\end{construction}

Since $\bfrho^{\bfs_0',\bft'}$ and $\bfrho^{\bfs_0,\bft}$ both correspond to the same element of $H_1(F(\mathcal{Z}))$, it follows from the construction that $\dd \tilde{B}$ is in fact an element of $H_1(F(\mathcal{Z}))$.  To obtain $\bfx', \bfy', B'$ as in Section~\ref{subsubsec:sameidem}, we now repeat Construction~\ref{con:doubling}, where we replace $\bfx, \bfy$ with $\tilde{\bfx}, \tilde{\bfy}$, $B$ with $\tilde{B}$, $\Sigma$ with $\Sigma \cup RN(Z)$, and $\bfs_0$ with $\bfs_0 \cup \bft$.  We relabel $\beta'_1,\ldots,\beta'_\ell$ as $\beta^r_1,\ldots,\beta^r_c,\beta'_1,\ldots,\beta'_{\ell-c}$, maintaining the order listed, where $\ell = k + \# \{i \mid i \in \bfs_0 \cup \bft, h_i \neq 0\}$. Here, the $r$ stands for refinement data. 
While $\tilde{\bfx}, \tilde{\bfy}$ are not generators of a Heegaard Floer chain complex and $\tilde{B}$ is not a domain, this is not necessary to carry out the topological construction.  Next, we repeat Constructions~\ref{con:doublepoints} and \ref{con:sameidempotent} mutatis mutandis.

\begin{figure}[htb!]
\centering
\labellist
	\pinlabel $\d \Sigma$ at 62 -5
	\pinlabel $\d \Sigma$ at 142 -5
	\pinlabel $\tilde{Z}$ at 177 -5
	\pinlabel $\d \Sigma$ at 270 -5
	\pinlabel $\tilde{Z}$ at 299 -5
\endlabellist
\includegraphics[scale=1.2]{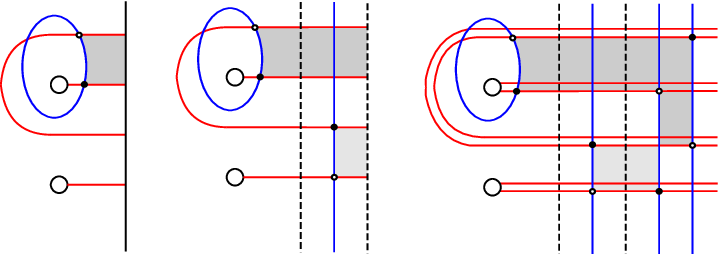}
\vspace{10pt}
\caption[]{Left, a domain $B$ in $\Sigma$. Center, $\tilde{B}$. Right, the associated domain $B'$ in $\Sigma'$. Note that the small rectangle has multiplicity $-1$, while the rest of $B'$ has multiplicity $1$, corresponding to $h_1 = -1$, $h_2 = +1$. The refinement neighborhood is contained between the two dashed lines.}
\label{fig:examplesdiffidem}
\end{figure}
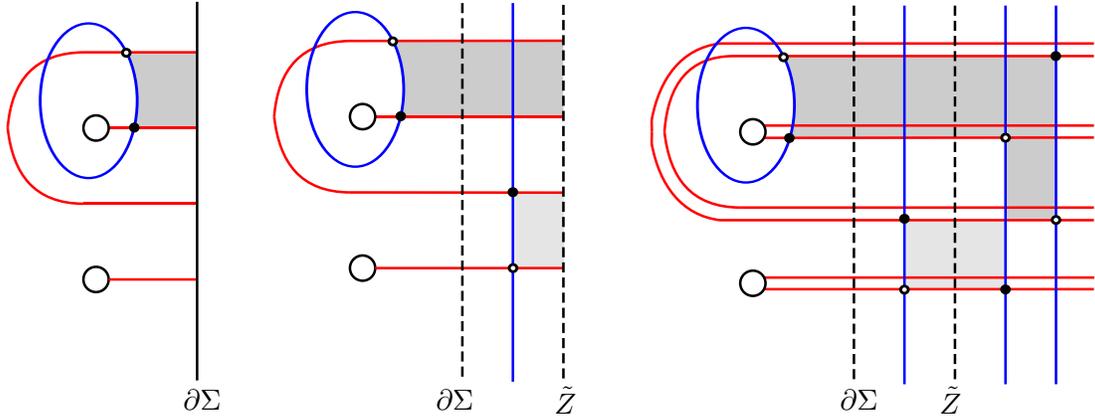

We will compare the gradings of $\bfx, \bfy \in \mfS(\cH)$ with those of $\bfx', \bfy' \in \mfS(\cH')$, where $\bfx', \bfy', $ and $\cH'$ are as described above. Recall that $g(B) = ( -e(B) - n_\bfx(B) - n_\bfy(B); h_1, \dots, h_{2k})$ and $(h_1, \dots, h_{2k}) = \d^\d\tilde{B} \in H_1(F(\cZ))$.

Let $S_0 = \{ \alpha^-_i \mid i \in \bfs_0 \}$ and $T = \{ \alpha^-_i \mid i \in \bft \}$. 

\begin{lemma}
Let $\bfx, \bfy, \bfx', \bfy'$ be as above.  Then, 
\begin{equation}\label{eqn:diffgrA}
 \gr_A(\bfy) - \gr_A(\bfx) = M(\bfy') - M(\bfx') + c + m( [\bfrho^{\bfs'_0, \bft'}], S_0 \cap T) + \sum_{i \notin \bfs_0 \cup \bft } h_i + \sum_{\substack{ i \in \bfs_0 \cup \bft \\ h_i \neq 0}} (|h_i| -1) \pmod 2, 
 \end{equation}
where $M(\bfx')$ denotes the Maslov grading of $\bfx'$ modulo $2$. 
\end{lemma}
\begin{proof}
First, we analyze the contribution to the change in relative gradings from $\bfx$ and $\bfy$ to $\tilde \bfx$ and $\tilde \bfy$ coming from the refinement neighborhood.\footnote{While $\tilde \bfx, \tilde \bfy$ are not generators of a Heegaard Floer chain complex, there is an obvious extension of the definition of $\gr_A$.}  Again, we have $\sgn_A(\sigma_\bfx) + \sgn_A(\sigma_\bfy) = \inv(\sigma_\bfy\sigma_\bfx^{-1}) \pmod 2$, where we use $\sigma_\bfx^{-1}$ to denote the inverse of $\sigma_\bfx$ restricted to its image. Note that $\sigma_\bfy\sigma_\bfx^{-1}$ is an injection from $[g]$ to a $g$-element subset of $[g+k]$, since $I_A(\bfx) = \bfs_0$, and $\sigma_{\tilde \bfy} \sigma_{\tilde \bfx}^{-1}$ is a permutation of a $(g+c)$-element subset of $[g+k]$. The set $[g+k]$ is in bijection with the $\alpha$-circles, ordered as $\alpha^c_1, \dots, \alpha^c_{g-k}, \alpha'_1, \dots, \alpha'_{2k}$.

We consider the change in the number of inversions of $\sigma_{\tilde \bfy}\sigma^{-1}_{\tilde \bfx}$ relative to $\sigma_{\bfy}\sigma^{-1}_{\bfx}$; this change comes from the refinement neighborhood. We claim that
	\begin{equation}\label{eqn:invST} 
\inv(\sigma_{\tilde \bfy}\sigma^{-1}_{\tilde \bfx}) - \inv(\sigma_{\bfy}\sigma^{-1}_{\bfx}) = c + m( [\bfrho^{\bfs'_0, \bft'}], S_0 \cap T) \pmod 2. 
	\end{equation}
Since we have added intersection points to $\bfx$ (respectively $\bfy$) to obtain $\tilde{\bfx}$ (respectively $\tilde{\bfy})$, we will count the number of newly introduced inversions. Note that by our ordering, the first $g - k$ intersection points of $\bfx, \bfy, \tilde{\bfx}, \tilde{\bfy}$ (ordered according to the $\beta$-circles on which they sit) are on the original $\alpha$-circles in $\Sigma$. Therefore, the inversions that we have introduced consists of inversions of the form $(i, j)$ where $g - k + 1 \leq i<j \leq g + k$ and either
\begin{enumerate}
	\item \label{it:diffgrA1} $\sigma_{\tilde \bfy}\sigma^{-1}_{\tilde \bfx}(i) \in [g+1, g+k]$ and $\sigma_{\tilde \bfy}\sigma^{-1}_{\tilde \bfx}(j) \in [g-k+1, g]$ 
	\item \label{it:diffgrA2} $\sigma_{\tilde \bfy}\sigma^{-1}_{\tilde \bfx}(i) \in [g-k+1, g]$,  $\sigma_{\tilde \bfy}\sigma^{-1}_{\tilde \bfx}(j) \in [g-k+1, g]$ and $\sigma_{\tilde \bfy}\sigma^{-1}_{\tilde \bfx}(j) < \sigma_{\tilde \bfy}\sigma^{-1}_{\tilde \bfx}(i)$.
\end{enumerate}
See Figure \ref{fig:diffgrA}. Note that by our definition of $\rho^{s_n, t_n}$ (recall that each $s_n$ corresponds to one of the first $k$ $\alpha$-arcs), there are no new inversions of the form $\sigma_{\tilde \bfy}\sigma^{-1}_{\tilde \bfx}(i) \in [g+1, g+k]$ and $\sigma_{\tilde \bfy}\sigma^{-1}_{\tilde \bfx}(j) \in [g+1, g+k]$.  The number of pairs $(i, j)$ satisfying \eqref{it:diffgrA1} is $c^2$. The set of pairs $(i, j)$ satisfying \eqref{it:diffgrA2} is equivalent to

\begin{equation}
\label{eqn:pairsij}
	\{ (u, v) \mid 1 \leq v \leq c, u \in \bfs_0 \cap \bft, s_v < u\},
\end{equation}
where $u= \sigma_{\tilde \bfy}\sigma^{-1}_{\tilde \bfx}(i)$ and $s_v=\sigma_{\tilde \bfy}\sigma^{-1}_{\tilde \bfx}(j)$. Indeed, if $(i,j)$ is a new inversion satisfying \eqref{it:diffgrA2}, then $\sigma_{\tilde \bfy}\sigma^{-1}_{\tilde \bfx}(j)$ must equal $s_n$ for some $n$ and $\sigma_{\tilde \bfy}\sigma^{-1}_{\tilde \bfx}(i)$ must be in both $\bfs_0$ and $\bft$.

Notice that $u \in \bfs_0 \cap \bft$ and $s_v<u$ if and only if $u \in \bfs_0 \cap \bft$ and $s_v < u < t_v$. The number of elements in $\{ (u, v) \mid 1 \leq v \leq c, u \in \bfs_0 \cap \bft, s_v < u < t_v\}$ is $m( [\bfrho^{\bfs'_0, \bft'}], S_0 \cap T)$. This completes the proof of the claim that \eqref{eqn:invST} holds.

\begin{figure}[htb!]
\centering
\subfigure[]{
\begin{tikzpicture}[scale=0.50]
	\draw[step=1, black!30!white, very thin] (0, 0) grid (10, 10);
	\draw[-] (2, 0) -- (2, 10);
	\draw[-, dashed] (6, 0) -- (6, 10);
	\draw[-] (0, 2) -- (10, 2);
	\draw[-, dashed] (0, 6) -- (10, 6);
	\filldraw (0.5, 2.5) circle (2pt);
	\filldraw (1.5, 4.5) circle (2pt); 
	\filldraw (4.5, 1.5) circle (2pt); 
	\filldraw (6.5, 5.5) circle (2pt);
	\filldraw (7.5, 0.5) circle (2pt); 
	\filldraw (9.5, 3.5) circle (2pt); 
	
	\draw [decorate,decoration={brace,amplitude=4pt}] (-0.2, 0) -- (-0.2, 1.95) node [midway, left, xshift=-2pt] {$g-k$};
	\draw [decorate,decoration={brace,amplitude=4pt}] (-0.2, 2.05) -- (-0.2, 5.95) node [midway, left, xshift=-2pt] {$k$};
	\draw [decorate,decoration={brace,amplitude=4pt}] (-0.2, 6.05) -- (-0.2, 9.95) node [midway, left, xshift=-2pt] {$k$};
	\draw [decorate,decoration={brace,amplitude=4pt,mirror}] (0, -0.2) -- (1.95, -0.2) node [midway, below, yshift=-2pt] {$g-k$};
	\draw [decorate,decoration={brace,amplitude=4pt,mirror}] (2.05, -0.2) -- (5.95, -0.2) node [midway, below, yshift=-2pt] {$k$};
	\draw [decorate,decoration={brace,amplitude=4pt,mirror}] (6.05, -0.2) -- (9.95, -0.2) node [midway, below, yshift=-2pt] {$k$};
\end{tikzpicture}
}
\hspace{20pt}
\subfigure[]{
\begin{tikzpicture}[scale=0.50]
	\draw[step=1, black!30!white, very thin] (0, 0) grid (10, 10);
	\draw[-] (2, 0) -- (2, 10);
	\draw[-, dashed] (6, 0) -- (6, 10);
	\draw[-] (0, 2) -- (10, 2);
	\draw[-, dashed] (0, 6) -- (10, 6);
	\filldraw (0.5, 2.5) circle (2pt);
	\filldraw (1.5, 4.5) circle (2pt); 
	\filldraw (2.5, 6.5) circle (2pt); 
	\filldraw (3.5, 7.5) circle (2pt);
	\filldraw (4.5, 1.5) circle (2pt); 
	\filldraw (5.5, 9.5) circle (2pt); 
	\filldraw (6.5, 5.5) circle (2pt);
	\filldraw (7.5, 0.5) circle (2pt); 
	\filldraw (9.5, 3.5) circle (2pt); 
	
	\draw [decorate,decoration={brace,amplitude=4pt}] (-0.2, 0) -- (-0.2, 1.95) node [midway, left, xshift=-2pt] {$g-k$};
	\draw [decorate,decoration={brace,amplitude=4pt}] (-0.2, 2.05) -- (-0.2, 5.95) node [midway, left, xshift=-2pt] {$k$};
	\draw [decorate,decoration={brace,amplitude=4pt}] (-0.2, 6.05) -- (-0.2, 9.95) node [midway, left, xshift=-2pt] {$k$};
	\draw [decorate,decoration={brace,amplitude=4pt,mirror}] (0, -0.2) -- (1.95, -0.2) node [midway, below, yshift=-2pt] {$g-k$};
	\draw [decorate,decoration={brace,amplitude=4pt,mirror}] (2.05, -0.2) -- (5.95, -0.2) node [midway, below, yshift=-2pt] {$k$};
	\draw [decorate,decoration={brace,amplitude=4pt,mirror}] (6.05, -0.2) -- (9.95, -0.2) node [midway, below, yshift=-2pt] {$k$};
\end{tikzpicture}
}
\caption[]{Left, an example of $\sigma_\bfy \sigma_\bfx^{-1}$; the domain corresponds to the vertical axis. Right, the corresponding example of $\sigma_{\tilde \bfy} \sigma_{\tilde \bfx}^{-1}$. Here, $\bfs_0=\{1,2,3,4\}$, $\bft=\{3,5,6,8\}$, and $c=3$.}
\label{fig:diffgrA}
\end{figure}
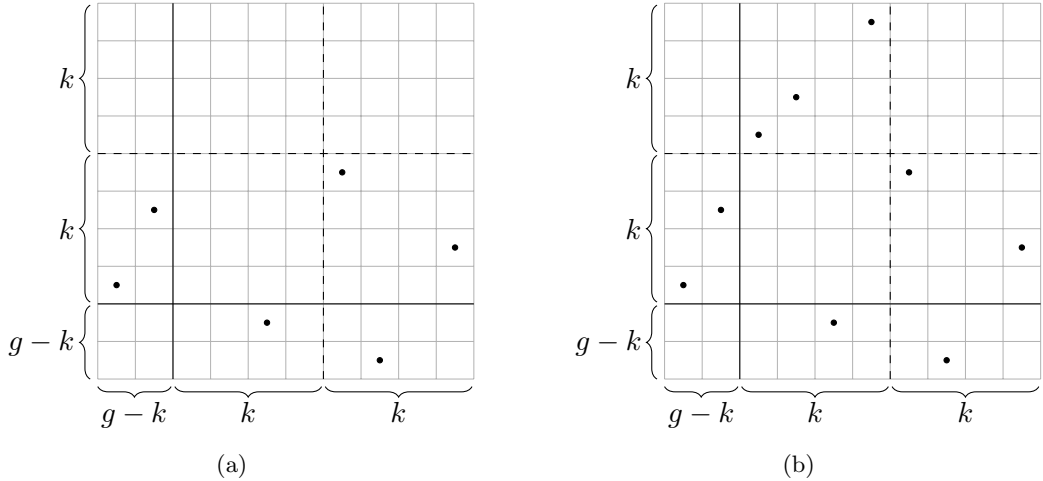

As we obtain $\bfx', \bfy'$ from $\tilde \bfx, \tilde \bfy$ by the same procedure as in Section~\ref{subsubsec:sameidem} where we obtained $\bfx', \bfy'$ from $\bfx, \bfy$, we can now repeat the arguments of Lemma~\ref{lem:grAMsame} (with $\bfs_0$ replaced by $\bfs_0 \cup \bft$) to obtain the remaining terms, completing the proof.  
\end{proof}

As before, we have that
\begin{equation}\label{eqn:diffeuler}
e(B') = e(B),
\end{equation}
since $B'$ differs from $B$ by a collection of rectangles, which have Euler measure zero.

Let $S'_0 = \{ \alpha'^-_i \mid i \in \bfs'_0 \}$ and $T' = \{ \alpha'^-_i \mid i \in \bft' \}$.  We again compare the multiplicities of the generators along the domains $B$ and $B'$:
\begin{lemma}\label{lem:nxydiff}
Let $\bfx, \bfy, \bfx', \bfy', B, B'$ be as above.  Then, \begin{align}
\label{eqn:diffnxy}	n_{\bfx'}(B') + n_{\bfy'}(B') - \big( n_\bfx(B) + n_\bfy(B) \big) &= \frac{1}{2} \sum_i h_i  + \sum_{\substack{ i \in \bfs_0 \cup \bft \\ h_i \neq 0}} (|h_i| -1) + \sum_{i < j} h_i h_j L([\rho_i], [\rho_j]) \\
	 \nonumber	&\quad - \frac{c^2}{2} + m([\dd B], S'_0) + m([\dd B], T') \pmod 2.
\end{align}
\end{lemma}
\begin{proof}
We are first interested in the final three terms of \eqref{eqn:diffnxy}.  These will all come from the intersection points in the refinement neighborhood, the $b^\pm_i$.  More precisely, we will show that 
\[
\sum_{1 \leq i \leq c} (n_{b^+_i}(\tilde{B}) + n_{b^-_i}(\tilde{B})) = - \frac{c^2}{2} + m([\dd B], S'_0) + m([\dd B], T').
\]  
The rest of the proof will then follow as in Lemma~\ref{lem:samenxy}.  	

Observe that the only contributions of $\tilde{B}$ to the multiplicities of $b^\pm_i$ come from the $R_{s_j,t_j}$ and the productwise extension of $\dd B$ across the refinement neighborhood.  First, each $-R_{s_i,t_i}$ clearly contributes $-\frac{1}{4}$ to each of $n_{b^+_i}$ and $n_{b^-_i}$.  The only other contribution of $-R_{s_i,t_i}$ is $-1$ to $n_{b^-_j}$ for $j > i$.  See Figure \ref{fig:Rsiti}.  Hence, 
\begin{equation*}
\sum_{1 \leq i \leq c} \sum_{1 \leq j \leq c} (n_{b^+_i}(-R_{s_j,t_j}) + n_{b^-_i}(-R_{s_j,t_j}))
= \Bigg(\sum_{1 \leq i \leq c} -\frac{1}{2}\Bigg) + \sum_{1 \leq i \leq c} \Bigg(\sum_{i < j \leq c} -1 \Bigg) 
 = -\frac{c^2}{2}.    
\end{equation*}  
It thus remains to compute the contribution of the productwise extension of $\dd B$ to the $n_{b^\pm_i}$.  The intersection points $b^+_i$ (respectively $b^-_i$) are translates of the points in $T'$ (respectively $S'_0$) along $\alpha'_{t_i}$ (respectively $\alpha'_{s_i}$). Hence it is straightforward to check that the contribution to $\sum_i n_{b^+_i}(\tilde{B})$ (respectively $\sum_i n_{b^-_i}(\tilde{B})$) from the productwise extension of $B\cap \d \Sigma$ is $m([\dd B], T')$ (respectively $m([\dd B], S'_0)$).  As discussed, the proof now follows as in Lemma~\ref{lem:samenxy}.
\end{proof}

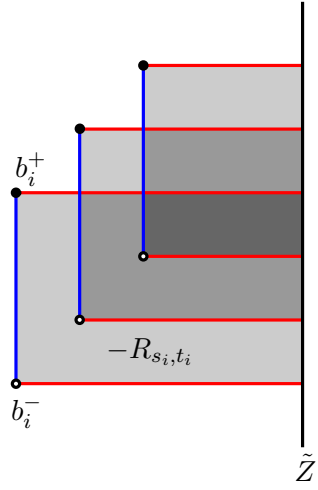
\begin{figure}[htb!]
\centering
\labellist
	\pinlabel $\tilde{Z}$ at 74 -5
	\pinlabel $b^-_i$ at 4 9
	\pinlabel $b^+_i$ at 5 70
	\pinlabel $-R_{s_i,t_i}$ at 35 24
\endlabellist
\includegraphics[scale=1.5]{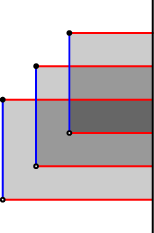}
\vspace{10pt}
\caption[]{The sum $\sum_i -R_{s_i, t_i}$ when $c=3$. 
}
\label{fig:Rsiti}
\end{figure}

For what follows, let $\vec{h} = (h_1,\ldots,h_{2k}) \in H_1(F(\cZ))$.  
\begin{lemma}\label{lem:mhc}
With notation as above, 
\begin{equation}\label{eqn:mhc}
	2m([\dd B], T') = \Big( \sum_{i \in \bft '} h_i  \Big) + c \pmod 2.
\end{equation} 
\end{lemma}
\begin{proof}
Recall that $\d^\d \tilde{B} = \vec{h}=(h_1, \dots, h_{2k}) \in H_1(F(\cZ))$. We observe the following.
\begin{enumerate}
	\item $[\dd B] = \vec{h} + [\bfrho^{\bfs'_0, \bft'}]$.
	\item Recall that $\rho_i$ is the Reeb chord from $\alpha'^-_i$ to $\alpha'^+_i$.  Note that $m(\rho_i, \alpha'^-_j) = \frac{1}{2} \delta_{i,j} \pmod 1$, where $\delta_{i,j}$ is the Kronecker delta function.  Therefore, $m(\vec{h}, T') = \sum_{i \in \bft'} \frac{1}{2}h_i \pmod 1$.  
	\item Let $j \in \bft'$. Then $j>s_i$ for all $1 \leq i \leq c$ and 
\[
m(\rho^{s_i,t_i},\alpha'^-_j) = \begin{cases} 0 & \text{ if } j > t_i \\ 
	\frac{1}{2} & \text{ if } j = t_i\\ 
	1 & \text{ if } j < t_i.\end{cases}  
\]
Here we recall that in the definition of $\alpha'^-_j$, this point is given positive orientation.  Consequently,  
\begin{equation}\label{eqn:mrho}
m([\bfrho^{\bfs'_0,\bft'}],T') = \frac{c^2}{2},
\end{equation}
since $T' = \{\alpha'^-_i \mid i \in \bft'\}$ and $|T'| = c$. We note for future use that we similarly have 
\begin{equation}\label{eqn:mrho2}
m([\bfrho^{\bfs'_0,\bft'}],S_0') = \frac{c^2}{2}.
\end{equation}
From \eqref{eqn:mrho}, it follows that $m([\bfrho^{\bfs'_0, \bft'}],T') = \frac{1}{2}c \pmod 1$.  
\end{enumerate}
Combining these observations, we obtain 
\begin{align*}
	2m([\dd B], T') &= 2m(\vec{h} + [\bfrho^{\bfs'_0, \bft'}], T') \\
			&= 2m(\vec{h}, T') + 2m([\bfrho^{\bfs'_0, \bft'}], T') \\
			&= \Big( \sum_{i \in \bft '} h_i  \Big) + c \pmod 2.\qedhere
\end{align*}
\end{proof}

We also observe that
\begin{align}\label{eqn:mbs} 
\nonumber	L([\bfrho^{\bfs_0,\bft}], [\dd B] ) &= m([\dd B], \partial [\bfrho^{\bfs_0,\bft}]) \\
&= m ([\dd B], T) -m ([\dd B], S_0)  \\
\nonumber	&= m ([\dd B], T') -m ([\dd B], S'_0).
\end{align}

\begin{proof}[Proof of Proposition~\ref{prop:diffidem}]
Combining \eqref{eqn:Mclosed}, \eqref{eqn:diffgrA}, \eqref{eqn:diffeuler}, \eqref{eqn:diffnxy}, \eqref{eqn:mhc}, and \eqref{eqn:mbs}, we see that
\begin{align*}
	\gr_A(\bfy) - \gr_A(\bfx) &= M(\bfy') - M(\bfx') + c + m( [\bfrho^{\bfs'_0, \bft'}], S_0 \cap T) + \sum_{i \notin \bfs_0 \cup \bft } h_i + \sum_{\substack{ i \in \bfs_0 \cup \bft \\ h_i \neq 0}} (|h_i| -1) \\
		&= -e(B') -n_{\bfx'}(B') -n_{\bfy'}(B') + c + m( [\bfrho^{\bfs'_0, \bft'}], S_0 \cap T) + \sum_{i \notin \bfs_0 \cup \bft } h_i + \sum_{\substack{ i \in \bfs_0 \cup \bft \\ h_i \neq 0}} (|h_i| -1) \\
		&= -e(B) -n_\bfx(B) -n_\bfy(B) - \frac{1}{2} \sum_i h_i  + \sum_{\substack{ i \in \bfs_0 \cup \bft \\ h_i \neq 0}} (|h_i| -1) - \sum_{i < j} h_i h_j L([\rho_i], [\rho_j])+ \frac{c^2}{2} \\
	 	&\quad  - m([\dd B], S'_0) - m([\dd B], T') + c + m( [\bfrho^{\bfs'_0, \bft'}], S_0 \cap T) + \sum_{i \notin \bfs_0 \cup \bft } h_i + \sum_{\substack{ i \in \bfs_0 \cup \bft \\ h_i \neq 0}} (|h_i| -1) \\
		&= -e(B) -n_\bfx(B) -n_\bfy(B) - \frac{1}{2} \sum_i h_i  - \sum_{i < j} h_i h_j L([\rho_i], [\rho_j]) + \frac{c^2}{2} \\
	 	&\quad  + L([\bfrho^{\bfs_0,\bft}], [\dd B] ) - 2m([\dd B], T') + c + m( [\bfrho^{\bfs'_0, \bft'}], S_0 \cap T)  + \sum_{i \notin \bfs_0 \cup \bft } h_i \\
		&= -e(B) -n_\bfx(B) -n_\bfy(B) - \frac{1}{2} \sum_i h_i  - \sum_{i < j} h_i h_j L([\rho_i], [\rho_j]) + \frac{c^2}{2} + L([\bfrho^{\bfs_0,\bft}], [\dd B] )  \\
		&\quad + m( [\bfrho^{\bfs'_0, \bft'}], S_0 \cap T) + \sum_{i \notin \bfs_0 } h_i \pmod 2.
\end{align*}

We now verify that the end result of the above set of equalities is $f^\cZ_0 \circ g(B)$. Recall that our grading refinement data is $\psi(\bft) = \grunref(I(\bfs_0) \cdot a(\bfrho^{\bfs_0, \bft}))$. Since $I_A(\bfx) = \bfs_0$, we have that
\begin{align*}
	g(B) &= \grunref(B) \cdot (\grunref(I(\bfs_0) \cdot a(\bfrho^{\bfs_0, \bft})))^{-1} \\
	&= ( -e(B) -n_\bfx(B) -n_\bfy(B), [ \dd B]) \cdot ( -m([\bfrho^{\bfs_0, \bft}], S_0 ) , [\bfrho^{\bfs_0, \bft}])^{-1} \\
	&= ( -e(B) -n_\bfx(B) -n_\bfy(B) + m([\bfrho^{\bfs_0, \bft}], S_0 )+ L( [\dd B], -[\bfrho^{\bfs_0, \bft}]) , [ \dd B] - [\bfrho^{\bfs_0, \bft}] ),
\end{align*}
where the second term in the second equality follows from \eqref{eqn:unrefgralg} and the observation that $\iota(I(\bfs_0) \cdot a(\bfrho^{\bfs_0,\bft})) = -m([\bfrho^{\bfs_0,\bft}],S_0)$.

Since $[ \dd B] - [\bfrho^{\bfs_0, \bft}] = (h_1, \ldots, h_{2k})$, it follows from \cite[Proposition 15]{Petkovadecat} that
\begin{align*}
	f_0^\cZ \circ g(B) &= -e(B) -n_\bfx(B) -n_\bfy(B) + m ([\bfrho^{\bfs_0, \bft}] , S_0 ) + L( [\dd B], -[\bfrho^{\bfs_0, \bft}]) - \frac{1}{2} \sum_{i \in \bfs_0} h_i +\frac{1}{2} \sum_{i \notin \bfs_0} h_i \\
	& \quad + \sum_{i < j} h_i h_j L([\rho_i], [\rho_j]) \pmod 2.
\end{align*}

By \eqref{eqn:mrho2}, we have 
\begin{align*}
m ([\bfrho^{\bfs_0, \bft}] , S_0 ) &= m ([\bfrho^{\bfs'_0, \bft'}] , S_0 \setminus (S_0 \cap T) ) + m( [\bfrho^{\bfs'_0, \bft'}], S_0 \cap T)		\\
& = \frac{c^2}{2} + m( [\bfrho^{\bfs'_0, \bft'}], S_0 \cap T).
\end{align*}

Recall that $L([\rho_i], [\rho_j])$ is integral and that $L$ is skew-symmetric.
These results combine to yield
	\[ \gr_A(\bfy) - \gr_A(\bfx) = f_0^\cZ \circ g(B) \pmod 2, \]
as desired.
\end{proof}

  
\subsubsection{$\CFD$}\label{subsec:CFDagree}
Having established that the two $\Ztwo$-gradings agree on $\CFA$ in the presence of a generator in the base idempotent, we proceed to show that they agree on arbitrary $\CFD$. We will prove this using the fact that both $\Ztwo$-gradings respect the pairing theorem, i.e.,  Proposition~\ref{prop:respectspairing} and \cite[Proposition 22]{Petkovadecat}. 

\begin{figure}[htb!]
\labellist
	\pinlabel $\cZ$ at 62 5
	\pinlabel $B_1$ at 48 45 \pinlabel $B_2$ at 67 45
		\scriptsize
	\pinlabel $\bfx_1$ at 36 30 
	\pinlabel $\bfy_1$ at 36 60 
	\pinlabel $\bfx_2$ at 103 60 
	\pinlabel $\bfy_2$ at 103 48.5 
	\endlabellist
\includegraphics[scale=1.5]{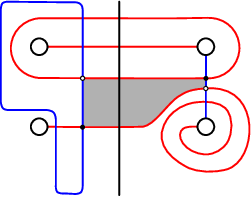}
\caption{An example of the closed Heegaard diagram $\cH$ constructed in the proof of Theorem~\ref{thm:gradingsagree}\eqref{gradingsagreethm:2}. Notice that, as the $\beta$-circle on the left has been isotoped to meet every $\alpha$-curve,  there is a generator in the base idempotent (in this case, $\bfx_1$) so that $B_1\in\pi_2(\bfx_1,\bfy_1)$ (in particular, this is non-empty). }
\label{fig:CFDCFA-example}
\end{figure}
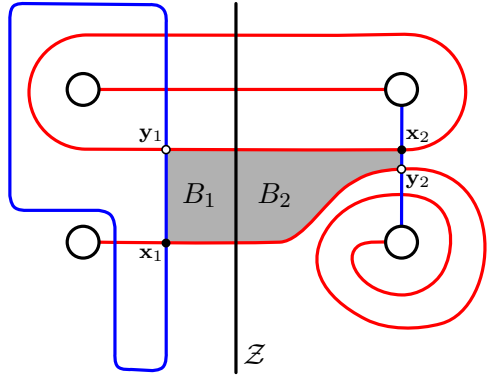

\begin{proof}[Proof of Theorem~\ref{thm:gradingsagree}\eqref{gradingsagreethm:2}]
For notational convenience, we relabel $\cH$ as $\cH_2$ and $\bfx, \bfy$ as $\bfx_2, \bfy_2$ and $\mfs$ as $\mfs_2$.  We begin by choosing a bordered Heegaard diagram $\cH_1$ for a bordered three-manifold $Y_1$ with $\d \cH_1= \cZ =-\d \cH_2$ and generators $\bfx_1, \bfy_1 \in \mfS(\cH_1)$ such that $\bfx_0 = \bfx_1 \otimes \bfx_2$ and $\bfy_0 = \bfy_1 \otimes \bfy_2$ are generators in $\mfS(\cH_1 \cup \cH_2, \mfs_0)$ for some $\mfs_0$. One can construct such a bordered Heegaard diagram as follows (compare Figure \ref{fig:CFDCFA-example}).  Let $Y_2$ denote the bordered 3-manifold corresponding to $\cH_2$.  Let $Y_1$ denote a bordered handlebody of genus $k$ with boundary $F(\cZ)$ such that the map from $H_1(F(\cZ))$ to $H_1(Y_1 \cup Y_2)$ is identically 0.  Now let $\cH_1$ be a bordered Heegaard diagram for $Y_1$.  Further, after isotopy, we can guarantee that every $\beta$-circle intersects every $\alpha$-arc. Now take $\bfx_1$ (respectively $\bfy_1$) to be any generator with $I_A(\bfx_1) = I_D(\bfx_2)$ (respectively $I_A(\bfy_1) = I_D(\bfy_2)$).  By construction, we have that $\bfx_1 \otimes \bfx_2$ and $\bfy_1 \otimes \bfy_2$ are generators for $\widehat{CF}(\cH_1 \cup \cH_2)$. Further, these homological conditions guarantee that these generators correspond to the same $\spinc$ structure, $\mfs_0$, on $Y_1 \cup Y_2$.    

Let $B_0 \in \pi_2(\bfx_0, \bfy_0)$. Then we can decompose $B_0$ as the union of $B_1 \subset \cH_1$ and $B_2 \subset \cH_2$. More generally, given $B_i \in \pi_2(\bfx_i, \bfy_i)$ such that $B_1$ and $B_2$ agree along their boundary, we write $B_1 \natural B_2$ to denote the associated homology class  in $\pi_2(\bfx_1 \otimes \bfx_2, \bfy_1 \otimes \bfy_2)$, which is guaranteed to exist by \cite[Lemma 4.32]{LOT}.

By \cite[Proposition 22]{Petkovadecat} we have that
	\begin{align} \label{eqn:indf}
		M(\bfx_0) - M(\bfy_0) = f_0^{\cZ} \circ g(B_1) + f_0^{\cZ} \circ R \circ g(B_2) \pmod 2,
	\end{align}
where $M$ denotes the Maslov grading modulo $2$. We have that
	\begin{align*}
		 M(\bfx_1 \otimes \bfx_2) - M(\bfy_1 \otimes \bfy_2) &= \gr_A(\bfx_1) - \gr_A(\bfy_1) + \gr_D(\bfx_2) - \gr_D(\bfy_2) \\
			&= f_0^{\cZ} \circ g(B_1) + \gr_D(\bfx_2) - \gr_D(\bfy_2) \pmod 2,
	\end{align*}
where the first equality follows from Proposition \ref{prop:respectspairing} and the second equality follows from Proposition \ref{prop:diffidem} and the discussion preceding it.  (Here, we are using the fact that each $\beta$-circle intersects every $\alpha$-arc, so there is necessarily a generator of  $\CFA(\cH_1)$ with associated idempotent given by the base idempotent.)   Combining this with \eqref{eqn:indf} yields
\[ \gr_D(\bfy_2) - \gr_D(\bfx_2) = f_0^{\cZ} \circ R \circ g(B_2) \pmod 2. \]
Note that this equation in fact holds for any $B_2 \in \pi_2(\bfx_2, \bfy_2)$ since any two elements in $\pi_2(\bfx_2, \bfy_2)$ differ by a periodic domain, which by \cite[Theorem 13]{Petkovadecat} has image zero under the map $f_0^{\cZ} \circ R \circ g$.  This completes the proof.    
\end{proof}


\subsubsection{The general case of Theorem \ref{thm:gradingsagree}\eqref{gradingsagreethm:1}}\label{subsec:generalCFAagree}
We now complete the proof of Theorem \ref{thm:gradingsagree}\eqref{gradingsagreethm:1}, that the two $\Ztwo$-gradings on $\CFA$ agree, without the hypothesis that we have a generator in the base idempotent.

\begin{proof}[Proof of Theorem~\ref{thm:gradingsagree}\eqref{gradingsagreethm:1}]
We repeat the proof of Theorem~\ref{thm:gradingsagree}\eqref{gradingsagreethm:2}, reversing the roles of $\CFA$ and $\CFD$. Namely, we know that the two $\Ztwo$-gradings agree on $\CFD$ as well as on $\widehat{CF}$ for any choices of $\cH_2$ and $\cH_0$ respectively. Note that in the proof of Theorem~\ref{thm:gradingsagree}\eqref{gradingsagreethm:2}, we did not need a generator $\bfx$ of $\CFD(\cH_2)$ with $I_D(\bfx) = I(\bfs_0)$.  Combined with the relevant pairing theorems, we conclude that the two $\Ztwo$-gradings agree on all $\CFA$.
\end{proof}

\subsection{Bimodules} \label{sec:grequivbimodules}
The non-commutative gradings on $\CFAA$ and $\CFDD$ are induced by the restriction and induction functors respectively, as in \cite[Subsection 2.4.2]{LOTbimodules}.  Similarly, the gradings $\gr_{AA}$ and $\gr_{DD}$ defined in Section~\ref{sec:background} are induced from $\gr_A$ and $\gr_D$.  Since $\gr_A$ and $f^\cZ_0 \circ g$ (respectively $\gr_D$ and $f^\cZ_0  \circ R \circ g$) agree on $\CFA$ (respectively $\CFD$) by Theorem~\ref{thm:gradingsagree}, it follows that the induced gradings on $\CFAA$ and $\CFDD$ agree as well.  

We now generalize Propositions \ref{prop:respectspairing} and \ref{prop:CFDApairing} to show that the pairing of $\CFAA$ and $\CFDD$ respects the $\Ztwo$-grading in an appropriate manner.  Let $\cH_{12}$ (respectively $\cH_{23}$) be an arced bordered Heegaard diagram for a manifold with boundary parameterized by $\cZ_1$ and $\cZ_2$ (respectively $\cZ_2$ and $\cZ_3$).  Let $g_1$ (respectively $g_2$) denote the genus of $\cH_{12}$ (respectively $\cH_{23}$) and let $k_i$ denote the genus of $F(\cZ_i)$.   



\begin{proposition} \label{prop:CFAACFDDgrpairing}
Given $\cH_{12}$, $\cH_{23}$ as above, we have 
\begin{equation}\label{eqn:CFAACFDDpairing}
	 \CFDA(\cH_{12} \cup \cH_{23}) \simeq \CFAA(\cH_{12}) \boxtimes \CFDD(\cH_{23})
\end{equation}
and given $\bfx \in \CFAA(\cH_{12})$ and $\bfy \in \CFDD(\cH_{23})$, we have that 
\begin{equation} \label{eqn:DADDAAgr}
	 \gr_{DA}(\bfx \otimes \bfy) = \gr_{AA}(\bfx) + \gr_{DD}(\bfy) + g_1(k_1+k_3)+t(k_1+k_2)  \pmod 2,
\end{equation}
where $t = |\bfs|$ and $I(\bfs)$ is the minimal idempotent in $\cA(\cZ_1)$ such that $\bfx \cdot (I(\bfs), \bfI)= \bfx$. In particular, if $k_1=k_2=k_3$, then
\[ \gr_{DA}(\bfx \otimes \bfy) = \gr_{AA}(\bfx) + \gr_{DD}(\bfy) \pmod 2.\]
\end{proposition}

\begin{proof}
The ungraded version of \eqref{eqn:CFAACFDDpairing} is \cite[Theorem 12]{LOTbimodules}.  The graded statement is a straightforward modification of the proof of Proposition \ref{prop:respectspairing}, where the curves in $\cH_{12} \cup \cH_{23}$ naturally inherit orientations and orders from $\cH_{12}$ and $\cH_{23}$.
\end{proof}

\begin{remark}\label{rmk:pairingorderchoices}
More generally, one may pair elements in $\sfMod_{\cA(\cZ_1),\cA(\cZ_2)}$ and ${}^{\cA(\cZ_3),\cA(\cZ_2)}\sfMod$, or $\sfMod_{\cA(\cZ_2),\cA(\cZ_1)}$ and either ${}^{\cA(\cZ_2),\cA(\cZ_3)}\sfMod$ or ${}^{\cA(\cZ_3),\cA(\cZ_2)}\sfMod$. Analogous results to \eqref{eqn:DADDAAgr} holds, i.e., $\gr_{DA}(\bfx \otimes \bfy) = \gr_{AA}(\bfx) + \gr_{DD}(\bfy)$, up to a shift which depends on $k_1, k_2, k_3, g_1, g_2,$ and $t$. Similarly, one may pair other types of modules, and analogous results hold.
\end{remark}

We turn our attention to $\CFDA$.  First, we will extend Petkova's mod 2 reduction of the non-commutative grading to $\CFDA$.  Then, we will show that this agrees with the grading $\gr_{DA}$ defined in Section~\ref{sec:background}.    

We begin with some general discussion about non-commutative group gradings.  Let $(G_1, \lambda_1)$ and $(G_2, \lambda_2)$ each be groups with distinguished central elements. As in \cite[Definition 2.5.9]{LOTbimodules}, we define the group
	\[ G_1 \times_\lambda G_2 = G_1 \times G_2 / (\lambda_1=\lambda_2) \]
with the distinguished central element $\lambda=[\lambda_1]=[\lambda_2]$.

We use the notation of \cite[Section 6.5]{LOTbimodules}. Let $\cH$ be an arced bordered Heegaard diagram.  Define
	\[ G'_{AA}(\d \cH) = G'(\cZ_L) \times_\lambda G'(\cZ_R) \quad \textup{ and } \quad G_{AA}(\d \cH) = G(\cZ_L) \times_\lambda G(\cZ_R), \]
where $\cZ_L = \d_L \cH$ and $\cZ_R = \d_R \cH$ and we have fixed refinement data $\psi_{L,A}$ and $\psi_{R,A}$ for $\cA(\cZ_L)$ and $\cA(\cZ_R)$, respectively.

We will often write an element of $G'_{AA}(\d \cH)$ as $(n, \vec{h}^L, \vec{h}^R)$ where $n \in \frac{1}{2}\Z$, $\vec{h}^L \in H_1(Z'_L, \bfa_L)$, and $\vec{h}^R \in H_1(Z'_R, \bfa_R)$, and similarly for $G_{AA}(\d \cH)$.

Define a map
	\[ g': \pi_2(\bfx, \bfy) \rightarrow G'_{AA}(\d \cH) \]
by
	\[ g'(B) = (-e(B) -n_\bfx(B) -n_\bfy(B), \d^{\d_L}B, \d^{\d_R}B). \]
Given $B \in \pi_2(\bfx, \bfy)$, we define the refined grading $g(B) \in G_{AA}(\d \cH)$ to be
	\[ g(B) = \psi_{AA}(\bfx) \cdot g'(B) \cdot \psi_{AA}(\bfy)^{-1} \]
where
	\[ \psi_{AA}(\bfx) = (\psi_{L, A}(I_{L, A}(\bfx)), \psi_{R, A}(I_{R, A}(\bfx))) \in G'(\cZ_L) \times_\lambda G'(\cZ_R). \]
Recall from \cite[Definition 3.19]{LOTbimodules} that given refinement data $\psi$ for $\cA(\cZ)$, the \emph{reverse} of $\psi$, denoted $\overline{\psi}$ is given by
\begin{equation}\label{eqn:reverserefinement}
	\overline{\psi}(\bfsbar) = R(\psi(\bfs))^{-1} \quad \textup{ where } \ \bfsbar=[2k]\setminus \bfs,
\end{equation}
and is refinement data for $\cA(-\cZ)$ \cite[Lemma 3.20]{LOTbimodules}. Having fixed refinement data for $\cA(\cZ)$, we will always use the reverse $\overline{\psi}$ as our refinement data for $\cA(-\cZ)$. We will write $\psi_{L,D}$ for the reverse of $\psi_{L, A}$; in particular, $\psi_{L, D}$ is grading refinement data for $\cA(-\cZ_L)$.

Define 
	\[ G_{DA}(\d \cH) = G(-\cZ_L)^\op \times_\lambda G(\cZ_R). \]
Recall from \cite[Subsection 6.5]{LOTbimodules} that $\CFDA$ is graded by a right $G_{DA}(\d \cH)$-set; we review the construction.  Recall that $R$ defines an isomorphism between $G(\cZ)$ and $G(-\cZ)^\op$. Define $\Rtilde$ to be the canonical isomorphism
	\[ R \times_\lambda \bbI : G_{AA}(\d \cH) \rightarrow G_{DA}(\d \cH). \]
Then $\CFDA(\cH, \mfs)$ is graded by the right $G_{DA}(\d \cH)$-set
\[ S_{DA}(\cH, \bfx_0) = \Rtilde(P_{\bfx_0}) \backslash G_{DA}(\d \cH) \]
where $\bfx_0 \in \mfS(\cH, \mfs)$ is a fixed reference point and $P_{\bfx_0} = g(\pi_2(\bfx_0, \bfx_0)) \subset G_{AA}(\d \cH)$. Namely,
	\[ g(\bfx) = \Rtilde(P_{\bfx_0}) \cdot g(B) \quad \textup{ for } B \in \pi_2(\bfx_0, \bfx). \]

\subsubsection{Extending Petkova's homomorphism to $\CFDA$}
We begin by extending Petkova's homomorphism defined in \cite[Section 3]{Petkovadecat} to a homomorphism from $G_{AA}(\partial \cH)$ to $\Ztwo$. In fact, we will define such a homomorphism for each $\t$, $-(k_L+k_R) \leq \t \leq k_L+k_R$. Recall that we have the splitting
	\[ \CFDA(\cH) = \bigoplus^{k_L + k_R}_{\t = -(k_L + k_R)} \CFDA(\cH, \t). \]
We point this out since we will now be using all of the strands gradings, as opposed to the case of a single boundary component.


For each $\t$, we have that $G_{AA}(\d \cH)$ is generated by $\{ \lambda, \mu_{\t,1}, \dots, \mu_{\t,2k_L}, \nu_{\t,1}, \dots , \nu_{\t,2k_R} \}$
subject to the relations 
	\[ \mu_{\t,i} \mu_{\t,{j}} = \mu_{\t,{j}} \mu_{\t,{i}} \lambda^{2 (\rho^L_i \cap \rho^L_{j})}, \quad \nu_{\t,i} \nu_{\t,{j}} = \nu_{\t,{j}} \nu_{\t,{i}} \lambda^{2 (\rho^R_i \cap \rho^R_{j})}, \quad \mu_{\t,i}\nu_{\t,{j}} = \nu_{\t,{j}} \mu_{\t,i}, \quad \textup{and} \quad \lambda \textup{ central}, \]
where 
\begin{enumerate}
	\item $\rho^L_i \in H_1(F(\cZ_L))$ (respectively $\rho^R_i \in H_1(F(\cZ_R))$) denotes the cycle corresponding to $\alpha^a_{L,i}$ (respectively $\alpha^a_{R,i}$) and $\rho^L_i \cap \rho^L_{j}$ (respectively $\rho^R_i \cap \rho^R_{j}$) indicates the signed intersection number
	\item $\mu_{\t,i} = g(\bfI^L_{k_L+k_R-\t} \cdot a(\rho^L_i))$ and $\nu_{\t,i} = g(\bfI^R_{\t} \cdot a(\rho^R_i))$ with $\bfI^L_{k_L+k_R-\t}$ (respectively $\bfI^R_{\t}$) the unit in $\cA(\cZ_L, k_L+k_R-\t)$ (respectively $\cA(\cZ_R, \t)$).
\end{enumerate}

Define 
	\[ f_\t^{\cZ_L, \cZ_R} : G_{AA}(\d \cH) \rightarrow \Ztwo \]
by sending 
	\begin{align*}
		\lambda &\mapsto 1 \\
		\mu_{\t,i} &\mapsto 1 \\
		\nu_{\t,i} &\mapsto 1.
	\end{align*}
It follows from the group relations of $G_{AA}(\d \cH)$ that the above map is well-defined.  

We use the $\rho^L_i$ (respectively $\rho^R_i$) for an ordered basis of $H_1(F(\cZ_L))$ (respectively $H_1(F(\cZ_R))$).  With this, we can represent an element of $G_{AA}(\d \cH)$ by $(n,h^L_1,\ldots,h^L_{2k_L},h^R_1,\ldots,h^R_{2k_R})$, for $n \in \frac{1}{2}\Z$, $\vec{h}^L=(h^L_1,\ldots,h^L_{2k_L}) \in H_1(F(\cZ_L))$, and $\vec{h}^R=(h^R_1,\ldots,h^L_{2k_R}) \in H_1(F(\cZ_R))$. We let $\t' = k_L+k_R -\t$.

\begin{proposition}\label{prop:fZLZRgamma}
The homomorphism $f_{\t}^{\cZ_L, \cZ_R}: G_{AA}(\d \cH) \rightarrow \Ztwo$ satisfies
\begin{align*}
	 f_\t^{\cZ_L, \cZ_R}(n, h^L_1, \dots, h^L_{2k_L}, h^R_1, \dots, h^R_{2k_R}) &= n - \frac{1}{2} \sum_{i \in \bfs_{0, \t'}^L} h^L_i  + \frac{1}{2} \sum_{i \not\in \bfs_{0, \t'}^L} h^L_i + \sum_{i<j}h^L_i h^L_j L(\rho^L_i, \rho^L_j) \\
	& \quad - \frac{1}{2} \sum_{i \in \bfs_{0, \t}^R} h^R_i  + \frac{1}{2} \sum_{i \not\in \bfs_{0, \t}^R} h^R_i + \sum_{i<j}h^R_i h^R_j L(\rho^R_i, \rho^R_j),
\end{align*}
where $\bfs_{0, \t'}^L$ (respectively $\bfs_{0, \t}^R$) denotes the base idempotent in $\cA(\cZ_L, \t')$ (respectively $\cA(\cZ_R, \t)$).
\end{proposition}

\begin{proof}
The group $G_{AA}(\d \cH)$ agrees with the grading group obtained by the restriction functor applied to $\CFA(\d \cH_{\dr})$ \cite[Remark 6.34]{LOTbimodules}. Moreover, under this identification, the homomorphism $f^{\cZ_L, \cZ_R}_\t$ agrees with $f^{\cZ_L \# \cZ_R}_0$. The result now follows from \cite[Proposition 15]{Petkovadecat}.
\end{proof}

Since $G(\d \cH) = G(\cZ_L) \times_\lambda G(\cZ_R)$, we may consider $f_{\t}^{\cZ_L, \cZ_R}$ as a homomorphism $f_{\t}^{\cZ_L, \cZ_R}\co G(\cZ_L) \times_\lambda G(\cZ_R) \rightarrow \Ztwo$. Clearly, we have an induced set map 
\begin{equation}\label{eqn:fop}
f_{\t}^{\cZ_L, \cZ_R}\co G(\cZ_L)^\op \times_\lambda G(\cZ_R) \rightarrow \Ztwo,
\end{equation}
which, by abuse of notation, we also denote by $f_{\t}^{\cZ_L, \cZ_R}$. Moreover, since $\Ztwo$ is abelian, the aforementioned set map is in fact a homomorphism. In particular, if $\d_R \cH = \cZ_R$ and $\d_L \cH = -\cZ_L$, then $G_{DA}(\d \cH) = G(\cZ_L)^\op \times_\lambda G(\cZ_R)$, and so \eqref{eqn:fop} gives a homorphism
\[ 
f_{\t}^{\cZ_L, \cZ_R}\co G_{DA}(\d \cH) \rightarrow \Ztwo.
\]

\begin{lemma}
\label{lem:fsum}
Let $\cH_1$ and $\cH_2$ be arced bordered Heegaard diagrams with two boundary components such that $\d_R \cH_1 = \cZ_M = -\d_L \cH_2$, where $M$ stands for middle. Let $\cZ_L = -\d_L \cH_1$, $\cZ_R = \d_R \cH_2$, and $\cH = \cH_1 \cup \cH_2$. Let $B_i \in \pi_2(\bfx_i, \bfy_i)$ for $\bfx_i, \bfy_i \in \mfS(\cH_i, \mfs_i)$ such that $\d^\d_R B_1 = -\d^\d_L B_2$ and both $\bfx=\bfx_1 \otimes \bfx_2 \neq 0$ and $\bfy = \bfy_1 \otimes \bfy_2 \neq 0$ in $\mfS(\cH)$. Suppose that $I_{R, A}(\bfx_2) \in \cA(\cZ_R, \t)$ and let $B = B_1 \natural B_2 \in \pi_2(\bfx, \bfy)$. Then
\begin{align*}
	f_\t^{\cZ_L, \cZ_R} &\co G_{DA}(\d \cH) \rightarrow \Ztwo \\
	f_\t^{\cZ_L, \cZ_M} &\co G_{DA}(\d \cH_1) \rightarrow \Ztwo \\
	f_\t^{\cZ_M, \cZ_R} &\co G_{DA}(\d \cH_2) \rightarrow \Ztwo
\end{align*}
satisfy
		\[ f_\t^{\cZ_L, \cZ_R} \circ \Rtilde \circ g(B)= f_\t^{\cZ_L, \cZ_M} \circ \Rtilde \circ g(B_1) + f_\t^{\cZ_M, \cZ_R} \circ \Rtilde \circ g(B_2). \]
Analogous statements hold for $G_{AA}$ and $G_{DD}$, appropriate combinations thereof, and when one or both of $\cH_i$ has a single boundary component.
\end{lemma}

\begin{proof}
Note that 
\begin{align*}
	g'(B_1) &\in G'_{AA}(\d \cH_1)=G'(-\cZ_L) \times_\lambda G'(\cZ_M) \\
	g'(B_2) &\in G'_{AA}(\d \cH_2)=G'(-\cZ_M) \times_\lambda G'(\cZ_R). \\
\end{align*}

Let $\psi_{L, A}$ be refinement data for $\cZ_L$ (with reverse $\psi_{L,D}$), $\psi_{M, A}$ refinement data for $\cZ_M$ (with reverse $\psi_{M,D}$), and $\psi_{R, A}$ refinement data for $\cZ_R$. Since $I_{R, A}(\bfx_1) = I_{L, D}(\bfx_2)$, we have that 
\begin{align*}
	\psi_{M,D}(I_{L,A}(\bfx_2) &= R(\psi_{M,A}(I_{L,D}(\bfx_2)))^{-1}\\
		&= R(\psi_{M,A}(I_{R,A}(\bfx_1)))^{-1}.
\end{align*}
Similarly, $I_{R, A}(\bfy_1) = I_{L, D}(\bfy_2)$ and $\psi_{M,D}(I_{L,A}(\bfy_2))=R(\psi_{M,A}(I_{R,A}(\bfy_1)))^{-1}$. Then
\begin{align*}
	g(B_1) &= (\psi_{L,D}(I_{L,A}(\bfx_1)), \psi_{M,A}(I_{R,A}(\bfx_1)))  \cdot g'(B_1) \cdot (\psi_{L,D}(I_{L,A}(\bfy_1)), \psi_{M,A}(I_{R,A}(\bfy_1)))^{-1} \\
	g(B_2) &= (\psi_{M,D}(I_{L,A}(\bfx_2)), \psi_{R,A}(I_{R,A}(\bfx_2)))  \cdot g'(B_2) \cdot (\psi_{M,D}(I_{L,A}(\bfy_2)), \psi_{R,A}(I_{R,A}(\bfy_2)))^{-1} \\
		&= (R(\psi_{M,A}(I_{R,A}(\bfx_1)))^{-1}, \psi_{R,A}(I_{R,A}(\bfx_2)))  \cdot g'(B_2) \cdot (R(\psi_{M,A}(I_{R,A}(\bfy_1)))^{-1}, \psi_{R,A}(I_{R,A}(\bfy_2)))^{-1},
\end{align*}
where $g'(B_i) =  (-e(B_i) - n_{\bfx_i}(B_i) - n_{\bfy_i}(B_i), \d^\d_L B_i, \d^\d_R B_i)$. Note that $\d_L \cH_1 = -\cZ_L$, so we use the refinement data $\psi_{L,D}$, and similarly $\d_L \cH_2 = -\cZ_M$, so we use $\psi_{M,D}$. We also have that
	\[ g(B) = (\psi_{L,D}(I_{L,A}(\bfx_1)), \psi_{R,A}(I_{R,A}(\bfx_2)))  \cdot g'(B) \cdot (\psi_{L,D}(I_{L,A}(\bfy_1)), \psi_{R,A}(I_{R,A}(\bfy_2)))^{-1} \]
 since
 \begin{align*}
	I_{L,A}(\bfx)&=I_{L,A}(\bfx_1) & I_{L,A}(\bfy)&=I_{L,A}(\bfy_1) \\
	I_{R,A}(\bfx)&=I_{R,A}(\bfx_2) & I_{R,A}(\bfy)&=I_{R,A}(\bfy_2).
 \end{align*}
Then since $\d^\d_R B_1 = - \d^\d_L B_2$, we can write 
\begin{align*}
	g(B_1) &= (n_1, \vec{h}^L, \vec{h}^M) =  (n_1, h^L_1, \dots, h^L_{2k_L}, h^M_1, \dots, h^M_{2k_M}) \\
	g(B_2) &= (n_2, -\vec{h}^M, \vec{h}^R) = (n_2, -h^M_1, \dots, -h^M_{2k_M}, h^R_1, \dots, h^R_{2k_R}) \\
	g(B) &= (n, \vec{h}^L, \vec{h}^R) = (n, h^L_1, \dots, h^L_{2k_L}, h^R_1, \dots, h^R_{2k_R}),
\end{align*}
for some $n=n_1+n_2$, $\vec{h}^L$, $\vec{h}^M$, and $\vec{h}^R$. Here, we have chosen our bases for $H_1(F(\pm \cZ_M))$ such that $R(n, \vec{h}^M)=(n, \vec{h}^M)$.

It follows that
\begin{align*}
	f_\t^{\cZ_L, \cZ_R} \circ \Rtilde \circ g(B) &= n - \frac{1}{2} \sum_{i \in \bfs_{0, \t'}^L} h^L_i  + \frac{1}{2} \sum_{i \not\in \bfs_{0, \t'}^L} h^L_i + \sum_{i<j}h^L_i h^L_j L(\rho^L_i, \rho^L_j) \\
			& \quad - \frac{1}{2} \sum_{i \in \bfs_{0, \t}^R} h^R_i  + \frac{1}{2} \sum_{i \not\in \bfs_{0, \t}^R} h^R_i + \sum_{i<j}h^R_i h^R_j L(\rho^R_i, \rho^R_j) \\
		&= f_\t^{\cZ_L, \cZ_M} \circ \Rtilde (n_1, \vec{h}^L, \vec{h}^M) + f_\t^{\cZ_M, \cZ_R} \circ \Rtilde (n_2, -\vec{h}^M, \vec{h}^R) \\
		&= f_\t^{\cZ_L, \cZ_M} \circ \Rtilde \circ g(B_1) + f_\t^{\cZ_M, \cZ_R} \circ \Rtilde \circ g(B_2),
\end{align*}
where $\bfs_{0, \t'}^L$ (respectively $\bfs_{0, \t}^R$) is the base idempotent of $\cA(\cZ^L, \t')$ (respectively $\cA(\cZ^R, \t)$).

The proofs in the other cases are similar.
\end{proof}

\begin{lemma}
\label{lem:CFAAperiodicdomain}
Let $\cH$ be an arced bordered Heegaard diagram with $\d_L \cH = \cZ_L$ and $\d_R \cH = \cZ_R$. Suppose that $B \in \pi_2(\bfx, \bfx)$ and $I_{R, A}(\bfx) \in \cA(\cZ_R, \t)$. Then
	\[ f_\t^{\cZ_L, \cZ_R} \circ g(B) = 0. \]
\end{lemma}

\begin{proof}
The result follows from \cite[Theorem 13]{Petkovadecat} and the compatibility between $f^{\cZ_L,\cZ_R}_\gamma$ and $f^{\cZ_L \# \cZ_R}_0$ under the restriction functor applied to $\CFA$, as in Proposition~\ref{prop:fZLZRgamma}.
\end{proof}

\begin{lemma}\label{lem:fRtilde}
Let $\bfx_2 \in \mfS(\cH_2)$, where $\cH_2$ is an arced bordered Heegaard diagram, $B_2 \in \pi_2(\bfx_2, \bfx_2)$, $I_{R, A}(\bfx_2) \in \cA(\cZ_R, \t)$, $\d_L \cH_2 = -\cZ_M$, and $\d_R \cH_2 = \cZ_R$. Then 
\[ f_\t^{\cZ_M, \cZ_R} \circ \Rtilde \circ g(B_2) = 0. \]
Similarly, $f_\t^{\cZ_M, -\cZ_R} \circ RR \circ g(B_2) =0$ where $RR$ denotes the map $G(-\cZ_M) \times_\lambda G(\cZ_R) \rightarrow G(\cZ_M)^\op \times_\lambda G(-\cZ_R)^\op$ obtained by applying $R$ to each factor.
\end{lemma}

\begin{proof}
We prove the lemma using the pairing theorem, together with Lemma \ref{lem:CFAAperiodicdomain}. That is, we pair some $\CFAA(\cH_1)$ with $\CFDA(\cH_2)$ to obtain $\CFAA(\cH_1 \cup \cH_2)$ and a corresponding periodic domain in $\cH$, and then appeal to Lemma \ref{lem:CFAAperiodicdomain}.

Let $\cH_1$ be an arced bordered Heegaard diagram with $\d_R \cH_1 = \cZ_M$ such that there exists $\bfx_1 \in \mfS(\cH_1)$ and $B_1 \in \pi_2(\bfx_1 , \bfx_1)$ such that $\bfx = \bfx_1 \otimes \bfx_2 \neq 0$ and $\d^\d_R B_1 = -\d^\d_L B_2$, i.e., $B = B_1 \natural B_2 \in \pi_2(\bfx, \bfx)$. For example, one could choose $\cH_1$ to be $-\cH_2$, where we perform an isotopy to guarantee the existence of a generator $\bfx_1$ in $\mfS(\cH_1)$ with $I_{R, A}(\bfx_1) = I_{L, D}(\bfx_2)$, and $B_1$ the image of $B_2$ under the obvious map from $\cH_2$ to $\cH_1 = -\cH_2$ (followed by the map induced by isotopy, if necessary). Note that since $B_2$ is a periodic domain, $B_1$ is as well, and so it follows that $B_1 \in \pi_2(\bfx_1 , \bfx_1)$.

Let $\cZ_L = \d_L \cH_1$. By the version of Lemma \ref{lem:fsum} for pairing $\CFAA$ with $\CFDA$, we have
	\[ f_\t^{\cZ_L, \cZ_R} \circ g(B) = f_\t^{\cZ_L, \cZ_M} \circ g(B_1) + f_\t^{\cZ_M, \cZ_R} \circ \Rtilde \circ g(B_2). \]
Moreover, by Lemma \ref{lem:CFAAperiodicdomain}
\begin{align*}
	f_\t^{\cZ_L, \cZ_R} \circ g(B)  &= 0 \\
	f_\t^{\cZ_L, \cZ_M} \circ g(B_1)  &= 0,
\end{align*}
implying that
	 \[ f_\t^{\cZ_M, \cZ_R} \circ \Rtilde \circ g(B_2) = 0, \]
as desired.

The other case is similar.
\end{proof}

We are now ready to define the mod 2 reduction of the non-commutative gradings for type DA structures.  
\begin{definition}
Let $\cH$ be an arced bordered Heegaard diagram with $\d_L \cH = -\cZ_L$ and $\d_R \cH= \cZ_R$.  Fix $\bfx \in \mfS(\cH, \mfs)$ and suppose $I_{R, A}(\bfx) \in \cA(\cZ_R, \t)$. Then we define a relative $\Ztwo$-grading on $\CFDA(\cH, \mfs)$ by
\begin{equation}\label{eqn:DAgeometricalgebraic}
m_{DA}(\bfy) - m_{DA} (\bfx) = f_\t^{\cZ_L, \cZ_R} \circ \Rtilde \circ g(B)
\end{equation}
for $\bfy \in \mfS(\cH, \mfs)$ and $B \in \pi_2(\bfx, \bfy)$.
\end{definition}

\begin{remark}
Note that \eqref{eqn:DAgeometricalgebraic} is well-defined, by Lemma~\ref{lem:fRtilde}.
\end{remark}


\subsubsection{Equivalence of the two $\Ztwo$-gradings on $\CFDA$}

\begin{theorem}\label{thm:DAgradingsagree}
Let $\cH_2$ be an arced bordered Heegaard diagram with $\d_L \cH_2 = -\cZ_M$, $\d_R \cH_2 = \cZ_R$, $\bfx_2, \bfy_2 \in \mfS(\cH_2, \mfs_2)$, and $I_{R, A}(\bfx_2) \in \cA(\cZ_R, \t)$. Then
\[ \gr_{DA}(\bfy_2) - \gr_{DA}(\bfx_2) = f_\t^{\cZ_M, \cZ_R} \circ \Rtilde \circ g (B_2), \]
for $B_2 \in \pi_2(\bfx_2, \bfy_2)$.
\end{theorem}

\begin{proof}
The proof is similar to the proof of Theorem \ref{thm:gradingsagree}\eqref{gradingsagreethm:2}. Since we know that the two $\Ztwo$-gradings agree on $\CFAA$, we will pair our given diagram $\cH_2$ (thought of as a type DA structure) with some $\cH_1$ (thought of as a type AA structure) to obtain $\cH=\cH_1 \cup \cH_2$ (thought of as a type AA structure). Since we know the two $\Ztwo$-gradings agree on $\cH$ and $\cH_1$ (thought of as type AA structures), it will follow that they agree on $\cH_2$ as well.  Finally, we note that it suffices to prove the result for a single choice of $B_2 \in \pi_2(\bfx_2,\bfy_2)$; indeed, since any two elements in $\pi_2(\bfx_2, \bfy_2)$ differ by a periodic domain, it follows from Lemma \ref{lem:fRtilde} that $f_\t^{\cZ_M, \cZ_R} \circ \Rtilde \circ g (B_2)$ is well-defined. 

We choose an arced bordered Heegaard diagram $\cH_1$ with $\d_R \cH_1 = -\d_L \cH_2$ and generators $\bfx_1, \bfy_1 \in \mfS(\cH_1)$ such that $\bfx = \bfx_1 \otimes \bfx_2$ and $\bfy = \bfy_1 \otimes \bfy_2$ are generators in $\mfS(\cH_1 \cup \cH_2, \mfs)$ for some $\mfs$. One such choice of $\cH_1$ is the arced bordered Heegaard diagram in Lemma~\ref{lem:CFAAidgr} for the mapping cylinder of the identity on $F(\cZ_M)$. We have that $\cZ_M = \d_R \cH_1 = -\d_L \cH_2$, and let $\d_L \cH_1 = \cZ_L$.

Let $B \in \pi_2(\bfx, \bfy)$. Then we can decompose $B$ as the union of $B_1 \in \pi_2(\bfx_1,\bfy_1)$ and $B_2 \in \pi_2(\bfx_2,\bfy_2)$. By Lemma \ref{lem:fsum},
\[ f_\t^{\cZ_L, \cZ_R} \circ g(B) = f_\t^{\cZ_L, \cZ_M} \circ g (B_1) + f_\t^{\cZ_M, \cZ_R} \circ \Rtilde \circ g(B_2). \]
By Theorem \ref{thm:gradingsagree}\eqref{gradingsagreethm:1} together with the restriction functor,
\begin{align*}
	\gr_{AA}(\bfy) - \gr_{AA}(\bfx) &= f_\t^{\cZ_L, \cZ_R} \circ g(B)\\
	\gr_{AA}(\bfy_1) - \gr_{AA}(\bfx_1) &= f_\t^{\cZ_L, \cZ_M} \circ g (B_1).
\end{align*}
Further, by Proposition \ref{prop:CFAACFDDgrpairing} together with Remark \ref{rmk:pairingorderchoices}
\[ \gr_{AA}(\bfy) - \gr_{AA}(\bfx) = \gr_{AA}(\bfy_1) - \gr_{AA}(\bfx_1) + \gr_{DA}(\bfy_2) - \gr_{DA}(\bfx_2). \]
Combining these equations yields the desired result.
\end{proof}

We are now ready to prove Theorem \ref{thm:allthegradings}.

\begin{proof}[Proof of Theorem \ref{thm:allthegradings}]
Theorem \ref{thm:allthegradings} follows from Theorems \ref{thm:alggradingsagree}, \ref{thm:gradingsagree}, and \ref{thm:DAgradingsagree}, which equate our $\Ztwo$-gradings $\gr$, $\gr_D$, $\gr_A$, and $\gr_{DA}$ with (a generalization of) Petkova's $\Ztwo$ reduction of the $G$-set gradings of \cite[Chapter 10]{LOT} and \cite[Section 6.5]{LOTbimodules}. Since the $G$-set gradings are differential gradings (alternatively, type A or DA gradings, in the cases of $\CFA$ and $\CFDA$ respectively) and invariant up to the appropriate notion of equivalence (\cite[Theorem 10.39]{LOT}, \cite[Theorem 10]{LOTbimodules}), it follows that $\gr_D, \gr_A$, and $\gr_{DA}$ are as well. 

Note that although Theorems \ref{thm:gradingsagree} and \ref{thm:DAgradingsagree} rely on a particular choice of order and orientation on the $\alpha$-arcs (namely, that which is induced by the orientation of $Z$), in light of Remark \ref{rem:idemconj}, it follows that Theorem \ref{thm:allthegradings} holds for any choice of order and orientation on the $\alpha$-arcs.

The desired results hold for $\CFDD$ and $\CFAA$ by induction and restriction, respectively.
\end{proof}

\subsubsection{Hochschild homology and the $\Ztwo$-gradings on $\CFDA$}\label{app-sub-Hoch}
Suppose that $\cH$ is an arced bordered Heegaard diagram of genus $g$ for a three-manifold $W$ with $\d W = -F \amalg F$.  As usual, we suppose the genus of $F$ is $k$.  By identifying the two components of $\d W$ via the parameterizations, and doing 0-surgery on the closure of $\bfz$, we obtain a closed manifold, knot pair  $(W^\circ,K)$.   Further suppose that $-\d_L \cH = \cZ = \d_R \cH$.  Recall from Section~\ref{sec:background} that we can close up $\cH$ into a doubly-pointed Heegaard diagram $\cH^\circ$ for a knot $K$ in the three-manifold $W^\circ$.  Given a generator $\bfx \in \CFDA(\cH, \t)$ with $\iota \cdot \bfx \cdot \iota =\bfx$ for some minimal idempotent $\iota$, we let $\bfx^\circ$ denote the induced generator of $CH_*(\CFDA(\cH,\t))$ as defined in Section~\ref{sub:HH}.  We now can understand the behavior of $\gr_{DA}$ under Hochschild homology. 

Recall from Definition \ref{defn:sgnDA} that for $(g, k, k, D, A, \sigma)$ a type DA partial permutation with $|\Im(\sigma) \cap A| = \t + k$,
	\[  \sgn_{DA}(\sigma) = \inv(\sigma) + \sum_{i \in \Im(\sigma)} \# \{ j \ | \ j>i, j \notin \Im(\sigma), j \in D \} + (\t+k)(g-2k) \pmod 2. \]

\begin{proposition}
\label{prop:Hochschildgr}
The $\Ztwo$-grading $ \gr_{DA}$ on $\CFDA$ behaves well with respect to Hochschild homology; that is,
\[  \gr_{DA}(\bfy) -  \gr_{DA}(\bfx) +\t_\bfy -\t_\bfx = M(\bfy^\circ) - M(\bfx^\circ) \pmod 2, \]
where $I_{R,A}(\bfx) \in \cA(\cZ, \t_\bfx)$ and $I_{R,A}(\bfy) \in \cA(\cZ, \t_\bfy)$.
\end{proposition}

\begin{proof}
For each $\bfx \in \mfS(\cH)$ which induces a generator $\bfx^\circ$, we consider the associated (partial) permutations $\sigma_\bfx$ and $\sigma_{\bfx^\circ}$. The $\beta$-circles in $\cH^\circ$ are exactly the $\beta$-circles in $\cH$, and we order them the same way in $\cH^\circ$ and $\cH$. We order the $\alpha$-circles in $\cH^\circ$ as the closed up $\alpha$-arcs followed by the $\alpha$-circles, where the order on the set of $\alpha$-arcs and the set of $\alpha$-circles agrees with their order in $\cH$.

By Definition~\ref{defn:grDA}, it suffices to show that
\begin{equation}\label{eqn:sgnDAHH} 
\sgn_{DA}(\sigma_\bfx) + k + \t_\bfx = \inv (\sigma_{\bfx^\circ}) \pmod 2, 
\end{equation}
since $\sum_{x \in \bfx} o(x)  = \sum_{x \in \bfx^\circ} o(x)$ and $M(\bfx^\circ) = \inv (\sigma_{\bfx^\circ})+ \sum_{x \in \bfx^\circ} o(x) \pmod 2$.

We begin by considering how to obtain $\sigma_{\bfx^\circ}$ from $\sigma_\bfx$. For each $\ell \in [g]$ with $\sigma_\bfx(\ell)=j \in A$, we define $\sigma_{\bfx^\circ}(\ell)=j-g$ and set $\sigma_{\bfx^\circ}(\ell) = \sigma_\bfx(\ell)$ otherwise. We consider the change in the number of inversions. We measure these changes in ascending order of $j \in \Im(\sigma_\bfx) \cap A$. For each $j$ the number of inversions changes by
	\[ q=g-2k + \# \{ i \in D \mid i>j-g, i \in \Im(\sigma_\bfx) \} \pmod 2, \]
since there are exactly $q$ elements in the image of $\sigma_\bfx$ between $j-g$ and $j$. These $q$ elements consist of the $g-2k$ elements in $[g+2k] \setminus (D \cup A)$ and the elements in $D \cup \Im(\sigma_\bfx)$ which are greater than $j-g$. Recall that, by definition, $\t_\bfx +k =\# \{ j \in \Im(\sigma_\bfx) \cap A\}$.

Thus, the total change in inversions is
\begin{align*}
	\inv(\sigma_{\bfx^\circ}) - \inv (\sigma_\bfx) &= (\t_\bfx+k)(g-2k) + \sum_{j \in \Im(\sigma_\bfx) \cap A} \# \{ i \in D \mid i>j-g, i \in \Im(\sigma_\bfx) \} \\
		&= (\t_\bfx+k)(g-2k) + \sum_{\substack{j \notin \Im(\sigma_\bfx) \\ j \in D}} \# \{ i \in D \mid i>j, i \in \Im(\sigma_\bfx) \} \\
		&= (\t_\bfx+k)(g-2k) + \sum_{i \in \Im(\sigma_\bfx) \cap D} \# \{ j \in D \mid i>j, j \notin \Im(\sigma_\bfx) \} \\
		&= (\t_\bfx+k)(g-2k) + (\t_\bfx + k)(k - \t_\bfx) + \sum_{i \in \Im(\sigma_\bfx) \cap D} \# \{ j \in D \ | \ j>i, j \notin \Im(\sigma_\bfx) \} \\
		&= (\t_\bfx+k)(g-2k) + \t_\bfx + k + \sum_{i \in \Im(\sigma_\bfx) \cap D} \# \{ j \in D \ | \ j>i, j \notin \Im(\sigma_\bfx) \} \pmod 2, 
\end{align*}
where the penultimate equality follows from the reduction modulo 2 of the expression
\[ 
\sum_{i \in \Im(\sigma_\bfx) \cap D} \# \{ j \in D \ | \ j>i, j \notin \Im(\sigma_\bfx) \} + \sum_{i \in \Im(\sigma_\bfx) \cap D} \# \{ j \in D \mid i>j, j \notin \Im(\sigma_\bfx) \}  = (k-\t_\bfx) (\t_\bfx+k) 
\]
since
\[ k-\t_\bfx = \# \{i \in \Im(\sigma_\bfx) \cap D\} \quad \textup{ and } \quad  \t_\bfx+k=\#\{ j \in D \mid j \notin \Im(\sigma_\bfx)\}. \]
This establishes \eqref{eqn:sgnDAHH} which, as discussed, completes the proof. 
\end{proof}


\bibliographystyle{amsalpha}
\bibliography{references}

\end{document}